\documentclass[11pt]{amsart}

\usepackage{amscd,amssymb}
\usepackage{newlfont,amsmath,mathrsfs,enumerate,amssymb,array,amscd,xcolor}

\usepackage{amsmath,mathrsfs,enumerate,amssymb,array,amscd,xcolor}
\usepackage{graphicx}
\usepackage{amsmath,amssymb,enumerate}
\usepackage[all]{xy}

\theoremstyle{plain}

\newtheorem{theorem}{Theorem}[section]
\newtheorem{thm}[equation]{Theorem}
\newtheorem{prop}[equation]{Proposition}
\newtheorem{cor}[equation]{Corollary}

\newtheorem{lemma}[equation]{Lemma}

\newtheorem{defn}[equation]{Definition}

\numberwithin{equation}{section}

\newcommand{\Q}{\mathbb Q}
\newcommand{\Z}{\mathbb Z}
\newcommand{\R}{\mathbb R}
\newcommand{\C}{\mathbb C}

\newcommand{\A}{\mathbb A}

\def\Hom{{\mathrm{Hom}}}
\def\Aut{{\mathrm{Aut}}}

\def\SL{{\mathrm{SL}}}

\def\GSp{{\mathrm{GSp}}}

\def\PGSO{\mathrm{PGSO}}
\def\Sp{{\mathrm{Sp}}}
\def\Spin{{\mathrm Spin}}
\def\GU{{\mathrm {GU}}}
\def\SU{{\mathrm {SU}}}
\def\U{{\mathrm U}}
\def\GO{{\mathrm GO}}
\def\GL{{\mathrm{GL}}}

\def\Gal{{\mathrm Gal}}
\def\SO{{\mathrm{SO}}}
\def\Spin{{\mathrm{Spin}}}
\def\Ind{{\mathrm{Ind}}}
\def\Irr{{\mathrm{Irr}}}

\def\O{{\mathrm O}}

\def\Stab{{\mathrm{Stab}}}
\def\St{{\mathrm{St}}}

\def\Res{{\mathrm{Res}}}
\def\End{{\mathrm{End}}}
\def\der{{\mathrm{der}}}
\def\Gal{\mathrm{Gal}}
\def\Ker{{\mathrm{Ker}}}

\def\A{{\mathbb A}}

\def\R{{\mathbb R}}

\def\Z{{\mathbb Z}}

\def\H{{\cal H}} 

\begin{document}
 \title[Twisted Composition Algebra and Triality]{Twisted Composition Algebras   \\and  Arthur Packets for Triality $\Spin(8)$}

 \author{Wee Teck Gan and Gordan Savin}

\address{W.T.G.:   Department of Mathematics, National University of Singapore, 10 Lower Kent Ridge Road
Singapore 119076} \email{matgwt@nus.edu.sg}
\address{G. S.: Department of Mathematics, University of Utah, Salt Lake City, UT} \email{savin@math.utah.edu}
 \subjclass[2000]{11S90, 17A75,  17C40}
\keywords{triality $Spin(8)$, minimal representation, theta correspondence} 

 \maketitle

\section{\bf Introduction}
The purpose of this paper is to construct and analyze certain square-integrable automorphic forms on the quasi-split simply-connected groups $\Spin_8$ of type $D_4$ over a number field $F$. Since the outer automorphism group of $\Spin_8$ is $S_3$, these quasi-split groups are parametrised by \'etale cubic $F$-algebras $E$ and we denote them by $\Spin_8^E$ (to indicate the dependence on $E$). 
We shall specialize to the case when $E$ is a cubic field: this gives the so-called triality $\Spin_8$. 
\vskip 5pt

The square-integrable automorphic forms we construct are associated to a family of discrete Arthur parameters $\psi$ which are quite degenerate. Indeed, apart from the A-parameters of the trivial representation and the minimal representation of $\Spin_8^E$, the A-parameters we consider here  are the most degenerate among the rest. These A-parameters are analogs of the cubic unipotent A-parameters for the exceptional group $G_2$ studied in \cite{GGJ}. In particular, the component groups associated to these A-parameters can be the non-abelian group $S_3$, leading to high multiplicities in the automorphic discrete spectrum, as in \cite{GGJ}. 
\vskip 5pt

For each such A-parameter, we shall give a construction of the local A-packets and establish the global Arthur multiplicity formula. Both the local and global constructions are achieved using exceptional theta correspondences for a family of  dual pairs $H_C \times \Spin_8^E$  in an ambient adjoint group of type $E_6$ (considered with its outer automorphisms); these dual pairs are associated to {\em $E$-twisted composition algebras} of dimension $2$ over $E$. We shall in particular determine the local and global theta lifting completely. The automorphic forms constructed via these theta correspondences, though quite degenerate, can be cuspidal and have some special properties. For example, when one considers their Fourier coefficients along the Heisenberg maximal parabolic subgroup of $\Spin_8^E$ (corresponding to the branch vertex in the Dynkin diagram), one sees that these automorphic forms support only one orbit of generic Fourier coefficients: they are  {\em distinguished} in the sense of Piatetski-Shapiro. 
The relevant Fourier coefficients are parametrised by $E$-twisted composition algebras of $E$-rank $2$, as shown in our earlier work \cite{GS2} on twisted Bhargava cubes.
Such properties allow us to determine their multiplicity in the  automorphic discrete spectrum completely. 

\vskip 5pt

Because the objects mentioned above may be unfamiliar to the typical reader, and the precise results require a substantial amount of notation and language to state,  we will leave the precise formulation of the results to the main body of the paper and content ourselves with the rather cursory overview above. 
\vskip 5pt

We would however like to emphasize the pivotal role played  by the notion of a {\em twisted composition algebra}  (of rank $2$)  and its relation to embeddings of the cubic algebra $E$ into a degree $3$ Jordan algebra (of dimension $9$). This algebraic theory was created and developed by T. Springer (see \cite[Chap. 4]{SV} and \cite[\S 36]{KMRT}). Its relation with $\Spin_8^E$ has been explored in our earlier paper \cite{GS2} and we shall apply the algebraic results of \cite{GS2} to the study of automorphic forms here.  In addition, we also need  arithmetic results about twisted composition algebras and their automorphism groups, such as local and global Tate dualities, weak approximation and Hasse principles. These arithmetic results  are supplied by the papers of Tate \cite{T}, Voskresenskii \cite{V1,V2} and Prasad-Rapinchuk \cite{PR}.
These algebraic and arithmetic results, together with the representation theoretic results from exceptional theta correspondence, combine in rather intricate and  (to these authors)  utterly amazing ways to give the elegant Arthur multiplicity formula.
\vskip 5pt

 Given the length of the paper, it will be pertinent to provide a brief summary as a roadmap for the reader:
 \vskip 5pt
 
 \begin{itemize}
 \item[-] We introduce in \S \ref{S:structure} the group $G_E = \Spin_8^E$ and  its relevant structures, and give a description of its A-parameters in \S \ref{S:arthur}, reviewing Arthur's conjecture in the process.
 \vskip 5pt
 
 \item[-] The theory of twisted composition algebras is introduced in \S \ref{S:TCFJ}. Though this theory is due to Springer, we have needed to supplement it with some observations of our own. In particular, Proposition \ref{P:isom-torsor} plays an important role in the interpretation of our results in the framework of Arthur's conjecture.
 We then recall in \S \ref{SS:relation} our results from \cite{GS2} concerning  nondegenerate twisted  Bhargava cubes and supplement the discussion with results about degenerate cubes.
 \vskip 5pt
 
 \item[-] \S \ref{S:dual} is devoted to the construction of the various dual pairs that will be studied in this paper. It is followed by a detailed description of the Levi subgroup (of type $A_5$) of  the Heisenberg parabolic subgroup of  the adjoint group of type $E_6$ in \S \ref{S:levi}.
 \vskip 5pt
 
 \item[-] The minimal representation of the adjoint group of type $E_6$ is introduced in \S \ref{S:minimal} and its Jacquet module for the Heisenberg parabolic subgroup is determined in \S \ref{S:jac-E6},
 
 \vskip 5pt
 
 \item[-] In the spirit of the tower property of classical theta correspondence, we determine the mini-theta correspondence for the Heisenberg Levi subgroup  in \S \ref{S:mini}.
 This is based on relating it to a classical similitude theta correspondence for unitary groups. It is
 needed for the  study of the theta correspondence in $E_6$ which is carried out in \S \ref{S:main-theta}, after introducing some notations for representations of $G_E$ in \S \ref{S:LQ}.
 In particular, Theorem \ref{T:111} is the main local result of this paper in the nonarchimedean case. We recall in \S \ref{S:arch-theta} the analogous result in the archimedean case, but the proofs of Theorems \ref{T:arch1}, \ref{T:arch1.5} and \ref{T:arch2} there will be deferred to a joint paper with J. Adams and A. Paul.
 \vskip 5pt
 
 \item[-] After this, we move to the global setting, starting with  \S \ref{S:global-theta} which is devoted to the study of global theta correspondence. Here, we first need to understand the space of automorphic forms of the disconnected group 
 $H_C = \Aut_E(C)$, where $C$ is a twisted composition algebra of rank $2$. Not surprisingly, the automorphic multiplicity for $H_C$ can be $1$ or $2$. In \S \ref{S:APTCA}, we relate the relevant A-parameters to the theory of twisted composition algebras. The important  ingredients here are the local-global principles in Lemma \ref{L:loc}, the consequence of local Tate-Nakayama duality in Proposition \ref{P:TN} and the global Poitou-Tate duality in Proposition \ref{P:PT}. After this preparation, we interpret the space of global theta liftings in the framework of Arthur's conjecture in \S \ref{S:AMF}. More precisely,  we construct the local A-packets as well as their  bijection with characters of the local component groups, and then establish the Arthur multiplicity formula (AMF) for the space of global theta liftings in Theorem\ref{T:final1}. Finally, we show in Theorem \ref{T:final2} that the number  provided by the AMF is in fact the true discrete multiplicity of the relevant representation in the automorphic discrete spectrum of $G_E$.   For the interest of the reader, the following are some  examples of numbers which arise as such multiplicities:
 \[  2^n , \quad   \frac{2^n + 2 (-1)^n}{3} , \quad \frac{2^n + (-1)^{n+1}}{3}  \quad \text{for $n \geq 0$.} \]
 In particular, the multiplicities in the automorphic discrete spectrum are unbounded. The main source of these high multiplicities comes from the failure of Hasse principle for twisted composition algebras of $E$-dimension $2$, or alternatively, the failure of Hasse principle for Jordan algebras of dimension $9$.
 \vskip 5pt
 
 \item[-] We have included two appendices. In Appendix A, we consider an analogous theta correspondence for a dual pair $\SL_2(E)/ \mu_2 \times G_E$ in  $E_7$, associated to a rank 4 twisted composition algebra. This theta correspondence can be used to construct another family of Arthur packets for $G_E$, but we do not pursue this here. Instead, we only determine the theta lift of the trivial representation of $\SL_2(E)/\mu_2$ in Corollary \ref{C:GS3}; this result is used in our paper \cite{GS3}.  The long Appendix B is devoted to the study of unramified degenerate pricipal series representations of $G_E$ for the various maximal parabolic subgroups and the various possibilities of $E$.  Our approach is via the Iwahori Hecke algebra, and in each case, we determine the points of reducibility and the module structure at each such point. This allows us to introduce various interesting representations of $G_E$ with nonzero Iwahori-fixed vectors which intervene in the theta correspondence studied in the paper. In particular, we shall refer to the terminology and results of Appendix B in the description of theta lifting, for example in Theorem \ref{T:111}.
 \end{itemize}

 \vskip 5pt
 
We wrap up this introduction by mentioning some recent papers which are devoted to the (automorphic) representation theory of triality $\Spin_8$:
\vskip 5pt
\begin{itemize}
\item the paper \cite{L} of C.H. Luo on determining the unitary dual of the adjoint form of $G_E$ over $p$-adic fields;
\vskip 5pt

\item the papers \cite{Se1} and \cite{Se2} of A. Segal on the structure of degenerate principal series representations (which builds upon and complements our results in Appendix B) and poles of degenerate Eisenstein series of $G_E$;
\vskip 5pt

\item the paper \cite{La} of J.F. Lau on the determination of the residual spectrum of $G_E$.
\end{itemize}
It is interesting to relate  the local and global A-packets we construct here with the results of these other papers.

\vskip 10pt

\noindent{\bf Acknowledgments:} W.T.G. is partially supported by a Singapore government MOE Tier 1 grant R-146-000-320-114. G. Savin is partially supported by 
 a  National Science Foundation grant DMS-1901745. We thank Professor A. Skorobogatov for bringing the work of  Voskresenskii  to our attention and for useful communication.
  
The work for this paper was initiated in 2015 and  we are relieved to finally complete it. It is a pleasure to dedicate this paper to Dick Gross on the occasion of his 70th birthday. 
 It is through our collaboration with him that we have developed a deeper appreciation for the beauty of exceptional groups and exceptional algebraic structures.  
We are grateful to have the opportunity to learn from  him over the years and hope that there will continue to be such opportunities in the years to come.
\vskip 10pt

\section{\bf Structure Theory of $\Spin^E_8$}  \label{S:structure}

\subsection{\bf \'{E}tale cubic algebras.}
Let $F$ be a field of characteristic $0$ and with absolute Galois group  $\Gal(\overline{F}/F)$. 
An \'{e}tale cubic algebra is an $F$-algebra $E$ such that $E \otimes_F \overline{F} \cong \overline{F}^3$. More concretely, an \'{e}tale cubic $F$-algebra is of the form:
\[ E = \begin{cases}
\text{$F \times F \times F$;} \\
\text{$F \times K$, where $K$ is a quadratic field extension of $F$;} \\
\text{ a cubic field.} \end{cases} \]
Since the split algebra $F \times F \times F$ has 
automorphism group $S_3$ (the symmetric group on 3 letters),
the isomorphism classes of \'{e}tale cubic algebras $E$ over $F$ are naturally classified by the set of conjugacy classes of  homomorphisms 
\[  \rho_E : \mathrm{Gal}(\overline{F}/F) \longrightarrow S_3. \]
\vskip 10pt

By composing the homomorphism $\rho_E$ with the sign character of $S_3$, we obtain a quadratic character (possibly trivial) of $\Gal(\overline{F}/F)$ which corresponds to an \'{e}tale quadratic algebra $K_E$. We call $K_E$ the {\em discriminant algebra} of $E$. To be concrete, 
\[  K_E = \begin{cases}
F \times F, \text{  if $E = F^3$ or a cyclic cubic field;} \\
K, \text{  if $E = F \times K$;} \\
\text{the unique quadratic subfield in the Galois closure of $E$ otherwise.} 
\end{cases} \] 
We shall let $\chi_{K_E}$ denote the quadratic idele class character associated to $K_E$.
\vskip 10pt

 The \'etale cubic $F$-algebra $E$ possesses a natural cubic form $N_{E/F}: E \rightarrow F$ known as its norm form: for $a \in E$, $N_{E/F}(a)$ is the determinant of the   multiplication-by-$a$ map on the $F$-vector space $E$. Then there is a natural quadratic map 
 \begin{equation} \label{E:sharp0}
( -)^{\#}: E \longrightarrow E 
 \end{equation}
characterized by $a \cdot a^{\#} = N_{E/F}(a)$ for all $a \in E$. 
\vskip 10pt




\subsection{\bf Twisted form of $S_3$.}
Fix an \'{e}tale cubic $F$-algebra $E$. 
Then, via the associated homomorphism $\rho_E$, $\Gal(\overline{F}/F)$ acts on $S_3$ (by inner automorphisms) and thus defines a twisted form
$S_E$ of the finite constant group scheme $S_3$.  For any commutative $F$-algebra $A$, we have
\[  S_E(A) = \Aut_A(E \otimes_F A). \]
\vskip 10pt

 \vskip 5pt

\subsection{\bf $D_4$ root system.} 

Let $\Phi$ be a root system of type $D_4$ with a set of simple roots  $\Delta=\{\alpha_0, \alpha_1,\alpha_2,\alpha_3\}$. The highest root is $\beta_0 = \alpha_1+ \alpha_2 +\alpha_3 + 2 \alpha_0$.  The 
corresponding Dynkin diagram is 

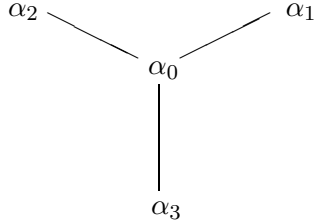
\begin{picture}(100,130)(-180,0) 

\put(19,21){$\alpha_3$} 

\put(22,30){\line(0,1){40}}

\put(18,74){$\alpha_0$}

\put(30,80){\line(2,1){34}}
\put(-20,97){\line(2,-1){34}}

\put(-35,96){$\alpha_2$}
\put(70,96){$\alpha_1$}

\end{picture}

\noindent Hence  the group $\Aut(\Delta)$ of diagram automorphisms  is identified with $S_3$ (the group of permutations of $\{1,2,3\}$).   
\vskip 5pt

\subsection{\bf Quasi-split groups of type $D_4$.}
Let $G = \Spin_8$ be a split, simply connected Chevalley group of type $D_4$. We fix a maximal torus $T$ contained in a Borel subgroup $B$ defined over $F$. The group $G$ is then generated by root groups $U_{\alpha}\cong \mathbb G_a$, where $\alpha \in\Phi$. Steinberg showed that one can pick the isomorphisms 
$x_{\alpha}: \mathbb G_a \rightarrow U_{\alpha}$ such that 
\[ 
[x_{\alpha}(u), x_{\alpha'}(u')]= x_{\alpha+\alpha'}(\pm uu')
\] 
whenever $\alpha+\alpha'$ is a root. Fixing such a system of isomorphisms for $\alpha \in \Delta$ is fixing an \'epinglage (or pinning) for $G$.  
By the discussion on page 40 in  \cite{FK}, commutators signs can be specified by choosing an 
orientation of the Dynkin diagram. There is a short exact sequence:
\[  \begin{CD}
1 @>>>  G_{ad} = \mathrm{Inn}(G)  @>>>  \Aut(G)  @>>>  \Aut(\Delta) = S_3 @>>> 1.  \end{CD} \]
As one can pick an orientation of the Dynkin diagram which is invariant under  $\Aut(\Delta)$,  one has a splitting 
$S_3 = \Aut(\Delta) \longrightarrow \Aut(G)$,
 where the action of $S_3$ permutes the root subgroups $U_{\alpha}$ and the isomorphisms $x_{\alpha}$. 
\vskip 5pt

Since $S_3$ is also the automorphism group of the split \'etale cubic $F$-algebra $F^3$, we see that
every cubic \'etale algebra $E$ defines a simply-connected quasi-split form $G_E$ of $G$, whose outer automorphism group is the finite group scheme $S_E$. It comes equipped with
a pair $B_E \supset T_E$ consisting of a Borel subgroup $B_E$  containing a maximal torus $T_E$, both defined over $F$. 
Moreover, we inherit a Chevalley-Steinberg system  of \'{e}pinglage relative to this pair and a splitting of the outer automorphism group 
\[  S_E \hookrightarrow \Aut(G_E). \]
 If $E$ is a cubic field, then $\Gal(\overline{F}/F)$ permutes the roots $\alpha_1$, $\alpha_2$ and $\alpha_3$ transitively. If $E = F \times K$ with $K$ a quadratic field, then without loss of generality, 
we assume that $\alpha_1$ is fixed, whereas $\alpha_2$ and $\alpha_3$ are exchanged by the Galois action. If $E$ is the split algebra, the Galois action on $\Phi$ is trivial.

 \vskip 5pt
 
 \subsection{\bf Center} \label{SS:center}
The center of the split group $G$ is 
\[  Z = \{  (z_1, z_2, z_3) \in \mu_2 \times \mu_2 \times \mu_2:  z_1z_2z_3 =1 \}. \]
By Galois descent, we deduce that the center of $G_E$ is
\[  Z_E = \mathrm{Res}^1_{E/F}(\mu_2)  = \mathrm {Ker} ( N_{E/F}: \mathrm{Res}_{E/F}(\mu_2) \longrightarrow \mu_2). \]
In particular, from the short exact sequence
\[ \begin{CD} 
1 @>>> Z_E @>>> G_E @>p>> G_E^{ad} @>>> 1,  \end{CD} \]
we deduce that
\begin{equation} \label{E:Gad}
  G_E^{\mathrm{ad}}(F) / p(G_E(F)) = {\Ker} ( H^1(F, Z_E) \longrightarrow H^1(F, G_E)). \end{equation}
The finite group scheme $Z_E$ will play an important role in this paper and we will see several other incarnations of it later on.
\vskip 5pt

\subsection{\bf L-group.}
The Langlands dual group of $G_E$ is the adjoint complex Lie group 
\[ G_E^{\vee} = \PGSO_8(\C).\]
It inherits a pinning from that of $G_E$.
The L-group $^LG_E$ is the semidirect product of $\PGSO_8(\C)$ with $\Gal(\overline{F}/F)$,
where the action of $\Gal(\overline{F}/F)$ on $\PGSO_8(\C)$ is via the homomorphism $\rho_E$ as pinned automorphisms.
 Thus there is a natural map
 \[  ^LG_E \longrightarrow \PGSO_8(\C) \rtimes S_3, \]
whose restriction to $\Gal(\overline{F}/F)$ is $\rho_E$.


 \vskip 10pt
 
\subsection{\bf $G_2$ root system.}
The subgroup of $G_E$ fixed pointwise by $S_E$  is isomorphic to the split exceptional group of type $G_2$. 
Observe that $B_0=G_2 \cap B_E$ is a Borel subgroup of $G_2$ and $T_0 = T_E \cap G_2$ is a maximal split torus of $G_2$.   Via the adjoint action of $T_0$ on $G_E$, we obtain the root system $\Phi_{G_2}$ of $G_2$, so that
\[  \Phi_{G_2} = \Phi|_{T_0}. \]
 We denote the short simple root of this $G_2$ root system by $\alpha$ and the long simple root by $\beta$, so that
\[  \beta = \alpha_0|_{T_0} \quad \text{and} \quad \alpha = \alpha_1|_{T_0} = \alpha_2|_{T_0} = \alpha_3|_{T_0}. \]
Thus, the short root spaces have dimension $3$, whereas the long root spaces have dimension $1$. 
For each root $\gamma \in \Phi_{G_2}$, the associated root subgroup $U_{\gamma}$ is defined over $F$ and the Chevalley-Steinberg system of \'{e}pinglage gives isomorphisms:
\[  U_{\gamma} \cong \begin{cases}
\Res_{E/F}\mathbb{G}_a, \text{   if $\gamma$ is short;} \\
\mathbb{G}_a, \text{  if $\gamma$ is long.} \end{cases} \]
When $E$ is a cubic field, $T_0$ is in fact the maximal $F$-split torus of $G_E$ and $\Phi_{G_2}$ is the relative root system of $G_E$.  
\vskip 10pt
 
For each $\gamma \in \Psi_{G_2}$, we shall also let $N_{\gamma}$ denote the root subgroup of $G_2$ 
corresponding to $\gamma$. In particular, 
\[ N_{\gamma} = U_{\gamma} \cap G_2. \]
Because  the highest root $\beta_0$ of the $D_4$-root system restricts to that of the $G_2$-root system, we shall let $\beta_0$ denote the highest root of the $G_2$-root system also.  
\vskip 15pt

\subsection{\bf Two parabolic subgroups.}
The $G_2$ root system gives rise to 2 parabolic subgroups of $G_E$. One of these is a maximal parabolic $P_E = M_E N_E$ known as the Heisenberg parabolic. Its unipotent radical $N_E$ is a Heisenberg group and its Levi subgroup $M_E$ is spanned by the 3 satellite vertices in the Dynkin diagram.
The other parabolic $Q_E = L_E U_E$ is a not-necessarily-maximal parabolic (it is not maximal over $\overline{F}$); its Levi subgroup $L_E$ is spanned by the branch vertex $\alpha_0$ and its unipotent radical $U_E$ is a 3-step unipotent group. 
We shall need to examine the structure of these 2 parabolic subgroups more carefully.
\vskip 15pt

\subsection{\bf The Heisenberg parabolic $P_E$.}  \label{SS:Heis}
Let us begin with the Heisenberg parabolic $P =M N$ of $G$. 
  Its unipotent radical $N$ is a 2-step nilpotent group with 
the center $Z=[N,N]=U_{\beta_0}$. As we explained in \cite{GS2}, The Levi factor $M$ can be identified with 
\[ 
\GL_2(F^3)^{\det}=\{ g=(g_1,g_2, g_3) ~|~ g_i\in \GL_2(F), \, \det(g_1)=\det(g_2)=\det(g_3) \}. 
\] 
We may also identify  $V=N/Z$ with $F^2 \otimes F^2 \otimes F^2$, so that the action of $M$ on $V$ corresponds to the standard action of $\GL_2(F^3)^{\det}$ twisted by $\det(g)^{-1}:=\det(g_i)^{-1}$ (for any $i$). 
Moreover, we can assume that the torus $T\subset M$ corresponds to the subgroup of 
$\GL_2(F^3)^{\det}$ consisting of $g=(g_1,g_2,g_3)$ where $g_i$ are diagonal matrices, and the standard basis elements of $F^2 \otimes F^2 \otimes F^2$ correspond to the basis of $N/Z$ given by the fixed pinning. 
\vskip 5pt

Thus, an element  $v\in V$ can be conveniently represented by a cube
 
 
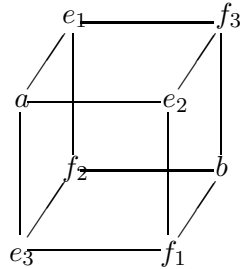
\begin{picture}(100,130)(-130,0) 

\put(18,16){$e_3$} 
\put(25,20){\line(1,0){50}}
\put(75,16){$f_1$}
\put(22,26){\line(0,1){46}}
\put(78,26){\line(0,1){46}}

\put(38,48){$f_2$} 
\put(45,50){\line(1,0){50}}
\put(96,48){$b$}
\put(42,56){\line(0,1){46}}
\put(98,56){\line(0,1){46}}

\put(25,25){\line(2,3){14}}
\put(82,25){\line(2,3){14}}

\put(20,74){$a$}
\put(25,76){\line(1,0){50}}
\put(76,74){$e_2$}

\put(38,106){$e_1$}
\put(45,106){\line(1,0){50}}
\put(96,106){$f_3$}

\put(25,81){\line(2,3){14}}
\put(82,81){\line(2,3){14}}

\end{picture}

\noindent 
where $a, \ldots  , b\in F$, and the vertices correspond to the standard basis in $F^2 \otimes F^2 \otimes F^2$.  We shall assume that the vertex marked by $a$ corresponds to $\alpha_0$, 
and that the vertex marked by $b$ corresponds to $\beta_0-\alpha_0$.  
The group $\Aut(\Delta)$ acts as the group of symmetries of the cube fixing these two vertices. We shall often write the cube as 
a quadruple 
\[ 
(a,e,f,b)
\] 
where $e=(e_1,e_2,e_3)$ and $f=(f_1,f_2,f_3)\in F^3$.

\vskip 10pt

 The quasi-split group $G_E$ contains a maximal parabolic $P_E=M_E N_E$ which is a form of $P$. The structure of $P_E$ can be determined  by  Galois descent. The highest root $\beta_0$ is invariant under 
$\Aut(\Delta)$, hence the center $Z_E$ is equal to the center $Z$ of $P$. The Levi factor $M_E$ can be identified with 
\[  \GL_2(E)^{\det} := \{ g\in \GL_2(E) :   \det(g)\in F^{\times} \},  \]
and 
\[ V_E :=   N_E/Z_E  \cong  U_{\beta} \times U_{\beta + \alpha} \times U_{\beta+2\alpha} \times U_{\beta+3\alpha}  
 \cong F \times E \times E \times F \]
can be identified  with the space of ``twisted cubes'' i.e. quadruples $(a,e,f,b)$ where $a,b\in F$ and $e,f\in E$. The cube 
\[ 
v_E=(1,0,0,-1) 
\] 
is called the distinguished cube. Its stabilizer in $M_E$ can be easily computed using Galois descent: 
\[ 
 \Stab_{{M}_E}(v_{E} ) \cong 
  E^1 \rtimes (\Z/2\Z)\] 
 where $E^1$ denotes the group of norm one elements in $E^{\times}$. In this isomorphism, $\alpha\in E^1$ corresponds to  
\[    \left( \begin{array}{cc}
\alpha & \\
  & \alpha^{-1} \end{array}  \right) \in \GL_2(E)^{\det} \] 
 and the nontrivial element in $\Z/2\Z$ corresponds to 
 \[ 
 w= \left( \begin{array}{cc}
0 & 1 \\
1  & 0 \end{array}  \right).
 \] 
Note that $P_E \cap G_2$ is the 
 Heisenberg maximal parabolic $P_0 = M_0N_0$ of $G_2$, with 
 \[ M_0  = G_2 \cap M_E \cong GL_2 \quad \text{ and } \quad N_0 = G_2 \cap N_E.\]
\vskip 10pt



 \vskip 15pt

\subsection{\bf  The 3-step parabolic $Q_E$.}  \label{SS:QE}
Now we come to the parabolic $Q_E$. The unipotent radical $U_E$ has a filtration
\[  \{1\} \subset U_E^{(1)} \subset U_E^{(2)} \subset U_E \]
such that 
 \[  U_E^{(1)} = U_{\beta_0} \times U_{\beta_0 - \beta} \]
is the center of $U_E$.   Further, 
\[  U_E^{(2)} = [U_E,U_E] =
 U_{\beta_0} \times U_{\beta_0 - \beta} \times U_{2\alpha+\beta} \]
is the commutator subgroup of $U_E$ and is abelian. 
In particular, $U_E$ is a 3-step unipotent group; hence we call $Q_E$ the 3-step parabolic.
Note that $Q_0= G_2 \cap Q_E = L_0\cdot U_0$ is the 3-step maximal parabolic of $G_2$, with 
\[ L_0 = G_2 \cap L_E \cong GL_2 \text{ and} \quad U_0 = G_2 \cap U_E. \] 
 One has an isomorphism
\[ L_E  \cong (\GL_2 \times \Res_{E/F} \mathbb{G}_m)^{\det} = \{ (g, e):  \det(g) \cdot N_{E/F}(e)  =1 \}.\]


\subsection{Nilpotent orbits} 
 Assume that $E$ is a field. In this subsection, we shall describe the nilpotent orbits of ${\mathrm Lie}(G_E)(F) = \mathfrak{g}_E(F)$ and the centralizers of the nilpotent elements.
 \vskip 5pt
 
 Let $t_E(F) = {\mathrm Lie}(T_E)(F)$ be the maximal  toral subalgebra in $\mathfrak g_E(F)$. 
  Let $e$ be a nilpotent element in $\mathfrak g_E(F)$ belonging to a nilpotent $G_E(\bar F)$-orbit $\Omega$. 
 By the Jacobson-Morozov theorem, the element $e$ is a member of an $sl_2$-triple $(f,h,e)$ defined over $F$, so that
  $h$ is a semi-simple element such that $[h,e]=2e$. 
We can assume that $h \in  \mathfrak t_E(\bar F)$ and lies in the positive chamber.  Then the values of the simple roots on $h$ are nonnegative integers and 
give a marking of the Dynkin diagram of type $D_4$; this marking parameterizes the orbit $\Omega$.  Note that the marking of the Dynkin diagram must necessarily be invariant under $\Aut(\Delta)$.  In fact, this condition is necessary and sufficient  (see \cite{Dj})  for a 
nilpotent orbit in $\mathfrak g_E(\bar F)$ to be defined over $F$ and to have an $F$-rational point. 
\vskip 5pt

The semisimple element  $h$ gives a $\mathbb Z$-grading $\mathfrak g_E=\bigoplus \mathfrak g_{E,i}$, with $e\in \mathfrak g_{E,2}$. 
Let $P_e=M_eN_e$ be the parabolic group such that the Lie algebra of $M_e$ is $\mathfrak g_{E,0}$.  By a result of Kostant, 
the centralizer $Z_{M_e}(e)$ of $e$ in $M_e$ is the reductive part of $Z_{G_E}(e)$. Moreover, by Galois cohomology,
the nilpotent $G_E(F)$-orbits contained in  $\Omega(F)$  are parametrized by
\[  {\Ker}(H^1(F, Z_{M_e}(e))\rightarrow H^1(F,G_E)). \]

We now list all nilpotent orbits $\Omega$ defined over $F$ and the corresponding $Z_{M_e}(e)$ (the reductive part of the centralizer $Z_{G_E}(e)$). 
First, we have  three Richardson orbits corresponding to the following diagrams:

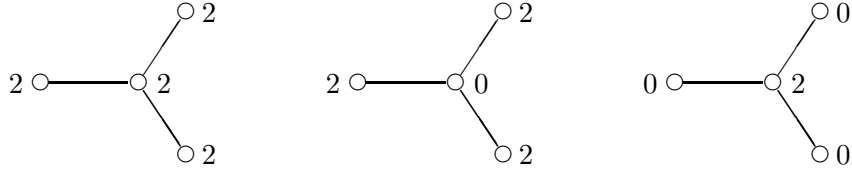
\begin{picture}(400,110)(-10,25)

\put(81,73){\line(2,-3){14}}

\put(30,72){2}

\put(42,76){\circle{6}}
\put(45,76){\line(1,0){30}}
\put(79,76){\circle{6}}

\put(86,72){2}

\put(97,103){\circle{6}}
\put(97,49){\circle{6}}

\put(81,73){\line(2,-3){14}}
\put(81,79){\line(2,3){14}}

\put(103,99){2}
\put(103,45){2}

\put(201,73){\line(2,-3){14}}

\put(150,72){2}

\put(162,76){\circle{6}}
\put(165,76){\line(1,0){30}}
\put(199,76){\circle{6}}

\put(206,72){0}

\put(217,103){\circle{6}}
\put(217,49){\circle{6}}

\put(201,73){\line(2,-3){14}}
\put(201,79){\line(2,3){14}}

\put(223,99){2}
\put(223,45){2}

\put(321,73){\line(2,-3){14}}

\put(270,72){0}

\put(282,76){\circle{6}}
\put(285,76){\line(1,0){30}}
\put(319,76){\circle{6}}

\put(326,72){2}

\put(337,103){\circle{6}}
\put(337,49){\circle{6}}

\put(321,73){\line(2,-3){14}}
\put(321,79){\line(2,3){14}}

\put(343,99){0}
\put(343,45){0}

\end{picture}

\noindent The first two diagrams correspond to the regular and the subregular orbit respectively, and the reductive part of the centralizer is 
the center of $G_E$ in each case. The third case is the most interesting. In this case $Z_{M_e}(e)$ is generally disconnected and its identity component is a 2-dimensional torus. In fact, $Z_{M_e}(e) =\Aut_E(C)$ where $C$ is an $E$-twisted composition algebra of $E$-dimension 2 (see later for this notion). 
We also have the three orbits corresponding to the following diagrams: 

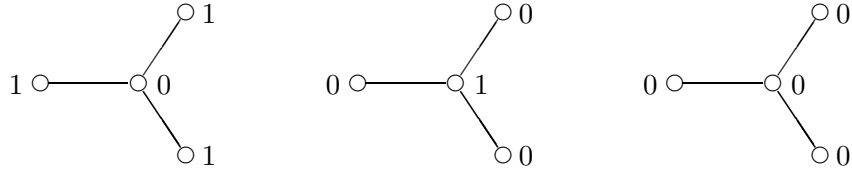
\begin{picture}(400,110)(-10,25)

\put(81,73){\line(2,-3){14}}

\put(30,72){1}

\put(42,76){\circle{6}}
\put(45,76){\line(1,0){30}}
\put(79,76){\circle{6}}

\put(86,72){0}

\put(97,103){\circle{6}}
\put(97,49){\circle{6}}

\put(81,73){\line(2,-3){14}}
\put(81,79){\line(2,3){14}}

\put(103,99){1}
\put(103,45){1}

\put(201,73){\line(2,-3){14}}

\put(150,72){0}

\put(162,76){\circle{6}}
\put(165,76){\line(1,0){30}}
\put(199,76){\circle{6}}

\put(206,72){1}

\put(217,103){\circle{6}}
\put(217,49){\circle{6}}

\put(201,73){\line(2,-3){14}}
\put(201,79){\line(2,3){14}}

\put(223,99){0}
\put(223,45){0}

\put(321,73){\line(2,-3){14}}

\put(270,72){0}

\put(282,76){\circle{6}}
\put(285,76){\line(1,0){30}}
\put(319,76){\circle{6}}

\put(326,72){0}

\put(337,103){\circle{6}}
\put(337,49){\circle{6}}

\put(321,73){\line(2,-3){14}}
\put(321,79){\line(2,3){14}}

\put(343,99){0}
\put(343,45){0}

\end{picture}

The first two orbits correspond to a short root  $\varphi: \mathfrak{sl}_2(E) \rightarrow \mathfrak g_E(F)$ embedding and a long root embedding $\varphi: \mathfrak{sl}_2(F) \rightarrow \mathfrak g_E(F)$ respectively. The reductive part of the centralizer is isomorphic to $\SL_2(F) \times Z$ and $\SL_2(E)$, 
respectively. The last diagram corresponds to the trivial orbit. 

Summarizing our findings, if $F$ is a local field, then $\Omega(F)$ consists of a single $G_E(F)$-orbit, except in one case when $G_E(F)$-orbits in $\Omega(F)$ are parameterized by $E$-isomorphism classes of $E$-twisted composition algebras $C$ of $E$-dimension 2.

\vskip 10pt

\subsection{\bf Unipotent orbits of ${^L}G_E$.}  \label{SS:uni}
We also need a description of the conjugacy classes of maps
\[  \SL_2(\C)  \longrightarrow {^L}G_E  \longrightarrow G_E^{\vee}  \rtimes S_3 \]
which are invariant under the $S_3$-action.  
These correspond to unipotent conjugacy classes of $G_E^{\vee} = \PGSO_8(\C)$ which are stable under the action of $S_3$.   
As in the previous subsection, these unipotent conjugacy classes  in turn correspond to markings of the $D_4$ Dynkin diagram which are invariant under the $S_3$-action.
In particular, such markings have been enumerated in the previous subsection.
\vskip 5pt

\vskip 15pt
\section{\bf Arthur Parameters of $\Spin_8^E$}  \label{S:arthur}
In this section, we shall enumerate the (elliptic) Arthur parameters for $G_E$ and single out a particularly interesting family of Arthur parameters.
Thus, in this section, we assume that $F$ is a number field and $E$ is a cubic field extension of $F$.

\vskip 5pt

\subsection{\bf A-parameters.}
An A-parameter for $G_E$ is a $G_E^{\vee}$-conjugacy class of homomorphism 
\[  \psi:  L_F \times \SL_2(\C) \longrightarrow   {^L}G_E = G_E^{\vee} \rtimes_{\rho_E}  {\Gal}(\overline{F}/F) \longrightarrow G_E^{\vee} \rtimes S_3,   \]
such that $pr_{S_3}  \circ \psi|_{L_F}  = \rho_E$, where $pr_{S_3}$ stands for the projection 
\[  pr_{S_3} :  G_E^{\vee}  \rtimes S_3   \longrightarrow S_3. \]
In particular, $\psi|_{\SL_2(\C)}$ is of the type  considered in Section \ref{SS:uni}.
 \vskip 5pt
 
 For each place $v$ of $F$, we have a conjugacy class of embeddings $L_{F_v}  \hookrightarrow L_F$, from which we obtain by restriction a local A-parameter
 \[  \psi_v:  L_{F_v} \times SL_2(\C) \longrightarrow G_E^{\vee}  \rtimes S_3. \]
 
 \vskip 5pt
\subsection{\bf Component groups.}
 For an A-parameter $\psi$, we set
 \[  S_{\psi}  =  \pi_0\left(  Z_{G_E^{\vee}}(\mathrm{Im}(\psi) )  \right). \]
 This is the global component group of $\psi$, and we say that $\psi$ is elliptic if $S_{\psi}$ is finite.
 Likewise, we have the local component group $S_{\psi_v}$.  There is a natural diagonal map
 \[  \Delta: S_{\psi}  \longrightarrow S_{\psi,\A} :=  \prod_v  S_{\psi_v}. \]
 Hence there is an induced pullback map
 \[  \Delta^*:   {\Irr} S_{\psi, \A}  \longrightarrow R(S_{\psi}),  \]
where $R(S_{\psi})$ denotes the (Grothendieck) representation ring of $S_{\psi}$.
\vskip 5pt

\subsection{\bf Arthur's conjecture.} \label{SS:arthur}
We briefly recall Arthur's conjecture.  Associated to each elliptic  A-parameter $\psi$, one expects to have the following:
\vskip 5pt

\begin{itemize}
\item for each place $v$ of $F$, a finite packet 
\[  \Pi_{\psi_v}  = \{  \pi_{\eta_v}:  \eta_v  \in {\Irr}  S_{\psi_v} \} \]
of unitary representations of finite length (possibly zero), indexed by the irreducible characters of the local component group $S_{\psi_v}$.

\item set 
\[  \Pi_{\psi} = \{  \pi_{\eta} = \bigotimes_v \pi_{\eta_v}:  \eta = \otimes_v \eta_v  \in \mathrm{Irr} S_{\psi,\A} \}, \]
and
\[  m_{\eta}  = \langle \Delta^* (\eta),  \epsilon_{\psi}  \rangle_{S_{\psi}} \]   
where $\epsilon_{\psi}$ is a certain quadratic character of $S_{\psi}$ (whose definition we won't recall here).
Then the automorphic discrete spectrum $L^2_{disc} $ of $G_E$ contains a submodule isomorphic to
\[  L^2_{\psi}  :=   \bigoplus_{\eta \in {\Irr} S_{\psi,\A} }  m_{\eta}  \cdot \pi_{\eta}. \]
\end{itemize}
Moreover, we have:
\[  L^2_{disc}  =  \bigoplus_{\psi}  L^2_{\psi} \]
where the sum runs over equivalence classes of elliptic A-parameters $\psi$. 
\vskip 5pt

\subsection{\bf Enumeration.}
In view of the above discussion, there are 6 families of A-parameters for $G_E$, according to the type of $\psi|_{\SL_2(\C)}$. We list them below, together with the component group $S_{\psi}$:
\vskip 5pt

\begin{itemize}
\item[(i)] $\psi|_{\SL_2(\C)}$ is the regular orbit:  $S_{\psi}$ is trivial and the resulting A-packet consists of the trivial representation (both locally and globally).

\item[(ii)]  $\psi|_{\SL_2(\C)}$ is the subregular orbit:  $S_{\psi}$ is trivial and the resulting local A-packet consists of the minimal representation.

\item[(iii)]  $\psi|_{\SL_2(\C)}$ is given by:
\[  \psi:  \SL_2(\C) \longrightarrow \SO_3(\C) \subset \SL_3(\C) \subset G_2(\C)  \subset G_E^{\vee}.  \]
This is the case of interest in this paper and we shall give a more detailed discussion in the next subsection.


\item[(iv)] $\psi|_{\SL_2(\C)}$ is given by
\[  \psi:  \SL_2(\C)  \hookrightarrow \SL_2(\C)  \times \SL_2(\C) \times \SL_2(\C)  \longrightarrow M_E^{\vee}  \subset G_E^{\vee}, \]
where the first map is the diagonal embedding.

\item[(v)]  $\psi|_{\SL_2(\C)}$ is a root $\SL_2$: we shall discuss this case briefly as well.

\item[(vi)]  $\psi|_{\SL_2(\C)}$ is the trivial map: this is the tempered case.  
\end{itemize}
\vskip 5pt

\subsection{\bf The case of interest.}   \label{SS:interest}
Now we examine the case of interest  (case (iii) above) in greater detail. The centralizer of $\psi(\SL_2(\C))$ in $G_E^{\vee}$  is isomorphic to the subgroup
\[ S \rtimes S_2  = \{ (a,b,c) \in (\C^{\times})^3: abc =1 \}  \rtimes S_2, \]
where the nontrivial element of $S_2$ acts on $S$ by inverting.  
Moreover, the group $S_3 = \Aut(\Delta)$ commutes with $\psi(\SL_2(\C))$ and $S_2$ and acts on $S$ by permuting the coordinates.  
Thus we have an embedding
\[  S \rtimes  (S_2 \times S_3)  \longrightarrow G_E^{\vee} \rtimes S_3.  \]
\vskip 15pt

To give an A-parameter $\psi$ of this type is thus equivalent to giving a map
\[  \psi:  L_F \longrightarrow S \rtimes (S_2 \times S_3). \]
 The composition of $\psi$ with the projection to $S_2 \times S_3$ gives a homomorphism $L_F \rightarrow S_2 \times S_3$ and thus determine an \'etale quadratic algebra $K$ and the fixed  \'etale cubic algebra $E$.  We shall say that $\psi$ is of type $(E, K)$. 
 \vskip 5pt

 To give an A-parameter of type $(E, K)$ amounts to giving a L-homomorphism
 \[ L_F  \longrightarrow  S \rtimes_{\rho_E \times \rho_{K} } W_F.  \]
 Now the group $S \rtimes_{\rho_E \times \rho_{K} } W_F$ is the L-group of a torus  
 \[  \tilde{T}_{E, K} =
  \{ x\in  (E \otimes_F K)^{\times}:  N_{E \otimes K / E}(x)  \in F^{\times}\} /  K^{\times}.   \]
  As shown in \cite{GS2},  this torus is the identity component of the $E$-automorphism group of any rank 2 $E$-twisted composition algebra $C$ with quadratic invariant $K_C $ satisfying 
  \[  [K_E] \cdot [K] \cdot [K_C] =1 \in F^{\times}/F^{\times 2}. \]
  By an exceptional Hilbert Theorem 90 \cite[Theorem 11.1]{GS2}, one has
  \[  \tilde{T}_{E,K}  \cong T_{E, K_C}  :=  \{ x\in  (E \otimes_F K_C)^{\times}:  N_{E \otimes K_C / E}(x)  = 1 = N_{E \otimes K_C/K_C}(x) \}. \]   
  Thus to give an A-parameter of type $(E, K)$ is to give a L-parameter for the torus $\tilde{T}_{E, K}$, taken up to conjugation by $S \rtimes S_2$.  In other words, it is to give an automorphic character of $\tilde{T}_{E,K}$ up to inverse.
\vskip 5pt

This suggests that the A-packet $\Pi_{\psi_v}$ or $\Pi_{\psi}$ can be constructed as a ``lifting" from $T_{E,K}$ to $G_E$.  The goal of this paper is to carry out such a construction, using the fact that there is a dual pair
\[  H_C \times G_E  \subset \Aut(E_6^J)   \]
where $H_C$ is the automorphism group of a rank $2$ $E$-twisted composition algebra (whose identity component is $\tilde{T}_{E,K}$) and  $E_6^J$ is an adjoint group of type $E_6$ (depending on a Freudenthal-Jordan algebra $J$ with $K_J = K$; see later).
\vskip 10pt

\subsection{\bf An example.}  \label{SS:example}
 The simplest A-parameter of type $(E, K)$ is determined by the natural map
 \[  \begin{CD}
 L_F @>\rho_{K} \times \rho_E>>  S_2 \times S_3  @>>> S \rtimes (S_2 \times S_3) @>>> G_E^{\vee}  \rtimes S_3. \end{CD} \]
 We denote this special A-parameter by $\psi_{E,K}$. 
Its global component group is thus
\[  S_{\psi_{E, K}}  = \begin{cases}
 \mu_3  \rtimes S_2  \cong S_3 \text{  if $K = F \times F$;} \\
 S_2 \text{  if $K$ is a field.} \end{cases} \]
 
 \vskip 5pt
 The local component groups $S_{\psi_{E, K, v}}$ are a bit more involved to describe, as they depend on the type of $E_v$ and $K_{v}$.  We list them  in the following table.
 \vskip 5pt
 
 \begin{center}
 \begin{tabular}{|c|c|c|}
\hline 
$E_v$ &  $K_v$ &  $S_{\psi_{E_v, K_{v}}}$     \\
& & \\
\hline 
field & field & $S_2$  \\
& & \\
\hline 
field & split & $S_3$  \\
& & \\
\hline 
$F_v \times K_{E,v}$  & $K_v$ splits  or $K_v  = K_{E,v}$ & $S_2$  \\
& & \\
\hline 
$F_v \times K_{E,v}$ & $K \ne K_{E,v}$ is a field & $\mu_2 \times S_2$ \\
& & \\
\hline 
$F_v \times F_v \times F_v$ & field  & $(\mu_2 \times \mu_2)  \times S_2$  \\
& & \\
\hline 
$F_v \times F_v \times F_v$ & split & $S_2$  \\
& & \\
\hline
\end{tabular}
\end{center} 
\vskip 5pt

Let's see what Arthur's conjecture implies for this particular A-parameter, specialising to the case when $K = F \times F$ is split:
\vskip 5pt

\begin{itemize}
\item if $E_v$ is a field, then
\[  \Pi_{\psi_{E, K,v}} = \{  \pi_{1,v}, \pi_{r,v}, \pi_{\epsilon,v} \} \]
\vskip 5pt

\item if $E_v = F_v \times K_{E,v}$ or $F_v^3$, then
\[  \Pi_{\psi_{E, K,v}} = \{  \pi_{1,v}, \pi_{\epsilon,v} \}. \] 
\end{itemize}
For  appropriate disjoint finite subsets $\Sigma_r$ and $\Sigma_{\epsilon}$ of the set of places of $F$, we thus have the representation
\[  \pi_{\Sigma_r, \Sigma_{\epsilon}}  =\left(  \bigotimes_{v \in \Sigma_r}  \pi_{r,v} \right) \otimes \left(  \bigotimes_{v \in \Sigma_{\epsilon} } \pi_{\epsilon,v} \right) \otimes \left( 
\bigotimes_{v \notin \Sigma_r \cup \Sigma_{\epsilon}}  \pi_1 \right) \]
in the global A-packet $\Pi_{\psi_{E, K}}$.  The multiplicity attached to this representation is the multiplicity of the trivial representation of $S_3$ in 
$(r^{\otimes |\Sigma_r|} )\otimes (\epsilon^{\otimes |\Sigma_{\epsilon}|})$.  A short computation using the character table of $S_3$ shows that this multiplicity is equal to
\[ \begin{cases}
 \frac{1}{6} \cdot (2^{|\Sigma_r|}  +  2 \cdot (-1)^{|\Sigma_r|}), \text{  if $\Sigma_r$ is nonempty;}\\
 \frac{1}{2} \cdot (1 + (-1)^{|\Sigma_{\epsilon}|}), \text{  if $\Sigma_r$ is empty.} \end{cases} 
  \]
 We shall see later how to construct this many automorphic realisations of $  \pi_{\Sigma_r, \Sigma_{\epsilon}}$, using exceptional theta correspondence.
\vskip 5pt

\subsection{\bf Root $\SL_2$}  \label{SS:rootSL2}
We consider briefly the case when $\psi|_{\SL_2(\C)}$ is a root $\SL_2$. 
We may assume that $\psi(\SL_2(\C))$ is the $\SL_2$ corresponding to the highest root which is $S_3$-invariant. Then the centralizer of $\psi(SL_2(\C))$ in ${^L}G_E$ is 
\[  ({^L}M_E)^{\mathrm{der}} \cong   \left( \SL_2(\C) \times \SL_2(\C) \times \SL_2(\C) \right) / \{ (a,b,c) \in \mu_2^3: abc =1 \}. \]
This is the L-group of  
\[   H  = \GL_2(E)^{\det}/  F^{\times}. \]
Hence to give such an elliptic A-parameter is to give an L-parameter
\[  \phi :  L_F \longrightarrow  {^L}H \]
which corresponds to an L-packet of $H = \GL_2(E)^{\det}/  F^{\times}$, or more simply to a cuspidal representation of $\GL_2(E)$ (with trivial restriction to $F^{\times}$).
\vskip 5pt

As we shall see in \S \ref{SS:4and8}, the group $H$ is  the $E$-automorphism group of a $E$-twisted composition algebra of $E$-rank 4. Indeed, given any $E$-twisted composition algebra $C$ of $E$-rank $4$, its automorphism group $H_C$ is an inner form of $H$ above and there is a dual pair (see \S \ref{SS:seesawdp})
\[  H_C \times G_E   \subset E_7^B, \]
where  $E_7^B$ is a group of type $E_7$ (associated to a quaternion algebra $B$). This suggests that the A-packets associated to $\psi$ as above can be constructed via exceptional theta lifting from $H_C$. We do not discuss this construction in this paper, but in Appendix A, we shall lay some algebraic and geometric groundwork to facilitate the further study of this case. In particular, we determine in Appendix A the theta lifting of the trivial representation of $H$ to $G_E$. This is needed for our paper  \cite{GS3}. 

\vskip 5pt

  \vskip 15pt

\section{\bf Twisted Composition and Freudenthal-Jordan Algebras}  \label{S:TCFJ}

As we alluded to in the introduction and \S \ref{SS:interest} above, the theory of {\em twisted composition algebras}  plays a fundamental role in this paper. In this section, we shall briefly recall this notion and its relation with Freudenthal-Jordan algebras. This theory is largely due to  Springer, though we shall need to supplement it with some results and observations of our own needed for our  application.
\vskip 5pt

\subsection{\bf Twisted composition algebra.}
For a given \'etale cubic $F$-algebra $E$, an $E$-twisted composition algebra $C$ is a vector space over $E$, equipped with a pair of tensors $(Q,\beta)$ where 
\begin{itemize}
\item $Q: C \longrightarrow E$ is a non-degenerate quadratic form on $C$, and  
\item $\beta: C \rightarrow C$ is a quadratic map
\end{itemize}
 such that 
\[ 
\beta(e \cdot x) =e^{\#} \beta(x), \quad Q(\beta(x))=Q(x)^{\#} \quad \text{and} \quad N_C(x):=b_Q(x,\beta(x)) \in F, \]
for all $e \in E$ and $x\in C$, where $b_Q(x,y)=Q(x+y)-Q(x)-Q(y)$ and $e^{\#}$ is defined in (\ref{E:sharp0}).
\vskip 5pt

Given two $E$-twisted composition algebras $(C, Q,\beta)$ and $(C', Q', \beta')$, an $E$-morphism between them is an $E$-linear map $\phi: C \longrightarrow C'$ such that
\[  Q' \circ \phi = Q \quad \text{and} \quad \beta' \circ \phi = \phi \circ \beta. \]
The automorphism group $\Aut_E(C, Q,\beta)$ of a twisted composition algebra $(C, Q,\beta)$ is an algebraic group over $F$.
 \vskip 5pt

These algebras were introduced by Springer and it is a fact that $\dim_EC  = 1$, $2$, $4$ or $8$. In this paper, we shall chiefly be concerned 
with the case where $\dim_EC =2$, though the case where $\dim_E C =  1$ or $4$ will also be considered. 
\vskip 5pt

\subsection{\bf Rank 1 case} \label{SS:rank1}
 When $\dim_E C  =1$, we may write $C = E \cdot v_0$ for a basis vector $v_0 \in C$. It is not difficult to see that the tensors  $(Q, \beta)$ are of the form
\[  Q_a( x \cdot v_0) = a^{\#} \cdot x^2 \quad \text{and} \quad \beta_a(x \cdot v_0)  =a \cdot  x^{\#}  \cdot v_0 \]
for some $a \in E^{\times}$.  We shall denote this rank $1$ $E$-twisted composition algebra by $C_a$. Its automorphism group is
\[  \Aut(C_a) = {\Res}^1_{E/F}(\mu_2) = \mathrm{Ker} ( N_{E/F}: {\Res}_{E/F}(\mu_2) \rightarrow \mu_2). \]
We have encountered this group before in \S \ref{SS:center}, as the center of the quasi-split group $G_E$, whence it was denoted by $Z_E$. 
The various interpretations of $Z_E$ account for the intricate and sometimes surprising connections between different objects we will encounter later on.
\vskip 5pt

\begin{lemma} \label{L:rank1}
The $E$-isomorphism classes of rank $1$, $E$-twisted composition algebras are parametrized by $E^{\times}/F^{\times} E^{\times 2}$ under the construction $a \mapsto C_a$.
\end{lemma}
\vskip 5pt
\begin{proof}
For $a, b \in E^{\times}$, $C_a$ is isomorphic to $C_b$ if and only if there exists $\lambda \in E^{\times}$ such that
\[  Q_b(\lambda v_0) =  Q_a(v_0) \quad \text{and} \quad \beta_b(\lambda v_0)  = \lambda \cdot \beta_a(v_0),  \]
i.e.
\[  a^{\#}/ b^{\#} =  \lambda^2 \quad \text{and} \quad  a/b = \lambda^{\#}/ \lambda. \]
In fact, the first requirement above is implied by the second (on taking $\#$ on both sides). Now observe that
\[  \lambda^{\#}/ \lambda = N_{E/F}(\lambda) / \lambda^2 \in F^{\times} \cdot E^{\times 2} \]
and conversely, for any $e \in E^{\times}$ and $f \in F^{\times}$,
\[  e^2 \cdot f =  \frac{(e^{\#}f)^{\#}}{e^{\#}f}. \]
Hence, we deduce that  
\[  F^{\times} \cdot E^{\times 2} = \{ \lambda^{\#}/\lambda: \lambda \in  E^{\times} \}, \]
 so that
\[  C_a \cong C_b  \Longleftrightarrow a/b \in F^{\times} \cdot E^{\times 2}. \]
\end{proof}
\vskip 5pt

 The lemma can also be shown via cohomological means. Namely, by considering the long exact sequence associated to the short exact sequence of algebraic groups
\[ \begin{CD}
1 @>>> Z_E= {\Res}^1_{E/F} \mu_2 @>>> {\Res}_{E/F} \mu_2 @>N_{E/F}>>\mu_2 @>>> 1, \end{CD} \]
 one sees that
 \[  H^1(F, Z_E)  = {\Ker}(N_{E/F}: E^{\times}/E^{\times 2} \longrightarrow F^{\times}/ F^{\times 2}). \]
 Then \cite[Prop. 18.34]{KMRT} shows that the map $\#$ gives an isomorphism of $E^{\times}/ F^{\times} E^{\times 2}$  with the kernel above.
 \vskip 5pt

\vskip 10pt

\subsection{\bf Rank $2$ case.}  \label{SS:rank2}
Every twisted composition algebra $(E, C,Q, \beta)$ has a cubic invariant:  the \'etale cubic algebra $E$. 
 On the other hand, when $\dim_E C = 2$, one can attach to it a quadratic invariant, i.e. an \'etale quadratic $F$-algebra $K_C$. Indeed, $K_C$ is determined by the requirement that the discriminant quadratic algebra of $Q$ is $E \otimes_F K_C$. 
 In fact, $C$ can  be realized on $L := E\otimes K_C$ with $Q$ and $\beta$  given by 
\[ 
Q(x)= e \cdot N_{E\otimes K_C/E}(x)  \quad \text{and} \quad \beta(x) =\bar{x}^{\#} \cdot e^{-1} \cdot \bar \nu  
\] 
for some $e\in E^{\times}$ and  $\nu \in K_C^{\times}$ satisfying 
\[ N_{E/F}(e)=N_{K_C/F}(\nu). \]
 Here $\bar{x}$ and $\bar \nu$ refer to the action of the non-trivial automorphism of 
$K_C$ on $x$ and $\nu$.   We shall denote this rank $2$ $E$-twisted composition algebra by $C_{e,\nu}$.  
For a more detailed discussion of this, see \cite{GS2}.

\vskip 5pt

Given an $E$-twisted composition algebra $C= C_{e, \nu}$ as above, consider its automorphism group $H_C = \Aut_E(C) \subset \GL_E(L)$. 
One has a short exact sequence
  \[  \begin{CD}
 1 @>>> (\Aut_EC)^0   @>>>\Aut_E(C)    @>>> S_2 @>>> 1 \end{CD} \]
with
  \[ \Aut_E(C)^0(F)   = T_{E, K_C}(F) :=  \{ x \in L^{\times}:  N_{L/ E}(x)  = 1   \quad \text{and} \quad  N_{L / K_C} (x) = 1\}. \] 
The identity component $H_C^0 = \Aut_E(C)^0$ is a 2-dimensional torus over $F$ depending only on $E$ and $K_C$ and  as $(e,\nu)$ varies, the algebraic subgroups $H_{C_{e,\nu}}^0 \subset \GL_E(L)$ are physically the same subgroup $T_{E,K_C}$. The conjugation action of $S_2$ on $H_C^0$ is by inversion.  In particular, the center of $H_C$ is
\begin{equation} \label{E:ZEHC}
 (H_C^0)^{S_2}  = H_C^0[2] =   {\Res}^1_{E/F} \mu_2  = Z_E. \end{equation}
 Hence, we see yet another incarnation of the finite algebraic group $Z_E$; the consequence of this incarnation will be explained in \S \ref{SS:embedding 1-2} and \S \ref{SS:cohom}.

\vskip 5pt
  The torus  $H_C^0 = \Aut_E(C)^0$  can be interpreted as the group $\Aut_L(C)$ of $L$-linear automorphisms of $C$.   
   It was observed in \cite{GS2} that  $C_{e,\nu}$ and $C_{e', \nu'}$ are $L$-linearly isomorphic if and only if there exists $x \in L^{\times}$ such that 
\begin{equation} \label{E:Lisom}
  e/e' = N_{L/ E}(x) \quad \text{and} \quad  \nu/ \nu' = N_{L/ K_C}(x), \end{equation}
in which case, multiplication-by-$x$ gives an $L$-linear isomorphism $\phi_x: C_{e,\nu} \longrightarrow C_{e',\nu'}$.  Moreover, the isomorphism $\phi_x$ induces an isomorphism
\begin{equation} \label{E:Adphi}
  \mathrm{Ad}(\phi_x): \Aut_E(C_{e,\nu}) \longrightarrow \Aut_E(C_{e',\nu'}) \end{equation}
It is easy to check that the restriction of this isomorphism to the identity components is the identity map on $T_{E, K_C}$.
In any case, we have shown:
\vskip 5pt

\begin{lemma}  \label{L:rank2}
The $L$-isomorphism classes of $E$-twisted composition algebras of rank $2$ and quadratic invariant $K_C$ are parametrized by 
\[  (E^{\times} \times K_C^{\times})^0 /  \mathrm{Im}( L^{\times} )   \]
where 
\[  (E^{\times} \times K_C^{\times})^0 = \{ (e, \nu) \in E^{\times} \times K_C^{\times}: N_{E/F}(e) = N_{K_C/F}(\nu) \} \]
and the map $L^{\times} \longrightarrow E^{\times} \times K_C^{\times}$ is given by 
\[  x \mapsto ( N_{L/ E}(x) , N_{L/ K_C}(x)).\] 
\end{lemma}
\vskip 5pt

This lemma can also be seen cohomologically. As was observed in \cite{GS2},  there is a short exact sequence of algebraic tori
\[  \begin{CD}
1 @>>> T_{E, K_C} @>>> {\Res}_{L/F} \mathbb{G}_m @>N_{L/E} \times N_{L/K_C}>>  ({\Res}_{E/F} \mathbb{G}_m \times {\Res}_{K_C/F} \mathbb{G}_m)^0 @>>> 1 \end{CD} \]
giving rise to an associated long exact sequence 
 \[  \begin{CD}
1@>>> T_{E, K_C}(F) @>>> L^{\times} @>>> (E^{\times} \times K_C^{\times})^0 @>>>  H^1(F, T_{E,K_C})  @>>> 1.  \\
@. @. @. @. @|  \\
 @. @. @. @. (E^{\times} \times K_C^{\times})^0 /  \mathrm{Im}(L^{\times}) 
\end{CD} \]

There is a natural action of $\Aut(K_C/F)$ (as group automorphisms) on $(E^{\times} \times K_C^{\times}) / \mathrm{Im}(L^{\times})$ with the action of the nontrivial element 
given by $(e, \nu) \mapsto (e, \bar{\nu})$. The orbits under this action parametrize the $E$-isomorphism classes of $E$-twisted composition algebras of rank $2$ with quadratic invariant $K_C$. Observe that since $N_{E/F}(e) = \nu \cdot \bar{\nu}$, 
\[  (e, \bar{\nu})  = (e^{-1}, \nu^{-1}) \in (E^{\times} \times K_C^{\times})^0/ \mathrm{Im}(L^{\times}). \]
Hence, the action of $S_2 = \Aut(K_C/F)$  on $H^1(F,  T_{E, K_C})$ is by inversion, and its fixed subgroup $H^1(F, T_{E, K_C})^{S_2}$ is the 2-torsion subgroup $H^1(F, T_{E, K_C})[2]$. 
\vskip 5pt

 Finally, note that the map
\[  H_C(F) := \Aut_E(C)(F)  \longrightarrow S_2  \]
need not be surjective.  Indeed, 
\[  H_C(F) \ne  H_C^0(F) \Longleftrightarrow [C] \in H^1(F, T_{E, K_C})[2], \]
that is, the $L$-isomorphism class of $C$  is fixed by $\Aut(K_C/F)$.
\vskip 5pt

  \subsection{\bf Freudenthal-Jordan algebras.}
Twisted composition algebras are closely related to Freudenthal-Jordan algebras; see \cite[Theorem 37.10]{KMRT} for a precise definition. 
Let $J$ be a Freudenthal-Jordan algebra;  it is a cubic Jordan algebra, so that   every element $a\in J$ satisfies a characteristic polynomial 
\[ 
X^3 - T_J(a) X^2 + S_J(a) X - N_J(a)\in F[X]. 
\] 
The maps $T_J$ and $N_J$ are called the trace and norm maps of $J$ respectively. The element 
\[ 
a^{\#}= a^2-T_J(a) a +S_J(a) 
\] 
is called the adjoint of $a$ and satisfies $a\cdot a^{\#}= N_J(a)$. The cross product of two elements $a,b\in J$ is defined by 
\[ 
a\times b =(a+b)^{\#} - a^{\#} - b^{\#}. 
\] 
The trace form $T_J$ defines a nondegenerate bilinear form $\langle x, y\rangle =T_J(xy)$ on $J$. We shall identify 
$J$ and $J^{\ast}$ using this bilinear form.    Let $(x,y,z)$ be the symmetric trilinear form associated to the norm form $N_J$, so that $(x,x,x)=6 N_J(x)$.  
For any $x,y\in J$, one has
\[   \langle x \times y, z \rangle  = (x,y,z). \]

\vskip 5pt

An etal\'e cubic algebra $E$ is an example of a Freudenthal-Jordan algebra. In general, it is a fact that $\dim_F J  = 1$, $3$, $6$, $9$, $15$ or $27$. In this paper, we shall largely be interested in the case where $\dim_F J  = 9$, though the case where $\dim_F J = 15$ will also be considered. 
\vskip 5pt

The split Freudenthal-Jordan algebra of dimension $9$ is simply the Jordan algebra $M_3^+$ of $3 \times 3$-matrices.  Its automorphism group is 
\[ \Aut(M_3^+)  = \mathrm{PGL}_3 \rtimes S_2, \]
with the nontrivial element of $S_2$ acting by $a \mapsto a^t$.  Hence, isomorphism classes of Freudenthal-Jordan algebras are classified by $H^1(F, {\Aut}(M_3^+))$. Since there is a natural homomorphism
 \[  H^1(F, {\Aut}(M_3)^+)  \longrightarrow H^1(F, S_2), \]
 one sees that to every Freudenthal-Jordan algebra $J$, one can attach an invariant which is an \'etale quadratic algebra $K_J$; this quadratic invariant determines the inner class of the group $\Aut(J)^0$ of type $A_2$.  More generally, if $J$ is a 9-dimensional Freudenthal-Jordan algebra, then $\Aut(J)$ sits in a short exact
 \[  \begin{CD}
 1 @>>> (\Aut J)^0 @>>>\Aut J  @>>> S_2  @>>> 1 \end{CD} \]
 where $\Aut (J)^0$  is an adjoint group of type $A_2$.  Note that the map
  \[  H_J =  \Aut(J)(F) \longrightarrow S_2 \]
 need not be surjective. 

\vskip 5pt

As explained in \cite[Prop. 37.6 and Theorem 37.12]{KMRT} and \cite[\S 4.2]{GS2}, a Freudenthal-Jordan algebra $J$ of dimension $9$ over $F$  is obtained from a pair $(B, \tau)$, where 
$B$ is a central simple algebra over $K = K_J$ of dimension $9$ and $\tau$ is an involution of second kind on $B$, as the subspace $B^{\tau}$ of $\tau$-symmetric elements, equipped with the Jordan product $x \circ y = (xy+yx)/2$.  For a fixed \'etale quadratic algebra $K$, this construction gives an  essentially surjective faithful functor of groupoids:
\[  \{ \text{$K$-isomorphism classes of $(B,\tau)$} \} \longrightarrow \{ \text{$F$-isomorphism classes of $J$ with $K_J = K$}  \} \]
(where $\dim_K B = 9 = \dim_F J$); it is fully faithful and thus  an equivalence if we allow $F$-linear isomorphisms on $(B,\tau)$ and not just $K$-linear ones. Thus  $\Aut(J)^0 = \Aut_K(B,\tau)$ and there is an $S_2$-action on the source given by
\[  (B, \tau) \mapsto (B^{\mathrm{op}}, \tau),  \]
so that the fibers of the map are precisely the $S_2$-orbits (and hence have size $1$ or $2$).  Further, $\Aut(J)^0(F) = \Aut(J)(F)$ if and only if the fiber of $J$ has size $2$, i.e. ($B,\tau) \ncong (B^{\mathrm{op}},\tau)$.

 \vskip 10pt
 
\subsection{\bf Springer decomposition.} \label{SS:springer}
Twisted composition algebras are  related to Freudenthal-Jordan algebras by the Springer construction. Suppose  we have an algebra embedding
\[  i: E \hookrightarrow J. \]
Then, with respect to the trace form $T_J$, we have an orthogonal decomposition
\[  J = E \oplus C \]
where $C = E^{\perp}$.  For $e \in E$ and $x \in C$, one can check that $e \times x \in C$. Thus, setting
\[  e \cdot x :=  -e \times x \]
 equips $C$ with the structure of an $E$-vector space. Moreover, for every $x\in C$, write 
 \[  x^{\#} = ( - Q(x) , \beta(x)) \in E \oplus C =J \]
 where  $Q(x) \in E$ and $\beta(x) \in C$. In this way, we obtain a quadratic form $Q$ on $C$ and a quadratic map $\beta$ on $C$. Then, by  \cite[Theorem 38.6]{KMRT},
  the triple $(C,Q,\beta)$ is an $E$-twisted composition algebra over $F$. Conversely, given an $E$-twisted composition algebra $C$ over $F$, the space $E \oplus C$ can be 
given the structure of a Freudenthal-Jordan algebra over $F$, by  \cite[Theorem 38.6]{KMRT} again. We recall in particular that for $(a,x) \in E \oplus C$,
\begin{equation} \label{E:sharp}
 (a,x)^{\#} = (a^{\#} - Q(x),  \beta(x) - a \cdot x). \end{equation}
\vskip 5pt

 This construction gives a bijection
\[  \{\text{$E$- isomorphism classes of $E$-twisted composition algebras} \} \]
\[  \updownarrow \]
\[ \{ \text{$H_J$-conjugacy classes of pairs $(J,  i: E \hookrightarrow J)$} \}  \]
\vskip 5pt
\noindent where $J$ is a Freudenthal-Jordan algebra of dimension $9$ and  $i: E \hookrightarrow J$ is an algebra embedding. 
Moreover, this bijection induces an isomorphism
\[  H_C := \Aut_E(C)  \cong   \Aut( i: E \hookrightarrow J), \]
where the latter group is the pointwise stabilizer in $\Aut(J)$ of $i(E) \subset J$. In other words, the Springer construction is an equivalence of groupoids.  
 If an $E$-twisted composition algebra $C$ corresponds to an embedding $i:  E \hookrightarrow J$ under this equivalence,  then one has:
 \begin{equation} \label{E:biquad}
  [K_E] \cdot [K_C] \cdot [K_J]  = 1 \in F^{\times}/ F^{\times 2}. \end{equation}
\vskip 5pt

One consequence of the Springer construction is that it gives us an alternative description of the torus $T_{E, K_C}$. 
 It was shown in \cite{GS2} that there is an isomorphism (an exceptional Hilbert Theorem 90), 
\[  \Aut_E(C_{e,\nu})^0 = T_{E, K_C}      \cong  \tilde{T}_{E, K_J}   = \{ x\in  (E \otimes_F K_J)^{\times}:  N_{E \otimes K_J / E}(x)  \in F^{\times}\} /  K_J^{\times}    \]
when $J = E  \oplus C_{e,\nu}$.  We will next recall how this isomorphism arises. 
\vskip 5pt

\subsection{\bf An isomorphism of tori} \label{SS:isom-tori}
 
Given an $E$-twisted composition algebra $C$ corresponding to an embedding $\iota: E \hookrightarrow J$, let us pick a pair $(B,\tau)$ over $K_J$ such that $J = B^{\tau}$. 
The embedding $\iota$ gives rise to an embedding of $K_J$-algebras compatible with involutions of second kind:
 \[  \tilde{\iota}:  E \otimes_F K_J \longrightarrow  B,  \]
where we have used the involution on $E \otimes K_J$ induced by the nontrivial automorphism of $K_J/F$.
This induces an embedding of algebraic groups
\[ \tilde{\iota}_*: (E \otimes K_J)^{\times}/  K_J^{\times} \longrightarrow  PB^{\times} = \Aut_{K_J}(B) \]
whose image is precisely the pointwise stabilizer of $\tilde{\iota}$ in $\Aut_{K_J}(B)$.
 The map $\tilde{\iota}_*$ restricts to give an isomorphism
 \[  \tilde{T}_{E, K_J} \cong  \Aut_{K_J}(B,\tau, \tilde{\iota})  \subset  \Aut_{K_J}(B, \tau).  \]
where
\[  \tilde{T}_{E, K_J} = \mathrm{Ker} \left( N_{K_J/F}:  (E \otimes K_J)^{\times}/K_J^{\times} \longrightarrow E^{\times}/F^{\times} \right). \]
Since
 \[   \Aut_{K_J}(B,\tau, \tilde{\iota})   = \Aut_F(J,\iota)^0  = \Aut_E(C)^0, \]
 we see that the choice of a $(B,\tau)$ with $J = B^{\tau}$ gives an isomorphism of algebraic groups 
 \[  \tilde{T}_{E,K_J} \longrightarrow  H_C^0 = \Aut_E(C)^0. \]
 If one had chosen $(B^{op},\tau)$ instead, the resulting isomorphism is the composite of the one for $(B,\tau)$ with the inversion map.
 If it turns out that $(B, \tau) \cong (B^{\mathrm{op}},\tau)$, then these two isomorphisms are conjugate by an element of $H_C(F) \setminus H_C^0(F)$. Thus, each $E$-twisted composition algebra $C$ with quadratic invariant $K_C$ comes equipped with a pair of isomorphisms of algebraic groups
 \[  \iota_C, \iota_C^{-1}:  H_C^0 \longrightarrow \tilde{T}_{E, K_J}, \]
 where $[K_E] \cdot [K_C] \cdot [K_J] = 1 \in F^{\times}/F^{\times 2}$. This gives a canonical isomorphism
 \[  [\iota_C]:  H_C^0(F)/H_C^0(F)^2 \cong \tilde{T}_{E, K_J}(F)/\tilde{T}_{E, K_J}(F)^2. \]
 \vskip 5pt
 
 In particular, if we consider $C =C_{e, \nu}$ and $J = E \oplus C_{e,\nu}$, then 
  we obtain   a pair of isomorphisms of algebraic tori
 \begin{equation} \label{E:enu-isom}
   \iota_{e, \nu}, \,  \iota_{e,\nu}^{-1}:  T_{E, K_C}  \cong \tilde{T}_{E, K_J}. \end{equation}
 We have:
 \vskip 5pt
 
 \begin{lemma} \label{L:indep}
 The pair of isomorphisms  in (\ref{E:enu-isom}) is independent of the choice of $(e, \nu)$.
  \end{lemma}
  \vskip 5pt
  
  \begin{proof}
  Suppose first that  $C_{e,\nu}$ and $C_{e', \nu'}$ are $L$-isomorphic,  with an $L$-isomorphism given by a multiplication-by-$x$ map $\phi_x$ as in  (\ref{E:Lisom}) and (\ref{E:Adphi}). Then it follows by the functoriality of Springer's construction that 
  \[  \iota_{e, \nu} = \iota_{e',\nu'}^{\pm 1} \circ Ad(\phi_x)|_{T_{E,K_C}}. \] 
  Here the sign $\pm$ arises because of the possibility of using a central simple algebra $B$ or $B^{op}$ in the construction of $\iota$. 
  We have observed after (\ref{E:Adphi}) that $Ad(\phi_x)$ is the identity map on $T_{E,K_C}$, so that $\iota_{E,\nu} = \iota_{e',\nu'}^{\pm 1}$. 
  
  \vskip 5pt
  
  Now given any two $C_{e,\nu}$ and $C_{e', \nu'}$, one knows that they become $L\otimes_F  \tilde{F}$-isomorphic over a finite Galois extension $\tilde{F}$ of $F$. 
  Hence the two pairs of isomorphisms $\iota_{e,\nu}^{\pm}$ and $\iota_{e',\nu'}^{\pm}$ of algebraic tori become equal after a base change to $\tilde{F}$. But then they are already equal over $F$. 
  \end{proof}
  \vskip 5pt

  Thus we have a canonical pair of isomorphisms 
  \begin{equation} \label{E:cano-isom}
   \iota, \iota^{-1}:  T_{E, K_C}  \cong \tilde{T}_{E, K_J}. \end{equation}
  This is the exceptional Hilbert 90 Theorem from \cite{GS2}. It gives a canonical isomorphism
  \[  [\iota]: T_{E, K_C}(F)/T_{E,K_C}(F)^2 \cong \tilde{T}_{E,K_J}(F)/ \tilde{T}_{E,K_J}(F)^2. \]
   \vskip 5pt

One consequence of this alternative description of $H_C^0$ is that its gives an alternative computation of $H^1(F,  H_C^0)$. In particular, it follows from 
\cite[Prop. 11.2]{GS2} that
\begin{equation} \label{E:H1T2}
  H^1(F, \tilde{T}_{E, K_J})[2]   =   E^{\times} /  F^{\times} N_{E \otimes K_J/ E}((E\otimes K_J)^{\times}). \end{equation}
This description of $H^1(F, T_{E,K_C})[2] = H^1(F, \tilde{T}_{E, K_J})[2]$   will be very helpful later on.

\vskip 5pt
\subsection{\bf Examples.}
As an example, consider the case where $E=F^3$, and $J=M_3(F)$ is  the Jordan algebra of $3\times 3$ matrices. We have a natural embedding of $F^3$ into $M_3(F)$ where 
$(a_1,a_2, a_3)\in F^3$ maps to the diagonal matrix with $a_1, a_2, a_3$ on the diagonal.  If $x\in M_3(F)$, then  $x^{\#}$ is the adjoint matrix. Thus it is easy to describe the 
structure of the twisted composition algebra $C$ in this case.  An element $x$ in $C$ is given by a matrix 
 \[ 
  x= 
  \left( \begin{array}{ccc} 
  0 & x_3 & y_2 \\
  y_3 & 0 & x_1 \\
  x_2 & y_1 & 0 
  \end{array} \right) . 
  \] 
  If we write $x=((x_1,y_1),(x_2,y_2),(x_3,y_3))$ then the structure of $F^3$-space on $C$ is given by 
  \[ 
  (a_1, a_2, a_3) \cdot ((x_1,y_1),(x_2,y_2),(x_3,y_3))= ((a_1x_1,a_1y_1),(a_2x_2,a_2y_2),(a_3x_3,a_3y_3))
  \] 
  for all $ (a_1, a_2, a_3) \in F^3$.  The structure of the twisted composition algebra on $C$ is given by 
  \[ 
  Q((x_1,y_1),(x_2,y_2),(x_3,y_3))=(x_1y_1, x_2y_2, x_3y_3) 
  \] 
  and 
  \[ 
  \beta((x_1,y_1),(x_2,y_2),(x_3,y_3))=( (y_2y_3, x_2x_3), (y_3 y_1, x_3 x_1), (y_1 y_2, x_1 x_2)). 
  \] 
This twisted composition algebra $(C, Q,\beta)$ has cubic invariant $F^3$ and quadratic invariant $F^2$. 
\vskip 10pt


Here is another example. Assume that $E$ is a cyclic cubic field extension of $F$, with Galois group 
generated by $\sigma$.  Let $D$ be a degree 3 central simple algebra over $F$ containing $E$ as a subalgebra.  Then as a vector space over $E$, $D$ has a basis $1, \varpi, \varpi^2$, for some element 
$\varpi\in D$ satisfying $\varpi x = \sigma(x) \varpi$, for all $x\in E$, and $\varpi^3=\lambda \in F^{\times}$. The corresponding $E$-twisted composition algebra is 
isomorphic to $C(\lambda)=E\oplus E$,  with 
\[ 
Q(x,y)= xy \text{ and } \beta(x,y) = (\lambda^{-1} y^{\#}, \lambda x^{\#}).
\] 
Moreover, $C(\lambda)$ has cubic invariant $E$ and quadratic invariant $F^2$ and is associated to $(e,\nu) = (1, (\lambda, \lambda^{-1}))$.
The algebra $D$ is split if and only if $\lambda$ is a norm of an element in $E^{\times}$. The group of $E$-automorphisms of $C(1)$ is 
\[ 
\Aut_E(C(1))= E^1 \rtimes (\Z/2\Z)
\]  
where $\alpha \in E^1$ acts on $C(1)$ by $(x,y) \mapsto (\alpha x, \alpha y)$, and the nontrivial element in $\Z/2\Z$ by $(x,y) \mapsto (y,x)$, for all $(x,y) \in C(1)$.

\vskip 5pt

\subsection{\bf When is $J$ division?}  \label{SS:division}
Following up on the last example above, one may consider the question: under what conditions on $(e, \nu)$ is  $J_{e,\nu} = E\oplus C_{e,\nu}$ associated to a division algebra? An answer for the general case is provided by \cite[Thm. 38.8]{KMRT}, but we provide an alternative treatment adapted to the rank 2 case here.  
\vskip 5pt

\begin{prop} \label{P:division}
Fix $(e, \nu) \in (E^{\times} \times K_C^{\times})^0$, so that $N_{E/F}(e) = N_{K_C/F}(\nu)$. Then the following are equivalent:
\vskip 5pt

\begin{itemize}
\item[(i)] $\nu \in N_{L/K_J}(L^{\times})$ (where $L = E \otimes K_C$);

\vskip 5pt

\item[(ii)]  $(e, \nu) = (e', 1) \in (E^{\times} \times K_C^{\times})^0/ \mathrm{Im}(L^{\times})$;
\vskip 5pt

\item[(iii)] $(e, \nu)  = (e', \nu')  \in (E^{\times} \times K_C^{\times})^0/ \mathrm{Im}(L^{\times})$, with $\nu' \in F^{\times}$;
\vskip 5pt

\item[(iv)] $[(e, \nu)] \in H^1(F, T_{E, K_C})[2]$;
\vskip 5pt

\item[(v)] $J= E \oplus C_{e,\nu}$ is not a division Jordan algebra.
\end{itemize} 
When these equivalent conditions  hold for $C$, $H_C(F) \cong H_C^0(F) \rtimes \Z/2\Z$. Indeed, for any $C = C_{e,\nu}$ with $\nu \in F^{\times}$, 
\[  \Aut_E(C_{e,\nu})  = T_{E, K_C} \rtimes \Aut(K_C/F)   \subset \GL_E(L). \]
In other words, these automorphism groups are physically the same subgroup of $\GL_E(L)$.

\end{prop}
\vskip 5pt

\begin{proof}
We first show the equivalence of the first four statements.
The implications (i) $\Longrightarrow$ (ii)  $\Longrightarrow $ (iii)  $\Longrightarrow $ (iv) are clear.
Assume that (iv) holds, so that   $[(e, \nu)] = [(e, \bar{\nu})]$. Then there exists $x \in L^{\times}$ such that 
\[ N_{L/E}(x) = x \cdot \bar{x} =1 \quad \text{and} \quad  \nu = \bar{\nu} \cdot  x\cdot x^{\#}. \]
Now the first condition implies that $x = \bar{z} / z$ for some $z \in L^{\times}$, which when substituted into the second gives $\nu \cdot N_{L/K_C}(z) \in F^{\times}$. 
Hence, replacing $(e, \nu)$ by an equivalent pair, we may assume that $\nu \in F^{\times}$, so that $N_{E/F}(e) = \nu^2$. 
 But then
\[  (e,\nu) =  (e \cdot N_{L/E}(e), \nu \cdot N_{L/K_C}(e)) = (e^3 , \nu^3) \in  (E^{\times} \times K_C^{\times})^0/ \mathrm{Im}(L^{\times}).  \]
Since $\nu^3 = N_{L/K_C}(\nu)$, we conclude that (i) holds. 
\vskip 5pt

We note that the equivalent conditions (i)-(iv) always hold when $E$ is not a field, for then the norm map $N_{L/K_C}: L^{\times} \longrightarrow K_C^{\times}$ is surjective.
\vskip 5pt

Finally, to check the equivalence with (v), note that  $J = E \oplus C_{e,\nu}$ is not a division Jordan algebra if and only if there exists nonzero $(a,x) \in E \oplus  C_{e,\nu}$ such that  $(a,x)^{\#} = 0$. By (\ref{E:sharp}), this is equivalent to 
\begin{equation} \label{E:sharp2}
  a^{\#} =  Q(x) =  e \cdot N_{L/E}(x) \quad \text{and} \quad   a \cdot x = \beta(x) = e^{-1} \cdot \bar{\nu}  \cdot \bar{x}^{\#}. \end{equation}
When $E$ is not a field, we can always take nonzero $(a,0)$ with $a^{\#} = 0$, so that $J$ is never a division algebra in this case. 
\vskip 5pt

We may henceforth assume that $E$ is a field.  
Suppose that (ii) holds, so that  $\nu =1$   and $N_{E/F}(e) =1$. Then we may take $(a,x) = (e, e^{\#})$; one checks that this satisfies the two equations in (\ref{E:sharp2}) and hence $J$ is not division.   We have thus shown (ii) $\Longrightarrow$ (v). 
\vskip 5pt

Conversely, we shall show (v) implies (i) (when $E$ is a field).
Assume that there is a nonzero $(a,x)$ such that the two equations  in (\ref{E:sharp2}) hold. Then $x$ must be nonzero (otherwise, we deduce by the first equation that $a^{\#} = 0$ and hence $a = 0$ since $E$ is a field).  
 Multiplying the two  equations in (\ref{E:sharp2}), we obtain
 \[  N_{E/F}(a) \cdot x = \bar{\nu}  \cdot N_{L/K_C}(\bar{x}) \cdot x, \]
 so that
 \begin{equation} \label{E:xnorm}
  x \cdot (N_{E/F}(a) - \bar{\nu} \cdot N_{L/K_C}(\bar{x}) ) = 0. \end{equation}
Hence, if $K_C$ is a field, so that $L$ is a field also, then we may cancel $x$ (noting that $x \ne 0$) to deduce that
 \[  \nu = N_{E/F}(a) \cdot N_{L/K_C}(x)^{-1}  \in N_{L/K_C}(L^{\times}). \] 
 On the other hand, if $K_C = F \times F$, then let  
 \[  x = (x_1, x_2) \in E\times E = L \quad \text{and} \quad \nu = (\nu_1, \nu_2) \in F^{\times} \times F^{\times}. \] 
 The two equations in (\ref{E:sharp2}) becomes:
 \[  a^{\#} = e \cdot x_1 x_2 \quad \text{and} \quad (ax_1, ax_2) =  e^{-1} \cdot (\nu_2 \cdot  x_2^{\#}, \nu_1 \cdot  x_1^{\#}). \]
 From this, we see that $a \ne 0$ (otherwise, the second equation would give $x_1 = x_2 = 0$ also), and hence $x_1, x_2 \in E^{\times}$. Hence, we may cancel $x$ in (\ref{E:xnorm}) as before and conclude that $\nu \in N_{L/K_J}(L^{\times})$, as desired.
\end{proof}

\vskip 10pt

   \subsection{\bf Embeddings}  \label{SS:embedding 1-2}
We record here some results that we will need later, concerning embeddings of rank 1 twisted composition algebras into rank 2 ones. 
\vskip 5pt

\begin{lemma} \label{L:embedding 1-2}
Let us fix
\vskip 5pt

\begin{itemize}
\item $a \in E^{\times}$ with corresponding rank $1$ $E$-twisted composition algebra $C_a  = E$ and
\item  an $E$-twisted composition algebra $C = C_{e,\nu}$ of rank $2$, corresponding to an embedding $E \hookrightarrow J$, with resulting Springer decomposition $J = E \oplus  C$.
\end{itemize}
There are  natural equivariant bijections between the following three $\Aut_E(C)$-sets (possibly empty)
\vskip 5pt
\begin{itemize}
\item[(a)]  the set of $E$-morphisms $f: C_a \longrightarrow  C$; 
\item[(b)]  the set of rank $1$ elements $x \in J$ (i.e. $x^{\#} =0$ but $x \ne 0$) of the form $x = (a,v) \in E \oplus C = J$;
\item[(c)]   the set
\[  X_{a,C}(F) = X_{a, e,\nu}(F)  = \{  x \in   L:= E \otimes K_C:  N_{L/E}(x) =   e^{-1} a^{\#} \, \text{and} \,  N_{L/K_C}(x) =  N_{E/F}(a) \cdot \nu^{-1} \}. \]
\end{itemize}
The bijection between (a) and (b)  is given by $f \mapsto (a,  f(1))$, whereas that between (b) and (c) is given by $x = (a, v) \mapsto v$. 
\end{lemma}   
  
   Note that the 3 sets are possibly all empty. For example, if $J$ is associated with a cubic division algebra, then there are no rank $1$ elements in $J$, so that the set in (b)
 is empty, and hence so are the other 2 sets. On the other hand, we note:
 
 \vskip 5pt
 
 \begin{lemma} \label{L:XaC}
 For any $a \in E^{\times}$, there exists a unique $E \otimes K_C$-isomorphism class $[C]$ such that $X_{a, C}(F)$ is nonempty. 
 This unique $E \otimes K_C$-isomorphism class is represented by $C_{a^{\#}, N_{E/F}(a)}$.  Hence we have  a group homomorphism
 \[    f:  E^{\times}/ F^{\times} E^{\times 2} \longrightarrow (E^{\times} \times K_C^{\times})^0/ \mathrm{Im}(L^{\times}) \]
 given by 
 \[  f(a)  = (a^{\#}, N_{E/F}(a))  \]
 and characterized by the requirement that  $C_a$ embeds into $C_{e, \nu}$ if and only if $(e, \nu) = f(a) \in H^1(F, T_{E, K_C})$.
The image of $f$  is equal to $H^1(F, T_{E, K_C})[2]$, i.e. consists precisely of those twisted composition algebras $C$ whose associated Jordan algebra is not division, whereas 
\[  \mathrm{Ker}(f) = \{ x^{\#}/ \bar{x}: x \in L^{\times} \,\,\text{and} \, \, N_{L_/K_C}(x) \in F^{\times} \}/ F^{\times} E^{\times 2} . \]
\end{lemma}

\begin{proof}
 It is clear that if $C= C_{a^{\#}, N_{E/F}(a)}$, then $1 \in X_{a,C}(F)$; this shows the existence of $C$ and that it has the desired form. 
For the uniqueness, suppose that $X_{a, e, \nu}(F)$ and $X_{a, e',\nu'}(F)$ are both nonempty. Then there exist $x, x'  \in L^{\times}$ such that
\[ N_{L/E}(x) = e^{-1}  a^\#  \quad \text{and} \quad  N_{L/K}(x) = N_{E/F}(a)\cdot  \nu^{-1} \]
and
\[ N_{L/E}(x') = {e'}^{-1}  a^\# \quad \text{and} \quad  N_{L/K}(x') = N_{E/F}(a)\cdot  {\nu'}^{-1}. \]
On dividing one equation by the other, we see that
\[   N_{L/E}(x'/x) = e/e'  \quad \text{and} \quad N_{L/K}(x'/x) = \nu/\nu'.  \]
This implies that $(e,\nu) =(e',\nu') \in H^1(F, T_{E,K_C})$, as desired.
 
\vskip 5pt

By Proposition \ref{P:division}, the image of $f$ consists of  twisted composition algebras  associated to non-division Jordan algebras $J$.   
On the other hand, to prove that any such $C$ is in the image of $f$, it suffices by Proposition \ref{P:division} to consider  $C = C_{e, 1}$, with $N_{E/F}(e)=1$. We claim that $f(e) = [(e,1)]$. Indeed,
\[  f(e) = (e^{\#}, N_{E/F}(e)) = (e^{-1}, 1) = (e,1) \in H^1(F, T_{E, K_C}). \]
We leave the statement about $\mathrm{Ker}(f)$ to the reader.
\end{proof}

\vskip 5pt

Since the image of the map $f$ in the above lemma is $H^1(F, T_{E, K_C})[2]$, we deduce from (\ref{E:H1T2}) that $f$ can be simply interpreted as the natural map
\begin{equation} \label{E:f2}
 f:  E^{\times}/F^{\times} E^{\times 2} \longrightarrow E^{\times} / F^{\times} N_{E \otimes K_J/E}((E\otimes K_J)^{\times}). \end{equation}

\vskip 5pt
 Finally, we note that  $X_{a,C} = X_{a,e,\nu}$ is an algebraic variety which is evidently a torsor for the torus $H_C^0 = T_{E, K_C}$.  If $X_{a,e,\nu}(F)$ is nonempty, then $H_C^0(F) = T_{E, K_C}(F)$ acts simply transitively on it. Thus, the action of $H_C(F)$ on $X_{a,e,\nu}(F)$ is transitive and the stabilizer of a point $x \in X_{a,e,\nu}(F)$ has order $2$, with the nontrivial element $h_x \in H_C(F) \setminus H_C^0(F)$.   For example, the stabilizer of $1 \in X_{a, C_{a^{\#}, N(a)}}(F)$ is $\Aut(K_C/F)$.
  Indeed, $h_x$ is the map on $C_{e,\nu} = E \otimes K_C$ given by
 \[  h_x : z \mapsto    \frac{x}{\overline{x}}  \cdot \overline{z}. \]
  If $x' \in X_{a,e,\nu}(F)$ is another element, then $x'  = t \cdot x$ for a unique $t \in H_C^0(F)$ and 
 \[  h_{x'}= t \cdot h_x \cdot t^{-1} = t^2 \cdot h_x. \]
 Thus the element $h_x$ gives a well-defined class in $(H_C(F) \setminus H_C^0(F)) / H_C^0(F)^2$ as $x \in X_{a,e,\nu}(F)$ varies. We record this as a lemma.
 \vskip 5pt
 
   \begin{lemma} \label{L:basept}
   Suppose that $f(a) = [C] \in H^1(F, T_{E, K_C})[2]$ so that $X_{a,C}(F)$ is nonempty. Then one obtains a class 
   \[  g_C(a) \in (H_C(F) \setminus H_C^0(F)) / H_C^0(F)^2 \]
    consisting of elements which stabilize some points in $X_{a,C}(F)$. 
   \end{lemma}
  \vskip 5pt
  
  \subsection{\bf Cohomological interpretation}  \label{SS:cohom}
  The embedding problem studied in the previous subsection can be given a rather transparent cohomological treatment. The map $f$ in Lemma  \ref{L:XaC} is a surjective homomorphism $H^1(F, Z_E) \longrightarrow H^1(F, T_{E, K_C})[2]$. This map can be obtained from our observation in (\ref{E:ZEHC}) that $T_{E, K_C}[2] = Z_E$. From the Kummer exact sequence 
  \[
   \begin{CD}
  1 @>>> Z_E @>>> T_{E, K_C} @>2>> T_{E, K_C} @>>>1,  \end{CD} \]
 one deduces the following fundamental short exact sequence
  \begin{equation} \label{E:keyses}  \begin{CD}
 1 @>>> T_{E, K_C}(F)^2 \backslash T_{E, K_C}(F) @>b>>  H^1(F, Z_E) @>f>> H^1(F, T_{E, K_C})[2] @>>> 1. \end{CD} \end{equation}
 The map $f$ here is precisely the one described in Lemma \ref{L:XaC}. This cohomological discussion also gives us a more conceptual description of $\mathrm{Ker}(f)$:
 \[  {\Ker}(f) = T_{E, K_C}(F)^2 \backslash T_{E, K_C}(F). \]
 The map $b$ can be described explicitly as follows. Given $t \in T_{E, K_C}(F) \subset L^{\times}$, since $N_{L/E}(t) =1$, we can write
 \[  t = \bar{y} / y \quad \text{  with $N_{L/K_C}(y) \in F^{\times}$ (since $N_{L/K_C}(t) =1$).} \]
 Then
 \[  b(t)  = y^{\#}/ \bar{y} \in E^{\times}/F^{\times} E^{\times 2}, \]
The reader can easily verify that $b(t)$  is independent of the choice of $y$ and is trivial if $t \in T_{E, K_C}(F)^2$.

\vskip 5pt

Here is another interesting observation arising from (\ref{E:keyses}) and Lemma \ref{L:basept}. Let us  fix $[C] \in H^1(F, T_{E, K_C})[2]$ and consider the fiber $f^{-1}([C])$ which is a $T_{E, K_C}(F)/T_{E,K_C}(F)^2$-torsor. Then we have:
\vskip 5pt

\begin{prop} \label{P:isom-torsor}
The map $a \mapsto g_C(a)$ (with $g_C(a)$ defined in Lemma \ref{L:basept}) gives an isomorphism 
\[  f^{-1}([C]) \longrightarrow (H_C(F) \setminus H_C^0(F)) / T_{E, K_C}(F)^2 \]
of $T_{E,K_C}(F)/ T_{E,K_C}(F)^2$-torsor.
\end{prop}
\vskip 5pt

\begin{proof}
Assume without loss of generality that $C=C_{e,\nu}$.
Since both $f^{-1}([C])$ and $(H_C(F) \setminus H_C^0(F)) / T_{E, K_C}(F)^2$ are torsors under $T_{E,K_C}(F)/ T_{E,K_C}(F)^2$, it suffices to show that if 
\[  a'  =  b(t) \cdot a  \in f^{-1}([C]), \]
then
\[  g_C(a')  = t \cdot g_C(a) \in (H_C(F) \setminus H_C^0(F)) / T_{E, K_C}(F)^2 . \]
Write 
\[  t = \bar{y}/y  \quad \text{  with $N_{L/K_C}(y) \in F^{\times}$},  \]
so that 
\[  b(t) = y^{\#}/\bar{y} \quad \text{and hence} \quad a'  = a \cdot y^{\#}/ \bar{y} . \] 
This implies in particular that 
\[ N_{E/F}(a') = N_{E/F}(a) \cdot N_{L/K_C}(y) \quad \text{and} \quad {a'}^{\#} = a^{\#} \cdot N_{L/E}(y). \]
Now suppose that $x \in X_{a,e,\nu}(F) \subset L^{\times}$, so that
\[ N_{L/E}(x) = e^{-1} a^{\#} \quad \text{and} \quad N_{L/K_C}(x) = N_{E/F}(a) \cdot \nu^{-1}. \]
Then one checks that $x' := xy \in X_{a', C}(F)$.  Hence, if $h_x$ and $h_{x'}$ are the nontrivial elements stabilizing $x$ and $x'$ respectively, then for any $z \in C$,
\[  h_{x'} (z) =  \frac{x'}{\bar{x}' } \cdot \bar{z} =  \frac{xy}{\bar{x} \bar{y}} \cdot \bar{z} = t^{-1} h_x (z). \] 
Thus we have
\[  h_{x'} = t \cdot h_x \in (H_C(F) \setminus H_C^0(F)) / T_{E, K_C}(F)^2. \]
\end{proof}

\vskip 5pt
Indeed, if $[C]$ is a nontrivial  element of $H^1(F, T_{E, K_C})[2]$, then $[C]$ generates a subgroup of order $2$ and we have a short exact sequence of abelian groups
\begin{equation} \label{E:ext1}
  \begin{CD}
1@>>>    T_{E,K_C}(F)/ T_{E, K_C}(F)^2 @>>>  f^{-1}( \langle [C] \rangle) @>>>  \langle [C] \rangle @>>>1. \end{CD} \end{equation}
On the other hand, with $C = C_{e, 1}$ (without loss of generality), one has another extension:
\begin{equation} \label{E:ext2} 
 \begin{CD} 
1@>>>    T_{E,K_C}(F)/ T_{E, K_C}(F)^2 @>>>  H_C(F) / T_{E, K_C}(F)^2 @>>> S_2 @>>> 1 \end{CD} \end{equation}
Then the following is a consequence of Proposition \ref{P:isom-torsor}:
\vskip 5pt 

\begin{prop} \label{P:isom-ext}
The two extensions (\ref{E:ext1}) and (\ref{E:ext2}) are  isomorphic via a canonical isomorphism of extensions defined as follows.  For any $a \in E^{\times}/F^{\times} E^{\times 2} = H^1(F, Z_E)$ with $f(a) = [C]$, the isomorphism sends  $a$ to $g_C(a)$.  
 \end{prop}
 \vskip 5pt

  \vskip 5pt

\vskip  5pt 
 \subsection{\bf Rank 4 and 8 cases}  \label{SS:4and8}
We conclude with a brief sketch of the rank $4$ and $8$ cases. The case $\dim_EC = 4$   corresponds to embeddings of Jordan algebras $E \longrightarrow J$ with $\dim_FJ = 15$. Examples of such $J$ are of the form $H_3(B)$, the Jordan algebra of $3 \times 3$-Hermitian matrices with entries in a quaternion algebra $B$.  This case is discussed in some detail in Appendix A below. We simply note here that the automorphism group of such a $C$ is
   \[  \Aut_E(C)  \cong  \left( \Res_{E/F} (B\otimes_F E)^{\times}  \right)^{\det} / F^{\times} \]
 where the RHS consists of elements in $(B \otimes E)^{\times}$ whose norm lies in $F^{\times}$. See \S \ref{SS:seesawdp} below.
 \vskip 5pt

 Finally, when $\dim_EC = 8$,  one has $\dim_FJ = 27$, so that $J$ is an exceptional Jordan algebra. An example is $J = H_3(\mathbb{O})$, the Jordan algebra of $3 \times 3$-Hermitian matrices with entries in an octonion algebra $\mathbb{O}$. When the octonion algebra is split, the automorphism group of such a $C$ is isomorphic to the group
 \[  G_E = \Spin_8^E.\]
 Moreover,  the action of $G_E$ on $C$ is (the Galois descent of) the sum of the 3 irreducible 8-dimensional representations of $\Spin_8$ over $\overline{F}$. 
It is no wonder that the structure of the group $G_E$ is intimately connected with the theory of twisted composition algebras. 
 
\vskip 15pt

\section{\bf Twisted Bhargava Cubes} \label{SS:relation}

 To connect the theory of twisted composition algebras with our earlier discussion on $G_E =\Spin_8^E$, let us recall the main result of \cite{GS2}. 
 
 \vskip 5pt
 
 \subsection{\bf Nondegenerate cubes}
 Recall the Heisenberg parabolic subgroup  $P_E = M_E \cdot N_E \subset G_E$ and the natural action of $M_E \cong \GL_2(E)^{\det}$ on the space $V_E = N_E/ [N_E, N_E]$ of $E$-twisted cubes. 
 Now we have \cite[Prop. 10.4]{GS2}:
 \vskip 5pt
 
 \begin{prop}  \label{P:orbits}
 The nondegenerate $M_E(F)$-orbits on $V_E(F)$ are in natural bijection with $E$-isomorphism classes of $E$-twisted composition algebras of rank $2$. More precisely,
 to every nondegenerate $E$-twisted   cube $\Sigma$, we attached in \cite{GS2} a pair $(Q_{\Sigma},\beta_{\Sigma})$  giving a structure of $E$-twisted composition algebra on $E\oplus E$, with an isomorphism
 \[  \mathrm{Stab}_{M_E(F)}(\Sigma)  \cong  \Aut_E(Q_{\Sigma},\beta_{\Sigma}). \] 
  If $g\in M_E(F) = \GL_2(E)^{\det}$ and $\Sigma'=g(\Sigma)$, then the pair $(Q_{\Sigma'},\beta_{\Sigma'})$ attached to $\Sigma'$ is obtained from $(Q_{\Sigma},\beta_{\Sigma})$ by the change of variables given by the matrix $g$, i.e.
  \[   Q_{\Sigma'} = Q_{\Sigma} \circ {^t}g \quad \text{and} \quad  \beta_{\Sigma'} = {^t}g^{-1} \circ \beta_{\Sigma} \circ {^t}g. \]
  Hence, 
  \[  g \in \mathrm{Stab}_{\GL_2(E)^{\det}}(\Sigma) \subset \GL_2(E)^{\det} \Longleftrightarrow {^t}g^{-1} \in \Aut_E(E^2, Q_{\Sigma}, \beta_{\Sigma}). \]
  \end{prop}
 \vskip 5pt
 
 In particular, if $F$ is a local field, then the $M_E(F)$-orbits of generic unitary characters of $N_E(F)$ are parametrized by $E$-twisted composition algebras (modulo $E$-isomorphisms). Likewise, when $F$ is a number field, the $M_E(F)$-orbits of (abelian) Fourier coefficients along $N_E$ are parametrised by   $E$-twisted composition algebras (modulo $E$-isomorphisms).
 \vskip 5pt
 
 We shall not need the general procedure to pass from $\Sigma$ to $(Q_{\Sigma},\beta_{\Sigma})$, but only for the so-called reduced cubes:
 
 \begin{prop} \label{P:reduced}
 (i) If  $\Sigma= (1, 0, f, b) \in V_E(F)$ (such a $\Sigma$ is called a reduced cube), then its associated pair $(Q_{\Sigma}, \beta_{\Sigma})$ is given by:
 \[  Q_{\Sigma}(x,y) = -fx^2 - bxy  + f^{\#} y^2
  \]
  and 
  \[  \beta_{\Sigma}(x,y) =(-by^{\#} - (fx) \times y , x^{\#} + fy^{\#})
  \] 
  so that $\beta_{\Sigma}(1,0)  = (0,1)$.
 
 \vskip 5pt
 
 \noindent (ii)  Conversely,  let $(C,Q,\beta)$ be 
  an $E$-twisted composition algebra of $E$-dimension 2. For  $v\in C$, set $\Delta(v):=N_C(v)^2 - 4 N_{E/F}(Q(v))\in F$.  
  Then  there exists $v\in C$ such that  $\Delta(v)\neq 0$. Moreover, the set  $\{ v, \beta(v)\}$ is an $E$-basis of $C$ if and only if $\Delta(v)\neq 0$. 
   Given such a $v \in C$ and identifying $C$ with $E\oplus E$ using the basis $\{ v,\beta(v) \}$,  
  the pair $(Q,\beta)$ corresponds to the reduced cube $(1,0,-Q(v),-N_C(v))$ under the recipe in (i). 
  \end{prop} 
  \vskip 5pt
  
  We record a corollary which will be used later, concerning isomorphisms between rank 2 twisted composition algebras:
  \vskip 5pt
  
   \begin{cor} \label{C:phs}
    Let $(C,Q,\beta)$ be an $E$-twisted composition algebra of $E$-dimension 2. Let $f\in E$ and $b\in F$, such that $b^2 + 4 N_{E/F}(f)\neq 0$. 
 Then the set of
 \[  \Omega_{C,f,b} := \{ v\in C: \,  \text{$Q(v)=-f$ and $N_C(v)=-b$} \} \]
 is a principal homogeneous space for $\Aut_E(C)$, which contains an $F$-rational point    if and only if 
 $(C,Q,\beta)$ is isomorphic to the $E$-twisted composition algebra $C_{\Sigma} = (E^2, Q_{\Sigma}, \beta_{\Sigma})$ defined by the reduced cube $\Sigma=(1,0,f,b)$.  
 Indeed,   there is an $\Aut_E(C)$-equivariant isomorphism 
 \[  \mathrm{Isom}_E(C_{\Sigma}, C) \longrightarrow  \Omega_{C, f,b} \]
 defined by
 \[  \phi \mapsto  \phi(1,0). \]
 \end{cor} 
 
 \begin{proof}  
 An $E$-linear isomorphism $\phi: C_{\Sigma} \longrightarrow C$ is determined by $v = \phi(1,0)$ (for $\phi(0,1)$ has no choice but to be equal to $\beta(v)$) and this $v \in C$ must satisfy 
 \[  \text{  $Q(v)=- f$, and $N_C(v)= -b$.  } \]
 Conversely, when $v\in C$ satisfies these two conditions, one checks using \cite[\S 3.1 and Lemma 3.2, eqn. (3.4)]{GS2} that the map $\phi$ 
 given by $\phi(1,0) = v$ and $\phi(0,1) = \beta(v)$   is an isomorphism of twisted composition algebras.
  \end{proof} 

Observe that $\mathrm{Isom}_E(C_{\Sigma}, C)$ has an action of $\Aut_E(C_{\Sigma}) \times \Aut_E(C)$ for which it is a torsor for each of the two factors. Hence, assuming $\mathrm{Isom}_E(C_{\Sigma}, C) $ is nonempty and  after fixing a base point $\phi_0 \in \mathrm{Isom}_E(C_{\Sigma}, C)$, one obtains an isomorphism 
\[ {\mathrm Ad}(\phi_0) :  \Aut_E(C_{\Sigma}) \cong \Aut_E(C). \]
 By transport of structure, we also see that $\Omega_{C,f,b}$ carries an action of $\Aut_E(C_{\Sigma}) \times \Aut_E(C)$. Let us describe the action of  
 $\Aut_E(C_{\Sigma})\cong \mathrm{Stab}_{\GL_2(E)^{\det}}(\Sigma)$ on $\Omega_{C, f,b}$ concretely. 
 \vskip 5pt
 
 \begin{lemma}  \label{L:phs}
 Given
 \[ g  = \left( \begin{array}{cc}
 p & q \\
 r & s \end{array} \right)  \in \mathrm{Stab}_{\GL_2(E)^{\det}}(\Sigma), \]
  so that ${^t}g^{-1} \in \Aut_E(E^2, Q_{\Sigma}, \beta_{\Sigma})$, and $v \in \Omega_{C, f,b}$ associated to $\phi \in \mathrm{Isom}_E(C_{\Sigma}, C)$,
one has
 \[ g \cdot v  = \phi \left( {^t}g \cdot (1,0) \right) = \phi(p,q) = p v + q \beta(v) \in \Omega_{C, f,b}. 
 \]
 \end{lemma}
 \vskip 5pt

 \subsection{\bf Degenerate cubes}
 It will be useful to have an understanding of the degenerate $M_E(F)$-orbits on $V_E(F) = N_E(F)/Z(F)$.
  The nontrivial degenerate orbits correspond to the nilpotent $G_E$-orbits  which are denoted by $A_1$, $2A_1$ and $3A_1$ in the 
Bala-Carter classification. Accordingly, we shall say that the corresponding elements in $V_E(F)$   are of rank 1, 2 or 3. We may refer to generic elements 
(non-degenerate cubes)  as rank 4  elements. The set of elements in $V_E$ of rank $\leq k$ is a Zariski closed subset.
For example, the elements of rank 1 are precisely the highest weight vectors, and the set of elements of rank $\leq 1$ can be described by a system of equations given in Proposition \ref{P:orbit} below (see also \cite[Prop. 11.2]{GS1}). 
\vskip 5pt

 We shall now describe  the $M_E(F)$-orbits of elements of rank 2 and 3.   
\vskip 5pt

\begin{prop}   \label{P:rank2or3} 
\begin{enumerate} 
\item Every  $M_E(F)$-orbit of rank 3 elements in $V_E= F\oplus E \oplus E \oplus F$ contains an element $(0,0, e, 0)$ where $e\in E^{\times}$.  
Two rank 3 elements $(0,0, e, 0)$ and $(0,0,f,0)$ belong to the same orbit if and only if $e/f \in F^{\times} E^{\times 2}$.   
\vskip 5pt

\item  Every  $M_E(F)$-orbit of rank 2 elements in $V_E= F\oplus E \oplus E \oplus F$ contains an element $(1,0, e, 0)$ where $e\in E$ such that 
$e\neq 0$ and $e^{\#}=0$. 
Two rank 2 elements $(1,0, e, 0)$ and $(1,0,f,0)$ belong to the same orbit if and only if $e/f \in (F^{\times})^2$.   
\end{enumerate} 
\end{prop} 
\begin{proof}  (1) Consider $\Sigma=(0,0,1,0)$. This element has rank 3 since, over $\bar F$, $1=(1,1,1) \in \bar F^3$ 
sits across three orthogonal root spaces, hence the notation 
$3A_1$.   A long but fascinating computation shows that the stabilizer $S_{M_E}(\Sigma)$ of $\Sigma$ in $M_E$ consists of all elements 
\[ 
\left( \begin{array}{cc} 
a & 0 \\
b & d \end{array} 
\right) 
\] 
where $ad\in F^{\times}$,  $d/d^{\#} =1$ and $T_{E/F}(b d^{\#}) =0$.   Let $T_E \subset M_E$ be the maximal torus of diagonal matrices in $M_E$.  The 
stabilizer $S_{T_E}(\Sigma)$ of $\Sigma$ in $T_E$ consists of matrices as above with $b=0$. Since 
\[ 
H^1(F, S_{M_E}(C)) = H^1(F, S_{T_E}(C)) 
\] 
it suffices to classify the orbits of $T_E$ on elements of the type $(0,0,e,0)$ where $e\in E^{\times}$.  On these elements, the diagonal matrices 
act by multiplication by $d/d^{\#}$. Since the set of all $d/d^{\#}$ is $F^{\times} E^{\times 2}$, (1) follows.  Statement (2) is proved in the same way, and we leave details 
to the reader.   

\end{proof}

\vskip 10pt 
\noindent 
Remark: If $E$ is a field, the set of $e\in E$ such that $e^{\#}=0$ consists only of $0$, so that  there are no rank 2 elements in $V_E$. If $E = F \times K$ with $K$ a field, the set of such $e$'s is one $F$-line, and it consists of  three $F$-lines if $E=F^3$.  This reflects the fact that $G_E(\bar F)$ has three orbits with Bala-Carter notation $3A_1$, 
permuted by the group of outer automorphisms.

 \vskip 10pt
 
\section{\bf Dual Pairs}  \label{S:dual}
In this section, we introduce the various dual pairs which we will study in this paper. In particular, we shall see that given a $E$-twisted composition algebra $C$, with corresponding embedding $i:  E \hookrightarrow J$ under the Springer decomposition, one may construct a dual pair:
\[  H_C \times  G_E = \Aut_E(C) \times \Spin_8^E \subset G_J, \]
where $G_J$ is a group we shall introduce in due course.
We shall first construct this dual pair on the level of Lie algebras. 
\vskip 10pt

\subsection{\bf Lie algebras.}
Let us begin with an arbitrary Freudenthal-Jordan algebra $J$ (not necessarily of dimension 9). 
Let $\mathfrak l_J\subset \End(J)$ be the Lie subalgebra preserving the trilinear form $(-,-,-)$ associated to the norm form $N_J$, i.e.  $a\in \End(J)$ lies in $\mathfrak l_J$ if and only if 
\[ 
(a\cdot x,y,z) + (x,a\cdot y,z) + (x,y,a\cdot z) =0 
\] 
for all $x,y,z\in J$. The trace form defines an involution $a\mapsto a^{\top}$ on $\mathfrak l_J$ by 
\[ 
\langle a\cdot x, y\rangle = \langle x, a^{\top}\cdot y\rangle 
\] 
for all $x,y\in J$. 
 \vskip 5pt 
With $\mathfrak{h} = \mathfrak{sl}(V)$ for $V$ a 3-dimensional vector space, the space 
\[ 
\mathfrak g_J= \mathfrak h \oplus \mathfrak l_J \oplus (V \otimes J) \oplus (V^{\ast} \otimes J)
\] 
has the structure of a simple Lie algebra, such that the above decomposition arises from a $\mathbb Z/3\mathbb Z$-grading. 
The brackets $[\mathfrak h \oplus \mathfrak l_J , V\otimes J]$ and $[\mathfrak h \oplus \mathfrak l_J , V^*\otimes J]$ are given by 
the natural action of $\mathfrak h \oplus \mathfrak l_J$ on  $V\otimes J$ and $V^*\otimes J$, with the action of $a\in \mathfrak l_J$ on the 
second factor of $V^*\otimes J$  is given by that of $-a^{\top}$. 
 The brackets 
 \[  \text{ $[V\otimes J, V\otimes J]\subseteq V^*\otimes J$ and  $[V^*\otimes J, V^*\otimes J]\subseteq V\otimes J$ } \]
 are defined by 
\[ 
[v\otimes x,u\otimes y]= - (v\wedge u)\otimes (x\times y) 
\] 
\[  
[v^*\otimes x,u^*\otimes y]= (v^*\wedge u^*)\otimes (x\times y) 
\] 
 respectively.
 \vskip 5pt
 
  The remaining bracket (between $V \otimes J$ and $V^{\ast} \otimes J$) is determined by the invariant Killing form. More precisely, the Killing form on $\mathfrak g_J$ is an extension of the  Killing form on  $\mathfrak h \oplus \mathfrak l_J$  (we shall specify the normalization later),  such that 
\[ 
\langle v\otimes x, u^* \otimes y\rangle = \langle v,u^*\rangle \cdot \langle x,y\rangle 
\] 
if $v\otimes x\in V\otimes J$ and $u^*\otimes y\in V^*\otimes J$,  
where $\langle v,u^*\rangle$ is the evaluation of $u^*$ on $v$ and $\langle x,y\rangle$ is the trace pairing on $J$.  Then the  bracket 
$[V\otimes J, V^*\otimes J]\subseteq \mathfrak h \oplus \mathfrak l_J$ is completely determined by:
\[ 
\langle [x,y],z\rangle = \langle[z,x], y\rangle 
\] 
for any $x,y,z\in \mathfrak g_J$.  We refer the reader to \cite{Ru} for explicit formulae in this case. However, if $ \langle v, u^*\rangle =0$, the bracket of 
$v\otimes x\in V\otimes J$ and $u^*\otimes y\in V^*\otimes J$  is contained in $\mathfrak h$, and is given by 
\[ 
[v\otimes x ,u^*\otimes y]=\langle x, y \rangle \cdot v\otimes u^* \in \mathfrak{sl}(V) 
\] 
and/or 
\[ 
[u^*\otimes y , v\otimes x ]=\langle x, y \rangle  \cdot u^*\otimes v \in \mathfrak{sl}(V^{*}) 
\] 
Explicitly, if $i\neq j$, 
\[ 
[e_i\otimes x,e_j^*\otimes y ]=\langle x, y \rangle  e_{ij}
\] 
\[
[e_j^*\otimes y ,e_i\otimes x]= \langle x, y \rangle  e^*_{ji}.  
\] 
We highlight two cases here:
\vskip 5pt

\begin{itemize}
\item[(a)]  If $J=F$, considered as a cubic algebra, so that $1\times 1= 2$ and $T_F(1)=3$,  then this construction returns  the simple split algebra $\mathfrak{g}$ of type $G_2$.
\vskip 10pt

\item[(b)]  If $J = E$ is a cubic \'etale algebra, then $\mathfrak l_E=E^0$, the subspace of trace 0 elements in $E$. The action of $x\in E^0$ on $e\in E$ is 
$x\cdot e= -2 xe$.  We fix a symmetric bilinear form on $\mathfrak l_E$ by $\langle x,x \rangle = 2 \cdot T_{E}(x^2)$. Then the Lie algebra $\mathfrak{g}_E$ is of type $D_4$; it is the Lie algebra of the group $G_E = \Spin_8^E$.
\end{itemize}

\subsection{Groups} 
 In order to explain the two appearances of 2 in (b) above, 
 let $J=E\oplus C$, where $C$ is $E$-twisted composition algebra (of arbitrary rank).  For $\alpha \in E^{\times}$, let $c_{\alpha}: J \rightarrow J$ be defined by  
\[ 
c_{\alpha} : (e,v) \mapsto (\alpha^{\#}/\alpha \cdot e, \alpha\cdot  v) 
\] 
for all $(e,v)\in E\oplus C$.  By (38.6) in  \cite{KMRT},  one has
\[ 
N_J((e,v))= N_E(e)+ N_C(v) -T_E( e \cdot Q(v)),
\]  
and it readily follows that 
\[  N_J(c_{\alpha}(e,v)) = N_E(\alpha) \cdot N_J(e,v), \]
so that $c_{\alpha}$ is a similitude map of $N_J$  with similitude  factor $N_E(\alpha)$. In particular,  if $\alpha$ has norm 1, then $c_{\alpha}$ preserves the norm $N_J$. Since $\alpha^{\#}=\alpha^{-1}$ (if $N_E(\alpha) =1$),    we can write 
 $c_{\alpha}(e,v)=(\alpha^{-2} e, \alpha v)$.   By passing to Lie algebras, we get 
 an embedding $\mathfrak l_E=E^0 \subseteq \mathfrak l_J$ where $x\in E^0$ acts on $J=E\oplus C$ by  
 \[ 
 x\cdot (e,v)= (-2xe,v) +(e,xv). 
 \] 
By setting $v=0$, we get the previously defined action of $\mathfrak l_E=E^0$ on $E$.
\vskip 5pt

On the other hand,  we fix the $\Aut(\mathfrak l_J)$-invariant form on 
$\mathfrak l_J$ so that the restriction to $\mathfrak l_E$ is $2 \cdot T_E(x^2)$. For example,  suppose that $J=M_3(F)$ and 
$E=F^3$ is diagonally embedded in $M_3(F)$. Then 
$\mathfrak l_J=\mathfrak{sl}_3 \oplus \mathfrak{sl}_3$, so that an element $(x,z) \in \mathfrak{sl}_3 \oplus \mathfrak{sl}_3$ acts on $y\in M_3(F)$ by 
$xy-yz$, and $\mathfrak l_E$ is the set of trace zero diagonal matrices $x$ embedded in $\mathfrak{sl}_3 \oplus \mathfrak{sl}_3$ as $(-x,x)$. 
\vskip 10pt

We embed $\Aut_E(C) \subset \Aut(J)$ so that it acts trivially on $E$, the first summand in $J=E\oplus C$.
\vskip 5pt

 \begin{prop} \label{P:isogeny} 
  Let $J=E\oplus C$.    
Every $F$-rational similitude map of $N_J$ commuting with the algebraic group $\Aut_E(C)$ is equal to $c_{\alpha}$ for some $\alpha\in E^{\times}$.  Likewise, every $F$-rational similitude map of $N_J$ commuting with the algebraic group $\Aut(J)$ is equal to $c_{\alpha}$ for $\alpha\in F^{\times}$. 
\end{prop} 
\begin{proof} 
Let $g$ be a $F$-rational similitude of $N_J$ commuting with $\Aut_E(C)$. Then $g$ preserves both summands $E$ and $C$ of $J$.  The algebra of $F$-rational endomorphisms  of $C$ commuting with the action of $\Aut_E(C)$ is $E$.  
Thus  $g=c_{\alpha}$ on $C$, for some $\alpha\in E^{\times}$. Let $g'=c_{\alpha^{-1}} \circ g$. Clearly, $g'$ belongs to the similitude group of $N_J$; however, since 
 $g'(0,v)=(0,v)$ for all $v\in C$,  the similitude factor is 1, i.e. $g'$ preserves $N_J$.  
 \vskip 5pt
 
 Now fix $e\in E$. Then  $g'(e,v)=(e',v)$ for all $v\in C$ and some $e'\in E$. 
 We want to show that $e=e'$. It suffices to do so over the algebraic closure $\bar F$. Since $g'$ preserves $N_J$, use $v=0$ to show first that 
$N_E(e)=N_E(e')$, and then 
$T_E(e\cdot Q(v))=T_E(e'\cdot Q(v))$ for all $v\in C$.  Since $Q$ is surjective over $\bar F$,  $T_E(ee'')=T_E(e'e'')$ for all $e''\in E\otimes \bar F$. Hence $e=e'$. 
Finally, if $g$ is a similitude that commutes with $\Aut(J)$, then it commutes with $\Aut_E(C)\subseteq \Aut(J)$, so $g=c_{\alpha}$. Since $\Aut(J)$ acts absolutely irreducibly on $J^0$, the space of trace 0 elements in $J$, $\alpha\in F^{\times}$. 
\end{proof} 

 \vskip 5pt 
 Let $G_J=\Aut(\mathfrak g_J)$. We note that $G_J$ is not necessarily connected.
  From the construction of the Lie algebra $\mathfrak g_J$, it is evident that $\Aut(J)\subseteq G_J$. 
Assume, furthermore,  that  $J=E\oplus C$ and $J\neq E$.  
The natural action of $\Aut_E(C)$ on $C$, extended trivially to $E\subset J$ gives an embedding $\Aut_E(C)\subset \Aut(J)$. 
 Hence we have a natural embeddings 
 \[ 
 \Aut_E(C)\subset \Aut(J)  \subset G_J.
 \]  
 We have also constructed inclusions of  $\mathfrak g \subseteq \mathfrak g_E \subseteq \mathfrak g_J$ of vector spaces.

 \begin{prop}  \label{P:include}
 The  inclusions  $\mathfrak g \subseteq  \mathfrak g_E \subseteq \mathfrak g_J$ are homomorphisms of Lie algebras, thus giving rise to inclusion of algebraic groups
 \[  G_2 \subset G_E = \Spin_8^E \subset G_J. \]
\end{prop} 
\begin{proof} 
Let $x,y\in E$. The cross product $x\times y$, computed in $J$,  is the same as the one computed in $E$. Hence 
 the bracket $[V\otimes E , V\otimes E]$ in $\mathfrak g_J$ coincides with the one in $\mathfrak g_E$. The bracket 
   $[V\otimes E , V^*\otimes E]$,  computed in $\mathfrak g_J$,  is fixed by $\Aut_E(C)(\bar F)$ hence  it is  contained in $\mathfrak h\oplus \mathfrak l_E$. 
   Since the Killing form on $\mathfrak h\oplus \mathfrak l_E$ is the restriction of the Killing form on $\mathfrak h\oplus \mathfrak l_J$ it follows, from the definition of the Lie brackets, that the two Lie brackets coincide. This shows that the inclusion $ \mathfrak g_E \subseteq \mathfrak g_J$ is a homomorphism. A similar 
   argument shows that the inclusion $\mathfrak g_E \subseteq \mathfrak g_J$. Indeed, the bracket 
   $[V\otimes F , V^*\otimes F]$,  computed in $\mathfrak g_J$,  is fixed by $\Aut(J)$ hence  it is  contained in $\mathfrak h$. 
   \vskip 5pt
   
  The inclusion of Lie algebras induce a corresponding inclusion of the corresponding connected algebraic subgroups of $G_J$, and we know  what these algebraic subgroups are up to isogeny.  It is clear that the algebraic subgroup associated to $\mathfrak{g}$ is $G_2$. 
   Over $\bar F$, under the adjoint action of $\mathfrak{g}_E$,   the algebra $\mathfrak g_J$ contains the three 8-dimensional fundamental representations of $\Spin_8$,  each 
occurring with multiplicity $\dim_E(C)$. This shows that the connected algebraic subgroup corresponding to $\mathfrak{g}_E$ is simply-connected and is thus isomorphic to $G_E = \Spin_8^E$. 
 \end{proof} 

 \subsection{Relative root system} 
We fix a basis $e_1, e_2, e_3$ of $V$ and let $\mathfrak t\subset \mathfrak h$ be the Cartan subalgebra consisting of diagonal matrices, with respect to this basis of $V$. Under the adjoint action of $\mathfrak t$, 
\[ 
\mathfrak g_J = \mathfrak g_{J,0} \oplus ( \bigoplus_{\alpha\in \Phi} \mathfrak g_{J,\alpha}) 
\] 
where $\Phi \subset \mathfrak t^*$ is a root system of type $G_2$. Note that 
\[  \mathfrak g_{J,0}=\mathfrak t \oplus \mathfrak l_J.\]
 The short root spaces are $Fe_i\otimes J$ or $Fe_i^*\otimes J$, so we have canonical identifications  with $J$ given by $x\mapsto e_2\otimes x$ and 
$x \mapsto e_2^* \otimes x$ respectively. 
 The long root spaces are one-dimensional and contained in $\mathfrak h$. In particular,  there are two choices for the basis vector : $e_{ij}$ or $e^*_{ji}$ ( $=-e_{ij}$  under the identifications
 $\mathfrak h=\mathfrak{sl}(V)=\mathfrak{sl}(V^*)$).  

\vskip 5pt
 In particular, when $J = E$,  $\mathfrak g_{E,0} =\mathfrak t \oplus \mathfrak l_E$ is 
a torus, and by choosing a set of positive roots in $\Phi$, we have constructed a Borel subalgebra in $\mathfrak g_E$, so that $\mathfrak g_E$ is quasi-split. 
Indeed, we have mentioned before that $\mathfrak{g}_E$ is the Lie algebra of $\Spin_8^E$. What we have done here is to give a direct construction of this Lie algebra, recover  some of the structure theory described in \S \ref{S:structure} from this construction and show that this Lie algebra fits into a family of such Lie algebras which is associated to a Freudenthal-Jordan algebra $J$.

\subsection{Two step parabolic subalgebra} \label{ss_parabolic} 
Let $s\in \mathfrak{sl}(V)$ be the diagonal matrix $(1,0,-1)$. The adjoint action of $s$ on $\mathfrak g_J$ gives a $\mathbb Z$-grading 
\[ 
\mathfrak g_J = \oplus_{n\in \mathbb Z } ~\mathfrak g_{J}(n). 
\] 
Then $\mathfrak g_J(n)\neq 0$ only for $n=-2,-1,0,1,2$. Let 
\[  \mathfrak m=\mathfrak g_J(0) \quad \text{and} \quad  \mathfrak n=\mathfrak g_J(1) \oplus \mathfrak g_J(2). \] 
Then $\mathfrak p=\mathfrak m \oplus \mathfrak n$ is a maximal parabolic subalgebra, with Levi subalgebra $\mathfrak{m}$ and nilpotent radical $\mathfrak n$. 
Let us examine the structure of each of these parts in turn.
\vskip 5pt

The Levi subalgebra $\mathfrak{m}$ has a decomposition
\[  \mathfrak{m}  = \mathfrak{t} \oplus \mathfrak{l}_J \oplus e_2 \otimes J \oplus e_2^* \otimes J. \]
The derived algebra 
\[ [\mathfrak m, \mathfrak m] = \mathfrak{l}_J \oplus e_2 \otimes J \oplus e_2^* \otimes J \]
 is generated by short root spaces  $e_2 \otimes J$ and $e^*_2\otimes J$.
The above decomposition also exhibits a (Siegel-type) parabolic subalgebra 
\[  \mathfrak{s} = ( \mathfrak{t} \oplus \mathfrak{l}_J) \oplus e_2 \otimes J  \subset \mathfrak{m} \]
with abelian nilpotent radical $e_2 \otimes J$. 
\vskip 5pt 
Considering now the nilradical $\mathfrak{n}$,  the center of 
$\mathfrak n$ is $[\mathfrak n,\mathfrak n]=\mathfrak g_J(2) = F e_{13}$. As an $\mathfrak m$-module, the quotient $\mathfrak n/[\mathfrak n,\mathfrak n]$ 
is isomorphic to 
\[ 
\mathfrak g_J(1)= Fe^*_{21} \oplus Fe_1\otimes J \oplus Fe^*_3 \otimes J \oplus F e_{23} = F\oplus J \oplus J \oplus F.
\] 
Henceforth, an element in $\mathfrak g_J(1)$ is a quadruple $(a,y,z,d)$ where $a,d\in F$ and  $y,z\in J$. 
  Using our formulae, we can describe this $\mathfrak{m}$-module. One sees that the Lie bracket of  $e_2 \otimes x \in e_2 \otimes J$ and $(a,y,z,d)$ is
\[ 
[e_2\otimes x,  (a,y,z,d)]= (0,ax, x\times y, \langle x, z\rangle) 
\] 
and the Lie bracket of  $e^*_2 \otimes x \in e_2^* \otimes J$ and  $(a,y,z,d)$ is
\[ 
[e^*_2\otimes x, (a,y,z,d)]= ( \langle x, y\rangle, x\times z,  dx, 0 ). 
\] 

If $J=E$, a cubic etal\'e algebra, then $\mathfrak g_E(1)$ is the space of $E$-twisted Bhargava cubes and  
$[\mathfrak m,\mathfrak m]$ is identified with $\mathfrak{sl}_2(E)$  by 
\[ 
\left(\begin{array}{cc} 
0 & x \\
0 & 0 \end{array} \right) \mapsto e_2 \otimes x 
\text { and } 
\left(\begin{array}{cc} 
0 & 0 \\
x & 0 \end{array} \right) \mapsto e^*_2 \otimes x
\] 
Let $P_J = M_J N_J$ be the parabolic subgroup associated to $\mathfrak{p}_J$. If we fix an embedding $E \hookrightarrow J$ of Jordan algebras, then we have a corresponding embedding
$\mathfrak{p}_E \hookrightarrow \mathfrak{p}_J$ of parabolic subalgebras such that   
\[  G_E \cap P_J = P_E \]
on the level of groups.
\vskip 10pt

\subsection{\bf 3-step parabolic subalgebra}  \label{SS:3step}
Now let $s\in \mathfrak{sl}(V)$ be the diagonal matrix $(1,1,-2)$. As above, the adjoint action of $s$ on $\mathfrak g_J$ gives a $\mathbb Z$-grading 
\[ 
\mathfrak g_J = \oplus_{n\in \mathbb Z } ~\mathfrak g_{J}(n). 
\] 
Then $\mathfrak g_J(n)\neq 0$ only for $n=-3,\ldots , 3$. Let 
\[ \mathfrak l=\mathfrak g_J(0) \quad \text{and} \quad  \mathfrak u=\mathfrak g_J(1) \oplus \mathfrak g_J(2)\oplus \mathfrak g_J(3). \] 
Then $\mathfrak q=\mathfrak l \oplus \mathfrak u$ is a parabolic subalgebra whose nilradical  $\mathfrak u$ is 3-step nilpotent. Note that 
\[ 
\mathfrak g_J(1)= F e_1\otimes J \oplus F e_2 \otimes J , \quad ~\mathfrak g_J(2)= F e_3^*\otimes J \quad \text{ and } \quad \mathfrak g_J(3)= Fe_{13} \oplus F e_{23}. 
\] 
 Let $Q_J=L_JU_J$ be the corresponding  parabolic subgroup in $G_J$. Thus, the unipotent radical $U_J$ has a filtration 
 \[ U=U_1\supset U_2\supset U_3 \quad \text{such that  $U_i/U_{i+1} \cong \mathfrak g_J(i)$ for all $i$.}\]
If we fix an embedding $E \hookrightarrow J$, then we have a corresponding embedding $\mathfrak{q}_E \hookrightarrow \mathfrak{q}_J$ of parabolic subalgebras such that 
\[   G_E \cap Q_J  = Q_E. \]
on the level of groups.

\vskip 10pt

\subsection{\bf See-saw dual pairs.}  \label{SS:seesawdp}
 To summarise the discussion in this section,  relative to an embedding $E \hookrightarrow J$,  we 
have constructed the following see-saw of dual pairs in $G_J$:

 \begin{picture}(100,82)(-130,10) 

\put(24,24){$G_2$} 

\put(75,24){$H_C=\Aut_E(C)$}

\put(37,36){\line(1,1){35}}

\put(24,74){$G_E$}
\put(37,71){\line(1,-1){35}}

\put(75,74){$H_J = \Aut(J)$}

\end{picture}
\vskip 15pt

We highlight two cases:
\vskip 5pt

\begin{itemize}
\item The particular case of interest in this paper is the case when $\dim_E C = 2$ or equivalently $\dim_F J = 9$.
 In this case, $G_J$ and $\Aut_E(C)$ are disconnected and we have a short exact sequence
 \[  \begin{CD}
  1 @>>>  \underline{G}_J^0  @>>> \underline{G}_J @>>>  S_2 @>>>1  \end{CD} \]
  where the identity component $G_J^0$ is an adjoint group of type  $E_6$ and whose inner class correspond to the quadratic algebra $K_J$. 
Note that on taking $F$-points, we have a map
 \[  G_J = \underline{G}_J(F) \longrightarrow S_2 \]
 which need not be surjective.   
 \vskip 10pt
 
 \item When $\dim_E C  =4$ (i.e. $\dim_F J  =15$),  then $G_J$ is an adjoint group of type $E_7$ associated to a quaternion $F$-algebra $B$. In this case, 
 \[  \Aut_E(C)  \cong  \left( {\Res}_{E/F} (B\otimes_F E)^{\times}  \right)^{\det} / F^{\times} \]
 where the RHS consists of elements in $(B \otimes E)^{\times}$ whose norm lies in $F^{\times}$. 
 \end{itemize}
\vskip 15pt

\section{\bf Levi Factor}  \label{S:levi}

In this section, we investigate some further properties of the dual pair  $H_C \times G_E$ in $G_J$, with $J = E \oplus C$ and $\dim_EC =2$.
The group $G_J$ has a (Heisenberg) maximal parabolic subgroup $P_J = M_J N_J \supset P_J^0  = M_J^0 \cdot N_J$,  whose Levi factor $M_J^0$ is of type $A_5$. Moreover,
\[  (H_C \times G_E) \cap P_J  =   H_C \times P_E,  \]
so that
\[  H_C \times M_E \longrightarrow M_J \]
is itself a dual pair in $M_J$. Indeed, if we intersect the seesaw diagram in \S \ref{SS:seesawdp} with $M_J$, we obtain the following seesaw diagram in $M_J$:

 \begin{picture}(100,82)(-130,10) 

\put(12,24){$\GL_2(F)$} 

\put(75,24){$H_C=\Aut_E(C)$}

\put(37,36){\line(1,1){35}}

\put(12,74){$\GL_2(E)^{\det}$}
\put(37,71){\line(1,-1){35}}

\put(75,74){$H_J = \Aut(J)$}

\end{picture}

\noindent  For our purposes, when $J$ is not a division algebra, we need to describe the Levi subgroup $M_J$ and the above embedding concretely. 
This is because of the need to relate the theta correspondence associated to $H_C \times M_E$ to a classical similitude theta correspondence. We treat the various cases in turn.

\vskip 5pt

\subsection{\bf Split case.}
 Suppose first that $J = M_3(F)$, so that $G^0_J$ is split.  In this case,  
\[  M^0_J= (\GL_1 \times \SL_6)/\mu_6 \]
 where $\mu_6$ is viewed as a subgroup of  $\GL_1 \times \SL_6$ by the map 
 \[  x\mapsto (x^3,x). \] 
A more convenient description is:
\[ 
M^0_J \cong (\GL_1 \times  \GL_6)/ \GL_1
\] 
 where  $\GL_1$  is viewed as a subgroup of   $\GL_1 \times \GL_6$ by the map $x\mapsto (x^3,x)$.   The character 
\[  \chi (x, g) = \det(g)/x^2 \]
 of $\GL_1 \times  \GL_6$ descends to $M^0_J$ and is a generator of $\Hom(M^0_J, \mathbb G_m)$. The character $\chi$ arises naturally when 
$M^0_J$ acts by conjugation on the center of $N_J$. 
\vskip 5pt

 If we identify $F^6=E^2$ (by choosing an $F$-basis of $E$), then 
 $M_E = \GL_2^{\det}(E)$ is naturally a subgroup of $\GL_6$. We define an embedding  $\GL_2(E)^{\det} \longrightarrow M_J$  by the map 
 \[  g\mapsto (\det(g), g). \]
 Note that $\chi(\det(g), g) =\det(g)$ since the determinant of $g$, viewed as an element in $\GL_6$ is $\det(g)^3$. 
On the other hand, since $K_J = F \times F$, one has $H_C^0 \cong E^{\times}/ F^{\times}$. 
The right-multiplication action of  $e\in E^{\times}$ on $E^2$  gives an embedding $E^{\times} \longrightarrow \GL_6$, so that any element  $e \in E^{\times}$ can be viewed as an element of $\GL_6$ denoted by the same letter. Thus we have a map 
$E^{\times}  \rightarrow \GL_1 \times \GL_6$ given by
\[  e\mapsto (N_{E/F}(e), e). \]  
 If $e\in F^{\times}$,  then the image is $(e^3,e)$. The map thus descends to an inclusion of 
$E^{\times}/F^{\times} \longrightarrow M_J$ and we have defined an embedding
\[   H^0_C \times M_E =  E^{\times}/ F^{\times}  \times \GL_2(E)^{\det} \hookrightarrow M_J^0 \]
when $J = M_3(F)$. Note that the character $\chi$ of $M^0_J$ is trivial on $E^{\times}/F^{\times}$. 

\vskip 5pt

\subsection{\bf Quasi-split case.}  \label{SS:qs}
Consider now the case when  $J=J_3(K)$, so that $G_J$ is quasi-split but not split. In this case,   
\[ \underline{M}_J^0 \cong  (\GL_1 \times \SU^K_6)/ {\Res}^1 \mu_{6,K} \]
where ${\Res}^1 \mu_{6,K} = {\Ker}( N_{K/F}:  {\Res}_{K/F} \mu_6 \rightarrow \mu_6)$  is viewed as a subgroup of   $\GL_1 \times \SU^K_6$ by the map $x\mapsto (x^3,x)$. 
 \vskip 5pt
 
   Fix an involution  $g\mapsto g^*$ 
 of $\GL_6(K)$ that defines the quasi-split form $\U^K_6$.  In particular, 
 $\det(g^*)=\det(g)^{-1}$ and $x^*=x^{-1}$ for any scalar matrix $x\in \GL_6$. 
 Consider the involution 
 \[ 
 \tau:  ( x, g) \mapsto (x\det(g)^{-1}, g^*)
 \] 
of $\GL_1 \times \GL_6$.  Since $\tau(x^3,x)=( x^{-3},x^{-1})$, for every $x\in\GL_1$, the involution $\tau$ descends to the quotient $ (\GL_1 \times \GL_6)/\GL_1$. 

\vskip 5pt

 Now $\underline{M}^0_J$ is the subgroup of 
 \[ {\Res}_{K/F} (\underline{M}^0_J \times_F K) \cong {\Res}_{K/F} ( \GL_1 \times \GL_6 / \GL_1) \] 
 fixed under the Galois action twisted by $\tau$. 
From our knowledge in the split case, we deduce  an exact sequence of algebraic groups, 
\[ 
1 \rightarrow \U_1^K \rightarrow (\Res_{K/F} \mathbb G_m \times\U_6^K )^{\dagger}\rightarrow M^0_J \rightarrow 1 
\] 
where $(\Res_{K/F} \mathbb G_m \times\U_6^K )^{\dagger}$ is the subgroup consisting of pairs $(x,g)$ such that 
\[ 
x/\sigma(x) = \det(g) \quad \text{with $1 \ne \sigma \in {\Aut}(K/F)$}.  
\] 
On the level of $F$-points, one has
\[ \begin{CD}
1 @>>> K^1 \rightarrow (K^{\times}  \times\U_6^K(F) )^{\dagger} @>>>  M^0_J(F)  @>>>  H^1(F, U^K_1) \cong F^{\times}/N_{K/F}(K^{\times}).
\end{CD} \] 
 Let 
\[ 
M^0_{J,K} = (K^{\times} \times \U_6^K(F))^{\dagger}/ K^1. 
\] 
so that $M^0_J(F)/M^0_{J,K} \subseteq F^{\times}/N_{K/F}(K^{\times})$.  We claim  that this is an isomorphism. 
The condition $x/\sigma(x) = \det(g)$ implies that $\chi(x,g)\in N_{K/F}(K^{\times})$, for all $(x,g) \in M^0_{J,K}$. On the other hand, the character 
$\chi: M^0_J(F) \rightarrow F^{\times}$ is surjective, and the claim follows.  Thus, we have an exact sequence of topological groups
\[  \begin{CD} 
1 @>>>  M^0_{J,K}  @>>> M^0_J(F) @>>> F^{\times}/N_{K/F}(K^{\times})@>>> 1. \end{CD} \]

\vskip 5pt

We would now like to describe the embedding of $\Aut_E(C) \times \GL_2(E)^{\det}$  into $M_J$. While this can be done by writing down some explicit formulas, we would like to view this embedding through the lens of a see-saw pair in the classical similitude theta correspondence.
For this, let us set up the relevant notation and recall the relevant background.
\vskip 5pt

\subsection{\bf Similitude dual pairs}  \label{SS:simi}
Here is the general setup.  For $a \in E^{\times}$, let
\[  W_a = E e_1\oplus E e_2 \]
be a 2-dimensional symplectic vector space over $E$ equipped with the alternating form
\[ 
 \langle e_1, e_2 \rangle_a = -  \langle e_2, e_1 \rangle_a = a. 
 \] 
 With respect to the basis $\{e_1, e_2\}$, we have an identification of the symplectic similitude group $\GSp(W_a)$ with $\GL_2(E)$.  The subgroup $\GSp(W_a)^{\det}$ of elements  
whose similitude factor lies in $F^{\times}$ is then identified with $M_E = \GL_2(E)^{\det}$. For $g \in \GL_2(E)^{\det}$, the corresponding similitude factor is
 \[ \lambda(g)= {\det}_E(g), \]
   where  $\det_E(g)$ refers to the determinant of $g$ considered an element of $\GL_2(E)$.
We write $\GL_2(E)^{\det}_K$  for the index 2  subgroup of elements whose similitudes lie in $N_{K/F}(K^{\times})$. Hence, we set
  \[ M_{E,K}  = \GL_2(E)^{\det}_K = \{ g \in  M_E = \GL_2(E)^{\det}: {\det}_E(g) \in N_{K/F}(K^{\times})\}. \]
 \vskip 5pt

 \noindent From this symplectic space $W_a$, we deduce the following 3 other spaces and groups:
 \vskip 5pt
 
 \begin{itemize}
 \item[(a)]  By restriction of scalars from $E$ to $F$, we obtain  a 6-dimensional symplectic space ${\Res}_{E/F}(W_a)$ with alternating form $\mathrm{Tr}_{E/F} \circ \langle-, -\rangle_a$.  One has a natural inclusion of similitude groups:
 \[ M_E=  \GL_2(E)^{\det} = \GSp(W_a)^{\det}  \hookrightarrow \GSp({\Res}_{E/F}(W_a)) \cong \GSp_6(F). \]
 We write  $\GSp(\Res_{E/F} W_a)_K$ for the index 2  subgroup of elements whose similitudes lie in $N_{K/F}(K^{\times})$.  
 
    \vskip 5pt
 
  \item[(b)] With  $L=E\otimes K$,  the 2-dimensional $L$-vector space
 \[  V_a = W_a \otimes_E L \]
 is naturally equipped with a skew-Hermitian form induced by the alternating form on $W_a$, with $\langle-, -\rangle_a$ given by the same formula as above on the basis $\{e_1, e_2\}$.  
 Then we have
 \[ \GL_2(E)^{\det} =  \GSp(W_a)^{\det} \hookrightarrow   \GU(V_a)^{\det}  \]
 where the superscript $\det$ refers to those elements whose similitude (which a priori lies in $E^{\times}$) belongs to $F^{\times}$.
 \vskip 5pt
 
 \item[(c)]   As above, by considering restriction of scalars from $L$ to $K$, we see that ${\Res}_{L/K}(V_a)$ is a 6-dimensional $K$-vector space equipped with the skew-Hermitian form ${\mathrm Tr}_{L/K} \circ \langle-, -\rangle_a$. This 6-dimensional skew-Hermitian space over $K$ is also the one naturally induced from the symplectic space ${\Res}_{E/F}(W_a)$ over $F$, in the same way as $V_a$ is obtained from $W_a$.   One has a natural inclusion of unitary similitude groups:
  \[   \GU(V_a)^{\det} \hookrightarrow \GU({\Res}_{L/K}(V_a)),\]
  In fact, both
  similitude maps here   have image equal to $F^{\times}$, but  we shall consider the index $\leq 2$ topological subgroups of elements whose similitude lies in $N_{K/F}(K^{\times})$, denoted by:
  \[     \GU(V_a)^{\det}_K \hookrightarrow \GU({\Res}_{L/K}(V_a))_K. \]
  Observe that
  \[  \GU({\Res}_{L/K}(V_a))_K =  ( K^{\times} \times  \U({\Res}_{L/K}(V_a)) ) / \nabla K^1 \]
  with $\nabla K^1 = \{ (z,z^{-1}): z\in K^1 \}$.
  \end{itemize}
  \vskip 5pt
  
  Summarizing, starting with $W_a$, we have the following containment diagram for the 4 groups we introduced:
     \begin{equation} \label{E:seesaw-right}
 \xymatrix{
&   \GU({\Res}_{L/K}(V_a))_K  &  \\
\GU(V_a)^{\det}_K  \ar[ur] &   & \GSp({\Res}_{E/F}(W_a))_K \ar[ul] \\
&    \GSp(W_a)_K^{\det}=  \GL_2(E)_K^{\det} \ar[ul] \ar[ur] &  }  
\end{equation} 
These groups appear in the classical similitude theta correspondence, and we proceed next to describe the other member of the relevant dual pairs, namely those lying on the other side of a seesaw diagram. 
  \vskip 10pt

 Regard $K$ as a rank 1 Hermitian space (relative to $K/F$) with the form $(x, y) \mapsto x \cdot \sigma(y)$. Then $\GU(K)  = K^{\times}$ and $ \GU({\Res}_{L/K}(V_a))_K$ form a similitude dual pair. Here it is necessary to consider the index $\leq 2$ subgroup $\GU({\Res}_{L/K}(V_a))_K$ as opposed to $\GU({\Res}_{L/K}(V_a))$, because the similitude map on $\GU(K)$ has image $N_{K/F}(K^{\times})$. Starting from this rank 1 Hermitian space, one deduces the following 3 spaces and groups:
 \vskip 5pt
 
 \begin{itemize}
 \item[(a')] By restriction of  scalars from $K$ to $F$, we regard $K$ as a 2-dimensional $F$-vector space with quadratic form $N_{K/F}$, with similitude group
 \[  \mathrm{GO}(K, N_{K/F}) \cong  K^{\times} \rtimes \langle \tau\rangle,\]
 with $\tau$ acting on $K^{\times}$ as the unique nontrivial automorphism $\sigma$ of $K/F$. Then $\mathrm{GO}(K, N_{K/F}) \times \GSp({\Res}_{E/F}(W_a))_K$ is a similitude dual pair.
 \vskip 5pt
 
 \item[(b')]  By base change from $F$ to $E$, we obtained a rank $1$ Hermitian space (relative to $L/E$) over $L$, so that $\GU(L)^{\det} \times \GU(V_a)_K^{\det}$ forms a similitude dual pair.
 \vskip 5pt
 
 \item[(c')] By restriction of scalars ${\Res}_{E/F}$  on the space in  (b') or the base change from $F$ to $E$ of the space $(K, N_{K/F})$ in (a'), we obtain the quadratic space 
 $(L, N_{L/E})$ of dimension $2$ over $E$, with similitude group
 \[ \mathrm{GO}(L, N_{L/E})^{\det} :=  \mathrm{GSO}(L, N_{L/E})^{\det} \rtimes \langle \tau \rangle \cong (L^{\times})^{\det}  \rtimes \langle \tau \rangle. \]
 This group form a similitude dual pair with  $M_{E,K} = \GL_2(E)^{\det}_K = \GSp(W_a)^{\det}_K$. 
  \end{itemize}
 Summarizing, starting from a rank $1$ Hermitian space (relative to $K/F$), one have the following diagram
   \begin{equation} \label{E:seesaw-left}
 \xymatrix{
&   \mathrm{GO}(L, N_{L/E})^{\det} = (L^{\times})^{\det} \rtimes \langle \tau \rangle &  \\
 \GU(L) = L^{\times} \ar[ur] &   & \mathrm{GO}(K, N_{K/F}) = K^{\times} \rtimes \langle \tau \rangle  \ar[ul] \\
&    \GU(K) = K^{\times}  \ar[ul] \ar[ur] &  }  
\end{equation}
  As mentioned above, the groups in (\ref{E:seesaw-left})  form a seesaw diagram of dual pairs with the corresponding group in (\ref{E:seesaw-right}). We shall only make use of the 
 groups at the top and bottom of the diagrams, so that we have a similitude seesaw pair:
 \begin{equation} \label{E:seesaw-top-bottom} 
  \xymatrix{
  \mathrm{GO}(L, N_{L/E})^{\det} = (L^{\times})^{\det} \rtimes \langle \tau \rangle  \ar@{-}[dr] &
   \mathrm{GU}({\Res}_{L/K}(V_a))_K =  ( K^{\times} \times  \U({\Res}_{L/K}(V_a)) ) / \nabla K^1  \ar@{-}[dl]  \\
    \GU(K) = K^{\times}  & M_{E,K} = \GL_2(E)^{\det}_K }
\end{equation}
 
   \vskip 5pt
   
   \subsection{\bf Embedding}
  We can now describe the embedding 
  \[  \Aut_E(C)^0 \times \GL_2(E)^{\det} \hookrightarrow M^0_{J}. \]
  Recall that we are considering 
  \[  [C] \in H^1(F, \tilde{T}_{E, K})[2] \cong E^{\times}/ F^{\times} N_{L/E}(L^{\times}) \quad  \quad \text{(by (\ref{E:H1T2})).} \]
  Take any $a \in E^{\times}$ representing the class of $[C]$, so that we have the above constructions of similitude dual pairs using $a \in E^{\times}$.
 Recall further that one has a natural isomorphism of algebraic groups
  \[  
  M_J^0 \cong ({\Res}_{K/F} \mathbb{G}_m   \times  \U ({\Res}_{L/K}(V_a)) )^{\dagger}/  \U_1^K. \]
 Now there is a natural map (with finite kernel) of algebraic groups
  \begin{equation} \label{E:f} 
   f: \GU({\Res}_{L/K}(V_a)) =  ( K^{\times} \times  \U ({\Res}_{L/K}(V_a))_K ) / \nabla \U_1^K  \longrightarrow  (K^{\times} \times  \U ({\Res}_{L/K}(V_a))_K)/  \U_1^K, \end{equation}
  given by
  \[  (z,g) \mapsto (z^{-3}, g). \]
  The restriction of this map to the subgroup $M_E$  (see (\ref{E:seesaw-top-bottom})) gives the embedding of algebraic groups
 \[  \GL_2(E)^{\det} \hookrightarrow M^0_{J}.  \]
 When restricted to the topological subgroup $M_{E,K} = \GL_2(E)^{\det}_K$, the map $f$ is given by the formula
 \[ g\mapsto (z^{-3}, gz^{-1}), \]
 where  ${\det}_E(g)=N_{K/F}(z)$.  Observe that this is clearly well defined, as $z$ is unique up to $K^1$.
 
 \vskip 5pt
 
   On the other hand, we have the natural isomorphism of algebraic groups 
 \[ \Aut_E(C) \cong  (L^{\times})^{\det} \rtimes \langle \tau \rangle / K^{\times} \cong  \mathrm{GO}(L, N_{L/F})^{\det} /  \GU(K), \]
 which is a quotient of the two algebraic groups appearing on the LHS of the seesaw diagram in (\ref{E:seesaw-top-bottom}). Hence
 \begin{equation} \label{E:ULUK}
   \Aut_E(C)^0 \cong    \GU(L)^{\det} / \GU(K) \cong \U(L)/ \U(K). \end{equation}
The embedding
\[  \Aut_E(C)^0  \cong \U(L)/\U(K) \hookrightarrow M_J^0 \]
is given by 
\[   e \mapsto ( N_{L/K}(e), e), \]
where $e \in \U(L)$ acts on ${\Res}_{E/F}(V_a)$ through its scalar multiplication action on $V_a = Le_1 \oplus Le_2$. 
\vskip 5pt

It is  useful to note the following lemma which says that the last isomorphism in (\ref{E:ULUK}) continues to hold on the level of $F$-rational points.
 
  \begin{lemma}  \label{L:L1K1}
  The inclusion  $L^1 \subset (L^{\times})^{\det}$ gives an isomorphism $L^1/K^1 \cong (L^{\times})^{\det}/K^{\times}$. 
 \end{lemma} 
 \begin{proof} 
 We have a long exact sequence 
 \[ 
 1 \rightarrow K^1 \rightarrow L^1 \rightarrow   (L^{\times})^{\det}/K^{\times}\rightarrow H^1(F, \U(K)) \rightarrow H^1(F, \Res_{E/F} \U(L)) 
 \] 
 so we need to show that the last arrow is injective. To that end, the map 
 \[ 
 N_{L/K} : \Res_{E/F} \U(L)\rightarrow \U(K)
 \] 
 gives 
 \[ 
 H^1(F, \U(K)) \rightarrow H^1(F, \Res_{E/F}  \U(L)) \rightarrow H^1(F, \U(K)) 
 \] 
 such that the composite is multiplication by $3$. Since $H^1(F, \U(K))$ is a $2$-group, the composite is the identity. This proves the lemma. 
 \end{proof} 
 \vskip 5pt

 The lemma implies that, for any $x \in (L^{\times})^{\det}$, $N_{L/E}(x) \in N_{K/F}(K^{\times})$. Thus, the embedding
 \[  (L^{\times})^{\det}  / K^{\times}   \hookrightarrow M_{J}^0(F) \]
 takes value in the index $\leq 2$ subgroup $M_{J,K}^0$ and 
 is given by the formula
 \[  x \mapsto ( N_{L/K}(x/z) , x/z), \quad \text{where $N_{L/E}(x) = N_{K/F}(z)$.} \]
 Again this is well-defined as $z$ is determined up to an element of $K^1$.

  \vskip 5pt
  
We have thus described the embedding of algebraic groups
\[  H^0_C \times M_E \hookrightarrow M^0_J. \] 
This embedding depends only on $a \in E^{\times}/ F^{\times} N_{L/E}(L^{\times})  = H^1(F, \tilde{T}_{E,K})[2]$.
On the level of points, it gives the embedding
  \[  H^0_C(F) \times M_{E,K}  =   (L^{\times})^{\det}/ K^{\times} \times \GL_2(E)^{\det}_K     \hookrightarrow  M^0_{J,K}. \]  
  Though  the embedding could have been written down via formulas, without mention of the framework of similitude dual pairs, this framework will help us in \S \ref{S:mini} to relate the  mini-theta correspondence associated to this commuting pair of groups   by reducing it to the classical similitude theta correspondence. So we shall have occasion to return to the material in \S \ref{SS:simi} later on.
  \vskip 5pt

\subsection{\bf Siegel parabolic}
Recall that the Lie algebra $\mathfrak{m}$ has a Siegel parabolic subalgebra $\mathfrak{s}$. This gives rise to a Siegel parabolic subgroup 
\[  S_J  \subset M_J \]
whose Levi factor is of type $A_2 \times A_2$ and whose unipotent radical can be identified with $J$. 
Moreover, $H_C \subset S_J$ and the intersection of $M_E$ with $S_J$ is a Borel subgroup of $M_E$.  If we identify $M_E$ with $\GL_2(E)^{\det}$, we may assume that   
$S_J \cap M_E$ is the Borel subgroup of upper triangular matrices. 
\vskip 10pt

\section{\bf Minimal Representation}  \label{S:minimal}
In this section, we assume that $F$ is a non archimedean local field. 
Let $\Pi$ be the minimal representation of $G_J(F)$ (see \cite{GS1}).  
In this section, we recall the relevant properties of $\Pi$ that we need.
 We first note that  the algebraic group $G_J$ is not connected, but the minimal representation $\Pi$ in \cite{GS1} is  a 
 representation of the subgroup $G_J^0(F)$ of $G_J(F)$. Thus there are two ways of extending $\Pi$ to $G_J(F)$ and we shall first need to specify the extension we use below.

\vskip 5pt

\subsection{\bf Extending the minimal representation}
Recall  the Heisenberg parabolic subgroup  $P_J=M_JN_J$  of $G_J$, with $Z$ the center of $N_J$ and and let 
\[  \chi : M_J \rightarrow F^{\times} \]
be the character  of $M_J$ given by the action of $M_J$ on $Z$. By composition with  $\chi$, we may regard any character $\mu$  of $F^{\times}$ as a character of $M_J(F)$. Henceforth, we shall write $\mu$ in place of $\mu \circ \chi$ for a character of $M_J(F)$.
\vskip 5pt

 Now we consider the degenerate principal series representation of $G_J(F)$:
\[   I_J(s_0):= {\Ind}_{P_J}^{G_J} \chi_J = {\Ind}_{P^0_J}^{G^0_J}  \chi_J     \quad \text{(unnormalized induction)} \]
where
\[  \chi_J =   \omega_{K/F} \cdot |-|^{s_J}  \]
  with $\omega_{K/F}$  the quadratic character associated to $K = K_J$  by local class field theory and $s_J$   given by the following table: 
 \[ 
 \begin{array}{|c|c|c|c|} 
 \hline 
 G_J & E_6 & E_7 & E_8 \\
 \hline 
 s_J & 2 & 3 & 5 \\
 \hline 
 \end{array} 
 \] 
 \vskip 15pt
 
 \noindent   The minimal representation $\Pi$ of $G_J^0(F)$ is the unique irreducible subrepresentation of  $I_J(s_0)$, regarded as a representation of $G_J^0(F)$. 
 This unique irreducible submodule is thus stable under the action of $G_J(F)$ and this defines the extension of $\Pi$ to $G_J(F)$.   When we regard $I_J(s_0)$ as a space of functions on $G_J^0(F)$ transforming under $(P_J^0(F), \chi_J)$ on the left, the action of $G_J^0(F)$ is by right translation whereas the action of $p \in P_J(F)$ is given by:
 \[  (p \cdot f) (h) = \chi_J(p) \cdot f(p^{-1}h p) \quad \text{  for  $h \in G_J^0(F)$ and $f \in I_J(s_0)$.} \]
 This describes the action of $G_J(F) = P_J(F) \cdot G_J^0(F)$.
\vskip 10pt

\subsection{\bf Restriction of $\Pi$ to $\bar P_J$.} \label{SS:minimal-restrict}
 The restriction of $\Pi$ to $\bar P_J$ sits in a short exact sequence
 \[ 
 0 \rightarrow C_c^{\infty}(\Omega) \rightarrow \Pi_{\bar Z}  \rightarrow \Pi_{\bar N_J} \rightarrow 0,
 \] 
where  $\Omega \subset  N_J/Z$ is the minimal nontrivial (highest weight) $M_J$-orbit. 
\vskip 5pt

To describe the action of $\bar P_J$ on  $C_c^{\infty}(\Omega)$, 
 let $\langle \bar n, n\rangle $ be the natural pairing of  $\bar N_J/\bar Z$ and $N_J/Z$  and fix a non-trivial additive character $\psi$ of $F$. Then the action is given as follows. For $f\in C_c^{\infty}(\Omega)$,
\vskip 5pt

\begin{itemize}
\item $\bar n \in \bar N_J/Z$ acts   by 
 \[ 
 \Pi(\bar n)f(n)= \psi(\langle \bar n, n\rangle) \cdot f(n).
 \] 
 \item  $m\in M_J$ acts    by
  \[ 
 \Pi(m)f(n)= \chi_J(m)  \cdot f(m^{-1}nm).
 \] 
   \end{itemize}
 
\vskip 5pt
\subsection{\bf The minimal orbit $\Omega$.}  \label{SS:minimal orbit}
  Recall from \ref{ss_parabolic} that we have an identification  
\[ 
W_J:= N_J/Z_J=F \oplus J \oplus J \oplus F. 
\]

 By  \cite[Proposition 11.2]{GS1}, we have the following description of $\Omega$:
  
 \begin{prop}  \label{P:orbit} 
 A non-zero element $\omega= (a, x,y,d)\in N_J/Z_J$ is in the minimal $M_J$-orbit  $\Omega$ if and only if 
 \[ 
 x^{\#}= ay, ~ y^{\#}=dx \text{ and } l(x) \cdot l^{\ast}(y)= ad \text{ for all $l\in L_J$} 
 \] 
 where $x\cdot y$ is the product in $J$, $L_J$ the group of linear transformations of $J$ preserving the norm form, 
 and $l^{\ast}$ the dual action of $L_J$ on $J^*\cong J$, with the identification given by the trace pairing. 
  In particular,  if $a=1$, then $\omega=(1,x,x^{\#},N_J(x))$.
 \end{prop} 
 \vskip 10pt

\noindent{\bf Erratum:}  
 In fact,  \cite[Proposition 11.2]{GS1} asserts that it suffices to use $x\cdot y=ad$ in place of the family of equations obtained by the $L_J$-action. 
This is false. Writing $W_J = N_J/Z_J$, the  $M_J$-module $S^2(W^*_J)$  is a direct sum of an irreducible module whose highest weight is equal to twice the highest 
weight of $W_J^*$, and the adjoint representation of $M_J$.  
The quadratic equations given here span the latter summand and hence give a complete set of generators.  
\vskip 5pt

Note however that in \cite[Proposition 11.2]{GS1}, only the proof of  the ``only if" statement  was given, as the  other direction was not used in \cite{GS1}. Hence this error does not affect any result in \cite{GS1}.

\vskip 5pt

\subsection{\bf The $M_J$-module $\Pi_{\bar N_J}$.}  \label{SS:levi}
 A complete description of the Jacquet module $\Pi_{\bar N_J}$ is given in \cite{GS1}. We have
 \[  \Pi_{\bar N_J}  \cong  \omega_{K/F} \cdot |-|^{-2} \oplus  |-|^{-3/2} \cdot \Pi_{M_J}  \]
 for an $M_J$-module $\Pi_{M_J}$ which is $0$ if $J$ is a division algebra  and is a unitary minimal representation of $M_J$ otherwise. We will assume that $J$ is not division henceforth and  describe the $M_J$-module $\Pi_{M_J}$ in some detail.

 \vskip 5pt

  \vskip 5pt

 Recall that $M^0_J(F)$ contains a subgroup $M^0_{J,K}$ of index $\leq 2$. 
  We first describe a representation of $M^0_{J,K}$, using the classical  theta correspondence for the pair 
  \[ \U(K)  \times \U({\Res}_{K/F}(V_a))   = \U_1(F) \times \U_6(F)  \]
  constructed in \S \ref{SS:simi}. 
   \vskip 5pt
   
   To give a Weil representation for this dual pair, we need to choose a character $\mu$ of $K^{\times}$ whose restriction to $F^{\times}$ is  the quadratic character $\omega_{K/F}$, which gives a splitting of the metaplectic cover over $\U_6(F)$.  Then we may consider the Weil representation $\omega_{\mu,\psi}$ for $\U_1  \times \U_6$ associated to the pair of splitting characters $(1, \mu)$ and a nontrivial additive character $\psi$ of $F$. 
     With respect to the choice of $(1,\mu)$ and $\psi$, the associated Weil representation $\omega_{\mu,\psi}$ can be realised on  $C_c^{\infty}(L)$, where $L=Le_2$ is a polarization of $V_a = L e_1 \oplus Le_2$.   The action of $\U(K) = K^1$ and the Siegel parabolic subgroup of $\U({\Res}_{K/F}(V_a))$ stabilizing $Le_1$ is given by the usual formulas in the Schrodinger model:
     \vskip 5pt
     
     \begin{itemize}
     \item The group $U_1=K^1$ acts geometrically on $C_c^{\infty}(Le_2)$: for $z \in K^1$,
     \[   (z \cdot f) (v)  = f(z^{-1} v). \]
     
     \item If $\GL_K(Le_2)$ is the Levi subgroup  that preserves the decomposition $V_a= Le_1 \oplus  Le_2$, the action of $g \in \GL_K(Le_2)$ is given by
   \[ 
  (g \cdot f)(v) = \mu(\det(g))\cdot  |N_{K/F}\det(g)|^{-\frac{1}{2}} \cdot f(g^{-1} v).   
  \]
 \item An element $u$ in the unipotent radical of the Siegel parabolic subgroup stabilizing $Le_2$ acts by:
 \[  (n \cdot f)(v)  = \psi( \langle n  , v \rangle_a) \cdot f(v). \]
 \end{itemize}
 In particular, we see the dependence on $a \in E^{\times}$ in the last formula above.
  If we replace $\mu$ by $\mu\cdot \beta$, where $\beta$ is a 
  character of $K^{\times}/F^{\times}$, then  the splitting of $\U_6(F)$ changes by $\alpha \circ \det$, where $\alpha$ is a character of $K^1$,determined by   $\beta$  via:  $\alpha(z/\sigma(z))=\beta(z)$. Moreover, for a fixed $\mu$, the Weil representation depends only on the orbit of $\psi$ under $N_{K/F}(K^{\times})$. 
    \vskip 5pt
    
   We can now consider the classical theta lift $\theta_{\mu}(1)$ of the trivial representation of $\U_1$, which is an irreducible representation of $\U_6(F)$ realized on the subspace
  \[ 
    C_c^{\infty}(Le_2)^{K^1} \subset  C_c^{\infty}(Le_2).
  \] 
Consider the representation of   $K^{\times} \times \U^K_6(F)$ on $C_c^{\infty}(Le_2)^{K^1}$ given by 
\[  \Pi_{M_{J,K}} := \mu^{-1} \boxtimes \theta_{\mu}(1). \]
It is a simple check that the restriction of  $\Pi_{M_{J,K}}$  to the subgroup 
  \[  \{ (x,g) \in K^{\times} \times \U^K_6(F): x/\sigma(x)=\det(g) \} \]
   is independent of $\mu$ and that it descends to a representation 
  of  $M^0_{J,K}$.  We extend this representation to $M_{J,K}$ by letting  $\tau$ act on  $f \in C_c^{\infty}(L)^{K^1}$ via
  \[  (\tau \cdot f)(v) = f(\sigma(v)). \]
  Thus we have a representation $\Pi_{M_{J,K}}$ of 
  $M_{J,K} = M^0_{J,K} \rtimes \langle\tau\rangle$ on $C_c^{\infty}(Le_2)^{K^1}$, which depends on the orbit of $\psi$ under $N_{K/F}(K^{\times})$. Now we have:
  \begin{equation} \label{E:PiMJ}
  \Pi_{M_J} \cong  \Ind_{M_{J,K}}^{M_J}  \Pi_{M_{J,K} }= \Ind_{M_{J,K}}^{M_J}  \mu^{-1} \boxtimes \theta_{\mu}(1). \end{equation}
  This representation is now independent of $\psi$ and  $\mu$.

\vskip 5pt
\subsection{\bf  Similitude theta lifting}  \label{SS:simi2}
It is in fact better to think of the representation $\Pi_{M_{J,K}}$ from the viewpoint of the similitude theta correspondence for the pair
  \[ \GU(K)  \times \GU({\Res}_{K/F}(V_a))_K   = K^{\times}  \times \GU_6(F) _K. \]
In particular, we may consider the similitude theta lift $\tilde{\theta}_{\mu}(1)$ of the trivial representation of $K^{\times}$; this representation is also realized on $C_c^{\infty}(Le_2)^{K^1}$, and is merely an extension of $\theta_{\mu}(1)$ to $\GU_6(F)_K$ with the center $K^{\times}$ acting by the central character  $\mu^3$. 
Recall from (\ref{E:f}) the isogeny
\[  f:  \GU( {\Res}_{L/K}(V_a))_K   = (K^{\times} \times \U_6(F))/  \nabla K^1 \longrightarrow   (K^{\times} \times \U_6(F)) / K^1 \]
defined by
\[  f(z,g) = (z^{-3}, g). \]
Then we have;
\[   \tilde{\theta}_{\mu}(1)  = (\mu^{-1} \boxtimes \theta_{\mu}(1))\circ f = \Pi_{M_{J,K}} \circ f . \]
In other words, $\tilde{\theta}_{\mu}(1)$ factors through $f$ and when restricted to $(K^{\times} \times \U_6(F))^{\dagger}$ is independent of $\mu$. 
\vskip 5pt

From this viewpoint, the restriction of the $M_{J,K}$-module $\Pi_{M_{J,K}}$  to the commuting pair $H_C(F) \times \GL_2(E)^{\det}$ can be transparently described using the seesaw diagram (\ref{E:seesaw-top-bottom}). More precisely, we pick $a \in E^{\times}$ so that its class in $E^{\times}/ F^{\times} N_{L/E}(L^{\times}) =  H^1(F, T_{E, K_C})[2]$ (see (\ref{E:H1T2}) and (\ref{E:f2})) corresponds to $[C]$.
From the seesaw identity arising from  (\ref{E:seesaw-top-bottom}), the representation $\tilde{\theta}_{\mu}(1)$ is naturally a representation of   
\[  ((L^{\times})^{\det} \rtimes \langle \tau \rangle) / K^{\times} \times \GSp(W_a)^{\det} = \Aut_E(C)  \times   \GL_2(E)^{\det}. \]
This representation is precisely the restriction of $\Pi_{M_{J,K}}$ to  $\Aut_E(C)  \times   \GL_2(E)^{\det}$.

 \vskip 5pt
   
\subsection{\bf Some formulas.}  \label{SS:formulas}
  We write down some formulas for $\Pi_{M_{J,K}}$ which are relevant to us.
  \vskip 5pt
  
  \begin{itemize}
  \item An element $e\in L^1/K^1=\Aut_E(C)^{\circ}$ acts on 
  $f \in C_c^{\infty}(Le_2)^{K^1}$  by 
  \[  (e\cdot f)(v)=
  f(e^{-1} v).   
  \] 
  \vskip 5pt
  
  \item The element 
  \[ 
t(x) =   \left( \begin{array}{cc} 
  x & 0 \\
  0 & 1 
  \end{array} \right) \in \GL_2(E)^{\det}, 
  \] 
  with $x= N_{K/F}(z)$ for some $z\in K^{\times}$,  acts on $f \in C_c^{\infty}(L e_2)^{K^1}$  by 
  \[  (t(x) \cdot f) (v) = 
   |x|^{-\frac{3}{2}}  \cdot f(z^{-1} v).    
  \] 
  
  \item The element 
  \[  u(b) =  \left( \begin{array}{cc} 
  1 & b \\
  0 & 1 
  \end{array} \right) \in \GL_2(E)^{\det}, 
  \] 
acts by
\[  (u(b)f)( v)  =  \psi(   {\mathrm Tr}_{E/F}(a \cdot N_{L/E}(v) \cdot b))  \cdot f(v). \]
  \end{itemize}
 \vskip 5pt
 
 \noindent   The dependence of  the $H_C(F) \times \GL_2(E)_K^{\det}$-module $\Pi_{M_{J,K}}$ on $a \in E^{\times}$ is thus evident from the action of the unipotent radical of the upper triangular matrices in $\GL_2(E)^{\det}$. In particular,  one sees that the Whittaker support (relative to $\psi$) of $\Pi_{M_{J,K}}$ as an $\GL_2(E)^{\det}_K$-module is on the coset $a \cdot N_{L/E}(L^{\times}) \subset E^{\times}$.  Thus, the Whittaker support of the $\GL_2(E)^{\det}$-module
 \[  \Pi_{M_J}  = {\Ind}_{\GL_2(E)_K^{\det}}^{\GL_2(E)^{\det}} \Pi_{M_{J,K}} \]
 is on the coset $a \cdot F^{\times} N_{L/E}(L^{\times})$. This is the coset corresponding to $[C] \in H^1(F, T_{E, K_C})[2]$, by our choice of $a$.
        \vskip 5pt
    
   \subsection{\bf Split case.}
    If $K=F^2$.  Then    $K^1= \{ (x,y) \in F^2 ~ | ~xy=1\} \cong F^{\times}$,  $L=E^2$ and $L^1 \cong E^{\times}$.
    In this case, we can simplify the description of $\Pi_{M_J}$.
    \vskip 5pt
    
     If we apply a partial Fourier transform to $C_c^{\infty}(L)= C_c^{\infty}(E^2)$ with respect to the second factor $E$ of $L$, the action of $K^1\cong F^{\times}$  on 
    $C_c^{\infty}(L)$ becomes the action by homotheties. The representation $\Pi_{M_J}\cong \Pi^{\vee}_{M_J}$ is the maximal $F^{\times}$-invariant quotient of $C_c^{\infty}(L)$, 
    and is  isomorphic to the space of smooth functions $f$ on $L\setminus\{ 0\}$ such that 
   \[ f(x v)=|x|_F^{-3} f(v) \quad\text{   for all $v\in L\setminus \{0\}$ and $x\in F^{\times}$.} \]
   The restriction of $\Pi_{M_J}$ to 
 $ M_E\times \Aut_E(C)$ is given as follows. If  $ g\in \GL_2(E)^{\det}$ then
 \[ 
 \Pi_{M_J}(g) f(v) = |\det(g)|^{-\frac{3}{2}}\cdot  f(g^{-1}v), 
 \] 
where   $g^{-1}v$ is the natural action of $g^{-1} \in \GL_2(E)$ on $v\in E^2=L$. If $e\in E^{\times}$ then 
 \[ 
 \Pi_{M_J}(e) f(v) = |N_{E/F}(e)|^{-1} \cdot f(e^{-1}v), 
 \] 
where $e^{-1}v$ is the product of the scalar $e^{-1} \in E^{\times}$ and the vector $v\in E^2$. 
The involution  $\tau$  acts by the Fourier transform, viewing $f$ as a distribution on $C_c^{\infty}(L)$.
\vskip 10pt

\subsection{\bf Schr\"odinger model of $\Pi_{M_J}$} \label{SS:Schro}
The description we have given above for $\Pi_{M_J}$  allows one to relate the theta correspondence arising from  its restriction to the dual pair $H_C \times M_E$ in $M_J$ to the classical theta correspondence. As a minimal representation, $\Pi_{M_J}$ also has a Schr\"odinger model adapted to the Siegel parabolic subgroup $S_J \subset M_J$, which we will describe next. 
\vskip 5pt

As a representation of $S_J$, $\Pi_{M_J}$ sits in a short exact sequence
\[  \begin{CD}
0 @>>> C^{\infty}_c(J_{rk=1})    @>>> \Pi_{M_J} @>>> r_{S_J}(\Pi_{M_J}) @>>> 0 \end{CD} \]
where $J_{rk=1}$ denotes the set of rank 1 elements in $J$ and $r_{S_J}(-)$ denotes the (normalized) Jacquet module with respect to $S_J$.
The action of some elements of $H_C \times B_E  = (H_C \times M_C) \cap S_J$ on $C^{\infty}_c(J_{rk=1})$ can be described as follows:
\vskip 5pt

\begin{itemize}
\item For $b \in E$,  the upper triangular unipotent element $u(b) \in M_E(F) = \GL_2(E)^{\det}$ acts by
\[  (u(b) \cdot f) (x)  = \psi(\mathrm{Tr}_J(b x)) \cdot f(x)  = \psi(\mathrm{Tr}_{E/F}( b \cdot e)) \cdot f(x) \]
where $x = (e, v) \in E \oplus  C = J$ has rank $1$ and $\psi$ is a fixed nontrivial additive character of $F$.
 \vskip 5pt

\item   For $h \in H_C(F)$, $h$ acts by
\[  (h \cdot f) (x) = f( h^{-1}x)  \]
where we have identified $H_C$ with the pointwise stabilizer of $E \subset J$, so that $H_C \subset \Aut(J)$.
\end{itemize}
\vskip 5pt

Observe that by Lemma \ref{L:embedding 1-2}, and Lemma \ref{L:XaC},  $x = (e,v) \in E \oplus C $ has rank $1$ if and only if the map $f$ in Lemma \ref{L:XaC} sends $e$ to $[C] \in H^1(F, T_{E, K_C})[2] = E^{\times}/F^{\times} N_{L/E}(L^{\times})$. In view of (\ref{E:f2}), this is equivalent to the coset $e \cdot F^{\times} N_{L/K}(L^{\times})$ being equal to that of $[C]$.  So the Whittaker support of $\Pi_{M_J}$ as a  $\GL_2(E)^{\det}$-module is as we had determined in \S \ref{SS:formulas} via  the classical theta correspondence.

\vskip 5pt

The description of the minimal representation $\Pi_{M_J}$ given here will be used for the study of the theta correspondence for $H_C \times M_E$ in \S \ref{S:mini}. 
This is necessary for the study of the theta correspondence for $H_C \times G_E$, which  will be carried out in \S \ref{S:main-theta}.

 \vskip 10pt
 
  \section{\bf Jacquet functors for $E_6$}  \label{S:jac-E6}
 
 In this section, we continue to assume that $F$ is a nonarchimedean local field. The goal of this section is to describe the (un-normalized) Jacquet module $\Pi_{\bar{N}_E}$ as a representation of $M_E \times \Aut_E(C)$. Here, recall that
  $P_E=M_E N_E = P_J \cap G_E$ is the Heisenberg parabolic subgroup in $G_E$  and  $N_J$ and $N_E$ share the center $Z$. 
 Let 
 \[ \Omega^{\perp} = \{ x \in \Omega: \text{$x$ is perpendicular to $\bar N_E/\bar Z$} \}. \]
 Then  we have an exact sequence 
\[ 
 0 \rightarrow C_c^{\infty}(\Omega^{\perp}) \rightarrow \Pi_{\bar N_E}  \rightarrow \Pi_{\bar N_J} \rightarrow 0 
 \]  
 Thus, we need to:
 
 \vskip 5pt
 
 \begin{itemize}
 \item  determine the set $\Omega^{\perp}$ and  describe $C_c^{\infty}(\Omega^{\perp})$ as a module for $M_E \times \Aut_E(C)$; we shall do this in this section.
 \vskip 5pt
 
 \item study the theta correspondence for $M_E \times \Aut_E(C)$ with respect to  $\Pi_{M_J}$: we shall study this in the next section.
 \end{itemize}
 \vskip 5pt
 
 Now as a $\GL_2^{\det}(E)$-module, the orthogonal complement of $\bar N_E/\bar Z$  in $N/Z$ is  given by the natural action of  $\GL_2^{\det}(E)$
 on  $C\oplus C= E^2 \otimes _E C$ via its action on $E^2$. 
 Thus, an  $\omega\in \Omega^{\perp}$ is of the form $(0,x,y,0)$ where $x$, $y\in C$ such that 
 \[  x^{\#}= (-Q(x), \beta(x) )= 0 = y^{\#},\quad  \text{and} \quad  x\cdot y=0 \in J. \] 
  Now we note the following proposition, which uses the structure theory of twisted composition algebras:
  
  \begin{prop} \label{P:vanishing} 
 If $x\in C$ is such that $Q(x)=0$ and $N_C(x)= b_Q(x,\beta(x))=0$, then $x=0$ except when  
\begin{enumerate} 
\item $E=F^3$ and $J=M_3(F)$. 
\item $E=F\times K$, where $K$ is a field, and $J=M_3(F)$. 
\item $E=F\times K$, where $K$ is a field, and $J=J_3(K)$. 
\end{enumerate} 
Hence $\Omega^{\perp}$ is empty unless we are in the three cases above.
\end{prop} 
\vskip 5pt

\begin{proof} It suffices to look at the cases when $Q$ is isotropic. If $K_C$ is a field, then the norm $N_{E\otimes K_C/E}$ is isotropic only when 
$E=F\times K$ and $K_C=K$. Since $K_E=K$, it follows that $K_J=F^2$. Hence we are in the second case. If $K_C=F^{2}$, then $Q$ is always isotropic. 
The cases $E=F^3$ and $E=F\times K$  correspond to the first and third cases, respectively, in the statement of the proposition.
\vskip 5pt

 If $E$ is a field, then 
$C=E\otimes F^2=E\oplus E$ and, up to an invertible scalar, $Q(y,z)= yz$ and $\beta(y,z) = (z^{\#}, y^{\#})$, for $(y,z)\in C= E^2$.  Here $Q(y,z)=0$ implies $y=0$ or $z=0$. Assuming $z=0$, we see that  $b_Q((y,0), (0, y^{\#}) = yy^{\#} = N_{E/F}(y)=0$, which implies that $y=0$.  
\end{proof}

  \vskip 5pt
  
 Hence, to explicate $C^{\infty}_0(\Omega^{\perp})$, we need to treat the 3 cases highlighted in the proposition, and we shall deal with them in turn.

  \vskip 5pt 
 \subsection{\bf Case 1: $E=F^3$ and $J=M_3(F)$.}  \label{SS:jac-split}.
  In this case, $C$ is a split twisted composition algebra. Write 
  \[ 
  x=((x_1,y_1), (x_2,y_2), (x_3,y_3)),  \qquad \, y=((x'_1,y'_1), (x'_2,y'_2), (x'_3,y'_3))
  \] 
  and suppose that $(x,y) \in \Omega^{\perp}$.
   Let $X_i$, respectively $Y_i$, be the 2-dimensional 
  $F$-subspace of $C\oplus C$  consisting of  all pairs $(x,y)$ such that all coordinates except  $x_i$ and $x'_i$ are trivial,  respectively,  all coordinates except $y_i$ and $y_i'$ are 
  trivial. On each $X_i$ and $Y_i$, two of  the three $\SL_2(F)\subset M_E$ act trivially, and the quotient group, isomorphic to $\GL_2(F)$, acts via the standard representation. 
  
  \vskip 5pt
  
  The condition $x^{\#}=0$ holds if and only if there exists a pair of indices $i\neq j$ such that all coordinates of  $x$ are 0 except possibly for $x_i$ and $y_j$. An analogous statement holds for $y$: all coordinates are 0 except possibly for $x'_{a}$ and $y'_{b}$ for some $a \ne b$. The last condition, $x\cdot y=0$, implies that $i=a$ and $j=b$.  This can be easily seen by 
  writing $x$ and $y$ as matrices, say
  \[ 
  x= 
  \left( \begin{array}{ccc} 
  0 & x_3 & y_2 \\
  y_3 & 0 & x_1 \\
  x_2 & y_1 & 0 
  \end{array} \right)  
  \text{ and } 
  y=
  \left( \begin{array}{ccc} 
  0 & x'_3 & y'_2 \\
  y'_3 & 0 & x'_1 \\
  x'_2 & y'_1 & 0 
  \end{array} \right). 
  \] 
  Hence, if $(0, x,y, 0) \in \Omega^{\perp}$, then $(x,y)\in X_i \oplus Y_j$ for some $i \ne j$, and we have: 
  \[ 
  \Omega^{\perp} \cup \{ 0\} = \bigcup_{i\neq j} X_i \oplus Y_j. 
  \] 
 Let $X_i^*$ and $Y_i^*$ denote the corresponding punctured planes.  As $M_E$-module,  the space  $C_c^{\infty}(\Omega^{\perp})$  has a 2-step filtration with 
 submodule 
 \[ 
 \bigoplus_{i\neq j} C_c^{\infty} (X_i^* \times Y_j^*) 
 \] 
 and  quotient (via restriction)
 \[ 
 \bigoplus_ i C_c^{\infty} (X_i^*) \oplus  \bigoplus_j C_c^{\infty}(Y_j^*). 
 \] 
 The action of $M_E$ is geometric, with the same twist $\chi_J$ as the one-dimensional summand of $\Pi_{\bar{N}_J}$.  
 
 \vskip 5pt 
 \subsection{Case 2: $E=F\times K$ and $J=M_3(F)$.}
  In this case $K_C=K$, so $C=E\otimes K=K^3$. The structure of $E$-module on $C$ is given by 

\[  (f,e) \cdot (z_1, z_2, z_3)= (fz_1, e z_2, \bar e z_3) \]
where $(f,e)\in F\times K$ and $z=(z_1,z_2,z_3)\in K^3$. The composition algebra structure is given by 
\[ 
Q(z)= (N_K(z_1), z_2 \bar z_3) 
\] 
and 
\[ 
\beta(z_1,z_2,z_3)=(\bar z_2\bar z_3, \bar z_1 \bar z_2, \bar z_3 \bar z_1). 
\] 
This algebra $C$ can be obtained from the split algebra $C_s$ by Galois descent from $C_s\otimes K$ where the usual action of the Galois group of $K$ over $F$ is 
twisted by 
\[ 
\sigma((x_1,y_1), (x_2,y_2), (x_3,y_3))= ((y_1,x_1), (y_3,x_3), (y_2,x_2)). 
\] 
Note that $Q(z)=0$ implies that $z_1 = z_2 =0$ or $z_1 = z_3=0$.  For $i =2$ or $3$, let $Z_i$  be the two-dimensional $K$-plane in $C\oplus C$ consisting of 
pairs  $(z,z')$ such that $z_j=z'_j=0$ for all $j\neq i$. Now $\Omega^{\perp}$ is the union of the punctured planes $Z_2^*$ and $Z_3^*$.  This claim can be easily 
verified form the split case using Galois descent. 
The group $\GL_2(E)^{\det}$ acts on each plane via projection onto $\GL_2(K)^{\det}$, with $\SL_2(F)$ as the kernel. 
As $M_E$-module, the space  $C_c^{\infty}(\Omega^{\perp})$  is a direct sum 
 \[ 
 C_c^{\infty} (Z_2^*) \oplus C_c^{\infty}(Z_3^*). 
 \] 
The action of $M_E$ is geometric, with the same twist $\chi_J$ as the one-dimensional summand of $\Pi_{\bar{N}_J}$.

 \vskip 5pt 
 \subsection{\bf Case 3:  $E=F\times K$ and $J=J_3(K)$.}
   In this case $K_C=F^2$, so $E\otimes C=F^2 \times K^2$. If $z=((x_1,y_1), (x_2,y_2)) \in C$, then
 \[ Q(z)=(x_1x_2, y_1 y_2) \quad \text{and} \quad  
 \beta( z)= ((N_K(y_2), N_K(y_1)), (y_1\bar y_2 , x_1 \bar x_2)). 
 \] 
 This algebra $C$ can be obtained from the split algebra $C_s$ by Galois descent from $C_s\otimes K$ where the usual action of the Galois group of $K$ over $F$ is 
twisted by 
\[ 
\sigma((x_1,y_1), (x_2,y_2), (x_3,y_3))= ((x_1,y_1), (x_3,y_3), (x_2,y_2)). 
\] 
  In this case $Q(z)=0$  and $\beta(z)=0$ imply that $x_2=y_2=0$ and $x_1=0$ to $x_2=0$. Let $X_1$ (respectively $Y_1$)  be the plane in $C\oplus C$ consisting of 
 all elements  $(z, z')$ such that all coordinates of $z$ and $z'$ are 0 except $x_1$ and $x_1'$ (respectively,  except $y_1$ and $y'_1$).  Then $\Omega^{\perp}$ is the union of the punctured planes $X_1^*$ and $Y_1^*$.  Again, this claim can be easily 
verified form the split case using Galois descent. 
The group $\GL_2(E)^{\det}$ acts on each plane via projection onto $\GL_2(F)$, with $\SL_2(K)$ as the kernel. 
As $M_E$-module, the space  $C_c^{\infty}(\Omega^{\perp})$  is a direct sum 
 \[ 
 C_c^{\infty} (X_1^*) \oplus C_c^{\infty}(Y_1^*). 
 \] 
The action of $M_E$ is geometric, with the same twist $\chi_J$ as the one-dimensional summand of $\Pi_{\bar{N}_J}$.
 \vskip 15pt

 \section{\bf Mini Theta Correspondence}  \label{S:mini}
 In this section, we shall determine the local theta correspondence given by the $M_E \times \Aut_E(C)$-module $\Pi_{M_J}$ when $F$ is a nonarchimedean local field.  This is only relevant when $J = E \oplus  C$ is not a division algebra.  Understanding this mini-theta correspondence is necessary for our main goal of understanding the theta correspondence for $G_E \times \Aut_E(C)\subset G_J$. We begin by introducing notation for the irreducible representations of $H_C(F)$ and $M_E(F) = \GL_2(E)^{\det}$.
 
 \vskip 5pt
\subsection{\bf Representations of $\Aut_E(C)$.}
 Since $J = E \oplus C$ is not a division algebra, we see by Proposition \ref{P:division} that
 \[  H_C(F)  \cong  H_C^0(F) \rtimes \Z/2\Z \]
 where the action of $\Z/2\Z$ on $\Aut_E(C)^0$ is by inverting. Note however that the above isomorphism is not canonical and amounts to choosing an element (necessarily of order $2$) in $H_C(F) \smallsetminus H_C^0(F)$.
 \vskip 5pt
 
  The irreducible representations of $H_C(F)$ are not hard to classify:
  \vskip 5pt
  
  \begin{itemize}
  \item[(a)]   For every character $\chi$ of the torus $H_C^0(F)$ such that $\chi^2\neq 1$, we have a two dimensional representation 
 \[  \rho(\chi) = \mathrm{Ind}_{H_C^0(F)}^{H_C(F)} \chi  = \rho(\chi^{-1}). \] 
 Note that  $\rho(\chi) \cong \rho(\chi')$ if and only if $\chi^{\pm 1}  = \chi'$.  
 \vskip 5pt
 
 \item[(b)] For each  character $\chi$  such that $\chi^2=1$,  there are two extensions of $\chi$ to $H_C(F)$. If $\chi=1$,  these two representations are easily distinguishable from each other: one is trivial whereas the other is not. We denote them by  $1$ and $\epsilon = \epsilon_C$ (the sign character of $H_C
 (F)$) respectively.  
 \vskip 5pt
 
 \item[(c)]  When $\chi^2=1$ but $\chi\ne 1$, we can use the fixed isomorphism $ H_C(F)  \cong H_C^0(F) \rtimes \Z/2\Z $ to distinguish the two extensions. 
Namely,  we may denote the two extensions by $\rho(\chi)^+$ and $\rho(\chi)^-$,  where the sign denotes  the action of the nontrivial element of $\Z/2\Z$.  
\end{itemize} 

Note however that the labelling in (c) above is not really canonical. We shall see much later that one has a better parametrization. 
 This is based on the following canonical bijection of 2-element sets deduced  from Proposition \ref{P:isom-torsor}:
  \[  f^{-1}([C])/ b({\Ker}(\chi))  \longleftrightarrow (H_C(F) \smallsetminus H_C^0(F))/ {\Ker}(\chi). \]
 and the observation that any extension of $\chi$  is a nonconstant $\pm 1$-valued function on the RHS. 
 For this section, the labelling provided by (c) above is sufficient.

   \vskip 5pt

 \subsection{\bf Induced representations of $\GL_2(E)^{\det}$.}  \label{SS:induced-rep}
 Writing $E$ as a product  $\prod_i E_i$ of fields $E_i$,  we have a similar product $L = E \otimes K = \prod_i L_i$ with $L_i= E_i \otimes K$.
 Let $\omega_{L/E}$ be the quadratic character of $E^{\times}$ such that the restriction to each $E_i$ is the quadrate character corresponding 
 to the extension $L_i$.
 \vskip 5pt

  Now let $\chi$ be a unitary character of $E^{\times}$ and consider the induced representation  $\chi\times \omega_{L/E}$  of $\GL_2(E)$ in the notation of 
 Bernstein and Zelevinski. We shall need some simple results on  the restriction of $\chi\times \omega_{L/E}\cong \chi^{-1}\times \omega_{L/E}$  to $\GL_2(E)^{\det}$. 
 \vskip 10pt

\begin{prop} \label{P:restME}
 Let $\chi$ be a unitary character of $E^{\times}/F^{\times}$. In the following, `` the restriction" refers 
to the restriction of $\chi\times \omega_{L/E}$  to $\GL_2(E)^{\det}$. 
\begin{enumerate} 
\item Assume that $K=F^2$, and $\chi$ is a character of $E^{\times}$ trivial on $F^{\times}$. 
The restriction is irreducible unless $\chi^2=1$ and $\chi\neq 1$, in which case 
 it is a direct sum of 2 non-isomorphic irreducible representations. 
 \vskip 5pt
 
 \item Assume that $K$ is a field and $E=F\times K$.  Let $\chi$ is a character of $F^{\times}\times K^{\times}$ trivial on $F^{\times} \times K^1$.
  The restriction is  irreducible unless $\chi^2=1$ and $\chi\neq 1$, in which case 
 it is a direct sum of 2 non-isomorphic irreducible representations. 
 \vskip 5pt
 
 \item Assume that $K$ is a field, but $E\neq F\times K$. Let $\chi=1$. The restriction of $1\times \omega_{\omega_{L/E}}$  is a direct 
 sum of $2^{n-1}$ non-isomorphic irreducible representations where $n$ is the number of factors of $E$. 
 \end{enumerate} 
\end{prop} 

\begin{proof} 
 These statements can be deduced from the well known facts about representations of $\GL_2(E)$ and $\SL_2(E)$.  
We provide the details in the case when $E$ is a field and $K=F^2$; 
 the general case is treated by a similar argument.  
 \vskip 5pt
 
 The representation 
$\chi\times 1$ is irreducible when restricted to $\SL_2(E)$ (and hence to $\GL_2(E)^{\det}$) unless $\chi^2=1$ and $\chi\neq 1$. If $\chi^2=1$ and $\chi\neq 1$, then
$\chi\times 1$ reduces to two  non-isomorphic summands on $\SL_2(E)$ and also on the intermediate group consisting of elements $g\in \GL_2(E)$ such that $\det(g)$ is in the kernel of $\chi$. Since, by our assumption, $\chi$ is trivial on $F^{\times}$, the character $\chi$ is trivial on $\det(g)$ for $g\in \GL_2(E)^{\det}$.  
Thus $\chi\times 1$ is a sum of two non-isomorphic irreducible representations. 
\end{proof} 

\vskip 5pt

 \vskip 5pt 
 \subsection{\bf Theta lifting.}
  For every irreducible representation 
 $\rho$ of $H_C(F)$,  let $\Theta_M(\rho)$ be a representation of $M_E$ such that $\Theta_M(\rho) \otimes \rho$ is the  maximal $\rho$-isotypic quotient of $\Pi_{M_J}$. 
 We shall now give a description of $\Theta_M(\rho)$ for {\em unitary} representations $\rho$. The results are essentially a reformulation of the 
classical  similitude theta correspondence for the dual pair $\GO_2(E) \times \GL_2(E)$, together with an understanding of the restriction of representations from $\GL_2(E)$ to $\GL_2(E)^{\det}$ (as we did in the previous proposition). 
\vskip 5pt

 Recall from (\ref{E:PiMJ})  that 
 \[  \Pi_{M_J} = {\Ind}_{M_{J,K}}^{M_J} \Pi_{M_{J,K}},  \]
 with $\Pi_{M_{J,K}}$ equal to the restriction of the similitude theta lift of the trivial representation of $\GU(K) = K^{\times}$.
  From the seesaw diagram in (\ref{E:seesaw-top-bottom}), $\Pi_{M_{J,K}}$ is naturally a module for
\[  \mathrm{GO}(L, N_{L/E})^{\det} \times \GSp(W)^{\det}  = ((L^{\times})^{\det} \rtimes \langle\tau \rangle) \times \GL_2(E)^{\det}_K \]
which factors through to the quotient   
 \[  H_C(F) \times  \GL_2(E)^{\det}_K = ((L^{\times})^{\det} \rtimes \langle\tau \rangle)/ K^{\times} \times \GL_2(E)^{\det}_K.\]
Here we recall that (see Lemma  \ref{L:L1K1}) that 
  \[  H_C^0(F)= (L^{\times})^{\det}/K^{\times} =  L^1/ K^1 \]
and
\[ \mathrm{GL}_2(E)^{\det}_K = \{ g \in \GL_2(E)^{\det}:  \det(g) \in N_{K/F}(K^{\times}) \}. \]  
Thus, we need to understand the theta correspondence for $H_C(F) \times \GL_2(E)^{\det}_K$ arising from  $M_{J,K}$. 
Indeed, if we  let $\Theta_{M_K}$ denote this theta correspondence, then for any $\rho \in \mathrm{Irr}(H_C(F))$, 
\[  \Theta_M(\rho)  = {\Ind}_{\GL_2(E)^{\det}_K}^{\GL_2(E)^{\det}} \Theta_{M_K}(\rho). \]
 We have thus explained the reduction of the determination of the mini-theta correspondence to the similitude theta correspondence for
\[   \mathrm{GO}(L, N_{L/E})  \times \GSp(W)^+  \]
together with the understanding of the restriction of the theta lifts to the subgroup $\GSp(W)^{\det}$.
With our knowledge of the theta correspondence for $\mathrm{GO}_2 \times \GL_2^+$, this interpretation  immediately gives us the following:
 \vskip 5pt
 
 \begin{lemma}  \label{L:theta_M}
(i)  For any $\rho \ne \epsilon$ (the sign character of $H_C(F)$), $\Theta_M(\rho)$ is nonzero, whereas $\Theta_M(\epsilon) =0$.
 \vskip 5pt
 
 (ii) For an irreducible representation $\rho(\chi)$ of $H_C(F)$,where $\chi$ is a character of $H_C^0(F) = (L^{\times})^{\det}/K^{\times}$, $\Theta_M(\rho(\chi))$ is noncuspidal if and only if  $\chi|_{L^1}$ is trivial on all the anisotropic factors of $L^1 = \prod_i L_i^1$.
 \end{lemma}

\vskip 5pt

In the context of (ii) of the Lemma, we note:
\vskip 5pt
\begin{itemize}
\item if $K=F^2$, then $H_C^0(F) = E^{\times}/ F^{\times}$ and there are no anisotropic factors of $L^1$, so that $\Theta_M(\rho)$ is noncuspidal (as long as  $\rho \ne \epsilon$).
\vskip 5pt

\item if  $K$ is a field and $E=F\times K$, then $H_C^0(F) \cong K^1 \times K^{\times} /K^1\cong K^{\times}$, 
 and a character  $\chi$ trivial on anisotropic factors can be  identified with a character of $K^{\times}/K^1$. 
 \vskip 5pt
 
 \item if $K$ is a field and $ E \ne F \times K$, only $\Theta_M(1)$ is noncuspidal.
 
 \end{itemize}
 It will turn out that the theta lifts in these cases are contained in  the principal series representations we considered in Proposition \ref{P:restME}.
 \vskip 5pt 
 
 The following proposition continues our study of the mini-theta correspondence by refining Lemma \ref{L:theta_M}:
\vskip 5pt

 \begin{prop}  \label{P:theta_M} 
 For every irreducible unitary representation 
 $\rho\neq\epsilon$ of $H_C(F)$,   $\Theta_M(\rho)$ is an irreducible nonzero representation of $M_E$, whereas $\Theta_M(\epsilon) =0$. Moreover, if $\Theta_M(\rho) \cong \Theta_M(\rho')\ne 0$, then $\rho \cong \rho'$. More precisely:
 \vskip 5pt
 
 \begin{enumerate} 
 \item $\Theta_M(1)$ is an irreducible  summand of $1\times \omega_{L/E}$. 
 \vskip 5pt
 
 \item Let $K=F^2$ and $\chi$ be a character  of $H_C^0(F) \cong E^{\times}/F^{\times}$.  Then
 \[    \chi^2\neq 1 \Longrightarrow  \Theta_M(\rho(\chi))\cong \chi\times 1, \] 
 whereas
 \[  \text{ $\chi^2=1$ but $\chi\neq 1$} \Longrightarrow \Theta_M(\rho(\chi)^+)\oplus \Theta_M(\rho(\chi)^-)\cong \chi\times 1.\] 
 \vskip 5pt
 
 \item  Let $K$ be a field,  $E=F\times K$ and $\chi$  a character  of $H_C^0(F) \cong K^{\times}$ trivial on $K^1$.   Extend  $\chi$ to a character $\tilde \chi$ of 
 $F^{\times} \times K^{\times}$, so that it is trivial on the first factor.   Then:
 \[  \chi^2\neq 1 \Longrightarrow \Theta_M(\rho(\chi))\cong \tilde \chi\times \omega_{L/E}, \] 
 whereas
 \[ 
 \text{ $\chi^2=1$ but $\chi\neq 1$} \Longrightarrow   \Theta_M(\rho(\chi)^+)\oplus \Theta_M(\rho(\chi)^-)\cong \tilde\chi\times \omega_{L/E}.\] 
 
 \vskip 5pt
 
 \item For all other cases of the triple $(E, K, \chi)$ not covered above, $\Theta_M(\rho(\chi))$ is cuspidal.
 \end{enumerate} 
 \end{prop} 
 
 \begin{proof}   
 In view of Lemma \ref{L:theta_M},
 the main issue here is  the irreducibility of $\Theta_M(\rho)$ for $\rho \in \mathrm{Irr}(H_C(F))$. 
  We shall illustrate the argument in the case where $K$ is a field and $E = F^3$; the other cases are similar and sometimes easier.
 \vskip 5pt
 
 For the case under consideration, we have
\[  \Aut_E(C)^0(F) = (K^{\times} \times K^{\times} \times K^{\times})^{\det}/ K^{\times} \rtimes \langle \sigma\rangle, \]
where the superscript $\det$ refers to the subgroup of elements $(x,y,z)$ with $N_{K/F}(x) = N_{K/F}(y) = N_{K/F}(z)$.
Ignoring the element $\sigma$ for the moment, we are thus considering a triple similitude theta correspondence for $\mathrm{GSO}_2^K(F) \times \GL_2(F)_K$.
We record the following known results concerning this similitude theta correspondence:
\vskip 5pt

\begin{itemize}
\item[(a)]  If $\chi$ is a unitary character of  $\mathrm{GSO}_2^K(F) = K^{\times}$ such that $\chi|_{K^1}$ is not quadratic,  or equivalently $\chi^{\sigma}/\chi$ does not factor through $N_{K/F})$, 
then 
\[  \theta(\chi) \cong \theta(\chi^{\sigma})   \in {\Irr}(\GL_2(E)^{\det}) \]
is supercuspidal. Indeed,
\[  \tilde{\theta}(\chi) := {\Ind}_{\GL_2(F)_K}^{\GL_2(F)} \theta(\chi) \]
is an irreducible supercuspidal representation which is dihedral with respect to $K/F$ and no other quadratic fields, so that $\theta(\chi)$ remains irreducible when restricted to $\SL_2(F)$.
\vskip 5pt

 \vskip 5pt

\item[(b)]  if $\chi|_{K^1}$ is quadratic but nontrivial, or equivalently $\chi^{\sigma}/\chi$ is nontrivial but factors through $N_{K/F}$, then $\theta(\chi) = \theta(\chi^{\sigma})$ is an irreducible supercuspidal representation of $\GL_2(E)^{\det}$. Indeed,  
\[  \tilde{\theta}(\chi) := {\Ind}_{\GL_2(F)_K}^{\GL_2(F)} \theta(\chi) \]
is an irreducible supercuspidal representation which is dihedral with respect to $K/F$ and two other quadratic fields. Hence, 
 $\theta(\chi)$ decomposes as  the sum of two irreducible supercuspidal representations when restricted to $\SL_2(F)$:
\[  \theta(\chi) |_{\SL_2} =  \theta(\chi|_{K^1}^+) \oplus \theta(\chi|_{K^1}^-), \] 
where the two summands are the theta lifts (to $\SL_2(F)$) of the two extensions of $\chi|_{K^1}$ to $O_2^K(F)$.   Indeed, if we consider the index $2$ subgroup
\[  \GL_2(F)_K^{\chi} = \{  g \in \GL_2(F)_K: \det(g) = N_{K/F}(z), \,  \chi(z/ \sigma(z)) =1 \}, \]
then each of the two summands $\theta(\chi|_{K^1}^{\pm})$ is an irreducible $\GL_2(F)_K^{\chi}$-module.
\vskip 5pt

\item[(c)]  If $\chi|_{K^1} = 1$, or equivalently $\chi =\chi^{\sigma}$, then $\chi = \mu \circ N_{K/F}$ for some $\mu$ (well-determined up to multiplication by $\omega_{K/F}$) and 
$\theta(\chi)$ is one of the two irreducible summands of  the restriction of $\mu \times \mu \cdot \omega_{K/F}$ to $\GL_2(F)_K$. Moreover, these two summands remain irreducible when restricted to $\SL_2(F)$.
\vskip 5pt

\item[(d)]  $\theta(\chi) \cong \theta(\chi')$ if and only if $\chi' = \chi$ or $\chi^{\sigma}$.
\end{itemize}

\smallskip 

Now we are ready to analyze the triple similitude theta correspondence.
Let $\chi=(\chi_1,\chi_2,\chi_3)$ be a character of $(K^{\times})^3$ such 
that $\chi_1 \cdot \chi_2 \cdot \chi_3=1$.   We need to study the reduciblity of 
\[ \theta(\chi_1,\chi_2,\chi_3):=  \theta(\chi_1) \otimes \theta(\chi_2) \otimes \theta(\chi_3)  \]
 when restricted to $\GL_2(E)^{\det}_K$. We shall consider several cases in turn:

\vskip 5pt

\begin{itemize}
\item[(i)]  If  $\chi_i|_{K^1}$ is not quadratic nontrivial for all $i$, then  by (a) and (c) above, $\theta(\chi_i)$ remains irreducible when restricted to $\SL_2(F)$. Hence $\theta(\chi_1, \chi_2,\chi_3)$ is irreducible when restricted to $\GL_2(E)^{\det}_K$.
\vskip 5pt

\item[(ii)]  Assume now that exactly one of the $\chi_i|_{K^1}$ is quadratic nontrivial. Without loss of generality, suppose that
$\chi_3|_{K^1}$ is quadratic nontrivial but the other two restrictions are not. Then $\theta_K(\chi_1)$ and $\theta_K(\chi_2)$ are irreducible as $\mathrm{SL}_2(F)$-representations, while $\theta_K(\chi_3)$ is irreducible as $\mathrm{GL}_2(F)_K$-representation. It follows readily that $\theta(\chi_1,\chi_2,\chi_3)$ irreducible as an $\mathrm{GL}_2(E)_K$-representation. 
\vskip 5pt

\item[(iii)]  Assume next that exactly two of the $\chi_i|_{K^1}$ is quadratic nontrivial.  Without loss of generality, we may suppose
\[  \chi_1 |_{K^1} = \chi_2|_{K^1}= \mu \quad \text{and} \quad  \chi^3|_{K^1}  =1  \]
 for some quadratic character $\mu$ of $K^1$. In this case, by (b) above, we have
 \[  \theta(\chi_1)|_{\SL_2}=  \theta(\chi_2)|_{\SL_2}  = \theta(\mu^+) \oplus \theta(\mu^-) \]
 as $\SL_2(F)$-modules. Now   it is easy to check that 
\[ 
[ \theta(\mu^+) \otimes \theta(\mu^+) \oplus \theta(\mu^-) \otimes \theta(\mu^-)]  \otimes \theta(1) 
\] 
and 
\[ 
[ \theta(\mu^+) \otimes \theta(\mu^-) \oplus \theta(\mu^-) \otimes \theta(\mu^+)]  \otimes \theta(1) 
\] 
are  irreducible representations of $\mathrm{GL}_2(E)_K$.  In particular, $\theta(\chi_1,\chi_2,\chi_3)$ is the sum of two irreducible representations as $\GL_2(E)^{\det}_K$-modules.
\vskip 5pt

\item[(iv)]  Finally, we consider the case when $\mu_i:= \chi_i|{K^1}$ is quadratic nontrivial for all $i$; this case can only occur when the residue characteristic of $F$ is $2$. In this case,
\[  \theta(\chi_1,\chi_2,\chi_3) = 
  [\theta(\mu_1^+) \oplus \theta(\mu_1^-)] \otimes  [\theta(\mu_2^+) \oplus \theta(\mu_2^-)] \otimes  [\theta(\mu_3^+) \oplus \theta(\mu_3^-)] 
\] 
as $\SL_2(F)^3$-modules.  The key observation here is  that  each $\mathrm{GL}_2(F)_K^{\chi_i}$ acts irreducibly on $\theta(\chi_j)  = \theta(\mu_j^+) \oplus \theta(\mu_j^-)$ if $i\neq j$, and preserves each summand if $i=j$. 
Now it is easy to check that for every $\epsilon =\pm 1$, 
\[ 
 \oplus_{\epsilon_1\epsilon_2\epsilon_3=\epsilon} \theta(\mu_1^{\epsilon_1}) \otimes  \theta(\mu_2^{\epsilon_2}) \otimes  \theta(\mu_3^{\epsilon_3})  . 
\] 
is an irreducible representation of $\mathrm{GL}_2(F^3)_K$.  In particular, $\theta(\chi_1,\chi_2,\chi_3)$ decomposes as the sum of two irreducible $\GL_2(E)_K^{\det}$-modules. 
\end{itemize}

\vskip 5pt

With the above results, we can now 
complete the proof of the proposition when $E = F^3$ and $K$ is a field.  
Note that we are only concerned with the restriction of $\chi_1\times \chi_2 \times \chi_3$ to the subgroup:
\[ H_C^0(F)=  ((K^{\times})^3)^{\det} / \Delta K^{\times}  = (K^1)^3/ \Delta K^1. \]
So for example, we have:
\vskip 5pt

\begin{itemize}
\item $\chi_1 \times\chi_2 \times \chi_3$ restricts to the trivial character if and only if  $\chi_i|_{K^1} =1$ for each $i$. 
\vskip 5pt

\item The restriction $\chi$ of $\chi_1\times \chi_2 \times \chi_3$ is a nontrivial quadratic character if and only if $\chi_i|_{K^1}$ is quadratic for all $i$ and is nontrivial for some $i$. 
\end{itemize}
In particular, we see that the latter case corresponds precisely to the cases (iii) and (iv) analyzed above. In this case, there are thus two extensions $\chi^{\pm}$ of $\chi$ to $H_C(F)$ 
and (in view of Lemma \ref{L:theta_M}) $\Theta_{M_K}(\chi^{\pm})$ are both nonzero and hence are precisely the two irreducible summands of $\theta(\chi_1,\chi_2,\chi_3)|_{\GL_2(E)_K}$ described in (iii) and (iv) above.  
\vskip 5pt

Finally, from the properties of the similitude theta correspondence, we deduce that
\[   \chi' = \text{$\chi^{\pm 1}$ on $(K^1)^3$} \Longleftrightarrow \text{$\theta(\chi) = \theta(\chi')$ on $\GL_2(E)^{\det}_K$.} \] 
This concludes the proof of the proposition, at least in the case when $E = F^3$ and $K$ is a field. 

\end{proof}

 \vskip 10pt
 
 \subsection{\bf Whittaker models}
  For a fixed $C$, with associated Springer decomposition $J= E \oplus C$, we have obtained a subset
 \[ {\Irr}_C(M_E(F)) :=  \{ \Theta_{M,C}(\rho) \in {\Irr}(M_E(F)): \epsilon \ne \rho \in \mathrm{Irr}_{unit}(H_C(F)) \}  \subset {\Irr}(M_E(F)). \]
 Moreover, the representations in ${\Irr}_C(M_E(F))$ are infinite-dimensional and hence generic. 
  In this subsection, we investigate the Whittaker models supported by the representations in ${\Irr}_C(M_E(F))$. 
  This serves to complete our analysis of the mini-theta correspondence by specifying precisely the irreducible representations $\Theta_{M,C}(\rho)$.
    \vskip 5pt

 We had briefly alluded to the Whittaker support of $\Pi_M$ as an $\GL_2(E)^{\det}$-module in \S \ref{SS:formulas} and \S \ref{SS:Schro}, but let us be more precise here.
 Fix a nontrivial additive character $\psi$ of $F$. Then a generic character  for the unipotent radical of the upper triangular Borel subgroup $B_E$ of  $M_E(F) = \GL_2(E)^{\det}$ is of the form
 \[  u(b) \mapsto  \psi( \mathrm{Tr}_{E/F}(ab))   \quad \text{for some $a \in E^{\times}$.} \]
 We denote this generic character by $\psi_a$. Two such generic characters  $\psi_a$ and $\psi_{a'}$ are equivalent if they are conjugate by the action of the diagonal torus and we call an equivalence class a Whittaker datum for $M_E(F)$. 
 A short computation shows that the set of Whittaker data is parametrized by $E^{\times}/ F^{\times} E^{\times 2}=H^1(F, Z_E)$. Hence we have yet another interpretation of $H^1(F, Z_E)$:
 \[  
 \begin{CD}
 H^1(F, Z_E)  @=   \{ \text{Whittaker datum for $\GL_2(E)^{\det}$} \}    \\
 @|  @|   \\
  G_E^{ad}(F)/ \mathrm{Im}(G(E)) @= \{ \text{rank $1$ $E$-twisted composition algebras}\} \end{CD} \]

    \vskip 5pt
 
 We are interested in computing the twisted Jacquet module
 \[  (\Pi_{M_J})_{U_E, \psi_a}  \quad \text{  as a $H_C(F)$-module,}  \]
   For this purpose, we shall make use of the Schrodinger model of $\Pi_{M_J}$ introduced in \S \ref{SS:Schro} and the results of \S \ref{SS:embedding 1-2}.
To formulate the result, let us recall from Lemma \ref{L:embedding 1-2} the $H_C(F)$-set
\[  X_{a,C}(F) = \{  x \in C:  Q(x) = a^{\#} \, \text{and} \, \beta(x) = ax \} \]
which is in bijection with the set of embeddings $C_a \hookrightarrow C$, where $C_a$ is a rank 1 $E$-twisted composition algebra defined in \S \ref{SS:rank1}. 
Moreover, if $X_{a,C}(F)$ is nonempty, then it is a principal homogeneous space for $H_C^0(F)$, so that the stabilizer in $H_C(F)$ of any point in $X_{a,C}(F)$ has order $2$.  Now we have:
 \vskip 5pt

 \begin{lemma}
 Fix an $E$-twisted composition algebra $C$ of rank $2$, with associated Springer decomposition $J = E \oplus C$. For each $a \in E^{\times}$,  one has
 \[  (\Pi_{M_J})_{U_E, \psi_a}  \ne 0 \Longleftrightarrow  X_{a,C}(F) \ne \emptyset, \]
 in which case
 \[  (\Pi_{M_J})_{U_E, \psi_a} \cong {\Ind}_{\mathrm H_{C, x_a}(F)}^{H_C(F)}  1  \]
 where $x_a \in X_{a,C}(F)$  and $H_{C,x_a}(F) \cong \Z/2\Z$  is the stabilizer of $x_a$ in $H_C(F)$. 
 \end{lemma}

\vskip 5pt

\begin{proof}
From the Schr\"odinger model of $\Pi_{M_J}$ discussed in \S \ref{SS:Schro} and the results of Lemma \ref{L:embedding 1-2}, we see that
\[  (\Pi_{M_J})_{U_E, \psi_a} = C^{\infty}_c(J_{rk=1})_{U_E,\psi_a} = C^{\infty}_c(X_{a,C}(F)), \]
as $H_C(F)$-module. Since $X_{a,C}(F)  = H_C(F) \cdot x_a  \cong H_C(F)/ H_{C,x_a}(F)$, the result follows.
 \end{proof}
\vskip 5pt

Recall the map
\[  f:  H^1(F, Z_E) = E^{\times}/F^{\times} E^{\times 2} \longrightarrow H^1(F, T_{E, K_C})[2]. \]
For each $[C] \in H^1(F, T_{E, K_C})[2]$, we have
\[  f^{-1}([C])  = \{ a \in E^{\times} : X_{a,C}(F) \ne \emptyset \}. \]
Then the above lemma gives the following corollary:
\vskip 5pt

\begin{cor}  \label{C:Whit}
For any $\rho = \rho(\chi)  \in \mathrm{Irr}(H_C(F))$, $\Theta_{M,C}(\rho)_{U_E, \psi_a} = 0$ if  $f(a) \ne [C]$. On the other hand, if $f(a) = [C]$, then we have:
\vskip 5pt

\begin{itemize}
\item If $\chi^2 \ne 1$, $\dim \Theta(\rho(\chi))_{U_E, \psi_a}  =1$.
\vskip 5pt

\item If $\chi^2=1$, so that $\chi$ has two extensions $\tilde{\chi}$ to $H_C(F)$, then  
\[  \dim \Theta_{M,C}(\tilde{\chi})_{U_E,\psi_a}  = \begin{cases} 
1 \text{  if $\tilde{\chi}(g_C(a)) = 1$;} \\
0, \text{  if $\tilde{\chi}(g_C(a)) = -1$.} \end{cases} \]
where $g_C(a)$ is the nontrivial element in $H_{C,x_a}(F)$ for some $x_a \in X_{a,C}(F)$ (see Lemma \ref{L:basept}).
\end{itemize}
\end{cor}
\vskip 5pt

\subsection{\bf As $C$ varies}
In this final subsection,  we allow $[C]$ to vary over $H^1(F, T_{E, K_C})[2]$.  Then  by Lemma \ref{L:XaC}, we have a disjoint union
\[  E^{\times}/F^{\times}E^{\times 2}=  \bigsqcup_{[C]}  f^{-1}([C])  \]
where each $f^{-1}([C])$ is nonempty and is a  $T_{E,K_C}(F)/T_{E,K_C}(F)^2$-torsor.
 We deduce:
\vskip 5pt

\begin{cor}  \label{C:Whit-disjoint}
The union
\[  \bigcup_{[C] \in H^1(F, T_{E, K_C})[2]}   {\Irr}_C(M_E(F))   \subset {\Irr}(M_E(F)) \]
 is disjoint, since  the representations in different subsets have different Whittaker support. 
\end{cor}
\vskip 5pt

We can in fact refine this corollary.
A character $\chi$ of $T_{E,K_C}(F)$ or $\tilde{T}_{E, K_J}(F)$ gives rise to a character $\chi_C$ of each $H_C^0(F)$.
We then consider the $M_E(F)$-module
\[  \Pi_{M_E}[\chi] :=  \bigoplus_{[C] \in H^1(F, T_{E, K_C})[2]}   \Theta_{M,C}(\rho(\chi_C) )  \quad \text{with $\rho(\chi_C) = {\Ind}_{H_C^0(F)}^{H_C(F)} \chi_C$.} \]
 Then we have:
 \vskip 5pt
 
 \begin{cor}  \label{C:Whit-disjoint2}
 For each $a \in E^{\times}$,
 \[  \dim \Pi_{M_E}[\chi]_{U_E,\psi_a} = 1. \]
 In particular,
 \[  \bigoplus_{[C] \in H^1(F, T_{E, K_C})[2]}   \Theta_{M,C}(1)  =   1 \otimes \omega_{L/E}. \]
  \end{cor} 
 Indeed, one can show in general that $\Pi_{M_E}[\chi]$ is the restriction to $M_E(F) = \GL_2(E)^{\det}$ of an irreducible generic representation of $\GL_2(E)$.
 Together with our knowledge of the Whittaker support of the mini-theta lifts. 
 this has the following nice consequence. If $M_E^{ad}$ denotes the Levi subgroup of the Heisenberg parabolic subgroup in the adjoint quotient $G_E^{ad}$,  recall that
 \[  M^{ad}_E(F)/ \mathrm{Im}(M_E(F) )\cong G^{ad}_E(F) /  \mathrm{Im}(G_E(F) ) =  H^1(F, Z_E) = E^{\times}/F^{\times}E^{\times 2}. \]
 Hence, $H^1(F, Z_E)$ acts naturally on ${\Irr}(M_E(F))$ and also on $H^1(F, T_{E, K_C})[2]$ (via the projection $H^1(F, Z_E) \twoheadrightarrow H^1(F, T_{E,K_C})$).
 For an element $\alpha \in H^1(F, Z_E)$ and a character $\chi$ of $T_{E,K_C}(F)$, we then have
 \[  \Theta_{M,C}(\rho(\chi_C))^{\alpha}  \cong \Theta_{M, C^{\alpha}} (\rho(\chi_{C^{\alpha}})), \]
 where the superscript $\alpha$ denotes the two actions of $\alpha$ on the relevant objects mentioned above.  
 \vskip 10pt
 
\section{\bf Langlands quotients of $D_4$}  \label{S:LQ}

The purpose of this section is to write down some representations of $G_E$ that will appear in the  theta lifting from $H_C = \Aut_E(C)$ in terms of their Langlands data, and to 
give explicit realizations  of these representations in some cases. It thus provides the language needed to express the answer for the theta correspondence treated in the next section.
In fact, in Appendix B below, we consider the decomposition of unramified degenerate principal series representations of $G_E$ and introduce notations for many irreducible representations with nonzero Iwahori-fixed vectors, constructed via Hecke algebra considerations. These representations will also appear in this section and the next one.
\vskip 5pt 

\subsection{\bf Langlands quotient from $P_E$.}
As previously,  let $P_E=M_E N_E$ be the Heisenberg maximal parabolic subgroup. The modular character $\rho_{N_E}$ of $M_E$ is 
\[ 
\rho_{N_E}=|\det|^5. 
\] 
Let $\pi$ be a tempered representation of $M_E$. Using the normalized parabolic induction, 
we induce $\pi \otimes |\det|^s$ from $P_E$ to $G_E$, giving  a standard module if $s >0$. 
Let $J_2(\pi, s)$ be the corresponding Langlands quotient when $s>0$. The representation $J_2(\pi, s)$ is also the unique submodule of the representation obtained by 
inducing $\pi \otimes |\det|^s$ from the opposite parabolic $\bar P_E=M_E\bar N_E$. This point of view is more useful to us.   

\vskip 5pt 

\subsection{\bf Langlands quotient from $Q_E$.}
We shall also need some Langlands quotients attached to the 3-step parabolic subgroup $Q_E=L_EU_E$ corresponding to the middle vertex of the Dynkin diagram. Then 
\[ 
L_E\cong (\GL_2(F) \times E^{\times})^{\det}=\{ (g,e) ~|~ \det(g)=N_{E/F}(e)\}. 
\] 
Let $N_{E/F}$ also denote the character of $L_E$ obtained by projecting $L_E$ to $E^{\times}$ followed by the norm on $E$. The modular character  $\rho_{U_E}$ of $L_E$ is 
\[ 
\rho_{U_E}=|N_{E/F}|^3.
\]  
For   a tempered irreducible representation $\pi$ of $L_E$,
 consider the normalized parabolic induction of  $\pi \otimes |N_{E/F}|^s$ from $Q_E$ to $G_E$.  If $s>0$, this is a standard module and we let 
$J_1(\pi, s)$ be the corresponding Langlands quotient. 

\vskip 5pt

We shall need this parabolically induced representation when $\pi$ is one of the following representations:

\vskip 5pt

\begin{itemize}
\item $\pi =\St_E$ is the Steinberg representation of $L_E$  obtained by projecting $L_E$ to $\GL_2(F)$ and pulling back the Steinberg representation of $\GL_2(F)$. 

\vskip 5pt

\item If $E=F\times K$, then we  define a character of $E^{\times}$ equal to $\chi_K$ on the first factor $F^{\times}$ and trivial on the second factor $K^{\times}$. 
We can pull this character back to $L_E$, and  abusing notation,  denote it by $\chi_K$. Note that $\chi_K$ is of course a nontempered representation of $L_E$.
\end{itemize}
\vskip 5pt 

\subsection{\bf Degenerate principal series}
We shall also need the structure and constituents of various unramfied degenerate principal series representations induced from maximal parabolic subgroups. The  necessary results are provided in Appendix B below. We provide here a roadmap for where the various results are located there:
\vskip 5pt

\begin{itemize}
\item when $E$ is a field, the only maximal parabolic subgroups are $P_E$ and $Q_E$. The degenerate principal series associated to $P_E$ is denoted by 
\[  I(s)  =  {\Ind}_{P_E}^{G_E}  |\det|^s  \quad \text{(normalized induction).} \]
The points of reducibility and the module structure at those points are given in Theorem \ref{T:degen_E}. On the other hand, the degenerate principal series associated to $Q_E$ is denoted by
\[  J(s) = {\Ind}_{Q_E}^{G_E}  |N_{E/F}|^s \quad \text{(normalized induction).} \]
Its reducibility points and module structure is described in Theorem \ref{T:degen_E'}.
\vskip 5pt

\item when $E = F \times K$ where $K$ is a field, there are 3 families of degenerate principal series:  $B(s)$ (associated with the $B_2$-maximal parabolic), $A(s)$ (associated to the $A_2$-maximal parabolic) and $I(s)$ (associated to the Heisenberg parabolic, which is the $A_1 \times A_1$-parabolic). The points of reducibility for these are given in Theorem \ref{T:degen_KB}, Proposition \ref{P:degen_KA} and Proposition \ref{P:degen_KA'} respectively.
\vskip 5pt

\item when $E = F^3$ is split, the degenerate principal series has been studied to some extent in the literature, such as \cite{BJ} and \cite{WeRT}. We only need the results concerning $I(s)$ (associated to Heisenberg parabolic) summarized in Proposition \ref{P:degen}.
\end{itemize}
\vskip 5pt

\subsection{\bf $A_2$-parabolic.}
We shall need an explicit description of the quotients $J_2(\pi,s)$ in certain cases. Assume now that $E=F^3$. When writing  $M_E^{\der}=\SL_2 \times \SL_2 \times \SL_2$, we shall 
assume that the three $\SL_2$ correspond, respectively, to simple roots $\alpha_1$, $\alpha_2$ and $\alpha_3$. Let $\pi(\chi)$ (or $\pi(\chi)^{\pm}$)
 be a representation (or two) of $M_E$ 
corresponding to $\chi=(\chi_1, \chi_2, \chi_3)$, a character of $E^{\times}$, as \S \ref{SS:induced-rep}. In particular,  $\chi_1 \cdot \chi_2 \cdot \chi_3=1$. We shall assume that 
$\chi$ is unitary, so that $\pi(\chi)$ is tempered.  Consider the parabolic subgroup in  standard position corresponding to the $A_2$ diagram, containing the vertex 
corresponding to $\alpha_1$.  The character $\chi$ defines a unitary character $\mu_{\chi}$  (temporary notation) of the Levi subgroup given by
\[ 
\mu_{\chi}(\alpha_2^{\vee}(t))= \chi_3(t) \text{ and } \mu_{\chi}(\alpha_3^{\vee}(t))= \chi_2(t). 
\] 
Let $D(\chi)$ be the unitary representation of $G_E$ obtained by inducing (unitary induction) the character $\mu_{\chi}$. Since $D(\chi)$ is unitary, it is completely reducible. 
We now consider three cases:
\vskip 5pt
\begin{itemize}
\item Suppose that $\chi^2 \ne 1$.  By working out exponents (there are 32 of these), one sees that  $D(\chi)$ has a unique irreducible subrepresentation and hence is irreducible. Using exponents again, one may determine the Langlands parameter of $D(\chi)$. It turns out that 
 \[ 
 D(\chi)\cong J_2(\pi(\chi),1).
 \] 
 \vskip 5pt
 
 \item  Suppose that $\chi^2=1$ but $\chi \ne 1$. Then $D(\chi)$ has two irreducible summands:
 \[  D(\chi) =  J_2(\pi(\chi)^+,1) \oplus J_2(\pi(\chi)^-,1). \]
\vskip 5pt

\item Suppose that $\chi=1$. Then $D(1)$ has two irreducible summands.
 The unique spherical summand is isomorphic to $J_2(\pi(1),1)$.  
The exponents of the non-spherical summand can be determined. Indeed, the
spherical summand of $D(1)$ is also the quotient of $I(1/2)$, and the exponents of this quotient are known by Prop \ref{P:degen}. Then, using the exponents, one can determine the
Langlands parameter of the non-sperhcial summand. It turns out that the non-spherical summand is 
 isomorphic to $J_1(\St_E,1/2)$.  Hence
 \[  D(1) =  J_2(\pi(1),1) \oplus J_1(\St_E,1/2). \] 
\end{itemize}

\vskip 5pt 
\noindent 
\underline{Remark:} Despite the fact that $D(\chi)$ is defined by an arbitrary choice of the $A_2$ parabolic,  the Langlands parameter of $D(\chi)$ is independent of this 
choice. Hence the isomorphism class of $D(\chi)$ is, remarkably, independent of the choice, i.e. the isomorphism class of  $D(\chi)$ is invariant by the triality automorphism.

\vskip 5pt 
We need a similar discussion in the case $E=F\times K$. Let $\chi=(\chi_F , \chi_K)$ be a character of $E^{\times}$ trivial on the diagonally embedded $F^{\times}$. 
Consider the (unique) parabolic subgroup in standard position corresponding to the $A_2$ diagram. 
Now $\chi$ defines a character $\mu_{\chi}$ of  the Levi subgroup given by 
\[ 
\mu_{\chi}(\alpha_2^{\vee}(t))= \chi_K(t) \text{ for all  } t\in K^{\times}. 
\] 
Let $D(\chi)$ be the unitary representation of $G_E$ obtained by parabolically inducing  the character $\mu_{\chi}$  (unitary induction). The structure of $D(\chi)$  is similar to that in the split case discussed above. The only difference is that the non-spherical summand of $D(1)$ is the representation $V'_1$ (introduced in \S \ref{SS:EFK-1d} and \S \ref{SS:EFK-1dram} of Appendix B below) with a one-dimensional space of Iwahori-fixed vectors.  It is a Langlands quotient of a standard module induced from $B_2$-parabolic.
 
\vskip 5pt

 We summarize both cases in the following proposition. 
\vskip 5pt

\begin{prop} \label{P:relevant_representations} 
Assume that $E$ is not a field.  Let $\chi$ be a unitary character of $E^{\times}$ trivial on $F^{\times}$ and consider  the representation $D(\chi)$ induced from a parabolic subgroup 
of type $A_2$  as defined above. Then 
\begin{enumerate} 
\item If $\chi^2\neq 1$, then $D(\chi) \cong J_2(\pi(\chi),1)$. 
\item If $\chi^2=1$ but $\chi\neq 1$, then $D(\chi) \cong J_2(\pi(\chi)^+,1)\oplus J_2(\pi(\chi)^-,1)$. 
\item If $E=F^3$, then $D(1)\cong J_2(\pi(1), 1) \oplus J_1(\St_E,1/2)$. 
\item If $E=F\times K$, then $D(1)\cong J_2(\pi(1), 1) \oplus V'_1$ (where $V_1'$ is introduced in \S \ref{SS:EFK-1d} and \S \ref{SS:EFK-1dram}).
\end{enumerate} 
\end{prop} 

\vskip 5pt

As we see from the above proposition, we shall need to refer to representations of $G_E(F)$ which are constructed in Appendix B below, where we study the decomposition of unramified degenerate principal series representations of $G_E$. Some of these representations will appear in the theta lifting from $H_C$ which we shall consider next.

\vskip 10pt

 \section{\bf Theta correspondence for $E_6$}   \label{S:main-theta}
 In this section, we will study the theta correspondence for $H_C \times  G_E \subset G_J = \Aut(J)$, where $J = E \oplus C$ is a  Freudenthal Jordan algebra of dimension 9. 
 The main goal is the following theorem, whose proof will occupy the rest of this section.  
 
 \begin{thm}  \label{T:111}
 For every unitary irreducible representation $\rho$ of $H_C(F)$, $\Theta(\rho)$ is non-zero and irreducible.  If $\Theta(\rho)\cong \Theta(\rho')$, for 
 two irreducible representations $\rho$ and $\rho'$  of $H_C(F)$,  then $\rho\cong\rho'$.    More precisely:
\vskip 5pt
\begin{enumerate} 
\item If $J=D$ (a cubic division algebra), so that $E$ is a field, then $\Theta(1)=V'_1=J_1(\St_E, 1/2)$ (see \S \ref{SS:E-1d} and \S \ref{SS:E-1dram} for the definition of $V_1'$, as well as Theorem \ref{T:degen_E}) and $\Theta(\rho)$ is supercuspidal for all $\rho\neq 1$. 
\vskip 5pt
\item If $J\neq D$ and $\rho\neq \epsilon$, then $\Theta(\rho)=J_2(\Theta_M(\rho),1/2)$. 
\vskip 5pt
\item If  $J\neq D$ and $H_C^0$ is anisotropic, then $\Theta(\epsilon)$ is supercuspidal. Otherwise: 
\vskip 5pt
\begin{itemize} \item If $E=F^3$ and $J=M_3(F)$, then $\Theta(\epsilon)=J_1(\St_E, 1/2) $. 
\item If $E=F\times K$ and $J=M_3(F)$, then $\Theta(\epsilon)=V'_1$ (see \S \ref{SS:EFK-1d} and \S \ref{SS:EFK-1dram} for the definition of $V'_1$).
\item If $E=F\times K$ and $J=J_3(K)$, then $\Theta(\epsilon)=J_1(\St_E\otimes \chi_K, 1/2)$.
\end{itemize} 
\end{enumerate} 

\end{thm} 

\vskip 5pt

\subsection{\bf $E$-twisted cubes.}
Recall from \S \ref{SS:relation} that if $P_E = M_E N_E$ is the Heisenberg parabolic subgroup of $G_E$, then the representation of $M_E = \GL_2(E)^{\det}$ on 
\[  N_E /Z_E = F \oplus E \oplus E \oplus F \]
is the space of $E$-twisted Bhargava cubes. As we summarized in Proposition \ref{P:orbits},
 the $M_E$-orbits of nondegenerate cubes are parametrized by $E$-isomorphism classes of $E$-twisted composition algebra of dimension $2$ over $E$.   
Indeed, for any nondegenerate  cube $\Sigma$,  one attaches a twisted composition algebra structure $(Q_{\Sigma}, \beta_{\Sigma})$ on $C_{\Sigma} = E^2$, so that there is a natural isomorphism

\begin{equation} \label{E:stabilizer}
  \mathrm{Stab}_{M_E}(\Sigma) \cong  \Aut_E(C_{\Sigma}) \quad \text{given by $g \mapsto {^t}g^{-1}$.} 
  \end{equation}
 
\vskip 5pt
 
If we fix a nontrivial additive character $\psi$ of $F$, then the natural pairing between $N_E / Z_E$ and $\bar{N_E}/ \bar{Z_E}$ allows us to identify the unitary characters of $\bar{N}_E$ with elements of $N_E/Z_E$. In particular,   an $E$-twisted cube $\Sigma$ determines  a corresponding character $\psi_{\Sigma}$  of $\bar N_E$.

\subsection{\bf Twisted Jacquet module}
Let $\Pi = \Pi_J$ be the minimal representation of $G_J$.  
We have computed the Jacquet module $\Pi_{\bar N_E}$ in \S \ref{S:jac-E6}. In this subsection, we
determine  the twisted Jacquet module $\Pi_{\bar N_E, \psi_{\Sigma}}$ for the character $\psi_{\Sigma}$ of $\bar N_E$ attached to a nondegenerate $E$-twisted cube $\Sigma$.  Note that $\Pi_{\bar N_E, \psi_{\Sigma}}$ is naturally a representation of $\mathrm{Stab}_{M_E}(\psi_{\Sigma}) \times \Aut_E(C)$, and thus of $\Aut_E(C) \times \Aut_E(C)$ in view of (\ref{E:stabilizer}).
\vskip 5pt

 In \S \ref{SS:minimal-restrict}, we have seen that 
\[   C^{\infty}_c(\Omega)  \subset \Pi_{\bar Z_E} \subset  C^{\infty}(\Omega) \]
where $\Omega$ is the minimal $M_J$-orbit on $N_J/Z_E$, which can be identified with a set of unitary characters of $\bar N_J$. It follows from  the description of $\Pi_{\bar Z_E}$  given in (\ref{SS:minimal-restrict})  that
  \[ 
  \Pi_{\bar N_E, \psi_{\Sigma}}  \cong C_c^{\infty}(\Omega_{\Sigma})  
 \] 
where $\Omega_{\Sigma}$ is the set of elements $\omega\in \Omega$ such that  $\psi_{\omega}$ restricted to $N_E$ is $\psi_{\Sigma}$.   Based on our description of $\Omega$ in \S \ref{SS:minimal orbit}, the following proposition determines the set $ \Omega_{\Sigma}$ concretely.
\vskip 5pt

\begin{prop} \label{P:twisted} 
Let $J=E\oplus C$ be a Freudenthal Jordan algebra of dimension 9. Let $\Sigma$ be a nondegenerate $E$-twisted cube. Then $\Pi_{\bar N_E, \psi_{\Sigma}}=0$ unless $\Sigma$ belongs to the $M_E$-orbit corresponding to $C$ (i.e. $C_{\Sigma} \cong C$). If $C_{\Sigma} \cong C$, then
 \[ \Pi_{\bar N_E, \psi_{\Sigma}} \cong  \mu_K \otimes C^{\infty}_c({\mathrm Isom}(C_{\Sigma}, C)) \]
 where
 \begin{itemize}
 \item[-]  $\mu_K$ is  the restriction of $\chi_J$ to $\mathrm{Stab}_{M_E}(\Sigma)$;  in particular, $\mu_K$ is either trivial or the sign character of $\mathrm{Stab}_{M_E}(\Sigma) \cong H_C(F)$ depending on whether $\omega_{K/F}(-1) = +1$ or $-1$;
 \vskip 5pt
 \item[-] the action of $\mathrm{Stab}_{M_E}(\Sigma) \times \Aut_E(C)$  on $C^{\infty}_c(\mathrm{Isom}(C_{\Sigma}, C))$ is the regular representation (via (\ref{E:stabilizer})). 
 \end{itemize}
\end{prop} 
\begin{proof} 
Since every nondegenerate $M_E$-orbit contains reduced cubes, we may assume
without loss of generality  that $\Sigma$ is reduced, i.e. 
\[ \Sigma=(1,0,f,b), \]
The associated twisted composition algebra $C_{\Sigma}$ is then described in Proposition \ref{P:reduced}. 
\vskip 5pt

 Now the projection map
 \[   N_J/Z_J  = F \oplus J \oplus  J \oplus F \longrightarrow N_E/ Z_E  = F \oplus E \oplus E \oplus F \]
 induced by the restriction of characters is given by
 \[  (a, x,y,d) \mapsto (a  , - e_x, e_y, -d) \]
 where we have writtem
 \[  x = (e_x, c_x) \quad \text{and} \quad y= ( e_y, c_y) \in E \oplus C = J. \]
Hence,   if  $\omega=(a,x,y,d)\in \Omega_{\Sigma}$, so that $\psi_{\omega}$ restricts to $\psi_{\Sigma}$, then $a=1$, so that 
\[ \omega=(1,x,x^{\#}, N_J(x))  \quad \text{(by  Proposition \ref{P:orbit}). } \]
 Writing $x=(e, v) \in E \oplus C = J$ and noting that $(0,v)^{\#}=(-Q(v), \beta(v))$, we then deduce that 
  \[ e=0  \quad \text{and} \quad  Q(v)=-f. \]
  Finally, since $N_J(x)=  N_C(v)$, we also have
 \[  N_C(v)=- b. \]
 Hence, we have a natural ${\mathrm Stab}_{M_E}(\Sigma) \times \Aut_E(C)$-equivariant  identification
 \[  \Omega_{\Sigma} = \{ (v, \beta(v))  \in C^2 : Q(v) = -f  \text{ and }  N_C(v)  = -b \} \subset C^2 = E^2 \otimes_E C, \]
 where the action of $\Aut_E(C)$ is componentwise, whereas that of  $\mathrm{Stab}_{M_E}(\Sigma) \subset \GL_2(E)^{\det}$ is via the standard representation on $E^2$. 
 Thus, the $\mathrm{Stab}_{M_E}(\Sigma) \times \Aut_E(C)$-set $\Omega_{\Sigma}$ is nothing but the $\mathrm{Stab}_{M_E}(\Sigma) \times \Aut_E(C)$-set $\Omega_{C, f, b}$ studied in  Corollary \ref{C:phs} and Lemma \ref{L:phs}.  We thus deduce that $\Omega_{\Sigma}=\emptyset$, unless $C$ is isomorphic to $C_{\Sigma}$, in which case $\Omega_{\Sigma}$ is  identified with $\mathrm{Isom}(C_{\Sigma}, C)$ and  $\Pi_{\bar N_E, \psi_{\Sigma}} = C^{\infty}_c(\mathrm{Isom}(C_{\Sigma}, C))$ is the regular representation of $\mathrm{Stab}_{M_E}(\Sigma) \times \Aut_E(C)$ twisted by the quadratic character $\mu_K$. 
\end{proof}
\vskip 5pt 

If we fix a base point $\phi_0 \in \mathrm{Isom}(C_{\Sigma}, C)$, we get an isomorphism $\mathrm{Stab}_{M_E}(\Sigma)  \cong \Aut_E(C)$ and with respect to this, $\Pi_{\bar N_E, \psi_{\Sigma}}$ is the regular representation of $\Aut_E(C) \times \Aut_E(C)$. We assume that this isomorphism has been fixed henceforth. We remark also that the quadratic character 
$\mu_K$ is trivial when $K$ is not a field.  In any case, this extra twist will be quite innocuous for our purpose.

\vskip 5pt 

For later use, we shall now compute the twisted co-invariants for some degenerate cubes in the case when $\Aut_E(C)$ is anisotropic. Consider 
\[  \Sigma=(1,0, f,0) \quad \text{ with  $f^{\#}=0$.} \] 
We have:
\begin{itemize}
\item If $f=0$,  this cube belongs to the minimal  $G_E$-orbit ($A_1$). 
\item If $f\neq 0$ and $f^{\#}=0$, then $E$ is not a field. We consider the two cases:
\begin{itemize}
\item If $E=F+ K$ with  $K$  a field, then $f=(a, 0)$  and $\Sigma$ belongs to a $G_E$-orbit denoted by $2A_1$. 
\item If $E=F^3$ then  $f=(a,0,0)$, $(0,a,0)$ or $(0,0,a)$, reflecting the fact that $G_E$ has three orbits of type $2A_1$ over the algebraic closure, permuted by 
the outer automorphism group $S_3$. 
\end{itemize}
 The rational orbits of these types are parameterized by classes of squares, and $\Sigma$ belongs to the class of $a$. 
\end{itemize}
\vskip 5pt

\begin{prop} \label{P:twisted_II} 
Let $J=E\oplus C$ be a Freudenthal Jordan algebra of dimension 9. Assume that $\Aut_E(C)$ is anisotropic. Let $\Sigma=(1,0,f,0)$ be an $E$-twisted cube such that $f^{\#}=0$.  
Then 

\vskip 5pt

(i) $\Pi_{\bar N_E, \psi_{\Sigma}} \cong C_c^{\infty}(\Omega_{\Sigma})$, with
\[  \Omega_{\Sigma}= \{ v\in C ~|~ Q(v)=-f \text{ and } b_Q(v, \beta(v))=0\}. \]

\vskip 5pt

(ii) If $f=0$,  then $\Omega_{\Sigma}=\{0\}$.

\vskip 5pt

(iii) If $f\neq 0$,  then $\Omega_{\Sigma}$  is compact (possibly empty) 
and $\Aut_E(C)^{0}$ acts transitively on it.  
\end{prop} 
\begin{proof} 
The assertion (i) is clear. For (ii),  since $\Aut_E(C)$ is anisotropic, Proposition \ref{P:vanishing} implies that $\Omega_{\Sigma}={0}$ if $f = 0$. 
\vskip 5pt

The assertion (iii) can be checked by an explicit computation. There are  two cases to consider, depending on whether $E = F^3$ or $E = F \times K$ with $K$ a field.  We examine the case $E=F^3$ as an illustration. 
\vskip 5pt

When $E = F^3$, we have $C=K^3$ for a quadratic field extension $K$ of $F$. Moreover, $Q$ and $\beta$ are of the form 
\[ Q(x,y,z) = (N_{K/F}(x), N_{K/F}(y), N_{K/F}(z))  \quad \text{(up to an element in  $(F^3)^{\times}$)} \]
and  $\beta(x,y,z)= (yz, zx, xy)$.  Then 
\[ 
\Aut_E(C)^0 =\{ (x,y,z) \in (K^{\times})^3 ~|~ N_{K/F}(x)= N_{K/F}(y) = N_{K/F}(z) =  xyz=1 \}.   
\] 
 If $f=(a,0,0)\in F^3$, then  $\Omega_{\Sigma}=\{ (x,0,0) \in C ~|~ N_{K/F}(x)=a\}$, which is a principal homogeneous variety for the group of norm one elements in $K^{\times}$ (possibly with no $F$-rational points).

\end{proof}

\vskip 5pt

\subsection{\bf Nonvanishing and injectivity of theta lifts}

Using the above results, we can now begin our determination of the theta liftings from $\Aut_E(C)$ to $G_E$. 
\vskip 5pt

\begin{prop} \label{P:one}  
Fix an embedding $E \rightarrow J$, so $J=E\oplus C$.  Let $\rho$ be an irreducible representation of  $\Aut_E(C)$. Then 
\begin{enumerate} 
 \item[(i)]  $\Theta(\rho)\neq 0$. 
 \item[(ii)]  If  $\rho'$ is another irreducible representation of $\Aut_E(C)$, then 
 \[  \Theta(\rho)\cong \Theta(\rho') \Longleftrightarrow \rho'\cong\rho. \] 
 \end{enumerate} 
\end{prop}  
\begin{proof}  
Proposition \ref{P:twisted}  shows that as a module for $\Stab_{M_E}(\Sigma)$, 
\begin{equation} \label{E:spec}
\Theta(\rho)_{\bar N_E, \psi_{\Sigma}} = \begin{cases} 
0, \text{  if $C_{\Sigma} \ncong C$;} \\
\rho^{\vee} \cdot \mu_K, \text{  if $C_{\Sigma} \cong C$,} \end{cases} \end{equation} 
 Thus $\Theta(\rho)\neq 0$ and the second statement also follows. 
\end{proof} 

\vskip 5pt

\subsection{\bf Langlands parameters of theta lifts}
We shall construct an explicit subquotient of $\Theta(\rho)$, for  $\rho=1$ if $J=D$ and all \underline{unitary}  $\rho\neq \epsilon$ if $J\neq D$, using the mini theta correspondence.  
Recall that we have an  exact sequence 
\[ 
 0 \rightarrow C_c^{\infty}(\Omega^{\perp}) \rightarrow \Pi_{\bar N_E}  \rightarrow \Pi_{\bar N_J} \rightarrow 0. 
 \]  
Furthermore, $\Pi_{\bar N_J}$, as $M_E \times \Aut_E(C)$-module decomposes as 
\begin{equation} \label{E:PiN}
\Pi_{\bar N_J} =  |\det|^{-2}\otimes \omega_{K/F} \oplus |\det|^{-\frac{3}{2}} \otimes \Pi_{M_J} 
\end{equation} 
where $\omega_{K/F}$ is the quadratic character corresponding to $K = K_J$, viewed as a character of $M_E$ by precomposing $\det$, and $\Pi_{M_J}$ is the minimal representation 
of $M_J$ that has been described  in \S \ref{SS:levi}.  The summand $\Pi_{M_J}$ appears if and only if $J\neq D$ . The action of $\Aut_C(E)$ on the one-dimensional summand  is trivial.

\vskip 5pt 

Assume first that $E$ is a field and $J=D$, which is the easiest case. Then
 \[ \Pi_{\bar N_E}= \Pi_{\bar N_J}=|\det|^{-2} \]
 so $\Theta(\rho)_{\bar N_E}=0$ for all $\rho\neq 1$.  We shall see later in \S \ref{SS:local-cuspidal} that this vanishing implies the cuspidality of $\Theta(\rho)$; for now, we  shall deal with $\Theta(1)$. By Frobenius reciprocity,  we have a map from $\Theta(1)$ into the degenerate principal series representation $I(-1/2)$ (see \S \ref{SS:degpsI}) induced from the Heisenberg parabolic subgroup. The image of this map must be $V'_1=J_1(\St_E, 1/2)$ since 
$(V'_1)_{\bar N_{E}}=|\det|^{-2}$ (and the other irreducible constituents of $I(-1/2)$ have 2- or 3-dimensional space of   $\bar{N}_E$-coinvariants, by Theorem \ref{T:degen_E}).
Thus, $\Theta(1)$ contains $V'_1$ as an irreducible quotient and we shall see later that it is in fact irreducible.
\vskip 5pt 

Now assume $J\neq D$.  
We have seen in (\ref{E:PiN}) that there is an $M_E \times \Aut_E(C)$-equivariant surjection
\[  \Pi_{\bar N_E} \longrightarrow |\det|^{-3/2} \cdot \Pi_{M_J} \]
where $\Pi_{M_J}$ is the minimal representation of $M_J$.
We have also described in Proposition  \ref{P:theta_M} the theta correspondence for the pair $M_E \times \Aut_E(C)$ acting on $\Pi_{M_J}$.
For any $\rho \in \mathrm{Irr}(\Aut_E(C))$ with $\rho \ne \epsilon$,  its theta lift $\Theta_M(\rho)$ on $M_E$ is nonzero irreducible. Hence by Frobenius reciprocity, we obtain a nonzero equivariant map
\[  \Theta(\rho) \longrightarrow \Ind_{\bar P_E}^{G_E} |\det| \otimes \Theta_M(\rho)  \quad \text{(normalized induction),}  \]
with $\Theta_M(\rho)$ as described in Proposition \ref{P:theta_M}. Now the induced representation is essentially the dual of a standard module and hence contains a unique irreducible submodule $\tau$, which is the Langlands quotient $J_2(\Theta_M(\rho), 1)$. This Langlands quotient $\tau$ is thus an irreducible subquotient of $\Theta(\rho)$ when $\rho \ne \epsilon$.
\vskip 5pt

 \subsection{\bf Irreducibility of $\Theta(\rho)$ I}  We shall now complete the correspondence in the case when $\Aut_E(C)^0$ is isotropic. In this case, there exists 
 a non-trivial co-character $\lambda: F^{\times}  \rightarrow \Aut_E(C)^0$. The centralizer of $\lambda$ in $G_J$ is a Levi subgroup. The restriction of the minimal representation on any (maximal) Levi  subgroup is fairly easy to compute. Indeed, this is a standard technique in the theory of exceptional theta correspondences. With that in hand, $\Theta(\chi)$ is easy to compute for 
 every unitary character $\chi$ of $\Aut_E(C)^0$.  
 
 \vskip 5pt 
 We shall execute this strategy in detail in the split case, where  $E=F^3$ and $J=M_3(F)$, so that  $G_J$ is a split group and $G_E$ is the derived group of the $D_4$-parabolic in $E_6$. 
 Then $\Aut_E(C)^0\cong (F^{3})^{\times} / \Delta F^{\times}$  and we can fix this isomorphism as follows. By extending the $E_6$ diagram, we see that $D_4$ sits in 
 three Levi subgroups $G_1, G_2$ and $G_3$ in $E_6$ of type $D_5$. Let $\lambda_i: F^{\times} \rightarrow G_i$ be the co-character generating the center of $G_i$.  (These 
 co-characters are miniscule co-weights.)  They are each unique up to inverse,  but we can pick them so that $\lambda_1(t) \lambda_2(t)\lambda_3(t)=1$ for every $t\in F^{\times}$.  
 Now the map $(t_1, t_2, t_3) \rightarrow \lambda_1(t_1) \lambda_2(t_2)\lambda_3(t_3)$ gives the claimed isomorphism.  
 \vskip 5pt
 
 The restriction of the minimal representation $\Pi$ to a $D_5$ maximal 
 parabolic has been determined in \cite{MS}. In particular, the restriction to $G_1$ is given by an exact sequence 
 \[ 
 0 \rightarrow C_c^{\infty}(\omega) \rightarrow \Pi \rightarrow \Pi_1 \oplus \mathbb C  \rightarrow 0 
 \]  
 where $\omega$ is the highest weight orbit in a 16-dimensional Spin module for $G_1$, the action of $G_1$ is geometric,  and $\Pi_1$ is the minimal representation of $G_1$, twisted by 
 an unramified character. More precisely, the action of $\lambda_1(t)$ on $\Pi_1$ and $\mathbb C$ is given by  $|t|^s$ and $|t|^r$ for two non-zero real numbers. In particular, since these 
 characters are not unitary, the two terms will not contribute to $\Theta(\chi)$ for $\chi$ unitary.  Thus we can concentrate on $C_c^{\infty}(\omega)$. 
 \vskip 5pt
 
 The group $G_E$ has three 
 irreducible $8$-dimensional representations $V_1$, $V_2$ and $V_3$. We pick this numbering so that the restriction of the 16-dimensional Spin module for $G_1$ containing $\omega$ 
 decomposes as $V_2\oplus V_3$.  Let $\omega_i\subset V_i$ be the $G_E$-orbit of highest weight vectors. Then it is a simple exercise, using the Bruhat decomposition for $G_1$, to see that   $\omega$ decomposes into three $G_E$-orbits: 
 \vskip 5pt
 
 \begin{itemize}
 \item an open $G_E$-orbit $\omega_0 \subset \omega$, such that the stabilizer of a point in $\omega_0$ is the derived group of an $A_2$ parabolic subgroup, 
 \item $\omega_2\subset V_2 $ and 
 \item $\omega_3 \subset V_3$.  
 \end{itemize}
 Thus we have an exact sequence  of $G_E$-modules: 
  \[ 
 0 \rightarrow C_c^{\infty}(\omega_0) \rightarrow  C_c^{\infty}(\omega) \rightarrow  C_c^{\infty}(\omega_2) \oplus C_c^{\infty}(\omega_3) \rightarrow 0 . 
 \]  
 Of course, by the $S_3$-symmetry of the situation, $C_c^{\infty}(\omega_1)$ must also contribute in the restriction of $\Pi$.  Indeed, it is contained in $\Pi_1$, where 
 $\lambda_1(t)$ acts by the non-unitary character $|t|^s$. Hence $\lambda_i(t)$ acts on $C_c^{\infty}(\omega_i)$  by the same character, and these terms will not contribute to $\Theta(\chi)$ if 
 $\chi$ is unitary. In particular, we have shown that  for $\chi$ unitary, $\Theta(\chi)$  arises from $C_c^{\infty}(\omega_0)$,  whence it is clear that $\Theta(\chi)=D(\chi)$. 
 \vskip 5pt
 
 It is now easy to finish the argument. For example, for 
  two characters $1$ and $\epsilon$ of $\Aut_E(C)$,  we have just proved that 
 \[ 
 D(1)= \Theta(1) \oplus \Theta (\epsilon). 
 \] 
 On the other hand,  recall from Proposition \ref{P:relevant_representations}(3), that
  \[  D(1) = J_1(\St_E, 1/2) \oplus J_2(\pi(1),1). \]
  Since $\Theta(1) \supseteq  J_2(\pi(1),1)$ and $\Theta(\epsilon)\neq 0$, it follows that $\Theta(1) \cong  J_2(\pi(1),1)$ and $\Theta(\epsilon) \cong J_1(\St_E, 1/2))$.

\vskip 5pt

\subsection{Subregular nilpotent orbit}  Assume now that $\Aut_E(C)$ is anisotropic. We shall prove the irreducibility of the theta lift  $\Theta(\rho)$ by studying its restriction to 
$N_E$ in detail. However, in order to make this strategy work, we need to eliminate subregular nilpotent orbits as leading terms of the wave-front set 
of $\Theta(\rho)$. 
\vskip 5pt

The subregular nilpotent orbit is the Richardson orbit for the $3$-step parabolic subgroup $Q_E = L_E U_E$ corresponding to the the middle vertex of the Dynkin diagram for $D_4$, with $[L_E,L_E]=\SL_2(F)$. Recall from (\ref{SS:3step}) that 
there is a parabolic subgroup $Q_J = L_J U_J$ of $G_J$ whose intersection with $G_E$ is $Q_E$. The unipotent radical of its Lie algebra has a decomposition
 \[  \mathfrak u_J=\mathfrak g_J(1) \oplus \mathfrak g_J(2)\oplus \mathfrak g_J(3) \]
 with 
 \[ 
\mathfrak g_J(1)= F e_1\otimes J \oplus F e_2 \otimes J \cong J^2 , ~\mathfrak g_J(2)= F e_3^*\otimes J \cong J\text{ and } \mathfrak g_J(3)= Fe_{13} \oplus F e_{23}\cong F^2 
\] 
in the notation of (\ref{SS:3step}). The unipotent radical $U_J$ of $Q_J$ has a filtration 
 \[ U_J=U_1\supset U_2\supset U_3 \quad \text{such that  $U_i/U_{i+1} \cong \mathfrak g_J(i)$ for all $i$.}\]
  Hence, the minimal representation $\Pi $  has a filtration 
 \[  \Pi\supset  \Pi_1 \supset \Pi_2 \ldots \quad \text{such that $\Pi/\Pi_{i}= \Pi_{\bar U_i}$. } \]
 In particular, each quotient $\Pi_{i}/\Pi_{i+1}$ is naturally a 
 $\bar U_{i}/\bar U_{i+1}$-module. The group $\bar U_{i}/\bar U_{i+1}$ is abelian and its characters are parameterized by $\mathfrak g_J(i)$.  The characters of 
$\bar U_{i}/\bar U_{i+1}$ that appear as quotients of $\Pi_{i}/\Pi_{i+1}$ are in $\Omega_{\min}(F) \cap \mathfrak g_J(i)$ where $\Omega_{\min}$ is the minimal orbit 
in $\mathfrak g_J$.

 \vskip 5pt
The embedding  $E \subset J$ gives rise to   $G_E \subset G_J$ such that $Q_J \cap G_E = Q_E = L_E \cdot U_E$. In particular, we have an analogue of the above sequence of 
inclusions 
 \[ 
\mathfrak g_E(1)= F e_1\otimes E \oplus F e_2 \otimes E \cong E^2 , ~\mathfrak g_E(2)= F e_3^*\otimes E \cong E\text{ and } \mathfrak g_E(3)= Fe_{13} \oplus F e_{23}\cong F^2 .
\] 
Thus a character of $\bar U_E$ is specified by  a pair $(a,b)\in E^2\cong \mathfrak g_E(1)$. We say that the character is non degenerate if $a$ and $b$ are linearly independent over $F$.  
We now have: 
\vskip 5pt

\begin{lemma} \label{L:subegular} 
 Let $J=E\oplus C$ be a 9-dimensional Freudenthal Jordan algebra such that  $\Aut_E(C)$ is anisotropic. Let $\Pi$ be the minimal representation of 
$G_J$ and $\psi$  a non-degenerate character of $\bar U_E$. Then $\Pi_{\bar U_E, \psi}=0$.  

\end{lemma} 

\begin{proof}  The first step is to show that $\Pi_{[\bar U_E,\bar U_E]}=\Pi_{[\bar U_J, \bar U_J]}$. To that end, for $i=3, 2$, 
we need to show that there are no elements in $\Omega_{\min}(F) \cap \mathfrak g_J(i)$ perpendicular to 
$\mathfrak g_E(-i)$.  If $i=3$ there is nothing to prove, since  $\mathfrak g_E(-3)=\mathfrak g_J(-3)$.  
\vskip 5pt

If $i=2$, then $\mathfrak g_J(2)= Fe_3^* \otimes J\cong J$ and elements in $\Omega_{\min}(F) \cap \mathfrak g_J(2)$ perpendicular to  $\mathfrak g_E(-2)$ are given by $x\in C$, $x\neq 0$,  such 
that $x^{\#}=0$. But there are no such elements, since $\Aut_E(C)$ is anisotropic. 
\vskip 5pt

 As the next step, we need to show that no character of $\bar U_J$ in the minimal orbit restricts to a 
non-degenerate character of $\bar U_E$.  A character of $\bar U_J$ is specified by $(x,y)\in J^2\cong \mathfrak g_J(1)$, and the restriction to $\bar U_E$ is given by projecting $x$ and $y$ 
 on the first summand in the decomposition $J=E\oplus C$.  If $(x,y)$ is in 
 $\Omega_{\min}(F) \cap \mathfrak g_J(1)$ then $x$ and $y$ are linearly dependent over $F$, and hence so are their $E$-components. This completes the proof of the lemma.  

\end{proof} 

\vskip 5pt 

\subsection{Irreduciblity of $\Theta(\rho)$ II}   
We assume that $\Aut_E(C)$ is anisotropic and note the following consequence of Proposition \ref{P:twisted_II} : 
\begin{lemma} \label{L:N_spectrum} 
Let $J=E\oplus C$ be a Freudenthal Jordan algebra of dimension 9. Assume that $\Aut_E(C)$ is anisotropic. Let $\Sigma=(1,0,f,0)$ be an $E$-twisted cube such that $f^{\#}=0$.  
Then 
\begin{itemize} 
\item[(i)]  If $f=0$, then 
\[ \Theta(\rho)_{\bar N_E,\psi_{\Sigma}} \cong \begin{cases}
\mathbb{C}, \text{  if $\rho =1$;} \\
0, \text{  if $\rho \ne 1$.} \end{cases} \]
\vskip 5pt
   
\item[(ii)]  If $f\neq 0$, then $\Theta(\rho)_{\bar N_E,\psi_{\Sigma}}$ is finite-dimensional for any $\rho$. Moreover,   $\Theta(\epsilon)_{\bar N_E,\psi_{\Sigma}} =  0$.  
\end{itemize} 

\end{lemma} 

\vskip 5pt 
We can now prove that $\Theta(\rho)$ is irreducible. The first step is to show that $\Theta(\rho)$ has  its wave-front set supported on the orbit $A_2$, that is, the Richardson orbit 
for the parabolic $P_E$.  There are three larger families of orbits: the regular orbit, the subregular orbit and the Richardson orbits for parabolic subgroups of the type $2A_1$ and we deal with each in turn:
\vskip 5pt
\begin{itemize}
\item The subregular orbits are  eliminated by  Lemma \ref{L:subegular}. 
\vskip 5pt

\item We now deal with the regular orbit.  Assume that $\Theta(\rho)$ is Whittaker generic, where we are using  Whittaker characters of  a maximal unipotent subgroup
containing $\bar N_E$.  Observe that there are infinitely many Whittaker characters which restrict to the character $\psi_{\Sigma_0}$ of $\bar N_E$, where $\Sigma_0=(1,0,0,0)$. 
This contradicts  Lemma \ref{L:N_spectrum}(i) which shows that $\Theta(\rho)_{\bar N_E,\psi_{\Sigma}}$ is finite-dimensional.
\vskip 5pt

\item  The last case, which concerns the Richardson orbit for parabolic subgroups of type $2A_1$ and thus does not occur if $E$ is a field, is  treated similarly. In this case, there are infinitely many characters of the unipotent radical of the $2A_1$ parabolic which restrict to $\psi_{\Sigma}$, where  $\Sigma=(1,0,f,0)$ with $f\neq 0$ but $f^{\#}=0$.   This again contradicts the finite-dimensionality in Lemma \ref{L:N_spectrum}(ii).
\end{itemize}
This completes the first step of the argument. 
\vskip 5pt

The second step is to show that there are no irreducible subquotients of $\Theta(\rho)$ supported on smaller orbits: $3A_1$, $2A_1$, $A_1$ and the trivial orbit.  
The orbit $3A_1$ is not special, so we can disregard it.   We now consider the other possibilities in turn:
\vskip 5pt

\begin{itemize}
\item  Lemma  \ref{L:N_spectrum} and  the finite-dimensionality  of $\Theta(\rho)_{\bar N_E, \psi_{\Sigma}}$ for nondegenerate $\Sigma$ imply that $\Theta(\rho)$ has finite length. Together with the unitarity of $\Theta(\rho)$, this  implies that any irreducible subquotient of $\Theta(\rho)$ is a summand of the minimal representation $\Pi$. Hence, by the theorem of Howe and Moore,  the trivial representation of $G_E$ can not be a summand. 
\vskip 5pt

\item The remaining possible small summands are eliminated using the Fourier-Jacobi functor \cite{WeRT} for the 
Heisenberg parabolic $P_E$. The output of this functor is a $[M_E,M_E]=\SL_2(E)$-module. 
It is easy to check that the Fourier-Jacobi functor applied to $\Pi$ gives    the Weil representation $C_c^{\infty}(C)$ of
$\SL_2(E) \times \O(Q)$, where $\O(Q)$ is the orthogonal group for the quadratic form $Q$ on $C$.  On the other hand, the Fourier-Jacobi functor applied to an irreducible 
representation of $G_E$ with the wave-front set supported in $2A_1$ or $A_1$ gives a representation of $\SL_2(E)$ with the trivial action of $\SL_2(K)$ or 
$\SL_2(E)$ respectively.  Since the matrix coefficients of the Weil representation decay, $\SL_2(E)$ or any of its factors, cannot fix a vector in $C^{\infty}_c(C)$.
\end{itemize}
\vskip 5pt

Now we can complete the proof of the irreducibility of $\Theta(\rho)$ when $\Aut_E(C)$ is anisotropic. The wave-front set of every irreducible subquotient of $\Theta(\rho)$ is supported on orbits of the type $A_2$. However, we know that 
$\Theta(\rho)_{\bar N_E,\psi_{\Sigma}}$ is non-zero only for $\Sigma$ in a single $M_E$-orbit of non-degenerate cubes, in which case this space is an irreducible ${\mathrm Stab}_{M_E}(\Sigma)$-module. Thus there is room for only one irreducible representation in $\Theta(\rho)$.  This proves the desired  irreducibility of  $\Theta(\rho)$ in all cases. 

\vskip 5pt 
\subsection{\bf Cuspidality}  \label{SS:local-cuspidal}
It remains to prove that $\Theta(\rho)$ is supercuspidal  if $\Theta(\rho)_{\bar N_E}=0$.   This follows from Lemma \ref{L:N_spectrum} combined with the following proposition. 
 
\vskip 5pt

\begin{prop}  Let $\pi$ be an irreducible representation of $G_E$  such that $\pi_{\bar N_E}=0$ and  $\pi_{\bar N_E, \psi_{\Sigma}}=0$ for all $\Sigma= (1,0,f,0)$ such that 
$f^{\#}=0$. Then $\pi$ is supercuspidal.  
\end{prop} 
\begin{proof}  Consider the case $E=F \times K$. Let $\mathcal{Q}=\mathcal{L} \cdot \mathcal{U}$ be a maximal parabolic subgroup of $G_E$ such that $\pi_{\bar{\mathcal{U}}} \neq 0$. 
Because $\pi_{\bar N_E}=0$, there are two other maximal parabolic subgroups to consider. 
\vskip 5pt
\begin{itemize}
\item If $[\mathcal{L},\mathcal{L}] \cong \SL_3$, then $\pi_{\bar{\mathcal{U}}} \neq 0$ will admit a non-trivial functional for a character of $\bar U_{\mathcal{L}}$, the unipotent radical of a Borel subgroup of $\mathcal{L}$.  
This character can be inflated to $\bar{\mathcal{U}} \cdot \bar U_{\mathcal{L}}$ and then restricted to $\bar N_E$. The restriction is $\psi_{\Sigma}$ where $\Sigma=(a,0,0,0)$ for some 
$a\in F$. This contradicts the hypotheses of the proposition.
\vskip 5pt

\item  If $[\mathcal{L},\mathcal{L}]\cong SU_{2,2}$, then we take $\bar U_{\mathcal{L}}$ to be the unipotent radical of the maximal parabolic subgroup whose (derived) Levi subgroup 
is $\SL_2(K)$.  This is an abelian subgroup (it is the space of $2\times 2$ hermitian matrices) and $\pi_{\bar{\mathcal{U}}}$ will admit a non-trivial functional for a character of 
$\bar U_{\mathcal{L}}$. 
The rest of the argument goes in the same way as above, leading to $\psi_{\Sigma}$ with $\Sigma=(1, 0,f,0)$ for an $f$ such that $f^{\#} = 0$. 
\end{itemize}
We have thus dealt with the case $E = F \times K$.  The cases when $E$ a field  or $F^3$ are similar and easier. Indeed, for these cases, 
it suffices to assume  that $\pi_{\bar N_E}=0$ and  $\pi_{\bar N_E, \psi_{\Sigma}}=0$  for $\Sigma=(1,0,0,0)$ to conclude the desired cuspidality. 
\end{proof} 

\vskip 5pt

We have now completed the proof of Theorem \ref{T:111}. The following corollary gives an alternative description of $\Theta(1)$ and will be used in \cite{GS3}.
\vskip 5pt 
\begin{cor} Let $\chi$ be a quadratic character of $F^{\times}$. Let $I(\chi, s)$ be the degenerate principal series representation for $G_E$ associated to the Heisenberg 
parabolic subgroup $P_E=M_E N_E$. Then the co-socle of $I(\chi, 1/2)$ is a direct sum of the theta lifts $\Theta_C(1)$ over all isomorphism classes of twisted composition algebras $C$ of $E$-dimension $2$ with associated embedding 
$E\rightarrow J$ such that $K_J$ corresponds to $\chi$ by local class field theory.   
\end{cor} 
\vskip 5pt

\begin{proof}  
Consider any embedding $E \hookrightarrow J$ such that $\chi$ corresponds to $K_J$ by  local class field theory and write  $J=E+C$. Then we have the dual pair 
$G_E \times \Aut_E(C) \longrightarrow G_J$, and we may consider the big theta lift $\Theta_C(1)$ of the trivial representation of $\Aut_E(C)$.  By Theorem \ref{T:111}, we know that $\Theta_C(1)$ is irreducible. On the other hand, observe that $\Theta_C(1)$ maps nontrivially to $I(\chi, -1/2)$ (by using the one dimensional summand of $\Pi_{\bar N_J}$), and thus it is an irreducible submodule of $I(\chi, -1/2)$. Since $\bar N_E$ spectra of $\Theta_C(1)$ for 
non-conjugate embeddings $E \rightarrow J$ are different, we thus have a submodule
\[  \bigoplus_C  \Theta_C(1)  \hookrightarrow I(\chi, -1/2), \]
with the sum running over isomorphism classes of $C$'s considered here.
\vskip 5pt
 
 Now the corollary follows by counting: \, the number of classes of embeddings with $E$ and $K_J$ fixed, given by \cite[Prop. 12.1]{GS2}, 
is equal to the number of representations in  the socle of $I(\chi, -1/2)$, which is given by  \cite[Thm 4.1]{Se2}.  For example, if $\chi\neq 1$, and $E=F+K$, where $K$ is a field, then we have one class of  embeddings if $K\cong K_J$ 
and two otherwise. These two cases can be characterized by $\chi\circ N_{K/F}=1$ and   $\chi\circ N_{K/F}\neq 1$ respectively, and correspond to the cases (6) and (7)
in  \cite[Thm. 4.1]{Se2}. However, the conditions were mistakenly stated there as $\chi\circ N_{E/F}=1$ and   $\chi\circ N_{E/F}\neq 1$, when in fact it was what we wrote here.  
\end{proof}

\vskip 15pt

\section{\bf Archimedean Theta Correspondence}  \label{S:arch-theta}
 
In this section, we consider the theta correspondence for $H_C \times G_E$ over archimedean local fields and formulate the analog of Theorem \ref{T:111}. The main theorems here are
Theorems \ref{T:arch1} and \ref{T:arch2}. The proofs of these theorems will appear in a separate paper, joint with Jeff Adams and Annegret Paul. 
\vskip 5pt

\subsection{\bf Real Freudenthal-Jordan algebras}
Assume first that $F=\mathbb R$; the case $F=\mathbb C$ will be dealt with at the end of this section.  Firstly, we enumerate the 
real Freudenthal-Jordan algebra $J$ of dimension 9:
\vskip 5pt

\begin{itemize}
\item For $K_J = \R^2$, we have $J = M_3(\mathbb R)$;
\item For $K_J = \C$, $J$ is given as the set of fixed points of involutions of the second kind on $M_3(\C)$. Involutions of the second kind  on $M_3(\mathbb C)$ arise from  nondegenerate Hermitian forms $h$ on $\mathbb C^3$, which we may assume to be given by: 
 \[  h=\epsilon_1 z_1 \bar z_1 + \epsilon_2 z_2 \bar z_2 + \epsilon_3 z_3 \bar z_3, \quad  \text{with $\epsilon_i=\pm 1$.} \] 
 There are 8 choices for signs, but we get only 4 different involutions, since $h$ and $-h$ give the same involution. In this way, we get 4 Jordan algebras $J_{\epsilon_1, \epsilon_2,\epsilon_3}$, but the 3 of them  
corresponding to $\{ \epsilon_1, \epsilon_2, \epsilon_3 \} = \{ + , -, -\}$ are isomorphic.  Hence, up to isomorphism, there are two such $J$'s:
\vskip 5pt
\begin{itemize}
\item $J = J_{3,0}(\C) = J_{+++}$;
\item $J = J_{1,2}(\C) = J_{-+-}$ 
 \end{itemize}
 \end{itemize}
We shall sometimes denote the last two cases of $J$ collectively as $J_3(\C)$.   
The group $G_J$ depends only on $K_J$. It is the split group if $K_J=\mathbb R^2$, and quasi-split if $K_J=\mathbb C$ \cite{LS15} .

 \vskip 5pt 
 
 \subsection{Embeddings of cubic algebras} 
 We shall next enumerate the $E$-twisted composition algebra of rank $2$ over $\R$
 by describing embeddings of cubic etal\'e algebras $E$ into $J$.  Note that there are 2 cubic etal\'e $\R$-algebras: 
 \[  E=\mathbb R^3  \quad \text{or} \quad   E=\mathbb R \times \mathbb C. \]
 We consider the various cases in turn:
 \vskip 5pt
 \begin{itemize}
 \item[(a)]  $J = M_3(\mathbb R)$:   in this case,  both $\R^3$ and $\R \times \C$ embeds into $M_3(\mathbb R)$ and these embeddings are unique up to conjugation. 
 \vskip 5pt
 
 \item[(b)]  $J = J_3(\C)$ and $E = \R^3$:  in this case, we may work with the 4 Jordan algebras  $J = J_{\epsilon_1, \epsilon_2, \epsilon_3}$ as described above. 
  For each of these $J$'s, there is  an embedding of $\mathbb R^3$ into $J$  as diagonal matrices. 
 Though 3 of these Jordan algebras are isomorphic (to $J_{1,2}(\C)$),  the three embeddings are not isomorphic.
 To conclude,  we get 4 classes of embeddings in all. 
  \vskip 5pt
  
  \item[(c)]  $J_3(\C)$ and $E = \R \times \C$: in this case,  $E$ does not embed into $J_{3,0}(\C)$
  and there is a unique embedding of   $E$ into $J_{1,2}(\C)$. 
  \end{itemize}
  We take this opportunity to correct a typo at the very end of \cite{GS2}, where it was incorrectly asserted in \cite[Pg. 1956]{GS2} that in the context (b), there are only 2 embeddings of $\R^3$ into $J_3(\C)$, even though the table on \cite[Pg 1954]{GS2} clearly shows that this set of embeddings have 4 elements. 
  \vskip 5pt

\subsection{\bf The torus $\Aut_E(C)^0$}  
 For each embedding $E \hookrightarrow J$, we have a decomposition $J=E\oplus C$. The 
 corresponding $H_C = \Aut_E(C)$ is always a semi-direct product $\Aut_E(C)^{0} \rtimes \mathbb Z/2\mathbb Z$ such that the conjugation action of the non-trivial element in 
 $\mathbb Z/2\mathbb Z$ on $\Aut_E(C)^{0}$ is the inverse involution. 
  The possible cases of the two-dimensional torus $\Aut_E(C)^{0}$  are tabulated in the following table, where $\mathbb T$ is the group of complex numbers of norm one. 
\[
\begin{array}{c||c|c} 
& E=\mathbb R^3 & E=\mathbb R \times \mathbb C \\
\hline 
K=\mathbb R^2 & (\mathbb R^{\times })^3/\Delta(\mathbb R^{\times}) & (\mathbb R^{\times } \times \mathbb C^{\times}) /\Delta(\mathbb R^{\times}) \\ 
\hline 
K=\mathbb C &  (\mathbb T)^3/\Delta(\mathbb T) & (\mathbb T \times \mathbb C^{\times}) /\Delta(\mathbb T)
\end{array} 
\]

\vskip 5pt 
\subsection{\bf Characters of $\Aut_E(C)^0$}
We introduce a refined notation for characters of these tori. 
\vskip 5pt
\begin{itemize}
\item A character $\chi$ of  $(\mathbb R^{\times })^3/\Delta \mathbb R^{\times}$ is a triple of 
characters $(\chi_1, \chi_2, \chi_3)$ of $\mathbb R^{\times}$ such that $\chi_1\cdot \chi_2\cdot \chi_3=1$. 
\vskip 5pt

\item A character of $\mathbb T$ is represented by an integer. Thus 
a character $\chi$ of  $(\mathbb T)^3/\Delta \mathbb T$ is represented by a triple of integers $(n_1, n_2, n_3)$ such that $n_1 + n_2 + n_3=0$. 
\vskip 5pt

\item In the remaining two cases 
a character of the torus is identified with a pair of characters $(\chi_{\mathbb R}, \chi_{\mathbb C})$, such that  $\chi_{\mathbb R}\cdot \chi_{\mathbb C}=1$ on  
$\Delta \mathbb R^{\times}$, and with a pair $(m, \chi_{\mathbb C})$, where $m \in \mathbb Z$, such that the restriction of $\chi_{\mathbb C}$ to $\mathbb T$  is given 
by $z\mapsto z^{-m}$. 
\end{itemize}

\vskip 5pt

\subsection{Representations of $\Aut_E(C)$}  
Let $\chi$ be a character of $\Aut_E(C)$. If $\chi\neq \chi^{-1}$, let $\rho(\chi)\cong \rho(\chi^{-1})$ be the unique irreducible representation of 
$\Aut_E(C)$ such that the restriction to 
$\Aut_E(C)^0$ is $\chi\oplus \chi^{-1}$. If $\chi=\chi^{-1}$, then $\chi$ extends to a character of $\Aut_E(C)$  in two ways, denoted by $\rho(\chi)^{\pm}$. 
These two representations are indistinguishable unless $\chi=1$, in which case  one extension is the trivial representation, denoted by $\rho(1)$, and the other the sign representation $\epsilon$.   Note that non-trivial quadratic characters $\chi$ appear only in the split case (where $E = \R^3$ and $K_J = \R^2$), since $\Aut_E(C)^0(\R)$ is connected as a real Lie group otherwise. 

\vskip 5pt

\subsection{Some tempered representations of $M_E$}  \label{packets} 
 To every unitary character $\chi$ of $\Aut_E(C)^0$, we shall attach 
a packet $P(E,K_J,\chi)=P(E,K_J,\chi^{-1})$ of tempered representations of $M_E\cong \GL_2(E)^{\det}$, 
obtained by restricting an irreducible representation of $\GL_2(E)$. 
We need additional notation. 
\vskip 5pt

\begin{itemize}
\item For a local field $F$ and a pair of characters  $(\mu_1,\mu_2)$  of  $F^{\times}$, 
let $\mu_1\times \mu_2$ be the unique infinite-dimensional subquotient of the principal series representation of $\GL_2(F)$
obtained by normalized parabolic induction from the pair of characters. 
\vskip 5pt
\item Let $\omega : \mathbb R^{\times} \rightarrow \{\pm 1\}$ be the sign character. It is the unique non-trivial quadratic character of $\mathbb R^{\times}$. 
\vskip 5pt
\item Let $\nu : \mathbb R^{\times} \rightarrow \mathbb R^{\times}$ be the identity character $\nu(x)=x$, for all $x\in \mathbb R^{\times}$. 
\vskip 5pt

\item For $n\in \mathbb Z$, the principal series representation $\nu^n \times \omega$, when restricted to $\SL_2(\mathbb R)$, contains a sum of two (limits of) 
 discrete series representations with the lowest $\SO_2$-types $\pm(|n|+1)$. 
 \end{itemize}
We can now describe the packet $P(E,K_J,\chi)=P(E,K_J,\chi^{-1})$ of tempered representations of $M_E\cong \GL_2(E)^{\det}$.
\smallskip 

\underline{Case $E=\mathbb R^3$ and $K_J=\mathbb R^2$.} Let $\chi=(\chi_1,\chi_2,\chi_3)$ be a unitary character of $(\mathbb R^{\times })^3/\Delta \mathbb R^{\times}$.  
The packet $P(E,K_J,\chi)$ consists of representations appearing in the restriction to $\GL_2(\mathbb R^3)^{\det}$ of
\[
(\chi_1\times 1) \otimes (\chi_2\times 1)\otimes (\chi_3\times 1). 
\]
 This representation is irreducible when restricted to $\SL_2(\mathbb R^3)$ unless $\chi_i=\omega$ for at least one $i$. The group 
$\GL_2(\mathbb R^3)^{\det}$ is large enough so that the restriction is still irreducible if precisely one $\chi_i$ is $\omega$. In view of the relation 
$\chi_1\cdot \chi_2\cdot \chi_3=1$, at most two $\chi_i$ can be $\omega$, and this is precisely when $\chi$ is a non-trivial quadratic character. Then and only then 
the packet consists of two elements. The standard intertwining operator provides an identification of $P(E,K_J,\chi)$ and $P(E,K_J,\chi^{-1})$. 

\smallskip 
\underline{Case $E=\mathbb R^3$ and $K_J=\mathbb C$.}
Let $\chi=(n_1,n_2,n_3)$ be a character of $\mathbb T^3/\Delta \mathbb T$.  
The packet $P(E,K_J,\chi)$ consists of representations appearing in the restriction to $\GL_2(\mathbb R^3)^{\det}$ of
\[
(\nu^{n_1}\times \omega) \otimes (\nu^{n_2}\times \omega)\otimes (\nu^{n_3} \times \omega). 
\]
The restriction to $\SL_2(\mathbb R^3)$ consists of 8 summands, hence the packet $P(E,K_J,\chi)$ consists of 4 representations.

\smallskip 
\underline{Case $E=\mathbb R\times \mathbb C$ and $K_J=\mathbb R^2$.}
 The restriction from $\GL_2(\mathbb R\times \mathbb C)$ to $\GL_2(\mathbb R\times \mathbb C)^{\det}$ is always irreducible, hence the packets are singletons. 
Let $\chi=(\chi_{\mathbb R}, \chi_{\mathbb C})$ be a unitary character of $(\mathbb R^{\times}\times \mathbb C^{\times} )/\Delta(\mathbb R^{\times})$.  
The packet $P(E,K_J,\chi)$ consists of the restriction to $\GL_2(\mathbb R \times \mathbb C)^{\det}$ of 
\[
(\chi_{\mathbb R}\times 1) \otimes (\chi_{\mathbb C}\times 1). 
\]

\smallskip 
\underline{Case $E=\mathbb R\times \mathbb C$ and $K_J=\mathbb C$.} 
We are again restricting from $\GL_2(\mathbb R\times \mathbb C)$ to $\GL_2(\mathbb R\times \mathbb C)^{\det}$ hence the packets are singletons. 
Let $\chi=(m, \chi_{\mathbb C})$ be a unitary character of $(\mathbb T\times \mathbb C^{\times} )/\Delta \mathbb T$.  
The packet $P(E,K_J,\chi)$ consists of the restriction to $\GL_2(\mathbb R \times \mathbb C)^{\det}$ of 
\[
(\nu^m \times \omega) \otimes (\chi_{\mathbb C}\times 1). 
\]

\smallskip 
Summarizing, we have 4 families of tempered packets 
$P(E,K_J,\chi)=P(E,K_J,\chi^{-1})$ of $\GL_2(E)^{\det}$, parameterized by unitary characters $\chi$ of $\Aut_E(C)^0$. If $E=\mathbb R^3$ and $K_J=\mathbb C$, then 
$|P(E,K_J,\chi)|=4$. As a part of our correspondence result, we will see that the 4 members of this packet are naturally parameterized by the 4 embeddings 
$\mathbb R^3 \rightarrow J_3(\mathbb C)$. 
If $\chi$ is a non-trivial quadratic character (this happens only  if $E=\mathbb R^3$ and $K_J=\mathbb R^2$) then $|P(E,K_J,\chi)|=2$. Let $\pi^{+}(\chi), \pi^{-}(\chi)$ be its constituents. 
Otherwise $|P(E,K_J,\chi)|=1$ and its unique element will be denoted by $\pi(\chi)$. 
\vskip 5pt

\subsection{Main result}  Let $V$ be the Harish-Chandra module of the minimal representation of $G_J$. Consider the dual pair 
$G_E \times \Aut_E(C)$ corresponding to an embedding $E\rightarrow J$.  For every irreducible representation $\rho$ of $\Aut_E(C)$ let 
\[ 
\Theta(\rho) = V /\cap_{\varphi \in \Hom(V, \rho)} {\Ker}(\varphi) 
\] 
where $\varphi$ are homomorphisms in the sense of Harish-Chandra modules. We note that $\Theta(\rho)$ is naturally a $(\mathfrak g_E, K_E)$-module, where 
$K_E$ is the maximal compact subgroup of $G_E$. 
The following will be proved in a joint paper with Jeff Adams and Annegret Paul, though we note that the second bullet, when $\Aut_E(C)$ is compact, is contained in Loke's thesis \cite{Lo}.

\smallskip

\begin{thm}  \label{T:arch1}
 Let $G_E\times \Aut_E(C)$ be the dual pair arising from an embedding $E\rightarrow J$.  
Let $\chi$ be a unitary character of $ \Aut_E(C)^{0}$. 
\begin{itemize} 
\item  If $E\rightarrow J$ is not one of the 4 embeddings $\mathbb R^3 \rightarrow J_3(\mathbb C)$, then 
 $\Theta(\rho(\chi))\cong J_2(\pi(\chi), 1)$, unless  $\chi$ is quadratic and non-trivial, in which case we have $\Theta(\rho^{\pm}(\chi))\cong J_2(\pi^{\pm}(\chi),1)$. 
 \vskip 5pt
 
\item If $E\rightarrow J$ is one of the 4 embeddings $\mathbb R^3 \rightarrow J_3(\mathbb C)$, then 
 $\Theta(\rho(\chi))\cong J_2(\pi, 1)$, where $\pi\in P(E,K_J,\chi)$. As we run through all 4 embeddings $\mathbb R^3 \rightarrow J_3(\mathbb C)$, 
 $\pi$ runs through the 4 representations in $P(E,K_J,\chi)$.
 \end{itemize} 

\end{thm} 

The representation $\Theta(\epsilon)$ is always irreducible, and can be described as it sits in a degenerate principal series representations, along with 
$\Theta(\rho(1))$. 
Let $I_E(s)$ denote the (normalized) degenerate principal series for $G_E$ where we induce $|\det|^s$ from $P_E$. Let $I_E(\omega, s)$ be the quadratic twist 
of this series, i.e. we induce $\omega(\det) \cdot |\det|^s$. (Recall that $\omega$ is the sign character of $\mathbb R^{\times}$.)  The following result is due to 
Avner Segal \cite[Appendix A]{Se2}, but formulated  with our interpretation in terms of theta lifts.  

\begin{thm} \label{T:arch1.5}
 Let $\Theta_{E\rightarrow J}(\rho)$ denote the theta lift of $\rho$ in the correspondence arising from the embedding $E\rightarrow J$. 

\begin{itemize} 
\item For every $E$, we have an exact sequence 
\[ 
0 \rightarrow \oplus \Theta_{E\rightarrow J_3(\mathbb C)} (\epsilon) \rightarrow I_E(1/2) \rightarrow \Theta_{E\rightarrow M_3(\mathbb R)} (\rho(1)) \rightarrow 0 .
\] 
\item For every $E$, we have an exact sequence 
\[ 
0 \rightarrow  \Theta_{E\rightarrow M_3(\mathbb R)} (\epsilon) \rightarrow I_E(\omega, 1/2) \rightarrow \oplus \Theta_{E\rightarrow J_3(\mathbb C)} (\rho(1)) \rightarrow 0 .
\] 
\end{itemize} 
Here $J_3(\mathbb C) = J_{3,0}(\C)$ or $J_{1,2}(\C)$ is any Jordan algebra with $K_J = \C$, and   the sum in both sequences is over the isomorphism 
classes of embeddings of $E$ into $J_{3,0}(\C)$ or $J_{1,2}(\C)$ (recall that there is one class if $E=\mathbb R\times \mathbb C$, and four if $E=\mathbb R^3$). 
 
  \end{thm} 

\vskip 10pt  
\subsection{\bf Complex case}
Assume now that $F=\mathbb C$. In this case $E=\mathbb C^3$ is the only possible case. We have:

\begin{thm} \label{T:arch2}
 Let $\chi=(\chi_1,\chi_2,\chi_3)$ be a unitary character of $(\mathbb C^{\times})^3/\Delta \mathbb C^{\times}$. 
Let $\pi(\chi)$ be the tempered representation of  $M_E=\GL_2(\mathbb C^3)^{\det}$ defined as in the real  split case. Then 
$\Theta(\rho(\chi))= J_2(\pi(\chi))$ if $\chi \ne 1$ and $\Theta(1) \oplus \Theta(\epsilon)\cong D(1)$ is the degenerate principal series for an $A_2$ parabolic subgroup,.  
\end{thm}

\vskip 15pt

\section{\bf Global Theta Lifting}  \label{S:global-theta}
In this section, let $E/F$ be a cubic field extension of number fields, so that $G_E$ is a so-called triality $\Spin_8$.
We shall consider the global theta correspondence for the dual pair 
\[ H_C  \times G_E = \Aut_E(C) \times \Spin_8^E  \longrightarrow G_J \]
 associated to a twisted composition algebra $C$ over $F$ with $\dim_E C =2$, corresponding to an embedding of Jordan algebras $E \hookrightarrow J$, for some Freudenthal-Jordan algebra $J$ of dimension $9$ over $F$.
\vskip 5pt

\subsection{\bf Hecke characters of $\tilde{T}_{E,K}$.}
Recall from \S \ref{SS:isom-tori} that $H_C^0$ is isomorphic to the 2-dimensional torus
\[  \tilde{T}_{E,K} \cong   {\Ker} \left( N_{K/F}:  ({\Res}_{E \otimes K / F} \mathbb{G}_m)  /  ({\Res}_{K/F} \mathbb{G}_m) \longrightarrow 
({\Res}_{E/F} \mathbb{G}_m) / \mathbb{G}_m \right), \]
so that
\[  \tilde{T}_{E,K}(F)  = {\Ker} \left(  N_{K/F} : (E \otimes K)^{\times}/ K^{\times} \longrightarrow E^{\times}/ F^{\times} \right). \]
Before describing the automorphic representation theory of $H_C = \Aut_E(C)$, let us record some relevant facts about automorphic characters of $\tilde{T}_{E,K}$. 
\vskip 5pt

\begin{prop} \label{P:auto-char}
(i) The torus $\tilde{T}_{E,K}$ satisifies the weak approximation property. As such, any two Hecke characters $\chi$ and $\chi'$ of $\tilde{T}_{E,K}$ such that $\chi_v = \chi'_v$ for almost all $v$ are equal.  
\vskip 5pt

(ii) Let $\chi$ and $\chi'$ be two unitary Hecke characters of $\tilde{T}_{E,K}$ such that for almost all $v$, either $\chi'_v = \chi_v$ or $\chi'_v = \chi_v^{-1}$. Then $\chi' = \chi$ or $\chi' = \chi^{-1}$. 
\end{prop}
\vskip 5pt

\begin{proof}
(i) By a result of Voskresenskii \cite{V2}, any tori of dimension $2$ over $F$ satisfies the weak approximation property. 
\vskip 5pt

(ii) Assume first that $K = F \times F$ is split. Then $\tilde{T}_{E,K} =( {\Res}_{E/F} (\mathbb{G}_m) / \mathbb{G}_m$, so that
$\tilde{T}(F) = E^{\times}/ F^{\times}$. We may thus regard $\chi$ and $\chi'$ as Hecke characters of $E^{\times}$.  
Consider now the principal series  representations 
\[    \text{$\pi_{\chi} := \pi(\chi, \chi^{-1})$ and $\pi_{\chi'} := \pi(\chi', \chi'^{-1})$ of $\mathrm{PGL}_2(\mathbb{A}_E)$. } \]   
These are irreducible automorphic representations which are nearly equivalent to each other under our hypothesis. 
If these two principal series representations are locally equivalent for places of $E$ outside a finite set $S$, then
we have an equality of partial Rankin-Selberg L-functions:
\[ 
L^S( s , \pi_{\chi} \times \pi_{\chi})  = L^S(s, \pi_{\chi'} \times \pi_{\chi}), \]
which is more explicitly written as:
\[  \zeta^S(s)^2 \cdot  L^S(s, \chi^2) \cdot L^S(s, \chi^{-2})  = L^S(s, \chi' \chi)\cdot L^S(\chi' \chi^{-1}) \cdot L^S(s, \chi'^{-1} \chi) \cdot L^S(s, \chi'^{-1} \chi^{-1}). \]
Now the LHS has a pole at $s= 1$ and hence so must the RHS. This implies that $\chi' = \chi$ or $\chi^{-1}$, as desired.

\vskip 5pt

Assume now that $K$ is a field. We shall invoke the base change from $F$ to $K$. We claim that the norm maps 
\[  \tilde{T}_{E,K}(K_v) \longrightarrow   \tilde{T}_{E,K}(F_v) \quad \text{and} \quad  \tilde{T}_{E,K}(\A_K) \longrightarrow   \tilde{T}_{E,K}(\A_F) \]
are surjective. Since 
\[  \tilde{T}_{E,K} \times_F K  \cong (E \otimes K)^{\times}/ K^{\times}, \]
this surjectivity claim allows one to reduce to the case of split K treated above, by composing $\chi$ and $\chi'$ with the norm map. 
\vskip 5pt

To show the surjectivity of the local norm map, we shall treat the most nondegenerate case where $L_v := E_v \otimes K_v$ is a field; the other cases are easier. 
Then the norm map
\[  \tilde{T}_{E,K}(K_v) = L_v^{\times} / K_v^{\times} \longrightarrow \tilde{T}_{E,K}(F_v) = {\Ker}  \left(  N_{L_v/E_v} : L_v^{\times}/ K_v^{\times} \longrightarrow E_v^{\times}/F_v^{\times}\right)  \]
is given by
\[  x \mapsto x/ \sigma(x)  \quad \text{where $\sigma \in \Aut(L_v/E_v) = \Aut(K_v/F_v)$. } \]
We thus need to show that
\[  \{ y \in L_v^{\times}: N_{L_v/E_v}(y) \in F_v^{\times} \} = K_v^{\times} \cdot  \{ z \in L_v^{\times}: N_{L_v/E_v}(z) =1 \}. \]
For this, we need to observe that if $y \in L_v^{\times}$ satisfies $N_{L_v/E_v}(y) \in F_v^{\times}$, then in fact $N_{L_v/E_v}(y) \in N_{K_v/F_v}(K_v^{\times})$. This in turn follows from the fact that the natural map
\[  F_v^{\times}/ N_{K_v/F_v}(K_v^{\times})  \longrightarrow  E_v^{\times}/ N_{L_v / E_v}(L_v^{\times}) \]
is an isomorphism (using the fact that $E_v$ is an odd degree extension of $F_v$).

\vskip 5pt

To deduce the surjectivity of the adelic norm map from the local ones, it suffices to note that at places $v$ of $F$ unramified over $L$, the local norm map remains surjective when all the local fields are replaced by their ring of units. 
\end{proof}
\vskip 5pt

\subsection{\bf Automorphic representations of $\Aut_E(C)$.}  \label{SS:auto-HC}
Recall that  one has a short exact sequence of algebraic groups
\[ \begin{CD}
1 @>>>  H^0_C   @>>>  H_C  @>>>   \mu_2@>>> 1  \end{CD} \]
From this, one obtains:
\[ \begin{CD}
 1 @>>>  H^0_C(F) @>>>  H_C(F) @>>> \mu_2(F)  \\
 @. @VVV  @VVV  @VVV  \\
1 @>>>  H^0_C(\A)  @>>>   H_C(\A)  @>>>   \mu_2(\A).    \end{CD} \]
 Because $E$ is a field, the torus $H^0_C$ is anisotropic so that 
\[  [H_C^0] := H^0_C(F) \backslash H^0_C(\A) \quad \text{ and }   [H_C] := H_C(F) \backslash H_C(\A) \]
 are compact.  The automorphic representations of $H^0_C$ are unitary automorphic characters which are classified by global class field theory. We will need to discuss the  automorphic representations of the disconnected algebraic group $H_C$.
\vskip 5pt

Let $\mathcal{A}(H^0_C)$ denote the space of automorphic forms on $H^0_C$. Since $H_C(F)$ acts naturally on $H^0_C(\A)$ by conjugation (preserving $H^0_C(F)$), we have a natural action of $H_C(F)$ on $\mathcal{A}(H^0_C)$ by
\[  (\gamma \cdot f)(t)  = f(\gamma^{-1} t \gamma) \quad \text{  for $\gamma \in H_C(F)$, $t \in H^0_C(\A)$ and $f \in \mathcal{A}(H^0_C)$.} \]
Since $H_C^0$ is abelian,  this action factors through the quotient $H_C(F)/ H_C^0(F) \hookrightarrow \mu_2(F)$.
We now consider two cases, depending on whether this last injection is surjective or not.
\vskip 5pt

 \vskip 5pt

\begin{itemize}
\item[(a)]  $H_C^0(F) = H_C(F)$.   In this case,  $C$ corresponds to an embedding $E \hookrightarrow J$ with $J$ a division algebra. At the nonempty finite set $\Sigma_C$ of places $v$ where $J \otimes_F F_v$ is division,  we have $H_C^0(F_v) = H_C(F_v)$.
\vskip 5pt

Let $\chi = \otimes_v \chi_v$ be a unitary  automorphic character of the torus $H^0_C$, so that 
\[ \chi  : [H^0_C] = H_C(F) \backslash H_C(F) \cdot H^0_C(\A)  \longrightarrow S^1, \]
and hence $\C \cdot \chi \subset \mathcal{A}(H^0_C)$.
 Consider the induced representation
\[  V_C(\chi) := \mathrm{ind}_{ H_C(F) \cdot H^0_C(\A)}^{H_C(\A)}   \chi = \mathrm{ind}_{H^0_C(\A)}^{H_C(\A)} \chi. \]
Then an element in $V(\chi)$ is a smooth function
\[  f :  H_C(F) \backslash H_C(\A)  \longrightarrow \C  \]
such that 
\[  f(tg)  = \chi(t)  \cdot f(g)  \quad \text{  for any $t \in H^0_C(\A)$ and $g \in H_C(\A)$.} \]
Hence we have:
\[  V_C(\chi)  \hookrightarrow \mathcal{A}(H_C).  \]
 As an abstract representation, $V_C(\chi)$ is the multiplicity-free direct sum of all irreducible representations of $H_C(\A)$ whose abstract restriction to $H^0_C(\A)$ contains $\chi$. 
 Indeed,  if one considers the restrictions of functions from $H_C(\A)$ to $H_C^0(\A)$,  the submodule $V_C(\chi)$ is characterizted as the subspace of functions whose restrictions are contained in $\C \cdot \chi \subset \mathcal{A}(H_C^0)$. 
\vskip 5pt

Thus one has the following description of $\mathcal{A}(H_C)$:
\[  \mathcal{A}(H_C) = \bigoplus_{\chi} V_C(\chi), \]
which is an orthogonal direct sum with $\chi$ running over the automorphic characters of $H_C^0$. 

\vskip 5pt
We note that $\mathcal{A}(H_C)$ is not multiplicity-free as a representation of $\H_C(\A)$. Indeed, if $\chi$ and $\chi'$ are two distinct automorphic characters of $H_C^0$, then $V_C(\chi) \cong V_C(\chi')$ as abstract representations if and only if the following two conditions hold:
\vskip 5pt

\begin{itemize}
\item for all $v \notin \Sigma_C$, $\chi'_v = \chi_v^{\pm 1}$,
\item for all $v \in \Sigma_C$,  $\chi'_v = \chi_v$. 
\end{itemize}
\vskip 5pt

By Proposition \ref{P:auto-char}(ii), the first condition implies that $\chi' = \chi^{\pm 1}$ and hence $\chi' = \chi^{-1}$ (since we are assuming that $\chi$ and $\chi'$ are distinct); this then implies by the second condition that $\chi_v^2 =1$ for all $v \in \Sigma_C$.
Thus, if $\chi$ is an automorphic character of $H_C^0 = T_{E,K}$, with the property that $\chi_v^2=1$ for all $v \in \Sigma_C$, but $\chi^2 \ne 1$, then $V_C(\chi) \cong V_C(\chi^{-1})$ as abstract representations, but $V_C(\chi)$ and $V_C(\chi^{-1})$ are orthogonal as subspaces of $\mathcal{A}(H_C)$; alternatively, one distinguishes them by their restriction as functions to $H_C^0$. Thus, $\mathcal{A}(H_C)$ has multiplicity-at-most $2$, but fails to have multiplicity one. What is interesting, however, is that even if the multiplicity of an irreducible representation $\rho$ in $\mathcal{A}(H_C)$ is $2$, there is a canonical decomposition of the $\rho$-isotypic submodule of $\mathcal{A}(H_C)$ into two irreducible summands. These summands are characterized by their restriction (as functions) to $H_C^0$ belonging to $\C \cdot \chi$ or $\C \cdot \chi^{-1}$ for a special $\chi$ as above.

\vskip 10pt

\item[(b)] $H_C(F)/H_C^0(F) \cong \mu_2(F)$. Then for every place $v$, $H_C(F_v)/ H_C^0(F_v) = \mu_2(F_v)$. In this case, the action of $H_C(F)/H_C^0(F) = \mu_2(F)$ on
 $\mathcal{A}(H_C^0)$ needs to be taken into account. 
\vskip 5pt

As before, let $\chi = \otimes_v \chi_v$ be a unitary  automorphic character of the torus $H^0_C$.  The action of $H_C(F)/ H_C^0(F)$ sends $\chi$ to its inverse $\chi^{-1}$. 
Hence, we consider the equivalence relation on automorphic characters of $H_C^0$ given by this action, i.e. modulo inversion. Denote the equivalence class of $\chi$ by $[\chi]$. 
\vskip 5pt

There are now two subcases to consider:

\vskip 5pt

\begin{itemize}
\item[(i)]   $\chi^2 = 1$, so that $\chi$ is fixed by $H_C(F)$ as an abstract representation and the equivalence class $[\chi]$ is a singleton. In this case,  $\chi$ is fixed by $H_C(F)$ as a function on $H^0_C(\A)$ and $\C \cdot \chi \subset \mathcal{A}(H^0_C)$ affords a representation $\chi^*$ of 
$H_C(F) \cdot H^0_C(\A)$ extending $\chi$, characterized by the requirement that $\chi^*$ is trivial on $H_C(F)$. 

 Consider the induced representation
\[  V_C[\chi] := \mathrm{ind}_{ H_C(F) \cdot H^0_C(\A)}^{H_C(\A)}   \chi^*. \]
Then an element in $V_C[\chi]$ is a smooth function
\[  f :  H_C(F) \backslash H_C(\A)  \longrightarrow \C  \]
such that 
\[  f(tg)  = \chi(t)  \cdot f(g)  \quad \text{  for any $t \in H^0_C(\A)$ and $g \in H_C(\A)$.} \]
Hence we have:
\[  V_C[\chi]  \hookrightarrow \mathcal{A}(H_C).  \]
 As an abstract representation, $V[\chi]$ is the multiplicity-free direct sum of all irreducible representations of $H_C(\A)$ whose abstract restriction to $H_C(F) \cdot H^0_C(\A)$ contains $\chi^*$. 
\vskip 5pt

\item[(ii)]  $\chi^2 \ne 1$, so that $\chi$ is not fixed by $H_C(F)$ as an abstract representation and $[\chi] = \{ \chi, \chi^{-1} \}$.   In this case,  the span of $\gamma \cdot \chi$, for all $\gamma \in H_C(F)$, is the 2-dimensional subspace
\[  W_{[\chi]} = \C \cdot \chi \oplus \C \cdot \chi^{-1} \subset \mathcal{A}(H^0_C) \]
 such that
\[  W_{[\chi]}  \cong \mathrm{ind}_{H^0_C(\A)}^{H_C(F) \cdot H^0_C(\A)}   \chi   \]
 as  $H_C(F) \cdot H^0_C(\A)$-module. 
 Consider the induced representation
\[  V_C[\chi]  = \mathrm{ind}_{ H_C(F) \cdot H^0_C(\A)}^{H_C(\A)}  W_{[\chi]}  \cong \mathrm{ind}_{H^0_C(\A)}^{H_C(\A)} \chi. \]
An element of $V_C[\chi]$ is thus  a function 
\[  \phi:  H_C(\A)  \longrightarrow W_{[\chi]}= \C \cdot\chi +  \C \cdot \chi^{-1}  \subset \mathcal{A}(H_C). \]
Setting
\[  f_{\phi}(h)  = \phi(h)(1),    \]
so that $f_{\phi}$ is the composition of $\phi$ with evaluation at $1 \in H_C(\A)$, we see that the map $\phi \mapsto f_{\phi}$ defines an embedding
\[ V_C[\chi] \hookrightarrow \mathcal{A}(H_C). \]
In this way, we shall regard $V_C[\chi]$ as a submodule of $\mathcal{A}(H_C)$ henceforth. 
 As an abstract representation, $V_C[\chi]$ is the multiplicity-free direct sum of all irreducible representations of $H_C(\A)$ whose restriction to $H^0_C(\A)$ contains $\chi$ and $\chi^{-1}$. 
\end{itemize}
\vskip 5pt

Now we have:
\[  \mathcal{A}(H_C) =  \bigoplus_{[\chi]}  V_C[\chi] \]
as $[\chi]$ runs over equivalence classes of automorphic characters $\chi$ of $H_C^0$.   The subspace $V_C[\chi]$ is characterized as the subspace of functions whose restriction to $H_C^0$ is contained in $W_{[\chi]} = \C \cdot \chi + \C \cdot \chi^{-1}$. 
We observe that in this case, the representation $\mathcal{A}(H_C)$ is multiplicity-free.

\end{itemize}
 
\vskip 5pt

\subsection{\bf Global minimal representation.}
To carry out the global theta correspondence, we need another ingredient: the global minimal representation of $G_J(\A)$. For each place $v$ of $F$, we have a local minimal representation $\Pi_v$ of $G_{J}(F_v)$ which is unramified for almost all $v$, so that we may set $\Pi = \otimes_v \Pi_v$. Using residues of Eisenstein series, it has been shown that there is an ($G_J^0(\A)$-equivariant) automorphic realisation 
\[  \theta:   \Pi \hookrightarrow   \mathcal{A}(G_J^0). \]
As before, the group $G_J(F)$ acts on $\mathcal{A}(G_J^0)$ via
\[  (\gamma \cdot \phi)(g)  = \phi(\gamma^{-1} g \gamma)  \quad \text{for $\gamma \in G_J(F)$ and $g \in G_J^0(\A)$.} \]
The embedding $\theta$ is easily checked to be $G_J(F) \cdot G_J^0(\A)$-equivariant. 
\vskip 5pt

We now recall  the main properties of the global minimal representation we shall use. 
Recall the Heisenberg parabolic subgroup $P_J = M_J \cdot N_J$ of $G_J$ with
\[   V_J := N_J^{ab}   =  F + J + J + F.  \]
Using a fixed character $\psi$ of $F \backslash \A$ and the natural pairing between $N_J$ and its opposite $\overline{N}_J$,  the elements of $V_J$ parametrizes automorphic characters of $\overline{N}_J(\A)$ (trivial on $\overline{N}_J(F)$).  
Let $\Omega \subset V_J$ be the minimal nonzero $M_J$-orbit in $V_J$.  
For $\phi \in \Pi$, one has the Fourier expansion
\[  \theta(\phi)_{\overline{Z}_J}(g)  =    \theta(\phi)_{\overline{N}_J} (g)  +  \sum_{x \in \Omega}  \theta(\phi)_{\overline{N}_J,  \psi_x} (g),\] 
where $\bar Z_J$ is the 1-dimensional center of $\overline{N}_J$.  If $M_{J,x}$ denotes the stabilizer of $x \in\Omega$ in the Levi subgroup $M_J$, then
the Fourier coefficient $\theta(\phi)_{\overline{N}_J, \psi_x}$ is left-invariant under $M^{der}_{J,x}(\A) := M_{J,x}(\A) \cap M_J^{der}(\A)$.  On the other hand, when restricted to $M_J(\A)$,  the constant term $\theta(\phi)_{\overline{N}_J}$ is an automorphic form on $M_J$. One has
\[  \theta(\phi)_{\overline{N}_J}  \in    \omega_{K_J/F} |-|^{-2}  \oplus |-|^{-3/2} \cdot \Pi_{M_J},  \]
where   $\Pi_{M_J} = 0$ unless $G_J$ (or equivalently $M_J$) is quasi split, in which case $\Pi_{M_J}$ is the global minimal representation of $M_J$.   
\vskip 5pt

\subsection{\bf Global theta lifts.}
For any automorphic form $f$ on $H_C$, and $\phi \in  \Pi$, we consider the associated global theta lift:
\[  \theta(\phi,f) (g)  = \int_{[H_C]}   \theta(h \cdot \phi)(g)  \cdot \overline{f(h)} \, dh,  \quad \text{with $g \in G_E(\A)$.} \]
Note that we have written $\theta(h \cdot \phi)(g)$ instead of $\theta(\phi)(gh)$ in the integral because $\theta(\phi)$ is only defined as a function of $G_J^0(\A)$. 
Observe however that for $\gamma \in H_C(F)$, 
\[ \theta(\gamma h\cdot \phi)(g)  = \theta(h \cdot \phi)( \gamma^{-1}g\gamma)  = \theta(h \cdot \phi)(g)  \quad \text{for $g \in G_E(\A)$.} \] 
In any case,  $\theta(\phi, f) \in \mathcal{A}(G_E)$. For
any irreducible summand $\rho \subset V(\chi)$, the global theta lift $\Theta(\rho)$ of $\pi$ is defined as the span of all $\theta(\phi, f)$ with $\phi \in \Pi$ and $f \in \rho$, so that
\[  \Theta(\rho)  \subset \mathcal{A}(G_E). \]

\vskip 5pt

\subsection{\bf Cuspidality.}
We first show the following analog of the tower property in classical theta correspondence.
\vskip 5pt

\begin{prop}  \label{P:global-cuspidality}
The global theta lift $\Theta(\rho)$ is contained in the space $\mathcal{A}_2(G_E)$ of square-integrable automorphic forms of $G_E$.  
Moreover, it is cuspidal if and only if the (mini-)theta lift  (via $\Pi_{M_J}$) of $\pi$ to $M_E$  is zero. 
\end{prop}
\vskip 5pt
\begin{proof}
To detect if $\Theta(\rho)$ is square-integrable or cuspidal, we need to compute the constant terms of a global theta lift $\theta(\phi, f)$ along the two maximal parabolic subgroups $\overline{P}_E = M_E \cdot \overline{N}_E$ and $\overline{Q}_E = L_E \cdot \overline{U}_E$ of $G_E$. Hence, we first compute the constant term $\theta(\phi, f)_{\overline{N}_E \cap \overline{U}_E}$ along the unipotent subgroup $\overline{N}_E \cap \overline{U}_E$.  We note that
\[   N_E / Z_E   =  F  \oplus  E  \oplus E \oplus F  \supset  (N_E \cap U_E ) / Z_E  =  0 \oplus E \oplus E \oplus F. \]
Recall that the Heisenberg parabolic subgroup $P_J = M_J \cdot N_J$ of $G_J$ satisfies $P_J \cap G_E =  P_E$, with $N_E \subset N_J$ such that
\[ V_E:= N_E /Z_E  \subset V_J := N_J / Z_E  =  F \oplus J \oplus J \oplus F,  \]
where the embedding $E \hookrightarrow J$ is such that $E^{\perp}  = C$. There is a natural projection map 
\[  pr: V_J  \longrightarrow V_E. \]  
which corresponds to the restriction of (automorphic) characters from $\overline{N}_J(\A)$ to $\overline{N}_E(\A)$. 

\vskip 5pt

For $\Omega \subset V_J$ the minimal $M_J$-orbit, let
\[  \Omega_0  = \{  x \in \Omega:   pr(x)  = (\ast, 0, 0, 0)  \in V_E \}. \]
Then one has
\begin{equation} \label{E:UN}
  \theta(\phi, f)_{\overline{N}_E \cap \overline{U}_E}(g)  = \int_{[H_C]} \overline{f(h)}  \cdot \left(  \theta(\phi)_{\overline{N}_J}(hg)  +  \sum_{x \in \Omega_0}   \theta(\phi)_{  \overline{N}_J, \psi_x}(hg)   \right) \, dh. \end{equation}
 To proceed further, we need to understand the set $\Omega_0$. Clearly, we have $\Omega_0 = \Omega_1 \cup \Omega_2$ where 
 \[  \Omega_1  = \{  x \in \Omega:   pr(x)  = (0, 0, 0, 0)  \in V_E \}   \]
 and
 \[ 
  \Omega_2  = \{  x \in \Omega:   pr(x)  = (t, 0, 0, 0) ,  t \ne 0 \}. \]
 By Proposition 8.1, and using the fact that $E$ is a field, we see that $\Omega_1$ is empty whereas $\Omega_2 = \{ (t, 0 , 0 , 0):  t \in F^{\times} \}$. 
 \vskip 5pt
 
 Hence, we see that 
 \[  \theta(\phi,f)_{\overline{N}_E}(g)  =  \int_{[H_C]} \overline{f(h)}  \cdot  \theta(\phi)_{\overline{N}_J}(hg)  \, dh.  \]
 Since 
 \[  \Pi_{\overline{N}_E}  = \omega_{K_J/F} \cdot |-|^{-2}  \oplus |-|^{-3/2} \cdot  \Pi_{M_J}, \]
with   $\Pi_{M_J}$ only present when $J$ is not division,  we deduce that  the constant term of $\Theta(\rho)$ along $\overline{N}_E$ vanishes unless $\rho$ is the trivial representation or if the (mini-)theta lift of $\rho$ to 
$M_E$ (via $\Pi_{M_J}$) is nonzero. One may check that if $\rho$ is trivial, then it does have  nonzero (mini-)theta lift to $M_E$, so that we may subsume the condition that $\rho$ is trivial into the second condition.    
\vskip 5pt

On the other hand, if $\psi_t$ is the automorphic character of $\overline{N}_J(\A)$ corresponding to $(t,0,0,0) \in \Omega_2(F)$ with $t \ne 0$, then $H_C(F)$ stabilizes $\psi_t$. This implies that in (\ref{E:UN}), 
\begin{equation} \label{E:WF}
 \theta(\phi)_{\overline{N}_J, \psi_t}(hg) = \theta(\phi)_{\overline{N}_J, \psi_t}(g), \end{equation}
so that the contribution of $\Omega_2$ to (\ref{E:UN}) vanishes  if  $f$ is not a constant function.
We have thus shown that if the mini-theta lift of $\rho$ to $M_E$ vanishes (so that $\rho$ is nontrivial in particular), then 
the constant term of $\Theta(\rho)$ along $\overline{N}_E \cap \overline{U}_E$ given in (\ref{E:UN}) vanishes, so that $\Theta(\rho)$ is cuspidal. 
\vskip 5pt

Conversely, it is clear from (\ref{E:UN}) and the above discussion that if the mini-theta lift of $\rho$ to $M_E$ is nonzero, then the constant term of $\Theta(\rho)$ along $\overline{N}_E$ is nonzero and hence $\Theta(\rho)$ is noncuspidal. To summarise, we have shown that $\Theta(\rho)$ is cuspidal if and only if the mini-theta lift of $\pi$ to $M_E$ (via $\Pi_{M_J}$) vanishes.  It remains to examine the case when $\Theta(\rho)$ is noncuspidal and show that $\Theta(\rho)$ is square-integrable nonetheless. 

\vskip 5pt

Suppose then that $\Theta(\rho)$ is not cuspidal, so that $\rho$ has nonzero (mini-)theta lift to $M_E$. For each parabolic subgroup $\overline{R}= \overline{P}_E$, $\overline{Q}_E$ or $\overline{B}_E = \overline{P}_E \cap \overline{Q}_E$, we consider the {\em normalized} constant term of $\Theta(\rho)$ along $\overline{R}$.  Since the Levi subgroup of $R$ is a product of groups of $\GL$-type, the strong multiplicity one theorem for $\GL_n$ implies that each of these normalized constant terms  is a direct sum of a cuspidal component and a noncuspidal component such that the two components are spectrally disjoint (i.e.  the system of spherical Hecke eigenvalues supported by the two parts are different). By the standard square-integrability criterion, we need to show that the (real parts of the) central characters appearing in the cuspidal component lie in the interior of the cone spanned by the positive simple roots which occur in the unipotent radical of $\overline R$.
\vskip 5pt

For the case $\overline{R} = \overline{P}_E$, the cuspidal component of the normalized constant term is contained in the mini-theta lift $\Theta_{M_J}(\rho)$ of $\rho$ to $M_E$. Since the center of $M_E$ is equal to the center of $M_J$, and the central character of $\Pi_{\overline{N}_J}$ is of the form $z \mapsto |z|^2$, this gives the desired positivity for the cuspidal component of $\Theta_{M_J}(\rho)$. By the results of \S 9.3 and Proposition 9.2, $\Theta_{M_J}(\rho) \otimes |\det|^{-1}$ is a summand of a tempered principal series representation of $M_E$.   Thus,  the noncuspidal component of $\Theta_{M_J}(\rho) \otimes |\det|^{-1}$ has normalised constant term consisting of unitary characters.  Since $|\det|$ corresponds to the highest root $3 \alpha + 2 \beta$, we have the positivity of cuspidal exponents along the Borel subgroup $\overline{P}_E \cap \overline{Q}_E$. 
\vskip 5pt

Finally, for the constant term along $\overline{Q}_E$, we claim that there are no cuspidal exponents. For if $\theta(\phi, f)_{\overline{U}_E}$ has nonzero projection to the space of cusp forms of $L_E$, then $\theta(\phi,f)_{\overline{U}_E}$ is in fact cuspidal and so has nonzero Whittaker-Fourier coefficients.  However, it follows from (\ref{E:WF})  that such Whittaker-Fourier coefficients all vanish, unless $f$ is a constant function. If $f$ is constant, then $\theta(\phi, f)$ has nonzero constant term along $\overline{B}_E$ (via our computation of the constant term along $\overline{P}_E$) and so $\theta(\phi,f)_{\overline{U}_E}$ cannot be nonzero cuspidal on $L_E$. 
    
    \vskip 5pt

Hence, we have shown that $\theta(\phi, f)$ is square-integrable. This completes the proof of Proposition \ref{P:global-cuspidality}.
 \end{proof}

\vskip 10pt
\subsection{\bf  Nonvanishing and Disjointness.}
We now consider the question of nonvanishing of the global theta lifting. We shall do this by computing the generic Fourier coefficients of $\theta(\phi, f)$ along the unipotent radical $\overline{N}_E$ of the Heisenberg parabolic subgroup  $\overline{P}_E$. 
These Fourier coefficients  are parametrised by generic cubes in $V_E(F)  = N_E(F)^{ab}$.
Recall that the $M_E(F)$-orbits of generic elements in $V_E(F)$ are parametrised by $E$-isomorphism classes of $E$-twisted composition algebras $A$.  For each such $A$, we let $\psi_A$ denote a character of $\overline{N}_E(\A)$ trivial on $\overline{N}_E(F)$ in the corresponding orbit; there is no loss of generality in assuming that $\psi_A$ corresponds to a reduced cube in $V_E(F)$, and note that the stabilizer of $\psi_A$ in $M_E$ is  isomorphic to $H_A =  \Aut_E(A)$. 
\vskip 5pt

Recall that if $N_J$ denotes the unipotent radical of the Heisenberg parabolic subgroup of $G_J$, then there is a natural projection map $pr: V_J = N_J^{ab}  \longrightarrow V_E$.  
This projection map corresponds to the restriction of characters from $\overline{N}_J(\A)$ to $\overline{N}_E(\A)$. 
Let $\Omega \subset V_J$ be the minimal nonzero $M_J$-orbit in $V_J$.  
Set
\[   \Omega_A  = pr^{-1}  (\psi_A )  \cap \Omega. \]
Then Corollary \ref{C:phs} says that $\Omega_A(F)$ is empty unless $A$ is $E$-isomorphic to  $C$, in which case, $\Omega_A(F)$ is a principal homogeneous space of $H_C(F)$. 
Thus, when $A \cong C$, we may fix an element $\tilde{\psi}_C \in \Omega_C(F)$, so that $\tilde{\psi}_C$ restricts to $\psi_C$ on $\overline{N}_E(\A)$. 
\vskip 5pt
 
Now we have:    

\vskip 5pt

\begin{prop}  \label{P:global-FC}
For $\phi \in \Pi$ and $f \in \rho \subset \mathcal{A}(H_C)$, $\theta(\phi,f)_{\overline{N}_E, \psi_A}$ vanishes (as a function on $G_E(\A)$) unless $A \cong C$, in which case
\[  \theta(\phi,f)_{\overline{N}_E, \psi_C} (g)  = \int_{H_C(\A)}  \theta(h \cdot \phi)_{\overline{N}_J,  \tilde{\psi}_C}( g)  \cdot \overline{f(h)} \, dh. \]  
 Moreover, there exist $\phi$ and $f$ such that $\theta(\phi,f)_{\overline{N}_E,\psi_C}(1) \ne 0$.  
\end{prop}
\vskip 5pt
\begin{proof}
We  have:
\begin{align}
   \theta(\phi,f)_{\overline{N}_E, \psi_A} (g)  &= \int_{[V_E]}  \overline{\psi_A(n)} \cdot   \left( \int_{[H_C]}   \theta(h \cdot \phi)_{\overline{Z}_J}( ng)  \cdot \overline{f(h)} \, dh \right) \, dn \notag \\  
&= \int_{[H_C]} \left(  \int_{[V_E]} \, \overline{\psi_A(n)} \cdot \sum_{\tilde{\psi} \in \Omega(F)}  \theta(h \cdot \phi)_{\overline{N}_J, \tilde{\psi}}(ng)   \, dn \right) \cdot \overline{f(h)} \, dh \notag \\
&= \int_{[H_C]} \,\left(  \sum_{\tilde{\psi} \in \Omega_A(F)}      \theta(h \cdot \phi)_{\overline{N}_J, \tilde{\psi}}(g) \right)  \cdot \overline{f(h)} \, dh.  \notag 
\end{align}
This gives the vanishing of $  \theta(\phi,f)_{\overline{N}_E, \psi_A}$ when $A \ncong C$ since $\Omega_A(F)$ is empty in that case. When $A = C$ and  $\tilde{\psi}_C \in \Omega_C(F)$, then we have an identification $H_C(F) \cdot \tilde{\psi}_C \cong \Omega_C(F)$, in which case
\begin{align}
  \theta(\phi,f)_{\overline{N}_E, \psi_C} (g) &=
\int_{[H_C]} \, \sum_{\gamma \in H_C(F)}      \theta( \gamma h \cdot \phi)_{\overline{N}_J, \tilde{\psi}_C}(g)  \cdot \overline{f(h)} \, dh \notag \\
&= \int_{H_C(\A)} \theta(h \cdot \phi)_{\overline{N}_J, \tilde{\psi}_C}(g) \cdot  \overline{f(h)} \, dh, \notag 
\end{align}
as desired. This proves the first statement.
\vskip 10pt

 To show the second statement, we need to understand the function $h \mapsto \theta(h \cdot \phi)_{\overline{N}_E, \tilde{\psi}_C}(1)$ as a function on $H_C(\A)$. 
 For a nonarchimedean place $v$ of $F$, a  property of the local minimal representation $\Pi_v$ is that
 \[ \dim \Hom_{\overline{N}_J(F_v)}(\Pi_v,  \tilde{\psi}_{C,v})  = 1.  \]
Moreover, a nonzero element of this 1-dimensional space can be constructed as follows. Recall that, at a nonarchimedean place $v$,  one has  \cite[Thm. 6.1.1]{KP}
\[  C^{\infty}_c(\Omega(F_v)) \hookrightarrow \Pi_{\overline{Z}_E(F_v)} \hookrightarrow C^{\infty}(\Omega(F_v)). \]
Thus elements of $\Pi_v$ gives rise to functions on the cone $\Omega(F_v)$. 
Then the evaluation map at $\tilde{\psi}_C \in \Omega(F_v)$ defines a nonzero element of  $\Hom_{\overline{N}_J(F_v)}(\Pi_v,  \tilde{\psi}_{C,v})$.
For $v$ outside some sufficiently large set $S$ of places of $F$, $\phi_v$
 is the unramified vector in $\Pi_v$, in which case the corresponding function $f_{0,v}$ on the cone $\Omega(F_v)$ has the following properties.
 The function $f_{0,v}$ is supported on the subset
 \[  \bigcup_{n \geq 0}  \varpi_v^n  \cdot \Omega(\mathcal{O}_v), \]
 is constant on each annulus $\varpi^n \cdot \Omega(\mathcal{O}_v)$, and takes value $1$ on $ \Omega(\mathcal{O}_v)$. Indeed, \cite{KP} gives an explicit formula for the value taken by $f_{0,v}$ on each annulus, but we won't need this here. 

\vskip 5pt 
 We need to understand the restriction of $f_{0,v}$ to the subset $\Omega_C(F_v)$. 
  Since $\Omega_C \subset \Omega \subset V_J$ is a Zariski closed subset of $V_J$, we see that for $v \notin S$ (with $S$ containing all archimedean places and enlarged if necessary), 
  \[    
     \left( \bigcup_{n \geq 0} \varpi_v^n \cdot \Omega(\mathcal{O}_v) \right) \cap \Omega_C(F_v)  = \Omega_C(\mathcal{O}_v) \subset \Omega(\mathcal{O}_v).  \] 
 Hence, for $v \notin S$, the restriction of $f_{0,v}$ to $\Omega_C(F_v) = H_C(F_v) \cdot \tilde{\psi}_{C,v}$ is the characteristic function of $H_c(\mathcal{O}_v)$.
 \vskip 5pt
 
 By the above discussion, we deduce that for $S$ sufficiently large and  with $F_S :=  \prod_{v \in S} F_v$, 
\[ \theta(\phi,f)_{\overline{N}_E,\psi_C}(1)  = \int_{H_C(F_S)} \theta(h \cdot \phi)_{\overline{N}_J,  \tilde{\psi}_C}( 1)  \cdot \overline{f(h)} \, dh. \]
 We need to show that  we can find some $f$ and $\phi$ such that the above integral is nonzero.  
 \vskip 5pt
 
To this end, we start with a fixed pair of $f$ and $\phi$ such that the integrand in the above integral is nonzero as a function of $h$. 
Now we consider an arbitrary Schwarz function $\varphi$ on $\overline{N}_J(F_S)$ and replace $\phi$ by the convolution $\varphi \ast \phi$ in the above formula. 
This gives:
\[   \theta(\varphi \ast \phi,f)_{\overline{N}_E,\psi_C}(1)  = \int_{H_C(F_S)} 
\widehat{\varphi_{\overline{Z}}}( h^{-1} \cdot \tilde{\psi}_C) \cdot  \theta(h \cdot \phi)_{\overline{N}_J,  \tilde{\psi}_C}( 1)  \cdot \overline{f(h)} \, dh, \]
where $\varphi_{\overline{Z}}$ is the constant term of $\varphi$ along $\overline{Z} \subset \overline{N}_J$ (which is a Schwarz function on $V_J(F_S) = N_J(F_S)/ Z(F_S)$) and 
$\widehat{\varphi_{\overline{Z}}}$ is its Fourier transform. Since $H_C(F_S) \cdot \tilde{\psi}_C = \Omega_C(F_S) \subset V_J(F_S)$ is a Zariski-closed subset, and 
$\widehat{\varphi_{\overline{Z}}}$ can be an arbitrary Schwarz function (as $\varphi$ varies), we see that the above  integral is nonzero for some choice of $\varphi$. 
 
 \vskip 5pt
 
 This completes the proof of the second statement. 
\end{proof}
\vskip 5pt

\begin{cor}  \label{C:global-injective}
\noindent (i) If $\rho \subset \mathcal{A}(H_C)$, then $\Theta(\rho) \subset \mathcal{A}_2(G_E)$ is a nonzero irreducible square-integrable automorphic representation of $G_E$. Moreover, $\Theta(\rho) \cong \Theta^{abs}(\rho) := \otimes_v \theta(\rho_v)$, where $\theta(\rho_v)$ denotes the local theta lift of $\rho_v$ to $G_E(F_v)$ (which is nonzero irreducible). 
\vskip 5pt

\noindent (ii) For an abstract irreducible representation $\rho$ of $H_C(\A)$, we have
\[ \dim \Hom_{H_C}(\rho, \mathcal{A}_2(H_C))  = \dim \Hom_{G_E}(\Theta^{abs}(\rho), \Theta(\mathcal{A}(H_C))) \]
where 
\[  \Theta(\mathcal{A}(H_C))  = \langle \theta(\phi, f): \phi \in \Pi_J, f \in \mathcal{A}(H_C) \rangle \subset \mathcal{A}_2(G_E). \]
\vskip 5pt

\noindent (iii)  If $\rho \subset \mathcal{A}(H_C)$ and $\rho' \subset \mathcal{A}(H_{C'})$ satisfy $\Theta(\rho) = \Theta(\rho')$ as submodules of $\mathcal{A}_2(G_E)$, then 
$C$ is $E$-isomorphic to $C'$ (so that $H_C \cong H_{C'}$) and $\rho =  \rho'$ as subspaces of $\mathcal{A}(H_C)$.
\end{cor}

\begin{proof}
(i)  This follows from Proposition \ref{P:global-cuspidality} and Proposition \ref{P:global-FC}. 
\vskip 5pt
 
(ii) This statement is often called the multiplicity-preservation of theta correspondence and in fact follows from  (i) and the local Howe duality theorem we established in our local study, which says that:
\[  \dim \Hom_{H_C \times G_E}(\Pi_J , \rho \otimes \Theta^{abs}(\rho)) \leq 1 \]
and
\[ \dim \Hom_{G_E}(\Theta^{abs}(\rho), \Theta^{abs}(\rho')) \leq \dim \Hom_{H_C}(\rho, \rho') \leq 1. \]
  In view of (i) and the local Howe duality theorem, the statement here is only interesting when $\mathcal{A}(H_C)$ is not multiplicity-free. To prove (ii), we define a pairing of finite-dimensional vector spaces:
\[   \mathcal{B} : \Hom_{H_C}(\rho, \mathcal{A}_2(H_C)) \times \Hom_{G_E}(\Theta^{abs}(\rho), \Theta(\mathcal{A}(H_C))) \longrightarrow \Hom_{H_C \times G_E}(\Pi_J \otimes \overline{\rho} \otimes \overline{\Theta^{abs}(\rho)}, \C)  \]
by
\[  \mathcal{B} (f , \iota) (\phi, v, w) =  \int_{[G_E]}  \theta(\phi, f(v))(g) \cdot \overline{\iota(w)(g)} \, dg  \]
for $\phi \in \Pi_J$, $v \in \rho$ and $w \in \Theta^{abs}(\rho)$.
The local Howe duality theorem says that the target space is $1$-dimensional (so we may identify it with $\C$).  
Now (i) and the local Howe duality theorem  imply that this $\C$-valued pairing is nondegenerate, giving us the desired equality of dimensions of the two Hom spaces on the left.

\vskip 5pt

(iii) It follows from Proposition \ref{P:global-FC} that for $\rho \subset \mathcal{A}(H_C)$, $\Theta(\rho)$ supports only one orbit of generic Fourier coefficients along $\bar N_E$, namely the orbit associated to $C$. Thus, if $\Theta(\rho) = \Theta(\rho')$, then we must have $C \cong C'$. The equality of $\rho$ and $\rho'$ now follows by (ii).
\end{proof}

  \vskip 5pt

\subsection{\bf Canonical decomposition}
 To finish this section, let us examine  the case when $H_C^0(F) = H_C(F)$: this is the case when $\mathcal{A}(H_C)$ has multiplicity 2. In this case, we have an orthogonal decomposition
\[  \mathcal{A}(H_C) = \bigoplus_{\chi} V_C(\chi) \]
as $\chi$ runs over automorphic characters of $H_C^0 = T_{E,K}$ and $V_C(\chi)$ is characterised as the subspace of functions whose restriction to $H_C^0$ is contained in $\C \cdot \chi$.  Each $V_C(\chi)$ is multiplicity-free and the occurrence of multiplicity $2$ is due to isomorphisms $V_C(\chi) \cong V_C(\chi^{-1})$ for those $\chi$ satisfying
\vskip 5pt

\begin{itemize}
\item $\chi^2 \ne 1$ but 
\item $\chi_v^2=1$ for the finitely many places $v$ where $H_C(F_v) = H_C^0(F_v)$.
\end{itemize}
\vskip 5pt

 For $\chi$ satisfying these two conditions,  and $\rho$ an abstract irreducible representation of $H_C(\A)$ which occurs in $V_C(\chi)$ and $V_C(\chi^{-1})$ and write $\rho_{\chi}$ for the corresponding submodule $\rho_{\chi} \subset V_C(\chi)$. Then the $\rho$-isotypic summand of $\mathcal{A}(H_C)$ has the canonical decomposition:
 \[  \mathcal{A}(H_C)(\rho) = \rho_{\chi} \oplus \rho_{\chi^{-1}}. \]
 On considering the global theta lifting, Corollary \ref{C:global-injective} gives a direct sum
 \[  \Theta(\rho_{\chi}) \oplus \Theta(\rho_{\chi^{-1}}) \subset \mathcal{A}_2(G_E) \]
 of two irreducible summands.  This gives a canonical decomposition of the $\Theta^{abs}(\rho)$-isotypic summand $\Theta(\mathcal{A}(H_C))[\Theta^{abs}(\rho)]$.
 One may ask how decomposition can be characterized directly on the side of $G_E$, i.e. without reference to $H_C$.  We shall address this question in the remainder of this section.
 \vskip 5pt

We have seen in Proposition \ref{P:global-FC} the Fourier coefficient formula  
\[  \theta(\phi,f)_{\overline{N}_E, \psi_C} (g)  = \int_{H_C(\A)}  \theta(h \cdot \phi)_{\overline{N}_J,  \tilde{\psi}_C}( g)  \cdot \overline{f(h)} \, dh. \] 
for $\phi \in \Pi_J$ and $f \in \rho_{\chi}$, where we recall that $\tilde{\psi}_C \in \Omega_{\psi_C}$.  
Let 
\[  S_{\psi_C}  = {\mathrm Stab}_{M_E}(\psi_C) \]
be the stabilizer of $\psi_C$ in $M_E$. Then we have an action of $S_{\psi_C} \times H_C$ on $\Omega_{\psi_C}$ for which $\Omega_{\psi_C}$ is a torsor for each of the two factors.
This gives an isomorphism
\[  \iota: S_{\psi_C} \cong H_C, \]
characterized by
\[  \iota(t) \cdot \tilde{\psi_C}  = t^{-1} \cdot \tilde{\psi}_C. \]
Now we may regard $ \theta(\phi,f)_{\overline{N}_E, \psi_C}$ as a function on $S^0_{\psi_C}(F) \backslash S^0_{\psi_C}(\A)$.  The following proposition, which strengthens Proposition \ref{P:global-FC} and is the global analog of ( \ref{E:spec}), describes this function explicitly.

\vskip 5pt

\begin{prop}
For $t \in S_{\psi_C}^0(\A) \cong H_C^0(\A)$ and  $f \in \rho_{\chi}$, we have
\[  \theta(\phi,f)_{\overline{N}_E, \psi_C}(t)  = \chi(\iota(t))^{-1}   \cdot \theta(\phi,f)_{\overline{N}_E, \psi_C}(1).\]
  In other words,
\[     \theta(\phi, f)_{\overline{N}_E. \psi_C} \circ \iota^{-1}  \in \C \cdot \chi^{-1} \subset \mathcal{A}(H_C). \]
\end{prop}
\begin{proof}
Write $\phi = \phi_{\infty} \otimes \phi^{\infty} \in \Pi_{J,\infty} \otimes \Pi_J^{\infty}$. With $\phi_{\infty}$ fixed, we consider the Fourier coefficient map
\[  \Pi_J^{\infty} \longrightarrow \C \]
given by
\[  \phi^{\infty} \mapsto   \theta(\phi_{\infty} \otimes \phi^{\infty})_{\overline{N}_J, \tilde{\psi}_C}(1). \]
As we have noted in the proof of Proposition \ref{P:global-FC},  there is a $P_J(\A^{\infty})$-equivariant map
\[  q:  \Pi_J^{\infty} \longrightarrow C^{\infty}(\Omega_{\A^{\infty}})  \]
so that 
\[     \theta(\phi)_{\overline{N}_J, \tilde{\psi}_C}(1) = \lambda(\phi_{\infty})  \cdot q(\phi^{\infty})(\tilde{\psi}_C). \]
for some $\lambda(\phi_{\infty}) \in \C$. Then  for $t \in S^0_{\psi_C}(\A^{\infty})$, we have:
\[  \theta(\phi)_{\overline{N}_J, \tilde{\psi}_C}(t) = \lambda(\phi_{\infty}) \cdot q(\phi^{\infty})( t^{-1} \cdot \tilde{\psi}_C)  = \lambda(\phi_{\infty}) \cdot q(\phi^{\infty})(\iota(t) \cdot \tilde{\psi}_C) =  \theta(\iota(t)^{-1} \cdot \phi)_{\overline{N}_J, \tilde{\psi}_C}(1).  \]
Hence,
\begin{align}
 \theta(\phi,f)_{\overline{N}_E, \psi_C} (t)  &= \int_{H_C(\A)}  \theta(h \cdot \phi)_{\overline{N}_J,  \tilde{\psi}_C}( t)  \cdot \overline{f(h)} \, dh \notag \\
 &= \int_{H_C(\A)}  \theta(\iota(t)^{-1} h \cdot \phi)_{\overline{N}_J,  \tilde{\psi}_C}( 1)  \cdot \overline{f(h)} \, dh \notag \\
&=  \int_{H_C(\A)}  \theta( h \cdot \phi)_{\overline{N}_J,  \tilde{\psi}_C}( 1)  \cdot \overline{f( \iota(t) h)} \, dh \notag \\ 
&=  \chi(\iota(t))^{-1} \cdot \int_{H_C(\A)}  \theta( h \cdot \phi)_{\overline{N}_J,  \tilde{\psi}_C}( 1)  \cdot \overline{f( h)} \, dh. \notag 
\end{align}
This proves the desired identity for $t \in S^0_{\psi_C}(\A^{\infty})$. However, both sides of the desired identity are automorphic functions of $S^0_{\psi_C} \cong H_C^0 
\cong \tilde{T}_{E,K}$.  The desired identity then follows by the weak approximation theorem (Proposition \ref{P:auto-char}(i)) for $\tilde{T}_{E,K}$. 
\end{proof}
\vskip 5pt

What the lemma says is that the consideration of the $\psi_C$-Fourier coefficient gives an $(\bar{N}_E ,\psi_C) \times S^0_{\psi_C}$-equivariant map 
\[  \Theta(\mathcal{A}(H_C))[\Theta^{abs}(\rho)] \longrightarrow \C \cdot \chi \oplus \C \cdot \chi^{-1} \subset  \mathcal{A}(S^0_{\psi_C}) \]
 The canonical decomposition of the codomain is given by the irreducible summands whose image is contained in $\C \cdot \chi$ or $\C \cdot \chi^{-1}$.
\vskip 15pt

\section{\bf A-parameters and Twisted Composition Algebras}  \label{S:APTCA}

In the next two sections, we relate the square-integrable automorphic representations constructed in the previous section to Arthur's conjecture for $G_E$. 
We begin by explicating the connections between twisted composition algebras and the relevant  class of A-parameters in this section.
\vskip 5pt

\subsection{\bf A-parameters.} 
We shall consider A-parameters
\[  \psi:  W_F \times \SL_2(\C)  \longrightarrow   \PGSO_8(\C) \rtimes S_3.  \]
 such that the centralizer of $\psi(\SL_2(\C))$  is isomorphic to the group 
\[  S \rtimes (S_2 \times S_3) =  (\C^{\times} \times \C^{\times} \times \C^{\times})^1 \rtimes  (S_2 \times S_3).  \]
We fix the isomorphism
\[  Z_{\PGSO_8 \rtimes S_3}(\psi(\SL_2(\C))) \cong S \rtimes (S_2 \times S_3) \]
throughout. Associated to such a $\psi$ is thus a map
\[  \rho = \rho_E \times \rho_K : W_F \longrightarrow S_2 \times S_3, \]
i.e. a pair $(E,K)$ consisting of an \'etale cubic $F$-algebra $E$ and an \'etale quadratic algebra $K$; we shall say that $\psi$ is of type $(E,K)$.
With the \'etale cubic algebra $E$ fixed,  $\psi$ is an  A-parameter for the group $G_E$.  
\vskip 5pt

If we let $W_F$ act on $S$ through the map $\rho$, then $S \rtimes W_F$ is the L-group of the torus  
\[  \tilde{T}_{E, K}  = \{ x\in  (E \otimes_F K)^{\times}:  N_{E \otimes K / E}(x)  \in F^{\times}\} /  K^{\times}.    \]
Hence, to  give  an A-parameter of type  $(E,K)$ is equivalent to giving an L-parameter
\[  \phi:  W_F \longrightarrow {^L}\tilde{T}_{E,K}^{\vee} \twoheadrightarrow S \rtimes (S_2 \times S_3)\]
modulo conjugacy by $S \rtimes S_2$, or equivalently  an automorphic character  of the torus $\tilde{T}_{E,K}$
 up to inverse, i.e. a pair of automorphic characters $[\chi] = \{\chi, \chi^{-1} \}$.
\vskip 5pt

To summarize, the A-parameters we are considering are determined by the triple $(E, K, [\chi])$.
We had already highlighted and discussed these A-parameters in \S \ref{SS:interest}. 
\vskip 5pt

\subsection{\bf Component groups}
An important structure associated to an A-parameter $\psi = \psi_{E,K, [\chi]}$ as above is its global and local component groups. The global component group is
\[  S_{\psi} = \pi_0(Z_{\PGSO_8}(\psi)) = \pi_0(Z_{S \rtimes S_2}(\phi)). \] 
On the other hand, for each place $v$ of $F$, one has the restriction $\psi_v$  of $\psi$ to $W_{F_v} \times \SL_2(\C)$ (the associated local A-parameter), and one has likewise the local component group
\[  S_{\psi_v} = \pi_0(Z_{\PGSO_8}(\psi_v)) = \pi_0(Z_{S \rtimes S_2}(\phi_v)). \] 
There is a natural diagonal map
\[  \Delta:  S_{\psi} \longrightarrow \prod_v S_{\psi_v} =: S_{\psi, \A}. \]
The following lemma gives a description of these component groups.
\vskip 5pt

 \vskip 5pt
 \begin{lemma} \label{L:component}
 Fix an A-parameter $\psi = \psi_{E,K, [\chi]}$ as above, with associated $\phi$.  For each place $v$ of $F$, one has an exact sequence
 \[  \begin{CD}
 1 @>>>  Z_S(\phi_v)  @>>>   Z_{S \rtimes S_2}(\phi_v)  @>>>   S_2  \end{CD} \]
 and this sequence is exact at the right if and only if the character $\chi_v$ associated to $\phi_v$ satisfies $\chi_v^2 =1$.
  Moreover, the abelian group  $Z_S(\phi_v)$ depends only on $(E_v, K_v)$ (i.e. is independent of $[\chi_v]$) and is given by
 \[  Z_S(\phi_v)  =  S^{W_{F_v}}  = (\tilde{T}_{E,K}^{\vee})^{W_{F_v}}. \]
 where the action of $W_{F_v}$ on $S = \tilde{T}_{E,K}^{\vee}$ is via the map $\rho: W_{F_v} \longrightarrow S_2 \times S_3$.
 Hence, one has
 \[  \begin{CD}
 1 @>>>  \pi_0(S^{W_{F_v}})  @>>>  S_{\psi_v}  = \pi_0(Z_{S \rtimes S_2}(\phi_v)) @>>>  S_2  \end{CD} \]
 with exactness on the right if and only if $\chi_v^2 =1$, in which case
 \[  S_{\psi_v} \cong   \pi_0(S^{W_{F_v}})  \rtimes S_2. \]
   \end{lemma}
 \vskip 5pt
The analogous result holds for the global parameter $\phi$. In \S \ref{SS:example}, we had considered an example of a family of such  $\psi$'s and tabulated the corresponding groups $S_{\psi_v}$. To simplify notations, we will henceforth set
\[  S_{\psi}^0 :=  \pi_0(S^{W_F}) \quad \text{and} \quad S_{\psi_v}^0 :=  \pi_0(S^{W_{F_v}}). \] 
\vskip 5pt

 \vskip 5pt

\vskip 5pt
 \subsection{\bf From A-parameters to twisted composition algebras.}
As we observed in \S \ref{SS:isom-tori}, the group $\tilde{T}_{E,K}$ is (canonically up to inverse) isomorphic to the identity component of the automorphism group of any $E$-twisted 
composition algebra $C$ with $\dim_E(C)  =2$ and quadratic invariant $K_C$ such that $[K_E] \cdot [K_C] \cdot [K] =1$. This motivates the following definition:

\vskip 5pt
\begin{defn} 
(i) Let $\Sigma_{E,K}$ denote the set of $E$-isomorphism classes of rank $2$ $E$-twisted composition algebras with quadratic invariant $K_C = [K_E] \cdot [K]$. 
\vskip 5pt

(ii) Let $\tilde \Sigma_{E,K}$ denote the set of $E \otimes_F K_C$-isomorphism classes of rank $2$ $E$-twisted composition algebras with quadratic invariant $[K_C]= [K_E] \cdot [K]$.
\end{defn}
\vskip 5pt

\noindent  Then any $C \in \Sigma_{E,K}$ corresponds under the Springer decomposition to an algebra embedding $E \hookrightarrow J$ for some $9$-dimensional Freudenthal-Jordan algebra $J$ with $K_J  =K$.     
\vskip 5pt

  The following long lemma summarizes  the discussion in \S \ref{S:TCFJ}, especially \S \ref{SS:rank2}, \S \ref{SS:springer}, \S \ref{SS:isom-tori} and \S \ref{SS:division} (see also  \cite[\S 11.5 and \S 11.6]{GS2}). 
   \vskip 5pt
   
   \begin{lemma}  \label{L:H1}
(i) There is a natural commutative diagram
\[  \begin{CD}
H^1(F, \tilde{T}_{E,K}) @=  \tilde \Sigma_{E,K} @=   \{ \text{isomorphism classes of triples $(B, \tau, \iota)$} \} \\
@VVV  @VVV @VVV  \\
  H^1(F, \tilde{T}_{E,K}) / S_2  @=\Sigma_{E,K} @=  \{ \text{equivalence classes of $\iota: E \hookrightarrow J$} \} 
  \end{CD} \]
  where the horizontal arrows are natural bijections (and hence written as equal signs). Moreover, 
  \vskip 5pt
  
  \begin{itemize}
  \item in the first row, for the triple $(B, \tau,\iota)$,
  \begin{itemize}
  \item $B$ is a central simple $K$-algebra of degree $3$;
  \item $\tau$ is an involution of second kind on $B$ (relative to $K/F$)
  \item $\iota : E \longrightarrow B^{\tau}$ is a Jordan algebra embedding.
    \end{itemize}
Two such triples $(B_1, \tau_1, \iota_1)$ and $(B_2, \tau_2, \iota_2)$ are equivalent if there is a $K$-algebra isomorphism $f: B_1 \cong B_2$ such that $\tau_2 \circ f = f \circ \tau_1$ 
and $f \circ \iota_1 = \iota_2$.
  \vskip 5pt
  
  \item  the group $S_2$ acts on $H^1(F, \tilde{T}_{E,K})$  by inverting; this action is described in terms of the other two sets in the row by 
 \[    C \mapsto C \otimes_{K_C, \sigma} K_C \quad \text{ on $\tilde{\Sigma}_{E,K}$} \]
 where $\sigma$ is the nontrivial element in $\Aut(K_C/F)$,  and
 \[   (B, \tau, \iota) \mapsto (B^{op}, \tau, \iota)  \quad \text{  on the last set.} \]
  \vskip 5pt

  \item in the second row, the second bijection   is via the Springer decomposition, so $\iota: E \hookrightarrow J$ refers to an embedding of Jordan algebras;
 \vskip 5pt
 
 \item the first two vertical arrows are the natural ones whereas the last vertical arrow is the forgetful map given by 
   \[  (B, \tau, \iota)  \mapsto  \iota.   \]
\end{itemize}
\vskip 5pt

(ii)  For any $C \in \Sigma_{E,K}$, its preimage in $\tilde{\Sigma}_{E,K}$ is an $S_2$-orbit and thus has 1 or 2 elements. Moreover, one has:
\[  \text{ Fiber over $C$ has 2 elements} \Longleftrightarrow  H_C(F) = H_C^0(F). \]
Thus, the restriction of the first vertical arrow gives a bijection from $H^1(F, \tilde{T}_{E,K})[2]$ onto its image.

\vskip 5pt

 (iii) If we pick any triple $(B, \tau, \iota)$ in the preimage of $C$, we obtain an isomorphism of algebraic tori over $F$:
  \[  \iota_{B,\tau} :  H_C^0 \longrightarrow     \tilde{T}_{E,K}. \]
  Hence, we have the following canonical bijection which gives another interpretation of $\tilde{\Sigma}_{E,K}$:
  \[  \tilde{\Sigma}_{E,K}  \longleftrightarrow  \{ \text{equivalence classes of  $(C, i)$} \} \]
where
\begin{itemize}
\item $C$ is an $E$-twisted composition algebra with quadratic invariant $K_C = [K_E] \cdot [K]$ and automorphism group $H_C$;
\item $i: H_C^0 \longrightarrow \tilde{T}_{E,K}$ is an isomorphism of $F$-tori, arising in the manner above;
\item  two pairs $(C,i)$ and $(C', i')$ are equivalent if and only if there is an isomorphism $j : C \longrightarrow C'$ of $E$-twisted composition algebras, inducing an isomorphism 
${\mathrm Ad}(j): H_C^0 \longrightarrow  H_{C'}^0$, so that $i' \circ {\mathrm Ad}(j)  = i$. 
\end{itemize}
       \end{lemma}
   \vskip 5pt
   
   \subsection{\bf Local fields}
  In particular, the above results apply to the case where $F$ is a number field, as well as the local completions $F_v$.  
  In \cite[\S 12]{GS2}, we have examined the case of a local field $F_v$ as an explicit example. Summarizing the results there, we note:
  \vskip 5pt
  
  \begin{lemma} \label{L:localH1}  
   Assume that $F_v$ is a local field.  We have two cases:
   \vskip 5pt
   
   \begin{itemize}
   \item[(i)]  If $(E_v, K_v) \ne (\text{field, split} )$,   then $H^1(F_v, \tilde{T}_{E_v,K_v})$ is an elementary abelian 2-group and  the action of $S_2$ on $H^1(F_v, \tilde{T}_{E_v,K_v})$ is trivial, so that
   \[  \Sigma_{E_v, K_v}  \longleftrightarrow \tilde{\Sigma}_{E_v, K_v}   \longleftrightarrow  H^1(F_v, \tilde{T}_{E_v,K_v}).  \]
  Hence, for any $C \in \Sigma_{E_v, K_v}$, its fiber in $\tilde{\Sigma}_{E_v, K_v}$ has 1 element and $H_C(F_v) \cong H_C^0(F_v) \rtimes \Z/2\Z$.
   \vskip 5pt
  
  \item[(ii)]  If $E_v$ is a field and $K_v$ is split (so that $F_v$ is nonarchimedean), one has isomorphisms
   \[  \tilde{\Sigma}_{E_v, K_v}  \cong  H^1(F_v, \tilde{T}_{E_v, K_v}) \cong \mathrm{Ker}(H^2(F_v, \mathbb{G}_m) \rightarrow H^2(E_v, \mathbb{G}_m))  \cong \Z/3\Z \]
   via
   \[  (B, \tau, \iota) \mapsto inv(B) \quad \text{(the invariant of $B$)} \]
   and the action of $S_2$ on $\Z/3\Z$ is by inverting.  Hence $\Sigma_{E_v, K_v}$ has 2 elements, corresponding to 
   \[ C_v^+ = ( E_v \hookrightarrow M_3(F_v)) \quad \text{and} \quad C_v^- = (E_v \hookrightarrow D_v^+) \]
   where $D_v^+$ denotes the Jordan algebra attached to a cubic division algebra  $D_v$ over $F_v$. The preimage of $C_v^-$ in $\tilde{\Sigma}_{E_v, K_v}$ has two elements (associated to $D_v$ and $D_v^{op}$) and in this case, $H_{C_v^-}(F_v) = H_{C_v^-}^0(F_v)$.  However, the choice of $D_v$ gives an isomorphism
   \[ \iota_{D_v}:  H_{C_v^-} \longrightarrow  \tilde{T}_{E_v, K_v}, \]
    with $\iota_{D^{op}_v} (-)  = \iota_{D_v} (-)^{-1}$.
    \end{itemize}
    
    Hence,  we have:
    \[  H^1(F_v, \tilde{T}_{E_v, K_v})[3] =\text{$1$ or $\Z/3\Z$} \]
and $H^1(F_v, \tilde{T}_{E_v, K_v}) /H^1(F_v, \tilde{T}_{E_v, K_v})[3] $ is an elementary abelian 2-group.
  \end{lemma}
   
   \subsection{\bf Local-global principles}
   When $F$ is a number field, there is a  commutative diagram of localisation maps
   \[  \begin{CD}
\tilde{\Sigma}_{E,K} @>{\tilde{\mathrm loc}}>>   \prod'_v \tilde{\Sigma}_{E_v, K_v} \\
@VVV  @VVV \\
   \Sigma_{E,K} @>{\mathrm loc}>>   \prod'_v  \Sigma_{E_v, K_v}. \end{CD} \]
\vskip 5pt

\noindent It will be necessary to explicate the image of ${\mathrm loc}$ and to determine the size of its fibers.  
\vskip 5pt

\begin{lemma}  \label{L:loc}
(1) Assume that $K = F \times F$ is split.
\vskip 5pt

\begin{itemize}
\item[(i)] One has a short exact sequence of abelian groups
\[  \begin{CD}
0 @>>> \tilde{\Sigma}_{E,K} @>\tilde{\mathrm loc}>> \bigoplus_v \tilde{\Sigma}_{E,v, K_v} @>\oplus_v {\mathrm inv}_v>>  \Z/3\Z  @>>> 0 \end{CD} \]

\vskip 5pt
\item[(ii)] Let $\mathfrak{C} = \{ C_v \}$ be  a collection of local twisted composition algebras, with $C_v = (E_v \hookrightarrow B^+_v)$, where $B_v$ is a central simple algebra of degree $3$ over $F_v$ which is split for almost all $v$, and let $S_{\mathfrak{C}}$ denote the set of places where $B_v$ is a cubic division algebra.  Then  we have: 
\[  \# {\mathrm loc}^{-1}(\mathfrak{C})  = \begin{cases} 
1 \text{  if $S_{\mathfrak{C}}$ is empty;} \\
 \left( 2^{\# S_{\mathfrak{C}}} + 2 \cdot (-1)^{\# S_{\mathfrak{C}}}\right) /6, \text{  if $S_{\mathfrak{C}}$ is nonempty.} \end{cases}  \]
In particular, $\mathfrak{C}$ lies in the image of ${\mathrm loc}$ if and only if $\# S_{\mathfrak{C}} \ne 1$. 
\end{itemize}

\vskip 5pt
(2) Assume that $K$ is a field. 
\vskip 5pt

\begin{itemize}
\item[(i)]   The map $\tilde{\mathrm loc}$ is bijective and the map ${\mathrm loc}$ is surjective. 
   \vskip 5pt
   
\item[(ii)]    Given a collection of local twisted composition algebras $ \mathfrak{C} = \{ C_v \}$, 
   let $S_{\mathfrak{C}}$ denote the finite set of places of $F$ where $E_v$ is a field, $K_v$ is split and $C_v = ( E_v \hookrightarrow D_v^+)$ with $D_v$ a division algebra of degree $3$ over $F_v$. Then we have:
   \[  \# {\mathrm loc}^{-1} (\mathfrak{C})  =  \begin{cases}
   1, \text{  if $S_{\mathfrak{C}}$ is empty;} \\
   2^{\# S_{\mathfrak{C}}-1}, \text{  if $S_{\mathfrak{C}}$ is nonempty.} \end{cases}   
    \] 
 \end{itemize}
 
 \vskip 5pt
 
 In both cases, the restriction of $\tilde{\mathrm loc}$ gives an isomorphism
 \[  H^1(F, \tilde{T}_{E,K})[2]  \cong  {\prod}_v' H^1(F_v, \tilde{T}_{E,K})[2].  \]

 \end{lemma}  
\vskip 5pt

\begin{proof}
(1i) Recalling that 
\[  \tilde{\Sigma}_{E,K} = H^1(F, \tilde{T}_{E,K}) = {\Ker}(H^2(F, \mathbb{G}_m) \longrightarrow  H^2(E, \mathbb{G}_m)), \]
the short exact sequence in (1i) is a consequence of global class field theory. 

\vskip 5pt
 (1ii) Given a set $S$ of places of $F$, there are 
 \[  \frac{2^{\# S} + 2 \cdot (-1)^{\# S} }{3} \]
 central simple $F$-algebras of degree $3$ which are ramified precisely at $S$; this is an interesting exercise which we leave to the reader. 
This number is thus the cardinality of the fiber of $\tilde{\mathrm loc}$ over a collection $\mathfrak{C}$ with $S_{\mathfrak{C}} = S$.  The action of $S_2$ on $\tilde{\Sigma}_{E,K}$ preserves this fiber and its action there is free, unless $S$ is empty (in which case the fiber is a singleton set and $S_2$ acts trivially). This proves (1ii).

\vskip 5pt

(2i) The map $\tilde{\mathrm loc}$ is injective by the Hasse principle for 2-dimensional tori, proved by Voskresenskii \cite{V1}. To show the surjectivity,  we make use of the moduli interpretation of $\tilde{\Sigma}_{E,K}$ as the set of tuples $(B, \tau, \iota)$ provided by Lemma \ref{L:H1}.
One has the local-global principle for odd degree division algebras equipped with involutions of second kind, which says that any collection $\{ (B_v, \tau_v) \}$ of local pairs comes from a unique global pair $(B, \tau)$. Equivalently, the natural map
\[  H^1(F, {\mathrm PU}_3^K) \longrightarrow \bigoplus_v H^1(F_v,   {\mathrm PU}_3^{K_v}) \]
is an isomorphism. In addition, for a fixed $(B,\tau)$ and a collection of local embeddings 
\[   \iota_v:  (E_v \otimes K_v, \sigma_v)  \longrightarrow ( B_v, \tau_v),  \quad \text{with $1 \ne \sigma_v  \in \Aut(K_v/F_v)$, } \]
  a local-global principle of Prasad-Rapinchuk \cite{PR} shows that there exists 
 \[   \iota:  (E \otimes K, \tau)  \longrightarrow ( B, \tau),  \]
 which localizes to $\iota_v$ for all $v$.  This shows the surjectivity of $\tilde{\mathrm loc}$.
 \vskip 5pt

The surjectivity of ${\mathrm loc}$ follows by that of $\tilde{\mathrm loc}$  and the surjectivity of the two vertical arrows. 
\vskip 5pt

(2ii) Given a finite set $S$ of finite places of $F$ which split over $K$, there are $2^{\# S}$ pairs $(B, \tau)$ of central simple $K$-algebras with an involution $\tau$ of the second kind,  with $B$ ramified precisely at places of $K$ lying over $S$. The $S_2$ action on these is free unless $S$ is empty (in which case the action is trivial). This proves (ii). 
 
\end{proof}
\vskip 5pt

In particular,  the map ${\mathrm loc}$ is not injective: this is the failure of the Hasse principle for twisted composition algebras which is ultimately responsible for the high multiplicities in the automorphic discrete spectrum of $G_E$.     
\vskip 10pt

\subsection{\bf Local Tate dualities} 
The connection between our A-parameters $\psi$ and twisted composition algebras is provided by the local and global Tate duality theorems. 
 We first note the local Tate-Nakayama duality theorem  (see \cite[\S 2]{K1} and \cite[Cor. 2.4]{Mi}).
 \vskip 5pt

\begin{lemma}  \label{L:local-tate}
Let $T$ be a torus over a local field $F_v$ with character group $X(T)  = \Hom(T, \mathbb{G}_m)$. Then one has a  
commutative diagram:
\vskip 5pt
\[ \begin{CD}
H^1(F_v, T) @>>> {\Irr} (H^1(F_v, X(T)) ) \\
@Ainj.AA   @AAinj.A \\
H^1(F_v, T)[2] @>>> {\Irr} (H^1(F_v, X(T))/ 2 H^1(F_v, X(T)) ) \\
@Asurj.AfA @AAsurj.A \\
H^1(F_v, T[2]) @>>> {\Irr}( H^1(F_v, X(T)/2 X(T)))  \\
@Ainj.AbA @AAinj.A \\
T(F_v)/T(F_v)^2  @>>> {\Irr}( H^2(F_v, X(T))[2]),
\end{CD} \]
\vskip 5pt

\noindent whose horizontal arrows are isomorphisms.  Here, in the left column, the maps $f$ and $b$  form a short exact sequence 
\[  \begin{CD}
1 @>>>  T(F_v)/T(F_v)^2 @>b>> H^1(F_v, T[2]) @>f>> H^1(F_v, ,T)[2] @>>> 1  \end{CD} \]
arising from the Kummer sequence 
\[  \begin{CD} 
  1 @>>> T[2] @>>> T @>2>> T @>>> 1, \end{CD} \]
  and the corresponding terms in the right column arises from the dual  short exact sequence
  \[  \begin{CD}
  1 @>>> X(T) @>2>> X(T) @>>> X(T)/2X(T) @>>> 1. \end{CD} \]
\end{lemma}

\vskip 10pt

We apply the above to our particular situation at hand. Fix an A-parameter $\psi = \psi_{E,K, [\chi]}$ as above and let $T = \tilde{T}_{E,K}$ for ease of notation. Then for each place $v$, we have the following canonical isomorphism \cite[\S 1]{K2}:
\[   H^1(F_v, X(T))  \cong  \pi_0((T^{\vee})^{W_{F_v}})  = S^0_{\psi_v}, \]
where $T^{\vee}$ is the complex dual torus of $T$.
Hence, by Lemma \ref{L:localH1}, $S^0_{\psi_v}[3] =1$ or $\mu_3$. Let us set
\[     \bar{S}_{\psi_v}^0   = S_{\psi_v}^0 / S_{\psi_v}^0[3] \quad \text{and} \quad \bar{S}_{\psi_v} =  S_{\psi_v}/ S_{\psi_v}[3].  \]
These are elementary abelian 2-groups, and we have
 \[    H^1(F_v, X(T)) / 2 H^1(F_v, X(T))  \cong     \bar{S}_{\psi_v}^0.  \]
Further,
\[  T[2] = Z_E, \quad \text{and} \quad  H^1(F_v, X(T)/2 X(T)) = H^1(F_v, Z({G_E^{\vee}}^{sc})),   \]
where $Z({G_E^{\vee}}^{sc})$ is the center of ${G^{\vee}_E}^{sc} = \Spin_8(\C)$.  Replacing these terms, the diagram in Lemma \ref{L:local-tate} now becomes:
\vskip 5pt

\begin{equation} \label{E:local-tate}
 \begin{CD}
H^1(F_v, T) @= {\Irr} (S^0_{\psi_v} ) \\
@Ainj.AA   @AAinj.A \\
H^1(F_v, T)[2] @= {\Irr} (\bar{S}^0_{\psi_v} ) \\
@Asurj.AfA @AAsurj.A \\
H^1(F_v, Z_E) @= {\Irr}(  H^1(F_v, Z({G_E^{\vee}}^{sc}))) \\
@Ainj.AbA @AAinj.A \\
T(F_v)/T(F_v)^2  @= {\Irr}( H^2(F_v, X(T))[2]),
\end{CD} \end{equation}

\vskip 5pt

Now, if $\chi_v^2 \ne 1$,  then $S_{\psi_v} = S_{\psi_v}^0$ and the first row of (\ref{E:local-tate}) already gives a bijection
\[  \mathrm{Irr}(S_{\psi_v}) \longleftrightarrow H^1(F_v, \tilde{T}_{E_v,K_v}). \]
Assume now that $\chi_v^2 =1$. In this case, $S_{\psi_v} \cong S_{\psi_v}^0 \rtimes S_2$ and we shall try to understand $\mathrm{Irr}(S_{\psi_v})$, or rather the subset $\mathrm{Irr}(\bar{S}_{\psi_v})$,  in terms of Lemma \ref{L:local-tate} and (\ref{E:local-tate}).
 
 \vskip 5pt
 
 To bring the component group $S_{\psi_v}$ into the picture, consider the projection
 \[  p: {G^{\vee}_E}^{sc} = \Spin_8(\C)  \longrightarrow {G^{\vee}_E} = \PGSO_8(\C) \]
 Taking the preimage of $S \rtimes S_2 \subset \PGSO_8(\C)$, we obtain
  the following commutative diagram of short exact sequences of $W_{F_v}$-modules:
\[  \begin{CD}
1 @>>>  Z({G^{\vee}_E}^{sc}) @>>> p^{-1}(S)  @>>> S  @>>> 1  \\
  @.  @| @VVV @VVV  \\
1@>>>  Z({G^{\vee}_E}^{sc}) @>>> p^{-1}(S \rtimes S_2) @>p>>  S \rtimes S_2 @>>> 1 \end{CD} \]
 where the action of $W_{F_v}$  is by conjugation via the map $\phi_v: W_{F_v} \longrightarrow S \rtimes (S_2 \times S_3)$ associated to $\psi_v$.  The  coboundary map in the long exact sequence then gives :
 \[ \begin{CD}
 S^0_{\psi_v} @>>>   H^1(F_v, Z({G^{\vee}_E}^{sc})) \\
 @VVV   @|  \\
 S_{\psi_v} @>\delta_v>> H^1(F_v, Z({G^{\vee}_E}^{sc}))  
 \end{CD} \]
 Because the target of the map $\delta_v$ is an elementary abelian 2-group (since $H^1(F_v, Z_E)$ is so), the map $\delta_v$ factors through 
the quotient $\bar{S}_{\psi_v}$ of $S_{\psi_v}$. Moreover, $\delta_v$ is injective on the index 2 subgroup $\bar{S}^0_{\psi_v}$; indeed, the map $\delta_v: \bar{S}^0_{\psi_v} \longrightarrow 
H^1(F_v, Z({G^{\vee}_E}^{sc}))$ is dual to the surjective map in the right column of (\ref{E:local-tate}). Hence $\mathrm{Ker}(\delta_v) \subset \bar{S}_{\psi_v}$ is either trivial or has order $2$ and we would like to determine precisely what it is. 
 \vskip 5pt
 
Together with (\ref{E:local-tate}),  the above gives rise to a group homomorphism
\begin{equation} \label{E:delta}
 \delta_v^*: H^1(F_v, Z_E) \cong  \mathrm{Irr}(  H^1(F_v, Z({G^{\vee}_E}^{sc}))) \longrightarrow \mathrm{Irr}(\bar{S}_{\psi_v}) \subset \mathrm{Irr}(S_{\psi_v}). \end{equation}
Thus, the diagram (\ref{E:local-tate}) can now be enhanced to:
\begin{equation} \label{E:local-tate2}
 \begin{CD}
H^1(F_v, T) @= \mathrm{Irr} (S^0_{\psi_v} )  @.  {} \\
@Ainj.AA   @AAinj.A    @. \\
H^1(F_v, T)[2] @= \mathrm{Irr} (\bar{S}^0_{\psi_v} )  @=  \mathrm{Irr}(\bar{S}^0_{\psi_v}) \\
@Asurj.AfA @AAsurj.A   @AAsurj.A \\
H^1(F_v, Z_E) @= \mathrm{Irr}(  H^1(F_v, Z({G^{\vee}_E}^{sc}))) @>\delta_v^*>>   \mathrm{Irr}(\bar{S}_{\psi_v})  \\
@Ainj.AbA @AAinj.A  @. \\
T(F_v)/T(F_v)^2  @= \mathrm{Irr}( H^2(F_v, X(T))[2]) @.  {} ,
\end{CD} \end{equation}
\vskip 5pt

What is the kernel of $\delta_v^*$?  Consider the fundamental  short exact sequence in the left column of (\ref{E:local-tate2}):
\begin{equation} \label{E:fundses}  \begin{CD}
1 @>>> T(F_v) /  T(F_v)^2 @>b>>  H^1(F_v, Z_E) @>f>> H^1(F_v, T)[2] @>>> 1. \end{CD} \end{equation}
We had first encountered this sequence in (\ref{E:keyses}).
Now $\chi_v$ is a character of the first term in the short exact sequence. Pushing out this sequence by $\chi_v$, one obtains:
\begin{equation} \label{E:fundses-chi} 
\begin{CD}
 1 @>>> \mu_2 @>>>  H^1(F_v, Z_E) / b(\mathrm{Ker}(\chi_v))  @>f_{\chi_v}>> H^1(F_v, T)[2]   @>>> 1 \end{CD} \end{equation}
 when $\chi_v \ne 1$.  Now we have:
\vskip 5pt

   \begin{prop}  \label{P:TN}
Fix a local A-parameter $\psi_v = \psi_{E_v, K_v, [\chi_v]}$. 
\vskip 5pt

\begin{itemize}
\item[(i)] There is a natural bijection
\[  \mathrm{Irr} \left(  S_{\psi_v}^0 \right) \longleftrightarrow  H^1(F_v, \tilde{T}_{E,K}). \]
\vskip 5pt

\item[(ii)]  Assume that $\chi_v^2 =1$, but $\chi_v \ne 1$. The natural map
\[  \delta_v: \bar{S}_{\psi_v} \longrightarrow  H^1(F_v, Z({G_E^{\vee}}^{sc})) \] 
is injective  and the dual map $\delta_v^*$ in (\ref{E:delta}) is surjective with kernel $b(\mathrm{Ker}(\chi_v))$, so that it  induces an isomorphism
  \[   H^1(F_v, Z_E) / b(\mathrm{Ker}(\chi_v))  \cong  \mathrm{Irr}(\bar{S}_{\psi_v}).
  \]
  Moreover, one has a commutative diagram of short exact sequence:
  \[  \begin{CD}
  1 @>>> \mu_2 @>>>  H^1(F_v, Z_E) / b(\mathrm{Ker}(\chi_v))  @>f_{\chi_v}>> H^1(F_v, \tilde{T}_{E,K})[2]   @>>> 1 \\
  @. @| @V\delta_v^*VV @| \\
  1 @>>>  \mu_2 @>>> \mathrm{Irr}(\bar{S}_{\psi_v}) @>{\mathrm rest}>>   \mathrm{Irr}(\bar{S}_{\psi_v}^0) @>>> 1,
  \end{CD} \]
where  the third vertical arrow is that given by (i).
\vskip 5pt

\item[(iii)] If $\chi_v =1$, then $\mathrm{Ker}(\delta_v) = \langle s_0 \rangle$ has order $2$ and hence one has a canonical element $s_0 \in \bar{S}_{\psi_v} \setminus \bar{S}^0_{\psi_v}$. In this case, $\delta_v^*$ induces an injection
\[  \delta_v^*:  H^1(F_v, Z_E)/ b(\tilde{T}_{E, K}(F_v)) = H^1(F_v, \tilde{T}_{E,K})[2] = \mathrm{Irr}(\bar{S}^0_{\psi_v}) \longrightarrow \mathrm{Irr}(\bar{S}_{\psi_v}) \]
which is a section to the restriction map $\mathrm{Irr}(\bar{S}_{\psi_v}) \rightarrow \mathrm{Irr}(\bar{S}^0_{\psi_v})$ and whose image consists of those characters of $\bar{S}_{\psi_v}$ which are trivial on $s_0$. 
\end{itemize}
  \end{prop}
\vskip 10pt

\subsection{\bf Global Tate duality}
We now consider the global analog of the above discussion.  We shall fix a global A-parameter $\psi  = \psi_{E,K, [\chi]}$ with global component group $S_{\psi}$ containing 
$S^0_{\psi} = \pi_0(S^{W_F})$ of index $\leq 2$.  Because $E$ is a field, we have
\[  S^0_{\psi} = \begin{cases} 
\mu_3, \text{  if $K = F \times F$;} \\
1, \text {  if $K$ is a field.} \end{cases} \]
So $S_{\psi}[3] = S_{\psi}^0[3] = S^0_{\psi} = 1$ or $\mu_3$, and as in the local case, we set
\[  \bar{S}_{\psi} = S_{\psi}/ S_{\psi}[3] \]
which is an elementary abelian $2$-group.
\vskip 5pt

Our discussion of local Tate duality allows us to reformulate the results of Lemma \ref{L:loc} in terms of characters of $S^0_{\psi}$:
\vskip 5pt

\begin{lemma} \label{L:loc2}
Writing $T = \tilde{T}_{E,K}$ for ease of notation, we have  the short exact sequence:
\[  \begin{CD}
1 @>>>  H^1(F, T) @>>>  \prod_v' H^1(F_v, T) @>>>  \mathrm{Irr}(\pi_0(S^{W_F}))  @>>> 1.  \\
   @.   @|       @|    @|   \\
   @. \tilde{\Sigma}_{E,K}  @. \mathrm{Irr}(S^0_{\psi,\A}) @.  \mathrm{Irr}(S^0_{\psi})
    \end{CD} \]
 In particular,  
\[  H^1(F, T)[2] \cong {\prod_v}' H^1(F_v, T)[2]   \cong  \mathrm{Irr}(\bar{S}^0_{\psi,\A}). \]
\end{lemma}
\vskip 5pt

After this recollection, 
we consider the following commutative diagram of short exact sequences.

\[  \begin{CD}
1 @>>> T(\A) / T(\A)^2 @>b>> \prod'_v H^1(F_v, Z_E)  @>f>>  \prod'_v H^1(F_v, T)[2]  @>>>1   \\
@. @AAA   @AsAA  @|  \\
1 @>>> T(F) / T(F)^2 @>>>    H^1(F, Z_E) @>>> H^1(F, T)[2]  @>>> 1    
\end{CD} \]
\vskip 10pt
 
 \noindent This diagram gives rise to the short exact sequence:
\[  \begin{CD}
1 @>>> T(F) \backslash T(\A) / T(\A)^2 @>b>>  b(T(F)) \backslash  \prod'_v H^1(F_v, Z_E)  @>f>>  H^1(F, T)[2] @>>> 1.  
\end{CD} \]
This is the global analog of the fundamental short exact sequence (\ref{E:fundses}) in the local setting. 
Moreover, it is equipped with a canonical section:  the map $s$ descends to give a section to $f$
\[   s:  H^1(F,T)[2] \longrightarrow  b(T(F)) \backslash  {\prod}'_v H^1(F_v, Z_E) . \]

\vskip 5pt

Now suppose we have a global A-parameter  $\psi = \psi_{E,K, [\chi]}$  as above.  We shall assume that $\chi^2 =1$ but $\chi \ne 1$, so that $\chi$ is a quadratic character of 
$T(F) \backslash T(\A) / T(\A)^2 $. Pushing out the last short exact sequence by $\chi$, we get a short exact sequence
\begin{equation} \label{E:globalSES}
 \begin{CD}
1 @>>> \mu_2 @>>> b(\mathrm{Ker}(\chi)) \backslash  \prod'_v H^1(F_v, Z_E) @>f_{\chi}>>  H^1(F, T)[2] @>>> 1. 
 \end{CD} \end{equation}
 Moreover, the above short exact sequence is equipped with a section $s_{\chi}$ of $f_{\chi}$.  
 \vskip 5pt
 
 We can also arrive at the above short exact sequence by using our local discussion in the previous subsection.
 We have:
 \[  \begin{CD}
 1 @>>>   \oplus_v \mu_2  @>>>  \prod'_v b_v(\mathrm{Ker}(\chi_v))\backslash H^1(F_v, Z_E) @>>> H^1(F, T)[2] @>>> 1. 
 \end{CD} \]
 Pushing this out by the sum map $\oplus_v \mu_2 \rightarrow \mu_2$  and denoting its kernel by $(\oplus_v \mu_2)^1$, we obtain
 \[  \begin{CD}
 1 @>>> \mu_2 @>>> \left( \prod'_v b_v(\mathrm{Ker}(\chi_v))\backslash H^1(F_v, Z_E) \right) / (\oplus_v \mu_2)^1 @>>> H^1(F, T)[2] @>>> 1, 
 \end{CD} \]
 which is the exact sequence in (\ref{E:globalSES}).
 \vskip 5pt
 
 To reformulate the above discussion in the language of characters of component groups, let us introduce the following notions.
 \vskip 5pt
 
 \begin{defn}
 Fix a global A-parameter $\psi = \psi_{E,K, [\chi]}$ with $\chi^2 =1$. 
 \vskip 5pt
 
 (i) For each place $v$,  the sign character of $S_{\psi_v}$ is the nontrivial character $\epsilon_v$ of $S_{\psi_v}/ S_{\psi_v}^0$. 
 \vskip 5pt
 
 (ii)  For any finite subset $\Sigma$ of places of $F$, we set
  \[  \epsilon_{\Sigma} = \prod_{v \in \Sigma} \epsilon_v \times \prod_{v \notin \Sigma} 1_v \] 
 and call $\epsilon_{\Sigma}$ a global sign character of $S_{\psi,\A}$. We say that $\epsilon_{\Sigma}$ is automorphic if it is trivial on $S_{\psi}$. This holds if and only if $|\Sigma|$ is even. The set of automorphic sign characters is a subgroup of $\mathrm{Irr}(\bar{S}_{\psi,\A})$.
 
 (iii) Set 
 \[  [\mathrm{Irr}(\bar{S}_{\psi,\A})]  =  \mathrm{Irr}(\bar{S}_{\psi,\A}) / \{ \text{automorphic sign characters} \}. \]
 \end{defn}
 \vskip 5pt
 
Summarizing the above discussion and applying global  Poitou-Tate duality \cite[Thm. 4.10]{Mi}, we obtain:

 \vskip 5pt
\begin{prop}  \label{P:PT}
Fix a global A-parameter $\psi = \psi_{E,K, [\chi]}$ as above with $\chi^2 =1$.

\vskip 5pt

\noindent (i) If $\chi \ne 1$, one has the following commutative diagram of short exact sequences:
\vskip 5pt
\[ \begin{CD}
1 @>>> \mu_2 @>>> b(\mathrm{Ker}(\chi)) \backslash  \prod'_v H^1(F_v, Z_E) @>f_{\chi}>>  H^1(F, T)[2] @>>> 1   \\
@.   @|  @VV{\delta^* = \prod_v \delta_v^*}V  @| @. \\
1 @>>>  \mu_2   @>>>    [\mathrm{Irr}(\bar{S}_{\psi,\A})] @>>> \mathrm{Irr}(\bar{S}^0_{\psi, \A}) @>>> 1 \\
@. @.  @VVrestV   @. \\
 @.   @.  \mathrm{Irr}(\bar{S}_{\psi})  @.
\end{CD} \]
which are equipped with a canonical section $s_{\chi}$ for $f_{\chi}$ given by the image of $H^1(F, Z_E)$.
 Finally, 
 \[  \mathrm{Ker}( {\mathrm rest}\circ \delta^* )= \mathrm{Im}(s) = \text{the image of $H^1(F, Z_E)$}. \]
 Equivalently, 
 \[   \mathrm{Ker}( {\mathrm rest}) =   \mathrm{Im}(\delta^* \circ s).\]
 
 \vskip 10pt
 
\noindent  (ii)  If $\chi=1$, the map $\delta^* = \prod_v \delta_v^*$ descends to give a section 
\[  \delta^*: H^1(F, T)[2] \longrightarrow \mathrm{Irr}(\bar{S}_{\psi,\A}) \longrightarrow [ \mathrm{Irr}(\bar{S}_{\psi,\A}) ] \]
  Then 
  \[   \mathrm{Ker}(rest)  = \mathrm{Im}(\delta^*)  \]
  where $rest: [ \mathrm{Irr}(\bar{S}_{\psi,\A})] \longrightarrow \mathrm{Irr}(\bar{S}_{\psi})$. 

  \end{prop}
\vskip 10pt

It is interesting to observe the following subtlety. When $\chi \ne 1$ in the above lemma. it is of course possible that $\chi_v =1$ for some places $v$. Let $\Sigma_{\chi}$ be the set of places where $\chi_v =1$. Then for places  $v \in \Sigma_{\chi}$, recall by Proposition \ref{P:TN}(iii) that the map
\[  \delta_v:  H^1(F_v, Z_E)/ b(\mathrm{Ker}(\chi_v)) = H^1(F_v, T)[2]  \longrightarrow \mathrm{Irr}(\bar{S}_{\psi_v}) \]
is only injective but not surjective: its image is a subgroup of index $2$. Hence,   we only have an injection
\[ \prod_v \delta_v^*:  \prod'_v H^1(F_v, Z_E)/ b_v(\mathrm{Ker}(\chi_v))  \hookrightarrow   \mathrm{Irr}(\bar{S}_{\psi,\A}). \]
However, the composite of this injection with the projection to $[ \mathrm{Irr}(\bar{S}_{\psi,\A})]$ is  surjective. This amounts to seeing that given any $\eta \in \mathrm{Irr}(\bar{S}_{\psi,\A})$, one can twist $\eta$ by an automorphic sign character to ensure that  at all places $v \in \Sigma_{\chi}$, $\eta_v$ belongs to the image of $\delta_v$.
\vskip 10pt

 \vskip 10pt

\section{\bf A-packets and Multiplicity Formula}  \label{S:AMF}
After this long preparation,  we are finally ready to define local and global Arthur packets and establish the Arthur multiplicity formula for the A-parameters 
$\psi = \psi_{E,K, [\chi]}$ considered above.
\vskip 5pt

\subsection{\bf Near equivalence classes and Arthur's conjectures}
A global A-parameter $\psi = \psi_{E,K, [\chi]}$ as above (with $E$ fixed) gives rise to a near equivalence class of representations of $G_E(\A)$. 
Namely,  for almost all places, $\psi_v$ is unramified and
\[  \psi_v\left( \mathrm{Frob}_v, \left( \begin{array}{cc}
 q_v^{1/2}   &  \\
 & q_v^{-1/2} \end{array} \right)\right) \in \PGSO_8(\C) \rtimes_E W_F  \]
gives a semisimple conjugacy class in $\PGSO_8(\C) \cdot {\mathrm Frob_v}$, which in turns determines an unramified representation of $G_E(F_v)$. 
We denote the associated near equivalence class in $\mathcal{A}_2(G_E)$ by $\mathcal{A}_{2,\psi}(G_E)$. 

\vskip 5pt

To a first approximation, Arthur's conjectures describe the structure of this submodule $\mathcal{A}_{2,\psi}(G_E)$. Though we have already discussed these conjectures in \S \ref{SS:arthur}, we highlight the two key points here for the convenience of the reader: 
\vskip 5pt

\begin{itemize}
\item (Local)  One expects to have a local A-packet $\Pi_{\psi_v}$, which is a finite multi-set over $\mathrm{Irr}(G_E(F_v))$
equipped with a map
\[  \Pi_{\psi_v} \longrightarrow \mathrm{Irr}(S_{\psi_v}). \]
We may thus view $\Pi_{\psi_v}$ as a finite length representation of $S_{\psi_v} \times G_E(F_v)$:
\[  \Pi_{\psi_v} = \bigoplus_{\eta_v \in \mathrm{Irr}(S_{\psi_v})}  \eta_v \otimes  \pi_{\eta_v}. \]
\vskip 5pt

\item (Global)  One has:
\[  \mathcal{A}_{2,\psi}(G_E)  \cong \bigoplus_{\eta \in \mathrm{Irr} S_{\psi, \A}}  \dim \Hom_{S_{\psi}}(\eta \circ \Delta, \C) \cdot  \pi_{\eta} \]
where
\[  \eta = \otimes_v\eta_v  \in \mathrm{Irr}(S_{\psi,\A}) = {\prod_v}' \mathrm{Irr}(S_{\psi_v})  \quad \text{and} \quad \pi_{\eta} := \otimes'_v \pi_{\eta_v}. \]
\end{itemize}
\vskip 5pt
We shall see that the square-integrable automorphic representations we have constructed by theta lifting in \S \ref{S:global-theta} verify  the above conjectures of Arthur.

\vskip 5pt

\subsection{\bf Theta lifts and near equivalence class}
Given a global  A-parameter $\psi = \psi_{E,K, [\chi]}$, we have
the pair  $\{ \chi,  \chi^{-1} \}$ of  automorphic characters of $ \tilde{T}_{E,K}$.
For any $C \in \Sigma_{E,K}$,    we have noted in \S \ref{SS:isom-tori} that there is a pair of isomorphisms
\begin{equation} \label{E:isom-tori-global}
  \iota_C, \, \iota_C^{-1} : H_C^0 \cong \tilde{T}_{E,K} \end{equation}
 of algebraic tori over $F$ (associated to the two choices of $(B,\tau, \iota)$ with $C$ corresponding to $E \hookrightarrow B^{\tau}$). 
 Pulling back $\chi$ and $\chi^{-1}$ via $\iota_C$, we  obtain a pair of
 automorphic characters $\chi^{\pm 1} \circ \iota_C$ of  $H^0_C = \Aut_E(C)^0$. Set 
 \[  V_C[\chi]  \subset \mathcal{A}(H_C) \]
 to be the submodule spanned by all irreducible summands whose restriction to $H_C^0$ contains $\chi \circ \iota_C$ or $\chi^{-1}\circ \iota_C$; this submodule is thus independent of the isomorphism $\iota_C$. 
In earlier sections, we have studied the theta lifting from $\mathcal{A}(H_C)$ to $\mathcal{A}_2(G_E)$. From our local results, one sees that the theta lift of the submodule $V_C[\chi]$  is contained in the near equivalence class $\mathcal{A}_{2, \psi}(G_E)$.  More precisely, Corollary  \ref{C:global-injective} gives
\vskip 5pt

\begin{prop}
Given $\psi = \psi_{E,K,[\chi]}$,
\[ V[\psi] :=  \bigoplus_{C \in \Sigma_{E,K}}    \Theta( V_C[\chi])   \subset \mathcal{A}_{2,\psi}(G_E). \]
Moreover, if $V_C([\chi]) = \oplus_\rho m_C(\rho) \cdot \rho$, then
\[  \Theta(V_C[\chi])  \cong \bigoplus_{\rho} m_C(\rho) \cdot \Theta^{abs}(\rho). \]
 \end{prop}

\vskip 5pt

 \vskip 5pt

\subsection{\bf Local A-packets.}
Our goal in the remainder of this section is to show that the submodule $V[\psi]$ in the above proposition can be described in the form dictated by Arthur's conjectures. 
Let us first collect together all the local components of the constituents of $V(\psi_{E,K, [\chi]})$.  
\vskip 5pt

\begin{defn}
Given $\psi = \psi_{E ,K, [\chi]}$, set
 \[  \Sigma_{E_v, K_v, [\chi_v]} = \{  (C_v, \rho_v)  \in \Sigma_{E_v, K_v} \times \mathrm{Irr}(H_{C_v}(F_v)): \text{$\rho_v|_{H_{C_v}^0(F_v)} \supset \chi_v \circ \iota_{C_v}$ or $\chi_v^{-1} \circ \iota_{C_v}$} \} \]  
 and
 \[ \Pi_{\psi_v} =  \{  \theta_{C_v}(\rho_v): (C_v , \rho_v) \in \Sigma_{E_v, K_v, [\chi_v]} \} \subset \mathrm{Irr}(G_{E_v}(F_v)). \]
   \end{defn}
 \vskip 5pt

 We have shown in Theorems \ref{T:111}, \ref{T:arch1}, \ref{T:arch1.5} and \ref{T:arch2}   that  for $(C_v, \rho_v) \in \Sigma_{E_v, K_v, [\chi_v]}$, 
  the theta lift $\theta_{C_v}(\rho_v)$ is nonzero irreducible.  Moreover, $\Pi_{\psi_v}$ is a set (rather than a multiset). It is clear that the set $\Pi_{\psi_v}$ contains all possible local component at $v$ of the constituents of $V(\psi_{E,K, [\chi]})$; this will be our definition of the local A-packet associated to $\psi_v$. Observe that, by definition,  there is a natural bijection
 \[  \Pi_{\psi_v} \longleftrightarrow \Sigma_{E_v, K_v, [\chi_v]}. \]

  \vskip 5pt

  \subsection{\bf The bijection $j_{\psi_v}$}  \label{SS:bijj}
  Our next task is to construct a natural bijection
   \[ \Pi_{\psi_v}  \longrightarrow \mathrm{Irr}(S_{\psi_v})\]
or equivalently  a bijection
 \[  j_{\psi_v}:   \mathrm{Irr} (S_{\psi_v}) \longleftrightarrow \Sigma_{E_v, K_v, [\chi_v]},  \]
 which then induces the deisred bijection with $\Pi_{\psi_v}$. To do this, we shall exploit
Lemma \ref{L:component},   Lemma \ref{L:localH1}, Proposition \ref{P:TN} as well as Proposition \ref{P:isom-torsor}. 
\vskip 5pt

Let us begin with some general observations:
\vskip 5pt

\begin{itemize}
\item[(a)] By restriction, one obtains (by Lemma \ref{L:component} and Proposition \ref{P:TN}(i)) a natural map 
\[
 \mathrm{Irr}(S_{\psi_v}) \longrightarrow (\mathrm{Irr}(S^0_{\psi_v}))/S_2  \cong H^1(F_v, \tilde{T}_{E,K}) /S_2 = \Sigma_{E_v, K_v}. \]
Hence, each $\eta_v \in  \mathrm{Irr}(S_{\psi_v})$ gives rise to a $C_{\eta_v} \in \Sigma_{E_v, K_v}$.
 \vskip 5pt
 
 \item[(b)] Suppose that $\chi_v^2 =1$ but $\chi_v \ne 1$. Then by Proposition \ref{P:TN}(ii), we have:
 \[  \begin{CD} 
 \mathrm{Irr}(\bar{S}_{\psi_v})  @=   H^1(F_v, Z_E)/ b(\mathrm{Ker}(\chi_v)). \\
 @VVV  @VVf_{\chi_v}V \\
  \mathrm{Irr}(\bar{S}^0_{\psi_v}) @=  H^1(F_v, \tilde{T}_{E_v, K_v})[2] \end{CD} \] 
 For any given $[C_v] \in H^1(F_v, \tilde{T}_{E_v, K_v})[2]$,  write
 \[  \mathrm{Irr}_{C_v}(\bar{S}_{\psi_v})  \longleftrightarrow f_{\chi_v}^{-1} ([C_v]). \]
 These are sets of size $2$.
 \vskip 5pt
 
 Now  Proposition \ref{P:isom-torsor} gives  a natural isomorphism of $\tilde{T}_{E_v, K_v}(F_v)/ \tilde{T}_{E_v, K_v}(F_v)^2$-torsors   
 \[  g_{C_v}:  f^{-1}([C_v])   \longrightarrow  (H_{C_v}(F_v) \smallsetminus H^0_{C_v}(F_v))/ \tilde{T}_{E_v, K_v}(F_v)^2, \]
 which induces a bijection
\[   g_{C_v,\chi_v}:   f^{-1}([C_v] )/ b(\mathrm{Ker}(\chi_v)) \longrightarrow  (H_{C_v}(F_v) \smallsetminus H^0_{C_v}(F_v))/ \mathrm{Ker}(\chi_v). \]
Taken together, we thus have a canonical bijection
\[  \mathrm{Irr}_{C_v}(\bar{S}_{\psi_v})  \longleftrightarrow (H_{C_v}(F_v) \smallsetminus H^0_{C_v}(F_v))/ \mathrm{Ker}(\chi_v). \]
\vskip 5pt

Hence, given $\eta_v \in \mathrm{Irr}_{C_v}(\bar{S}_{\psi_v})$ (so that $C_{\eta_v} = C_v$), $\eta_v$ corresponds to an element 
\[  a_{\eta_v} \in  f^{-1}([C_v] )/ b(\mathrm{Ker}(\chi_v)) \]
and then an element
\[  g_{C_v,\chi_v}(a_{\eta_v}) \in  (H_{C_v}(F_v) \smallsetminus H^0_{C_v}(F_v))/ \mathrm{Ker}(\chi_v). \]
On the other hand, the character $\chi_v \circ \iota_{C_v}$ of $H^0_{C_v}(F_v)$ has two extensions to $H_{C_v}(F_v)$, which are distinguished by the value $\pm 1$ they take on $g_{C_v}(a_{\eta_v})$.  We define
\[   \rho_{\eta_v} = \text{the extension of $\chi_v \circ \iota_{C_v}$ which takes value $+1$ on $g_{C_v}(a_{\eta_v})$} \]
and set
\[  j_{\psi_v}(\eta_v) = (C_{\eta_v},  \rho_{\eta_v}) \in \Sigma_{E_v, K_v, [\chi_v]}. \]
By Corollary \ref{C:Whit}, $\rho_{\eta_v}$ is also characterized as the unique extension of $\chi_v \circ \iota_{C_{\eta_v}}$ whose mini-theta lift to $\GL_2(E_v)^{\det}$ is supported on the Whittaker data in $a_{\eta_v} \cdot b(\mathrm{Ker}(\chi_v))$.
\vskip 5pt

\item[(c)] If $\chi_v =1$, then by Proposition \ref{P:TN}(iii), there is a canonical section
\[ \delta_v^*:  H^1(F_v, \tilde{T}_{E_v, K_v})[2] = \mathrm{Irr}(\bar{S}^0_{\psi_v}) \longrightarrow \mathrm{Irr}(\bar{S}_{\psi_v}). \]
So for the two extensions of a character $\eta_v^0$ of $\bar{S}^0_{\psi_v}$, there is a distinguished one contained in the image of $\delta_v^*$.
On the other hand, for any $[C_v]  \in H^1(F_v, \tilde{T}_{E_v, K_v})[2]$,  there is a distinguished extension of the trivial character $\chi_v\circ \iota_{C_v}$ from $H^0_{C_v}(F_v)$ to $H_{C_v}(F_v)$, namely the trivial character. Hence if  $\eta_v  = \delta_v^*(C_{\eta_v})$, we set
\[  \rho_{\eta_v}  = 1_{C_{\eta_v}} \quad \text{and}\quad \rho_{\eta_v \cdot \epsilon_v}  = \epsilon_{C_{\eta_v}} \] 
where $\epsilon_v$ is the sign character of $S_{\psi_v}$ and $\epsilon_{C_v}$ is the nontrivial (sign) character of $H_{C_v}(F_v)/H^0_{C_v}(F_v)$. 
\end{itemize}
\vskip 5pt

Hence, when $\chi_v^2=1$, we have defined in (b) and (c) above a canonical bijection
\begin{equation} \label{E:j-bar}
  \mathrm{Irr}(\bar{S}_{\psi_v}) \longleftrightarrow  \bar{\Sigma}_{E_v, K_v, [\chi_v]} = \{ (C_v,\rho_v)\in \Sigma_{E_v, K_v,[\chi_v]}: [C_v]^2 =1\}. \end{equation}
To complete  the construction of $j_{\psi_v}$, it will now    be convenient to consider different cases, depending on whether $(E_v, K_v) = (\text{field}, \text{split})$ or not, and whether $\chi_v^2 =1$ or not.
\vskip 5pt
 
 \begin{itemize}
 \item[(1)] Suppose first that $(E_v, K_v) \ne (\text{field}, \text{split})$. Then  $S_{\psi_v} = \bar{S}_{\psi_v}$ is an elementary abelian 2-group. If $\chi_v^2 =1$, the (\ref{E:j-bar}) already gives the construction of $j_{\psi_v}$. On the other hand, when $\chi_v^2 \ne 1$, then $S_{\psi_v} = S^0_{\psi_v}$. For $\eta_v \in \mathrm{Irr}(S_{\psi_v})$, we set
 \[  \rho_{\eta_v}= \mathrm{Ind}_{H^0_{C_{\eta_v}(F_v)}}^{H_{C_{\eta_v}}(F_v)}  \chi^{\pm 1} \circ \iota_{C_{\eta_v}}, \]
 recalling that $H_{C_v}(F_v) \ne H_{C_v}^0(F_v)$ for any $[C_v] \in \Sigma_{E_v, K_v}$.

 \vskip 5pt
 
 \item[(2)] Suppose now that $(E_v, K_v) = (\text{field}, \text{split})$, so that $v$ is necessarily a non-archimedean place of $F$. 
 We fix the map $\psi_v$ (as opposed to considering it as a conjugacy class of maps)  and suppose that 
$\psi_v|_{W_{F_v}}$ corresponds to the character $\chi_v$ (as opposed to $\chi_v^{-1}$) of $\tilde{T}_{E_v, K_v}$. 
 Then Proposition \ref{P:TN} and  Lemma \ref{L:localH1} give
  \[    \mathrm{Irr} (S^0_{\psi_v})  =  H^1(F_v,  \tilde{T}_{E_v, K_v}) = \tilde{\Sigma}_{E_v, K_v} =\mathrm{Br}_3(F_v) =  \Z/3\Z. \]
  Thus, an element $\eta_v \in  \mathrm{Irr} (S^0_{\psi_v}) $ gives rise to an $E_v$-twisted composition algebra $C_{\eta_v}$ and then a central simple algebra $D_{\eta_v} \in {\mathrm Br}_3(F_v)$ with an isomorphism 
  \[  i_{\eta_v}  = i_{D_{\eta_v}} : H_{C_{\eta_v}}(F_v) \longrightarrow \tilde{T}_{E_v, K_v}.\]
    Explicitly, we have two possible twisted composition algebras
 \[  C_v^+ = (E_v \hookrightarrow M_3(F_v)) \quad \text{and} \quad C_v^-  = (E_v \hookrightarrow D_v^+), \]
 where $D_v$ is any of the two cubic division $F$-algebras.    Moreover, the two isomorphisms $i_{D_v}$ and $i_{D_v^{op}}$ differ from each other by composition with inversion.
 We recall also that 
 \[  \text{$[H_{C^+_v}(F_v) : H_{C^+_v}^0(F_v)] =2$,  but  $H_{C_v^-}(F_v) = H_{C_v^-}^0(F_v)$. } \]

\vskip 5pt

 We now consider two cases:
 \vskip 5pt
 \begin{itemize}
 \item[(a)] $\chi_v^2 \ne 1$. In this case, one has  $S_{\psi_v}  = S^0_{\psi_v} = \mu_3$, so (\ref{E:j-bar}) tells us nothing in this case. 
 To specifiy the bijection 
\[ j_{\psi_v}:  \mathrm{Irr} (S_{\psi_v}) = \Z/3\Z  \longleftrightarrow \Sigma_{E_v, K_v, [\chi_v]}, \]
the trivial character of $S_{\psi_v}$ is sent to the element $(C_v^+,  \rho_{\eta_v}[\chi_v] ) \in \Sigma_{E_v, K_v,[\chi_v]}$, where $\rho_{\eta_v}[\chi_v]$ is defined as in case (1a) above. For a nontrivial character $\eta_v$ of $S_{\psi_v}$,  we set
 \[  j_{\psi_v}(\eta_v) = (C_v^-,   \chi_v \circ i_{\eta_v}). \]
 \vskip 5pt
 
 We note that the above recipe is independent of the choice of the representative $\psi_v$ in its conjugacy class.
 Indeed, if we had used the map $ \psi_v^{-1}$ (which corresponds to $\chi_v^{-1}$), then one has an equality of the component groups $S_{\psi_v} = S_{\psi_v^{-1}}$ as subsets of $S \rtimes (S_2 \times S_3)$. However, an element of the latter which conjugates $\psi_v$ to $\psi_v^{-1}$ induces not the identity automorphism of $S_{\psi_v}$ but the inverse automorphism. This implies that
 \[  j_{\psi_v}(\eta_v) = j_{\psi_v^{-1}}(\eta_v^{-1}), \]
  so that  the above recipe is independent of the choice of the representative map $\psi_v$ in its conjugacy class. A better language to express this is to work with the projective systems of $[\psi_v]$ and $[S_{\psi_v}]$, as we did in \cite[Prop. 3.2]{GS4}, where a similar situation arises.

 \vskip 10pt
 
 \item[(b)]  $\chi_v^2 =1$.  In this case, we have the short exact sequence
 \[  \begin{CD}
  1 @>>>  S_{\psi_v}^0= \mu_3 @>>>  S_{\psi_v} = S_3 @>>>  S_2  @>>> 1, \end{CD} \]
  so that $S_{\psi_v}$ is the nonabelian group $S_3$. Let us denote the irreducible representations of $S_3$ by $1$, $\epsilon$ (the sign character) and $r$ (the unique 2-dimensional irreducible representation).  Because $\bar{S}_{\psi_v}  = S_{\psi_v}/ S^0_{\psi_v} = S_2$, (\ref{E:j-bar}) already determines for us $j_{\psi_v}(1)$ and $j_{\psi_v}(\epsilon)$.  Hence we have no choice for $j_{\psi_v}(r)$:
   \[  j_{\psi_v}(r) = (C_v^-, \chi_v \circ \iota_{C_v^-}). \] 
    \end{itemize}

 \end{itemize}
 \vskip 10pt
 
   This completes our construction of a canonical bijection
 \[  j_{\psi_v} : \mathrm{Irr}(S_{\psi_v}) \longleftrightarrow \Sigma_{E_v, K_v,  [\chi_v]}, \]
 For any $\eta_v \in \mathrm{Irr}(S_{\psi_v})$, if $j_{\psi_v}(\eta_v) = (C_{\eta_v}, \rho_{\eta_v})$, 
we write
\[  \pi_{\eta_v} :=  \theta_{C_{\eta_v}}(  \rho_{\eta_v}) \in \mathrm{Irr}(G_E(F_v)). \]
    \vskip 10pt

\subsection{\bf Global A-packets}
We come now to the global setting.  
Without loss of generality, fix  a global A-parameter, or more precisely a map  
\[ \psi = \psi_{E,K, [\chi]} : W_F \longrightarrow S \rtimes (S_2 \times S_3) \subset \PGSO_8(\C) \rtimes S_3 \] 
and suppose that its restriction to $W_F$  corresponds to the Hecke character $\chi$ of the torus $\tilde{T}_{E,K}$. 
The $\PGSO_8(\C)$-conjugacy class of $\psi$ then corresponds to the pair $[\chi] = \{\chi, \chi^{-1} \}$ of Hecke characters of the torus $\tilde{T}_{E,K}$. 

\vskip 5pt

  As we explained in \S \ref{SS:bijj},  the local A-packets $\Pi_{\psi_v}$  are equipped  with canonical bijections 
\[  j_{\psi_v} : \mathrm{Irr}(S_{\psi_v}) \longleftrightarrow   \Sigma_{E_v, K_v. [\chi_v]} \longleftrightarrow \Pi_{\psi_v} \]
The global A-packet $\Pi_{\psi}$ associated to $\psi$ is simply the restricted tensor product of the local ones, so that
\[  \Pi_{\psi} = \{  \pi_{\eta} = \otimes_v'\pi_{\eta_v}:  \eta= \otimes_v \eta_v \in \mathrm{Irr}(  S_{\psi, \A}) \}. \]
  The irreducible summands of $V[\psi] \subset \mathcal{A}_{2,\psi}(G_E)$ are  isomorphic to elements of $\Pi_{\psi}$.

  \vskip 10pt

\subsection{\bf Multiplicity formula.}
Our remaining task is to verify that the Arthur multiplicity formula holds for $V[\psi]$. In other words, for each $\eta = \otimes_v \eta_v$, we need to determine the multiplicity of $\pi_{\eta}$ in $V[\psi]$.
Now 
\[  \pi_{\eta} \cong \otimes_v \theta _{C_{\eta_v}}(\rho_{\eta_v})  \quad \text{where $j_{\psi_v}(\eta_v) = (C_{\eta_v}, \rho_{\eta_v})$ for each $v$.} \]
 To determine the multiplicity of $\pi_{\eta}$ in $V[\psi]$, we consider the subset
 \[  \Sigma_{E,K, [\chi], \eta} \subset \Sigma_{E,K} \]
 consisting of those $C$'s satisfying:
 \vskip 5pt
 
 \begin{itemize}
 \item for each place $v$ of $F$, there is an isomorphism
 \[ \iota_v:  C_v := C \otimes_F F_v \cong   C_{\eta_v}.   \]
 Note that the isomorphism $\iota_v$ is unique up to composition by elements of $H_C(F_v)$, and so induces an isomorphism
 \[  \iota_v:  H_{C_v} \cong H_{C_{\eta_v}} \]
 which is well-determined up to conjugation. Hence,  $\rho_{\eta_v} \circ \iota_v$ is a well-defined element of $\mathrm{Irr}(H_{C_v}(F_v))$.
 In particular, we have a well-defined abstract irreducible representation 
 \[  \rho_{\eta,C} :=  \otimes_v'  \rho_{\eta_v} \circ \iota_v \quad \text{ of $H_C(\A)$}  \]
 such that 
 \[  \Theta^{abs}(\rho_{\eta,C}) \cong \pi_{\eta} \quad \text{as abstract representations.} \]
 \vskip 5pt
 
 \item the representation $\rho_{\eta,C}$ is automorphic and hence occurs in $V_C[\chi]$. 
   \end{itemize}
 \vskip 5pt
 
 To decide if $\rho_{\eta,C}$ is automorphic, an important role is played by the following diagram:

\[
\xymatrix@R=2pt{
H^0_{C_v}  \ar[dd]^{\iota_v}\ar[drr]^{\iota_{C,v}}& & \\
   & &   \tilde{T}_{E_v, K_v}  \\
 H^0_{C_{\eta_v}}  \ar[urr]_{\iota_{\eta_v}} & & 
 }
\]
Here $\iota_{C,v}$ is the localization of $\iota_C$ at the place $v$ and we recall that $\iota_C$ is well-determined up to conjugacy by $H_C(F)$, and likewise $\iota_v$ is well-determined up to conjugacy by $H_{C_v}(F_v)$. It is natural to ask if this diagram is commutative, or can be rendered such.
We have:
\vskip 5pt

\begin{lemma} \label{L:commu}
The above diagram commutes up to inverting, i.e. 
\[ 
\iota_{\eta_v} \circ \iota_v =  \iota_{C,v} \quad \text{or} \quad \iota_{C,v}^{-1}. \]
Hence, if $H_{C}(F_v) \ne H_C^0(F_v)$, then the above diagram is commutative by replacing $\iota_v$ by $\iota_v^{-1}$ if necessary.
In particular, if $H_C(F) \ne H_C^0(F)$, then the above diagram can be made commutative at all places $v$ (by appropriate choices of $\iota_v$ at each $v$).  
\end{lemma}

 \vskip 5pt
 
 For $C \in \Sigma_{E,K, [\chi], \eta}$, the  multiplicity $m_C(\rho_{\eta,C})$ of $\rho_{\eta,C}$ in $V_C[\chi]$ is in fact independent of $C$, by our discussion in \S \ref{SS:auto-HC}. We thus denote this multiplicity by $m(\rho_{\eta}) >0$.  Given this, we see that
 \[  \text{Multiplicity of $\pi_{\eta}$ in $V[\psi]$} =  m(\rho_{\eta}) \cdot \# \Sigma_{E,K,[\chi], \eta}. \]
 To establish the multiplicity formula,  we need to show that the above number is equal to
 \[ m_{\eta}: =  \langle \eta \circ \Delta , 1 \rangle_{S_{\psi}} = \frac{1}{\# S_{\psi}} \cdot \sum_{S \in S_{\psi}} {\mathrm tr} \left( \eta(\Delta(s)) \right). \]
 We consider the different cases of $\psi = \psi_{E,K, [ \chi]}$ in turn in the subsequent subsections.

  \vskip 5pt

 \subsection{\bf $K$ is a field and $\chi^2 \ne 1$}
 This is in some sense the most nondegenerate case, as all possible local scenarios we discussed in \S \ref{SS:bijj} can occur. However, it is also the least subtle case because
 \[  S_{\psi}  = \{ 1\} \quad \text{so that} \quad m_{\eta}  = \dim \eta.  \]
 Let $S_{\eta}$ denote the finite set of places $v$ of $F$ where $C_{\eta_v}$ is associated with a cubic division algebra; at these places, we have $(E_v, K_v) = (\text{field, split})$.   
  We have a decomposition
 \[  S_{\eta} =  S_{\eta}'  \sqcup S_{\eta}'' \]
 where $S_{\eta}'$ consists of those places $v$ where $\chi_v^2 =1$. Then
 \[  \dim \eta_v =\begin{cases} 
  1, \text{ if $v \notin S_{\eta}'$} \\
  2, \text{  if $v \in S_{\eta}'$} \end{cases} \]
  so that
  \[  m_{\eta} = \dim \eta = 2^{\# S_{\eta}'}. \]
 \vskip 5pt
 
 We now need to determine the size of $\Sigma_{E,K, [\chi], \eta}$. For $C \in \Sigma_{E,K, [\chi], \eta}$  corresponding to $E \hookrightarrow B^{\tau}$ 
 (for a central simple algebra $B$ over $K$ of degree $3$, equipped with an involution $\tau$ of the second kind), $B$ is ramified precisely at $v \in S_{\eta}$.  
 The number of possible $C$'s is, at this point,  
 \[  \begin{cases}
  2^{\# S_{\eta} -1} \text{  if $S_{\eta}$ is nonempty;} \\
   1, \text{  if $S_{\eta}$ is empty.} \end{cases} \]
 However, we also need to impose the condition that $\rho_{\eta,C}$ is  automorphic.   
  \vskip 5pt
  
  Assume first that $S_{\eta}$ is nonempty. For any $C \in \Sigma_{E,K, [\chi],\eta}$, we have $H_C(F) = H_C^0(F)$. From our discussion in \S \ref{SS:auto-HC}, the abstract representation   $\rho_{\eta,C}$ is automorphic if and only if its abstract restriction  to $H_C^0(\A)$ contains $\chi \circ \iota_C$ or $\chi^{-1} \circ \iota_C = \chi \circ \iota_C^{-1}$. 
  In other words, we need 
  \[    \rho_{\eta_v} |_{H^0_{C_{\eta_v}} } \circ \iota_v   \supset  \chi_v \circ \iota_{C,v}  \quad \text{for all places $v$:} \] 
  for one of the two choices of $\iota_C$.
  \vskip 5pt
  
  Now
  \[  \rho_{\eta_v}  |_{H^0_{C_{\eta_v}} }  = \begin{cases}  
  \chi_v \circ \iota_{\eta_v}  + \chi_v^{-1} \circ \iota_{\eta_v}  \text{  if $v  \notin S_{\eta}$ and $\chi_v^2 \ne 1$;} \\
  \chi_v \circ \iota_{\eta_v} , \text{  otherwise.} \end{cases} \]
  From this and Lemma \ref{L:commu}, we  see that the desired containment holds for any $v \notin S_{\eta}''$ for both choices of $\iota_C$. 
  \vskip 5pt
  
   It remains to consider the places in $S_{\eta}''$, where we need the following to hold:
   \[  \chi_v \circ \iota_{\eta_v} \circ \iota_v  = \chi_v \circ \iota_{C,v}. \]
    This identity fixes $\iota_{C,v}$ for every $v \in S_{\eta}''$. In other words, if $\iota_C$ is associated to $E \hookrightarrow B^{\tau}$ for a pair $(B,\tau)$, then the invariant of $B_v$ for every $v \in S_{\eta}''$ is fixed, and we only have the freedom to dictate the invariant of $B_v$  at $v \in S_{\eta}'$. 
    \vskip 5pt

    Hence, the number of possible $(B,\tau)$'s is $2^{\# S_{\eta}'}$ and
  \[   \# \Sigma_{E,K, [\chi], \eta} = \begin{cases} 
  2^{\# S_{\eta}'} , \text{  if $S_{\eta}''$ is nonempty;} \\
  2^{\# S_{\eta}' -1},  \text{  if $S_{\eta}''$ is empty.} \end{cases} \]
  On the other hand, by our discussion in \S \ref{SS:auto-HC},
  \[  m(\rho_{\eta}) = m_C(\rho_{\eta,C}) = \begin{cases}
  1,  \text{  if $S_{\eta}''$ is nonempty;} \\
  2,  \text{  if $S_{\eta}''$ is empty.} \end{cases} \]
  Taken together, we see that
  \[  m(\rho_{\eta}) \cdot  \# \Sigma_{E,K, [\chi], \eta}  = 2^{\# S_{\eta}'} = m_{\eta}, \]
  as desired.
\vskip 5pt

The case when $S_{\eta}$ is empty is dealt with similarly, with both quantities equal to $1$; we omit the details.

 \vskip 10pt

\subsection{\bf $K$ is a field and $\chi^2 =1$}
In this case 
\[  S_{\psi} = S_2. \]
 Given $\eta = \otimes_v \eta_v \in \mathrm{Irr}(S_{\psi, \A})$,  let $S_{\eta}$ be the finite set of places $v$ of $F$ where $C_{\eta_v}$ is associated with a cubic division algebra. 
Then  $\eta_v$ is the 2-dimensional representation $r$ of $S_{\psi_v} = S_3$  if $v \in S_{\eta}$, and $\eta_v$ is 1-dimensional otherwise. 
   Then
\[  m_{\eta} = \dim \Hom_{S_2} ( r^{ \otimes \# S_{\eta}} , \otimes_{v \notin S_{\eta}} \eta_v )  = 2^{\# S_{\eta} -1} \]
if $S_{\eta}$ is nonempty. On the other hand, if $S_{\eta}$ is empty, then 
\[  m_{\eta} = \frac{1}{2} ( 1 + (-1)^b)  = \begin{cases}
1 \text{  if $b$ is even;} \\
0 \text{  if $b$ is odd.} \end{cases} 
 \]
where $b$ is the finite number of places $v$ of $F$ where $\eta_v$ is nontrivial on $S_{\psi}$.  

\vskip 10pt

Assume first that $S_{\eta}$ is nonempty.
For any $C \in \Sigma_{E,K, [\chi], \eta}$, $H_C(F) = H_C^0(F)$, and if $C$ is associated with $E \hookrightarrow B^{\tau}$, then $B$ is ramified precisely at places in $S_{\eta}$.  Further, for  $\rho_{\eta,C}$ to be automorphic, we need to verify that, for one of the two choices of $\iota_C$, one has
\[  \rho_{\eta_v}   \circ \iota_v|_{H_C(F_v)}  = \chi_v \circ \iota_{C,v} \quad \text{for all places $v$.} \]
In fact, since $\chi^2=1$, it is immaterial which of the two $\iota_C$'s we use.  Now
 \[  \rho_{\eta_v}|_{H_{C_{\eta_v}}^0(F_v)} = \chi_v \circ \iota_{\eta_v}  \quad \text{for all $v$.} \]
 Hence the desired equality follows from Lemma \ref{L:commu} and the hypothesis that $\chi^2 =1$. In other words, $\rho_{\eta,C}$ is necessarily automorphic for any $C \in \Sigma_{E,K, [\chi], \eta}$, with $m_C(\rho_{\eta,C}) =1$. Hence,
 \[  \# \Sigma_{E,K, [\chi], \eta} = 2^{\# S_{\eta} -1} \]
 so that
 \[  m(\rho_{\eta}) \cdot \# \Sigma_{E,K, [\chi], \eta} = 2^{\# S_{\eta} -1} = m_{\eta} \]
 as desired. 
 \vskip 5pt
 
 Consider now the case when  $S_{\eta}$ is empty, so that $\eta$ is a character of $\bar{S}_{\psi,\A}$. In this case, $C_{\eta_v} \in H^1(F_v, \tilde{T}_{E_v, K_v})[2]$ for all $v$, and so by Lemma \ref{L:loc}, there is a unique $C_{\eta} \in H^1(F, \tilde{T}_{E,K})[2]$ so that $C_{\eta,v}   \cong C_{\eta_v}$ for all $v$, and we need to determine if $\rho_{\eta}$ is automorphic for $H_{C_{\eta}}$.
 For this, we shall appeal to Proposition \ref{P:PT} and Proposition \ref{P:isom-torsor}.
 \vskip 5pt
 
 By Proposition \ref{P:PT}, we see that $[\mathrm{Irr}(\bar{S}_{\psi,\A})]$ is divided into two equivalence classes, depending on whether the restriction to $\bar{S}_{\psi} = \mu_2$ is trivial or not. The distinguished class, with trivial restriction to $\bar{S}_{\psi}$,  is thus the one for which $m_{\eta} = 1$ (instead of $0$).
 Proposition \ref{P:PT} says that this distinguished class is 
  precisely the one which contains the image of a section $H^1(F, \tilde{T}_{E,K}) \rightarrow \mathrm{Irr}(\bar{S}_{\psi,\A})$. Equivalently, it is  the image of the natural map $H^1(F, Z_E) \rightarrow {\prod}'_v  H^1(F_v, Z_E) \rightarrow  [\mathrm{Irr}(\bar{S}_{\psi,\A})]$.
 
 \vskip 5pt
 
 For $\eta$ with $m_{\eta} =1$, there is thus an element $a_{\eta} \in H^1(F, Z_E) $ and an automorphic sign character $\epsilon_{\Sigma}$ such that for all places $v$, $\eta_v \cdot \epsilon_{\Sigma,v}$ corresponds to $a_{\eta}$ under the bijection 
 \[  H^1(F_v, Z_E)/ b(\mathrm{Ker}(\chi_v)) \supset f_{\chi_v}^{-1}([C_{\eta_v}]) \longleftrightarrow \mathrm{Irr}(\bar{S}_{\psi_v})  \]
 in Proposition \ref{P:TN}.  Observe that $m_{\eta \cdot \epsilon_{\Sigma}} = 1$ as well,  and $\rho_{\eta \cdot \epsilon_{\Sigma}} = \rho_{\eta} \cdot \epsilon_{C_{\eta}, \Sigma}$, where $\epsilon_{C_{\eta}, \Sigma}$ is the automorphic sign character of $H_{C_{\eta}}$ nontrivial precisely at places in $\Sigma$.
Hence in deciding the automorphy of $\rho_{\eta}$, there is no harm in assuming that $\Sigma$ is empty, by  replacing $\eta$ by $\eta \cdot \epsilon_{\Sigma}$ if necessary.
 \vskip 5pt

 By Proposition \ref{P:isom-torsor}, the element $a_{\eta} \in H^1(F, Z_E)$ gives rise to an element 
 \[  g(a_{\eta}) \in H_{C_{\eta}}(F_v) \smallsetminus H_{C_{\eta}}^0(F_v)  \quad \text{for each place $v$.}  \]
 Now $\rho_{\eta}$ is automorphic if and only if $\rho_{\eta}(g(a_{\eta}))  =1$. But its local component  $\rho_{\eta_v}$ is characterized by the property that 
 \[  \chi_v( g(a_{\eta}) ) =1 \quad \text{  for all $v$.} \] 
 In particular, $\rho_{\eta}$ is automorphic when $m_{\eta} = 1$, as desired. 
\vskip 5pt
On the other hand, if $m_{\eta} = 0$, it is clear that $\rho_{\eta}$ is not automorphic, since $\rho_{\eta}$ differs from an automorphic $\rho_{\eta'}$ by a twist of a global sign character of $H_{C_{\eta}}$ which is the local sign character at an odd number of places.
\vskip 10pt

\subsection{\bf $K$ is split and $\chi^2 \ne 1$.}
In this case, 
\[ S_{\psi} = \mu_3, \]
For a given $\eta \in \mathrm{Irr}(S_{\psi,\A})$, let $S_{\eta}$ be the finite set of places where $C_{\eta_v}$ is associated with a cubic division algebra. For $v \in S_{\eta}$, $E_v$ is necessarily a field. 
We have a decomposition
\[  S_{\eta} = S_{\eta}' \sqcup S_{\eta}'' \]
where $S_{\eta}'$ consists of those $v$ where $\chi_v^2 =1$.  Hence, for $v \in S_{\eta}'$, $S_{\psi_v} = S_3$ and $\eta_v$ is the 2-dimensional irreducible representation $r$ of $S_3$; at all other places, $\eta_v$ is 1-dimensional.  For places $v \in S_{\eta}''$, $S_{\psi_v} = \mu_3$ and we further decomposes
\[  S_{\eta}''  = S_{\eta, 1}'' \cup S_{\eta, 2}'' \]
where $S_{\eta,1}''$ consists of those $v$ such that $\eta_v$ corresponds to the element $1/3 \in \Z/3\Z = \mathrm{Irr}(\mu_3)$ and $S_{\eta, 2}''$  those $v$ such that $\eta_v$ corresponds to $2/3$.  For ease of notation, let us set
\[   a = \# S_{\eta}', \quad  b_1 = \# S_{\eta,1}'' \quad \text{and} \quad b_2 = \# S_{\eta, 2}''. \]
\vskip 5pt

Considering the pullback of $\eta_v$ to $S_{\psi}$, we have:
\[  \eta_v|_{S_{\psi} }= \begin{cases}
1 \text{  if $v \notin S_{\eta}$;} \\
\eta_v, \text{  if $v \in S_{\eta}''$;} \\
\text{the sum of the two nontrivial characters of $\mu_3$, if $v \in S_{\eta}'$.} \end{cases} \]
 Hence,
\[  m_{\eta} =  \frac{1}{3} \cdot \left( 2^a +  (-1)^a \cdot \zeta^{b_1 - b_2}  + (-1)^a \cdot \zeta^{b_2 - b_1} \right) \]
where $\zeta \in \C^{\times}$ is a primitive cube root of $1$. To further explicate the above formula, we have
\[  m_{\eta} = \begin{cases} 
(2^a + 2(-1)^a)/3,  \text{  if $b_1 - b_2 = 0 \mod 3$;} \\
(2^a + (-1)^{a+1} )/3, \text{  if $b_1-b_2 \ne 0 \mod 3$.} \end{cases} \]
In particular, if $S_{\eta}$ is empty (so that $a = b_1 = b_2 = 0$), we see that $m_{\eta} = 1$. 
\vskip 5pt

We now enumerate the set $\Sigma_{E,K,[\chi], \eta}$. Any $C \in \Sigma_{E,K,[\chi], \eta}$ corresponds to $E \hookrightarrow B^+$ for a central simple $F$-algebra $B$ ramified precisely at $S_{\eta}$. Assume first that $S_{\eta}$ is nonempty, so that $H_C(F) = H_C^0(F)$ for any $C \in \Sigma_{E,K,[\chi], \eta}$.  To check if $\rho_{\eta, C}$ is automorphic, we need to veify that, for one of the two choices of $\iota_C$, we have 
\[   \rho_{\eta_v} \circ \iota_v |_{H_C^0(F_v)} \supset  \chi_v \circ \iota_{C,v} \quad \text{for all places $v$.} \] 
Now 
\[  \rho_{\eta_v}|_{H_{C_{\eta_v}}^0(F_v)}  = \begin{cases}
 \chi_v \circ \iota_{\eta_v} + \chi_v^{-1} \circ \iota_{\eta_v}, \text{ if $v \notin S_{\eta}$ and $\chi_v^2 \ne 1$} \\
 \chi_v \circ \iota_{\eta_v}, \text{ otherwise.} \end{cases} \]
 So the desired containment holds at all places outside $S_{\eta}''$. 
\vskip 5pt

It remains to consider the places $v \in S_{\eta}''$. For such a $v$, we need to verify if
\[  \chi_v \circ \iota_{\eta_v} \circ \iota_v  = \chi_v \circ \iota_{C,v}. \]
This holds if and only if 
\[  \iota_{C,v}  = \iota_{\eta_v} \circ \iota_v. \]
In other words, if $\iota_C$ is associated to the associative algebra embedding $E \hookrightarrow B$, then the invariants of $B$ at $v \in S_{\eta}''$ are constrained by $\eta_v$ as follows:
\[  {\mathrm inv}(B_v) = \begin{cases}
1/3, \text{  if $v \in S_{\eta, 1}''$;} \\
2/3, \text{  if $v \in S_{\eta, 2}''$;} \\
\pm 1/3, \text{  if $v \in S_{\eta}'$;} \\
0, \text{  otherwise.} \end{cases} \]
We leave it as an amusing exercise to verify that the number of $B$'s satisfying these requirements  is equal to $m_{\eta}$ (with $m_{\eta}$ computed above). 
It follows that  
\[  \# \Sigma_{E,K, [\chi], \eta} = \begin{cases}
m_{\eta}, \text{  if $S_{\eta}''$ is nonempty;} \\
m_{\eta}/2, \text{  if $S_{\eta}''$ is empty.} \end{cases} \]
However, from the discussion in \S \ref{SS:auto-HC}, we have:
\[  m(\rho_{\eta}) = \begin{cases}
1, \text{  if $S_{\eta}''$ is nonempty;} \\
2, \text{  if $S_{\eta}''$ is empty.} \end{cases} \] 
Taken together, we thus conclude that, when $S_{\eta}$ is nonempty, 
\[  m(\rho_{\eta}) \cdot \# \Sigma_{E,K, [\chi], \eta} = m_{\eta}, \]
as desired. 
\vskip 5pt

Now consider the case when $S_{\eta}$ is empty. In this case, the only possible $C \in \Sigma_{E,K, [\chi], \eta}$ is $C^+$ corresponding to $E \hookrightarrow M_3(F)$, and $H_{C^+}(F)  \ne H_{C^+}^0(F)$. By our discussion in \S \ref{SS:auto-HC}, we see easily that $\rho_{\eta,C^+}$ is automorphic with $m_{C^+}(\rho_{\eta, C^+}) =1$. Hence 
\[m(\rho_{\eta}) \cdot   \# \Sigma_{E,K, [\chi], \eta} = 1 = m_{\eta} \] 
as desired. 
\vskip 10pt

\subsection{\bf $K$ is split and $\chi^2 =1$}
In this case, 
\[  S_{\psi} \cong  S_3 \]
and we fix an element $s_0$ in $S_3 \setminus \mu_3$, so that $S_{\psi} = \mu_3 \rtimes S_2$ and $\bar{S}_{\psi} = \langle s_0 \rangle$.  For all places $v$, we then have
$S_{\psi_v} = \pi_0(S^{W_{F_v}}) \rtimes \mu_2$.  
\vskip 5pt

Given an $\eta$, let $S_{\eta}$ be the finite set of places where $C_{\eta_v}$ is associated to a cubic division algebra. Then for $v \in S_{\eta}$, $S_{\psi_v} = S_3$   and $\eta_v$ is the 2-dimensional irreducible representation $r$ of $S_3$. For all other $v$, $\eta_v$ is 1-dimensional.  
On pulling back to $S_{\psi} = S_3$, we have
\[  \eta_v|_{S_{\psi}} = \begin{cases}
1, \text{  if $v \notin S_{\eta}$ and $\eta_v(s_0) =1$;} \\
\epsilon, \text{  if $v \notin S_{\eta}$ and $\eta_v(s_0) = -1$;} \\
r, \text{  if $v \in S_{\eta}$.} \end{cases} \]
Hence,
\[  m_{\eta} = \frac{1}{6} \cdot \left( 2^{\# S_{\eta}} + 2 \cdot (-1)^{\# S_{\eta}} \right)  \quad \text{if $S_{\eta}$ is nonempty,} \]
 and if $S_{\eta}$ is empty,
\[  m_{\eta} = \frac{1}{2} \cdot \left( 1 + (-1)^b \right) = \begin{cases} 
1, \text{  if $b$ is even;} \\
0, \text{  if $b$ is odd.} \end{cases}
  \]
where $b$ is the cardinality of the set of places $v$ where $\eta_v(s_0) = -1$.  
\vskip 10pt

We now consider the set $\Sigma_{E,K,[\chi], \eta}$. For $C \in \Sigma_{E,K, [\chi],\eta}$, associated to $E \hookrightarrow B^+$ say, we see that $B$ is ramified precisely at $S_{\eta}$. We know that 
\[  \# \{ B \in Br_3(F): \text{$B$ is ramified precisely at $S_{\eta}$} \} =  \frac{1}{3} \cdot \left( 2^{\# S_{\eta}} + 2 \cdot (-1)^{\# S_{\eta}} \right)  \]
 if $S_{\eta}$ is nonempty, and is 1 if $S_{\eta}$ is empty. 
\vskip 5pt

 Assume first that $S_{\eta}$ is nonempty, so that $H_C(F) = H_C^0(F)$ for any $C \in \Sigma_{E,K, [\chi], \eta}$. Then for $\rho_{\eta,C}$ to be automorphic, we need
\[  \rho_{\eta_v} \circ \iota_v |_{H_C(F_v)} \supset \chi_v \circ \iota_{C,v}  \quad \text{  for all $v$} \]
for one of the two choices of $\iota_C$. Now
\[ \rho_{\eta_v}|_{H^0_{C_{\eta_v}}(F_v)} = \chi_v \circ \iota_{\eta_v}, \]
so that $\rho_{\eta,C}$ is automorphic if and only if
\[  \chi_v \circ \iota_{\eta_v} \circ \iota_v = \chi_v \circ \iota_{C,v} \quad \text{for all $v$.} \]
By Lemma \ref{L:commu}, this holds automatically since $\chi^2=1$. Hence $\rho_{\eta,C}$ is always automorphic, with $m_C(\rho_{\eta,C}) =1$ (by the discussion in \S \ref{SS:auto-HC}),  and
\[
  \# \Sigma_{E, K, [\chi], \eta}  =
  = \frac{1}{6} \cdot \left( 2^{\# S_{\eta}} + 2 \cdot (-1)^{\# S_{\eta}} \right) 
  = m_{\eta} \] 
  as desired.
 \vskip 5pt
 On the other hand, if $S_{\eta}$ is empty, then the only possible $C \in \Sigma_{E,K, [\chi], \eta}$ is  $C^+$ corresponding to $E \hookrightarrow M_3(F)$.  This is treated in exactly the same way as the corresponding case when $K$ is a field, using the global Poitou-Tate duality summarized in Proposition \ref{P:PT}. We omit the details.
\vskip 10pt

To summarize, we have shown the following result which is one of the main  global theorems of this paper:
\vskip 5pt

\begin{thm}  \label{T:final1}
Let $\psi = \psi_{E,K., [\chi]}$ be a given global A-parameter of $G_E$ over a number field $F$. Let $\eta \in \mathrm{Irr}(S_{\psi, \A})$ be an irreducible character of 
its adelic component group with associated representation $\pi_{\eta}$ in the global A-packet $\Pi_{\psi}$. Then the multiplicity of $\pi_{\eta}$ in the submodule $V[\psi] \subseteq \mathcal{A}_{2,\psi}(G_E)$ is equal to 
\[  m_{\eta} = \dim \Hom_{S_{\psi}} (\eta \circ \Delta, \C). \]
\end{thm}

\vskip 10pt

\subsection{\bf  Main global theorem}
If $m_{disc}(\pi)$ denotes the multiplicity of an irreducible representation $\pi$ in the automorphic discrete spectrum $\mathcal{A}_2(G_E)$, then
the last theorem shows that 
\[   m_{disc}(\pi_{\eta}) \geq m_{\eta}  \quad \text{  for any $\eta \in \mathrm{Irr}(S_{\psi, \A})$.} \]
In this final subsection, we shall show the reverse inequality and hence strengthen this inequality to an equality. 
\vskip 5pt

The argument is analogous to that for the cubic unipotent A-packets of $G_2$ given in \cite{G}.
The proof will require two ingredients: one local and the other global in nature. We begin by describing these two ingredients. Hence, we fix a global A-parameter $\psi = \psi_{E,K,[\chi]}$ and $\eta = \otimes_v \eta_v \in \mathrm{Irr}(S_{\psi, \A})$, so that $\pi_{\eta} \cong \otimes_v \pi_{\eta_v} = \otimes_v \Theta^{abs}_{C_{\eta_v}}(\rho_{\eta_v})$.
\vskip 5pt

\begin{itemize}
\item (Local) For each place $v$ of $F$, and for each nondegenerate $E_v$-twisted Bhargava cube $\Sigma_v$ with associated character $\psi_{\Sigma_v}$ of $N_{E_v}(F_v)$, we have
\begin{equation} \label{E:local-input}
   \Hom_{N_{E_v}(F_v)} ( \pi_{\eta_v},  \psi_{\Sigma_v}) \cong 
   \begin{cases}
\rho_{\eta_v} \cdot \mu_{K_v}, \text{  if $C_{\Sigma_v} \cong C_{\eta_v}$;} \\
 0, \text{ otherwise,} \end{cases} \end{equation}
as a module for the stabilizer $M_{E_v, \Sigma_v}(F_v)$ of $\Sigma_v$. Here, $\mu_{K_v}$ is either the trivial character or the sign character of $M_{E_v, \Sigma_v}(F_v) \cong H_{C_{\eta_v}}(F_v)$ depending on whether $\omega_{K_v/F_v}(-1) = +1$ or $-1$.
\vskip 5pt

This result is Proposition \ref{P:twisted} in the nonarchimedean case. For archimedean $v$, note that the Hom space here refers to the space of continuous linear functionals of  $\pi_{\eta_v}$ (as a Casselman-Wallach representation).  The result for archimedean $v$ will be shown in a paper with J. Adams and A. Paul, where we studied the archimedean theta correspondence and prove the results  in \S \ref{S:arch-theta}. 
\vskip 5pt

\item (Global) Let 
\[   \Sigma_{E, K,\eta} = \{ C \in \Sigma_{E,K}: C_v \cong C_{\eta_v} \,\,\text{for all places $v$} \}. \]
For any embedding $f: \pi_{\eta} \hookrightarrow \mathcal{A}(G_E)$, there exists $C \in \Sigma_{E,K, \eta}$ such that the $\psi_C$-Fourier coefficient of $f(\pi_{\eta})$ is nonzero. 
We shall show this as a consequence of Proposition \ref{P:global-input} and Corollary \ref{C:global-input} below.
\end{itemize}
\vskip 5pt

Taking these two ingredients for granted, we proceed to show the reverse inequality. By the consideration of Fourier coefficients, we have a natural map
\[  \Hom_{G_E(\A)}( \pi_{\eta}, \mathcal{A}(G_E)) \longrightarrow \bigoplus_{C \in \Sigma_{E,K, \eta}} \Hom_{N_E(\A)}(\pi_{\eta}, \psi_C)^{M_{E, \psi_C}(F)} \]
The global ingredient shows that this map is injective, so that one has an upper bound
\[  m_{disc}( \pi_{\eta}) \leq \sum_{C \in \Sigma_{E,K, \eta}} \dim \Hom_{N_E(\A)}(\pi_{\eta}, \psi_C)^{M_{E, \psi_C}(F)}. \]
Here, we have used the fact that $\prod_v \mu_{K_v}$ is an automorphic character and hence is trivial on $H_C(F)$.
The local ingredient, on the other hand, shows that for each $C$, 
\[ \dim  \Hom_{N_E(\A)}(\pi_{\eta}, \psi_C)^{M_{E, \psi_C}(F)}  = \dim  \rho_{\eta}^{H_C(F)}  = \dim \Hom_{H_C(F)}(\rho_{\eta}^{\vee},\C). \]
The latter dimension is simply the automorphic multiplicity of $\rho_{\eta}^{\vee}$ in $\mathcal{A}(H_C)$. We have seen that this automorphic multiplicity is independent of $C \in \Sigma_{E,K,\eta}$ and have denoted it by $m(\rho_{\eta}^{\vee}) = m(\rho_{\eta})$.
 Hence, we obtain
 \[  m_{disc}( \pi_{\eta}) \leq m(\rho^{\vee}_{\eta}) \cdot \# \Sigma_{E,K,\eta} =  m_{\eta},  \]
 where the second equality is precisely what we showed when we verified the Arthur multiplicity formula for the space of global theta liftings.
 Summarizing, we have the following theorem which strengthens Theorem \ref{T:final1} and which is the main global theorem of this paper.
 \vskip 5pt
 
 \begin{thm}  \label{T:final2}
 Let $\psi = \psi_{E,K., [\chi]}$ be a given global A-parameter of $G_E$ over a number field $F$. Let $\eta \in \mathrm{Irr}(S_{\psi, \A})$ be an irreducible character of 
its adelic component group with associated representation $\pi_{\eta}$ in the global A-packet $\Pi_{\psi}$. Then
\[  m_{disc}(\pi_{\eta})   = \dim \Hom_{S_{\psi}} (\eta \circ \Delta, \C). \]
 \end{thm}
\vskip 5pt

It remains to establish the global ingredient above.
 For this, we recall the following notion from \cite{GS1}: when
$F_v=\mathbb R$ or $\mathbb C$, we say that a representation $\pi_v$ of $G_E(F_v)$ is weakly minimal  if the associated variety of its annihilator  in the universal enveloping algebra is the minimal nilpotent orbit.    Now we note:
\vskip 5pt

\begin{prop} \label{P:global-input}
Let $\pi \neq \mathbb C$ be an irreducible automorphic subrepresentation of $G_E$ such that $\pi_v$ is not weakly minimal for at least one 
archimedean place $v$. Then there exists a nondegenerate cube $C\in V_E(F)$ and $f\in \pi$ such that $f_{\overline N_E, \psi_C} \neq 0$.  
\end{prop}  
\vskip 5pt

\begin{proof} Let $f\in \pi$ and consider the Fourier expansion of the constant term $f_{\overline Z}$ along $\overline{V}_E = \overline N_E/\overline Z$. 
If this expansion is supported on cubes of rank one, then $\pi$ is weakly 
minimal in the sense of Definition 4.6 in \cite{GS1}. Then, by  \cite[Thm. 5.4]{GS1}, $\pi_v$ is weakly minimal at all archimedean places, which contradicts 
our assumption. Moreover, since $E$ is a field,  $V_E(F)$ has no rank 2 elements. 
Thus, $f_{\overline Z}$ has a non-trivial Fourier coefficient for a cube $C'$ of rank 3 or 4. 
\vskip 5pt

 If $C'$ is rank 3, then  by Proposition \ref{P:rank2or3}, we can assume that $C'=(0,0, e, 0)$ with $e\in E^{\times}$, 
    Let $U_E$ be the  
 unipotent radical of the 3-step maximal parabolic subgroup $Q_E$ in $G_E$, with $N_E$ and $U_E$ in standard position, such that 
  $\psi_{C'}$ restricts to  a non-trivial character of $[\overline U_E, \overline U_E]$. The character of $[\overline U_E, \overline U_E]$ thus obtained  is associated to an 
 $sl_2$-triple corresponding to the non-special nilpotent orbit $3A_1$ (see the introduction to \cite{JLS}). 
 By \cite[Cor. 6.6]{JLS} (the conditions of Lemma 4.3 there are satisfied since the orbit $3A_1$ is not special) 
 there exists $x\in F^{\times}$ such that , with $C=(x,0,e,0)$,  $f_{\overline N_E, \psi_C} \neq 0$
 for some $f\in \pi$. This proves the proposition.   
 \end{proof}  
 
\vskip 5pt

\begin{cor}  \label{C:global-input}
For any embedding $f: \pi_{\eta} \hookrightarrow \mathcal{A}(G_E)$, there exists $C \in \Sigma_{E,K, \eta}$ such that the $\psi_C$-Fourier coefficient of $f(\pi_{\eta})$ is nonzero. 
\end{cor}

\begin{proof}
By the local ingredient (\ref{E:local-input}), we see that the only possible nonzero nondegenerate Fourier coefficients supported by $f(\pi_{\eta})$ correspond to the finitely many $C \in \Sigma_{E,K,\eta}$. Hence the corollary follows from Proposition  \ref{P:global-input}. 
\end{proof}

\vskip 10pt

\section{\bf Appendix A: A theta correspondence for $E_7$}

In this section, we consider a dual pair $G_E \times H_C$ in the split adjoint group of type  $E_7$, where $H_C=\Aut_E(C)$ for a 4-dimensional  $E$-twisted composition algebra $C$.  This theta correspondence (and its version for inner forms) can be used to construct the A-packets corresponding to a root $\SL_2$, as we discussed briefly  in \S \ref{SS:rootSL2}. We will not launch into this detailed study in this paper. The main purpose of this appendix is simply to
 compute the theta lift of the trivial representation of $H_C=\SL_2(E)/\mu_2$; this result is needed in our paper \cite{GS3}.
\vskip 5pt

\vskip 5pt

\subsection{Twisted composition}  Assume that $B$ is  a composition algebra over  $F$. 
 Let $N(x)=x\bar x$ and $Tr(x)=x+\bar x$, be the norm and the trace on $B$. 
Then $C_B=B\oplus B \oplus B$ has a structure of an $F^3$-twisted composition algebra, given by 
\[ 
Q(x_1, x_2, x_3)= (N(x_1), N(x_2), N(x_3) ) 
\] 
\[ 
\beta(x_1, x_2, x_3)= (\bar x_2\bar x_3,  \bar x_3 \bar x_1, \bar x_1 \bar x_2)
\] 
\[ 
N_C(x_1, x_2, x_3) = Tr(x_3 x_2 x_1). 
\] 
The symmetric group $S_3$ acts on $C_B$ as $F$-automorphisms by permuting the three summands of $C_B$, with the action of odd permutations twisted by the map 
$(x_1,x_2,x_3) \mapsto (\bar x_1,\bar x_2,\bar x_3)$.  Let $E$ be a cubic etal\'e algebra over $F$. Since $\Aut(E/F)$ is isomorphic to a subgroup of $S_3$,  by fixing an embedding of $\Aut(E/F)$ into $S_3$, we obtain an $E$-twisted composition algebra  $C_B^E$ by Galois descent. 
\vskip 5pt

We shall now describe the group $\Aut(C_B^E)$ of automorphisms of $C_B^E$ for $B=M_2(F)$. In this case $\bar x$ is defined as the adjoint of the matrix $x$, 
\[ 
\bar x= \left(\begin{array}{rr} d & -b \\
-c & a 
\end{array} \right)  \text{ if } 
x= \left(\begin{array}{cc}  a & b \\
c & d
\end{array} \right) . 
\] 
Assume first that $E=F^3$.  Let 
\[ 
\GL_2(F^3)^{\det}=\{ (g_1, g_2, g_3) | \det g_1 = \det g_2 = \det g_3 \} 
\] 
This group acts on $C_B=B\oplus B \oplus B$ by 
\[ 
g(x_1, x_2, x_3)= (g_3 x_1 g_2^{-1}, g_1 x_2 g_3^{-1}, g_2 x_3 g_1^{-1} ).
\]   
It is fairly straightforward to check  that this action preserves $Q$ and $\beta$. An element $g$ acts trivially if and only if it belongs to 
$\Delta F^{\times}$. 
 The group $\GL_2(F^3)^{\det}/ \Delta F^{\times}$ is the group of $F$-points of the algebraic group $\SL_2(F^3)/\mu_2$.   
The action of $S_3$  on $C_B$ normalizes that of $\SL_2(F^3)/\mu_2$, on which it acts by permuting the 3 factors. 
 Hence, for a general cubic etal\'e algebra  $E$ over $F$, the group of $F$-automorphisms of $C_B^E$ (with $B=M_2(F)$)  is  
\[ 
\Aut_F(C_B^E) \cong  \SL_2(E)/\mu_2 \rtimes S_E,
\] 
and the group of $E$-automorphisms is its identity component
\[  \Aut_E(C_B^E) \cong \SL_2(E)/\mu_2.\]
Since 
\[ 
H^1(F, \SL_2(E)/\mu_2) \cong H^2(F, \mu_2)=Br_2(F)
\] 
we see that the $E$-isomorphism classes of $E$-twisted composition algebras $C$ of $E$-dimension 4 correspond to isomorphism classes of quaternion algebras. In particular, 
as $B$ varies over quaternion $F$-algebras, the algebras $C_B^E$ exhaust all $E$-isomorphism classes of $E$-twisted composition algebras of $E$-dimension 4.  
\vskip 5pt

Via the Springer decomposition,  we may connect the above discussion with the theory of Freudenthal-Jordan algebras of dimension $15$.
The split Jordan algebra of dimension $15$ is $J_s=F^3 \oplus C_{M_2(F)}$ and its automorphism group  is $\mathrm{PGSp}_6 =\Sp_6(F)/\mu_2$.  Since 
\[ 
H^1(F, \Sp_6(F)/\mu_2) \cong H^2(F, \mu_2)=Br_2(F), 
\] 
we see that the isomorphism classes of Freudenthal Jordan algebras of dimension 15 are parametrized by isomorphism classes of quaternion algebras as well. 
If $J$ is a form of $J_s$, let $[J]\in Br_2(F)$ denote  the 
corresponding Brauer class. Similarly, 
for $B \in Br_2(F)$, let $J_B$ be the corresponding Freudenthal-Jordan algebra.
  It is clear that $[J]=B$ if $J=E \oplus C_B$.

\vskip 10pt

\subsection{Some embedding problems} 

Let $C_B$ be an $E$-twisted composition algebra of $E$-dimension $4$. Every element $x$ in $C_B$ satisfies the quadratic equation 
\[ 
\beta^2(x)+ Q(x) \beta(x) - N_{C_B}(x) x = 0.
\] 
If we fix $e=Q(x)$ and $d=N_{C_B}(x)$, such that the cube $\Sigma=(1,0,-e,-d)$ is non-degenerate, then $x$ and $\beta(x)$ span an $E$-twisted subalgebra of $E$-dimension $2$, 
corresponding to the cube $\Sigma$.  Thus, in order to understand embeddings of the $E$-twisted composition algebras of $E$-dimension $2$ into $C_B$, it suffices to understand solutions  of the above equation.

\begin{prop} \label{P:orbit2} Assume  that $E=F^3$ and consider  $C_B$ with  $B=M_2(F)$. 
The  group $\Aut_E(C_B)=\GL_2(F^3)^{\det}/\Delta F^{\times}$ acts transitively on the set of elements $x\in C_B$ such that $Q(x)=0$, and $N_C(x)=1$. The stabilizer $\Stab_{\Aut_E(C_B)}(x_0)$ of 
\[ 
x_0=\left(  \left(\begin{smallmatrix}1 & 0 \\
0 & 0 \end{smallmatrix}\right), 
 \left(\begin{smallmatrix}1 & 0 \\
0 & 0 \end{smallmatrix}\right), 
  \left(\begin{smallmatrix}1 & 0 \\
0 & 0 \end{smallmatrix}\right) 
 \right)
\] 
is the quotient by $\Delta F^{\times}$ of the subgroup of $\GL_2(F^3)$ consisting of elements 
\[ 
\left(  \left(\begin{smallmatrix}a & 0 \\
0 & d \end{smallmatrix}\right), 
 \left(\begin{smallmatrix}a & 0 \\
0 & d \end{smallmatrix}\right), 
  \left(\begin{smallmatrix}a & 0 \\
0 & d \end{smallmatrix}\right) 
 \right). 
\] 
In particular, $\Stab_{\Aut_E(C_B)}(x_0)\cong F^{\times}$. 
The stabilizer of $x_0$ in $\Aut_F(C_B)=\Aut_E(C_B) \rtimes S_3$ is a semi-direct product   $F^{\times}\rtimes S'_3$, where $S'_3$ is a ``quadratic twist'' of $S_3$: we multiply any transposition in $S_3$ by 
\[ 
w=\left(  \left(\begin{smallmatrix}0 & 1 \\
1 & 0 \end{smallmatrix}\right), 
 \left(\begin{smallmatrix}0 & 1 \\
1 & 0 \end{smallmatrix}\right), 
  \left(\begin{smallmatrix}0 & 1 \\
1 & 0 \end{smallmatrix}\right) 
 \right). 
\] 
\end{prop} 
\begin{proof} 
Let $x=(x_1,x_2,x_3)\in C_B$ such that $Q(x)=0$, and $N_{C_B}(x)=1$. We want to show that $x$ is conjugated to $x_0$ by an element in $\GL_2(F^3)^{\det}$. 
Since $Q(x)=0$, we have $\det x_i =0$ for all $i$. Hence, we can write 
\[ 
x_1 = v_3 \cdot w_2^{\top} , x_2 = v_1 \cdot w_3^{\top} , x_3 = v_2 \cdot w_1^{\top} 
\] 
for some column vectors $v_i$ and $w_i$.  Note that 
\[ 
N_{C_B}(x) =Tr(x_3 x_2 x_1)=(w_1 \cdot v_1^{\top} )\cdot  (w_2^{\top} \cdot v_2) \cdot (w_3^{\top} \cdot  v_3)= 1. 
\] 
Hence, all vectors are non-zero, and we can pick $g_1, g_2, g_3\in SL_2(F)$, so that $g_i (v_i)= (1,0)^{\top}$ for all $i$. Thus, we can assume 
that $v_1= (1,0)^{\top}$ for all $i$. Since $(w_2 \cdot v_2^{\top} )\neq 0$, $w_i=(a_i,b_i)$ with $a_i\neq 0$. Hence, using the unipotent $g_i$ stabilizing 
$(1,0)^{\top}$, we can arrange all $b_i=0$. Thus $x$ is conjugate to  
\[ 
\left(  \left(\begin{smallmatrix}a_2 & 0 \\
0 & 0 \end{smallmatrix}\right), 
 \left(\begin{smallmatrix}a_3 & 0 \\
0 & 0 \end{smallmatrix}\right), 
  \left(\begin{smallmatrix}a_1& 0 \\
0 & 0 \end{smallmatrix}\right) 
 \right)
\] 
such that $a_1 a_2 a_3=1$.  But this element is conjugated to $x_0$ by a triple of diagonal matrices. 

The stabilizers can be computed directly. 
\end{proof} 

Let $C_B$ be an $E$-twisted composition algebra of $E$-dimension 4. For a nondegenerate $E$-twisted cube $\Sigma=(1,0-f,-b)$, consider the set
\[ 
\Omega_{\Sigma}=\{ v\in C ~|~ Q(v)=f, N_C(v)=b\} . 
\] 
Recall that to  $\Sigma$, we attach an $E$-twisted algebra $C_{\Sigma}$ of $E$-dimension $2$, equipped with a reduced basis $\{v, \beta(v)\}$. 
Any element $x\in \Omega_{\Sigma}(F)$ defines an $E$-embedding of $C_{\Sigma}$ into $C_B$,  where $v$ is sent to $x$. 
Hence $\Omega_{\Sigma}$ is in bijection with the set of embeddings $C_{\Sigma} \hookrightarrow C_B$. 
\vskip 5pt

\begin{cor} \label{C:orbit} Assume that $F$ is a local field, and $C_B$ is an $E$-twisted composition algebra of $E$-dimension 4. 
If $\Omega_{\Sigma}(F)$ is nonempty, then $\Aut_E(C_B)$ acts transitively on $\Omega_{\Sigma}(F)$.  
\end{cor} 
\begin{proof} 
Fix a point $v_0 \in \Omega_{\Sigma}(F)$. 
By Proposition \ref{P:orbit2},  $\GL_2(E)^{\det}$ acts transitively on $\Omega_{\Sigma}(\overline{F})$ (through its quotient $\GL_2(E)^{\det}/F^{\times} = H_{C_B}(F)$) and the stabilizer of $v_0 \in \Omega_{\Sigma}(F)$ is a maximal torus  $T_{v_0}$ in $\GL_2(F) \subset \GL_2(E)^{\det}$. Hence the $F$-rational orbits under $\GL_2(E)^{\det}$ is parametrized by $H^1(F, T_{v_0})$, which is trivial since $T_{v_0}  = {\Res}_{K/F} \mathbb{G}_m$ for some quadratic \'etale algebra $K$ over $F$.
The corollary follows.  
\end{proof} 

Next, we need to understand when $\Omega_{\Sigma}(F)$ is nonempty:

\begin{prop} \label{P:embeddings}  Let $B=M_2(F)$. Let $C$ be an $E$-twisted composition algebra of $E$-dimension $2$. Then $C$ embeds into $C_B^E$ if and only if 
$J=E\oplus C$ is not a division algebra.   
\end{prop} 
\begin{proof} If $J=E\oplus C$ is not a division algebra,  then by  \cite[Thm 38.8]{KMRT},  $J=J_3(K)$  for a quadratic \'etale algebra $K$ over $F$. Since $K$ embeds into $B= M_2(F)$, 
we deduce that  $J$ embeds into $J_3(B)$ and hence $C$ into $C_B^E$, where $J_3(B) = E \oplus C_B^E$. 
\vskip 5pt

Now assume that $J=E\oplus C$ is a division algebra. By tensoring with $K_E$ if necessary,
 we can assume without loss of generality  that $E$ is a cyclic field, with the Galois group generated by $\sigma$ of order 3. Then $C_B^E=M_2(E)$, and we have
\[  Q(x)= \det(x), \quad    \beta(x)=\bar  x^{\sigma} \bar x^{\sigma^2}, \quad \text{and} \quad 
 N_{C_B^E}(x) = Tr(x^{\sigma^2} x^{\sigma} x). 
\] 
  On the other hand,  there exists 
 $\lambda \in F^{\times}$ such that  $C \cong C(\lambda)=E \oplus E$, with 
 \[ 
 Q(a,b)= ab, \quad    \beta(a,b) = (\lambda^{-1} b^{\#}, \lambda a^{\#}) \quad \text{and} \quad  N_{C}(a,b) = \lambda N_E(a) + \lambda^{-1} N_E(b). 
\] 
 Moreover, since $E\oplus C$ is a division algebra,  $\lambda \notin N_{E/F}(E^{\times})$.
 \vskip 5pt
 
 Assume, for the sake of contradiction, that $C(\lambda)$ embeds into $C_B^E$. Let $x$ be the image of $(1,0)$.  
Since $Q_{C_B^E}(x)=Q_{C(\lambda)}(1,0)=0$, the determinant of $x$ is 0. 
Hence $x=v\cdot w^{\top}$ for two $2\times 1$ column vectors $v$ and $w$, with coefficients in $E$. One checks that 
\[ 
N_{C_B^E}(x)= N_{E/F}(w^{\top} \cdot v^{\sigma}). 
\] 
This implies that $N_{C(\lambda)}(1,0) = \lambda= N_{C_B^E}(x)$ is the norm of an element in $E^{\times}$, a contradiction. 

\end{proof}

\vskip 10pt

\subsection{$D_4$ geometry} 

Now let $O$ be the $8$-dimensional composition algebra of split octonions. The automorphisms group of $C_O$ is a semi-direct product of   the split simply connected group $G$ of type $D_4$ with  $S_3$. 
We remind the reader that $S_3$ acts on $C_O=O\oplus O \oplus O$ is by permuting the three summands of $C_O$, with a twist by the map 
$(x_1,x_2,x_3) \mapsto (\bar x_1,\bar x_2,\bar x_3)$ for odd permutations.  Tits \cite{Ti} has given a beautiful description of the  flag varieties for $G$ in terms of geometry of $C_O$. We follow the exposition of Weissman \cite{We}. 
\vskip 5pt
\
Fix a triple $(i,j,k)$ of integers $0\leq i,j,k \leq 2$. Let $\mathcal F_{ijk}$ be  the set of  subspaces 
\[ 
X\oplus Y \oplus Z\subseteq C_O 
\]  where $X,Y,Z $ are subspaces of $O$ of dimensions $i,j,k$, respectively, such that $N_O(X)=N_O(Y)=N_O(Z)=0$ and $XY=YZ=ZX=0$. 
Then $\mathcal F_{ijk}$ is a flag variety for $G$ with respect to a  parabolic $P=MN$, as  indicated in the following table, where $\Delta_M$ is the 
subset of simple roots ``contained'' in $M$. 
\[ 
\begin{array}{|c|c|}
 \hline 
i,j,k & \Delta_M  \\ 
\hline 
0,0,0 &  \{\alpha_0,\alpha_1, \alpha_2,\alpha_3\} \\ 
\hline 
1,0,0 & \{\alpha_0, \alpha_2,\alpha_3\} \\
\hline 
1,1,0 &  \{\alpha_0,\alpha_3\} \\
\hline 
1,1,1 &  \{\alpha_0\} \\ 
\hline 
2,1,1 &  \{\alpha_1\}  \\
\hline 
2,2,1 &  \{\alpha_1, \alpha_2\} \\
\hline 
2,2,2 &  \{\alpha_1, \alpha_2,\alpha_3\} \\ 
\hline \end{array} 
\]

Consider now, $C_O^E$,  the $E$-twisted version of $C_O$.  
 As we noted in \S \ref{SS:4and8},   
\[   G_E=\Aut_E(C_O^E). \]
 For $i=1$ or $2$, we define $\mathcal F_i$ to be the set of 
$E$-subspaces $V_i \subseteq C^E_O$ of dimension $i$ such that $V_i \otimes {\bar F}$ belongs to $\mathcal F_{iii}$ for $C^E_O \otimes {\bar F}\cong C_{O\otimes {\bar F}}$. 
A pair $V_1\subset V_2$ is a full flag if $E$ is a field.  
Let $P_i=M_i N_i$ be the stabilizer of $V_i$ in $G_E$.  Then 
\[ \text{$M_1^{\der}\cong \SL_2(F)$ (long root),} \quad  \text{$M_1/M_1^{\der}\cong \GL(V_1)=E^{\times}$} \]
and 
\[ \text{ $M_2^{\der}\cong \SL_2(E)$ (short root),} \quad \text{$M_2 \cong \GL(V_2)^{\det}$.} \]
 These claims can be easily checked over $\bar F$.  The modular characters are 
  \[ 
  \rho_{U_1}=|N_E|^3 \text{ and } \rho_{U_2}=|\det|^5 
  \] 
  We have degenerate principal series $J(s)$ and $I(s)$ corresponding to $P_1$ and $P_2$, respectively.

\vskip 10pt

\subsection{Dual pair} 
Now let $F$ be a nonarchimedean local field and $E$ a cubic etal\'e algebra over $F$. Let $C_B$ be the $E$-twisted  composition algebra of dimension $4$  associated to $B = M_2(F)$, with corresponding Springer decomposition $J_B = E \oplus C_B$.  By our discussion in \S \ref{S:dual}, this data gives rise to a dual pair
\[  G_E \times H_{C_B} \longrightarrow G_B := G_{J_B}  \]
where $G_E = {\mathrm Spin}_8^E$, $H_{C_B} = \Aut_E(C_B) \cong \SL_2(E)/\mu_2$ and $G_B$
is the split adjoint group of type $E_7$.  Our goal is to determine the theta lift  $\Theta(1)$, where $1$ is the trivial representation of $H_{C_B}(F)$.

\vskip 5pt
 
For this purpose, it will be more convenient to work with an alternative construction or description of the above dual pair which is adapted to the Siegel maximal parabolic subgroup in $G_B$ and which makes use of the interpretation of $G_E$ as the automorphism group of an $8$-dimensional $E$-twisted composition algebra. 
We give this alternative description next. 
\vskip 5pt

Let $S\cong \mathbb G_m$ be a maximal split torus in $\SL_2(F)/\mu_2 \subset H_C$. The torus $S$ gives a short 
$\Z$-grading of $\mathfrak g_B$ and $\mathfrak h_E$: 
\[ 
\mathfrak g_B= \bar{\mathfrak n}  \oplus  \mathfrak m \oplus \mathfrak n \text{ and } 
\mathfrak h_E= \bar{\mathfrak u}  \oplus  \mathfrak l \oplus \mathfrak u
\] 

Let $P=MN$ and $Q=LU$ be the corresponding maximal parabolic subgroups in $G_B$ and $H_C$ respectively.  
The unipotent radical $N$ is commutative and   can be identified with an exceptional Jordan algebra $J$. The Levi subgroup $M$ can be identified with with the similitude group of  the cubic  form $N_J$, with corresponding similitude character
\[ 
i_J : M \rightarrow F^{\times}. 
\] 
Now  the group $G_E$ is contained in $M$ and $J$, and  under its adjoint action on $N$, one has the decomposition  $N = J=E\oplus {C^E_O}$ 
where $C_O^E$ is the $E$-twisted composition algebra of $E$-dimension 8. 
\vskip 5pt

 Note that  the $M$-module $\bar N$ is  dual to $N$ and hence can be identified with $J^*$. Since $J^*$ is identified with $J$ using the trace form $T_J$,  we can identify both $N$ and $\bar N$ with $J$.   Under this identification, both $U$ and $\bar U$ are identified with $E\subset J$. 
The Levi factor $L$ is the centralizer of $G_E$ in $M$. By Proposition \ref{P:isogeny}, $L$ can be identified with $E^{\times}$. Indeed, for every 
 $\alpha \in E^{\times}$,  let $c_{\alpha}: J \rightarrow J$ be defined by  
\[ 
c_{\alpha} : (e,v) \mapsto (\alpha^{\#}/\alpha \cdot e, \alpha\cdot  v) 
\] 
for all $(e,v)\in E\oplus C^E_O$.  Then $c_{\alpha}$ is a similitude of $N_J$ with  $i_J(c_{\alpha})=N_E(\alpha)$. Henceforth, we fix an isomorphism $L\cong E^{\times}$ such that $\alpha\in E^{\times}$ acts on $\bar N$ as $c_{\alpha}$.  Using this identification, $i_J(\alpha)=|N_E(\alpha)|^{-1}$ and 
 the center of $M$ consists of  $\alpha\in F^{\times}$. 
\vskip 5pt

 \subsection{\bf Theta lift}
 Let $\Pi$ be the minimal representation of $G_B$.
Let $\Omega \subset  N = J$ be the set of elements of  rank 1, i.e. $x\in \Omega$ if and only if $x\neq 0$ and $x^{\#}=0$. 
As  $\bar P$-modules, we have an exact sequence \cite{MS} 
 \[ 
 0 \rightarrow C_c^{\infty}(\Omega) \rightarrow \Pi \rightarrow \Pi_{\bar N} \rightarrow 0 
 \] 
 where $\bar n \in \bar N$ acts on $f\in C_c^{\infty}(\Omega)$ by 
 \[ 
 \pi(\bar n)f(n)= \psi(\langle \bar n, n\rangle) f(n)
 \] 
 where $\langle \bar n, n\rangle $ is the natural pairing of 
 $\bar N$ and $N$,  and $m\in M$ by 
  \[ 
 \pi(m) f(n) = |i_J(m)|^{-2} f(m^{-1} n m). 
 \] 
 Moreover, $\Pi_{\bar N} =\Pi_M \otimes |i_J|^{-1} \oplus |i_J|^{-2} $, where $\Pi_M$ is a minimal representation of $M$, trivial on the center. It follows that  a central element 
  $\alpha\in F^{\times}$ acts as multiplication by $|\alpha|^3$ and $|\alpha|^6$ on the two summands. 
 \vskip 5pt
 
 Considering $\bar{U}$-coinvariants, we have a short exact sequence of $G_E \times L$-modules:
 \[ \begin{CD}
 0 @>>> C_c^{\infty}(\Omega^{\perp}) @>>> \Pi_{\bar U} @>>> \Pi_{\bar N} @>>>  0, \end{CD} 
 \] 
where $\Omega^{\perp}$ is the set of elements $x\in \Omega$ perpendicular to $E$, i.e. the set of $x=(0,v) \in E \oplus C^E_O$ such that 
\[  v\neq 0 \quad \text{and} \quad   x^{\#}=(-Q(v), \beta(v))=0. \]
  
 Assume, for simplicity, that $E$ is a field. Then $\Omega^{\perp}$ is the set of $v\in C_O^E$  spanning an $E$-line in $\mathcal F_1$. 
 Recall that $G_E$ acts transitively on $\mathcal F_1$.  Fix a line $V_1\in \mathcal F_1$, and 
  let $P_1$ be the stabilizer of   $V_1$.  Then $P_1^{\der}$ acts trivially on the line, and we identify $P_1/P_1^{\der}$ with $\GL(V_1)=E^{\times}$.  
  Summarizing, we have an isomorphism of $G_E \times L$-modules, 
\[ 
C_c^{\infty}(\Omega^{\perp}) \cong \Ind_{P_1}^{G_E} C_c^{\infty} (E^{\times}) \otimes |N_E|^2 
\] 
where  the induction is not normalized and $C_c^{\infty} (E^{\times})$ is the regular representation of $E^{\times} \times E^{\times} \cong P_1/P_1^{\der} \times L$, twisted by  the character $|N_E|^2$ of $L \cong E^{\times}$ as indicated.  
\vskip 5pt

\begin{prop} Let $\Theta(1)$ be the theta lift of the trivial representation of $H_{C_B}(F)=\GL_2(E)^{\det}/ \Delta F^{\times}$. Then $\Theta(1)$ is  a quotient  $J(1/2)$, the degenerate 
principal series representation associated to the parabolic $P_1$. 
\end{prop} 

\begin{proof}  Assume, for simplicity, that $E$ is a field.  
Since $\Theta(1) \otimes 1_{H_{C_B}}$ is a quotient of $\Pi$,  one sees  by passing to $\bar U$-coinvariants that
$\Theta(1)\otimes 1_{E^{\times}}$ is a quotient of $\Pi_{\bar U}$. Let $\alpha\in F^{\times}$  be in the center of $M$. Then $\alpha$ acts trivially on $1_{E^{\times}}$, and as 
 $|\alpha|^3$ and $|\alpha|^6$ on the two summands of $\Pi_{\bar N}$.  Hence $\Theta(1)\otimes 1_{E^{\times}}$ is a quotient of $\Ind_{P_1}^{G_E} C_c^{\infty} (E^{\times}) \otimes |N_E|^2$. 
Hence $\Theta(1)$ is a quotient of $\Ind_{P_1}^{G_E} |N_E|^2$ . Since $\rho_{N_1}^{1/2}=|N_E|^{3/2}$, it follows that  $\Ind_{P_1}^{G_E} |N_E|^2=J(1/2)$. 
\end{proof}

\begin{prop} \label{P:N_2_spectrum} 
Let $\Theta(1)$ be the theta lift of the trivial representation of $H_{C_B}(F) = \GL_2(E)^{\det}/ \Delta F^{\times}$.
 Let $\Sigma$ be a non-degenerate $E$-twisted cube, with associated   $E$-twisted composition algebra $C_{\Sigma}$. 
  Then 
  \begin{itemize} 
  \item $\Theta(1) _{\bar N_2,\psi_{\Sigma}}=0$ if $E\oplus C_{\Sigma}$ is a division algebra; 
  \item $\Theta(1) _{\bar N_2,\psi_{\Sigma}}=\mathbb C$ otherwise. 
  \end{itemize}  
\end{prop} 

\begin{proof}  
 The space of twisted coinvariants $\Pi_{\bar N_2, \psi_{\Sigma}}$ is computed exactly as in Proposition \ref{P:twisted}, giving 
\[\Pi_{\bar N_2, \psi_{\Sigma}}=C_c^{\infty}(\Omega_{\Sigma}(F))\] 
where $\Omega_{\Sigma}$ is as in Corollary \ref{C:orbit}. By the same corollary, if $\Omega_{\Sigma}(F)$ is nonempty, then it is a single 
$H_C(F)$-orbit, in which case $\Theta(1)_{\bar N_2, \psi_{\Sigma}}$ is  one dimensional. On the other hand, when $\Omega_{\Sigma}(F)$ is empty, $\Theta(1)_{\bar N_2, \psi_{\Sigma}} =0$. 
By Proposition  \ref{P:embeddings}, $\Omega_{\Sigma}(F)$  is empty precisely when 
$E\oplus C_{\Sigma}$ is a division algebra.
\end{proof}

\vskip 5pt

\begin{thm}  \label{C:GS3}
 Let $\Theta(1)$ be the theta lift of the trivial representation of $\GL_2(E)^{\det}/\Delta F^{\times}$. Then $\Theta(1)$ embeds as a submodule of the degenerate principal series $I(1/2)$. 
 If $E$ is a field, then $I(1/2) / \Theta(1)  \cong V_1'$ in the notation of Theorem \ref{T:degen_E}. Otherwise $\Theta(1) \cong I(1/2)$.   
\end{thm}

\begin{proof}
The minimal representation of $G_B$ is a submodule of a degenerate principal series representation induced from the Heisenberg parabolic subgroup of $G_B$. Via restriction of functions to $G_E$, one obtains a nonzero $H_C$-invariant and $G_E$-equivariant map 
\[  \Pi \twoheadrightarrow  \Pi_{H_E} = \Theta(1)  \longrightarrow I(1/2). \]
Since the spherical function restricts to a spherical function, the image must contain the submodule generated by the non-zero spherical vector in $I(1/2)$. This is the whole $I(1/2)$ 
unless $E$ is a field, by Propositions \ref{P:degen_KA'} and \ref{P:degen}. If $E$ is a field, the spherical vector generates the submodule whose quotient is $V_1'$.  Next, we need to use the fact that
\[ 
I(1/2)_{\bar N_2, \psi_{C_\Sigma}} \cong \mathbb C
\] 
 for all nondegenerate cubes $\Sigma$, which is a simple consequence of the Bruhat decomposition. Moreover, recall that 
$V_1'\cong \Theta_D(1)$ is the theta lift via the minimal 
representation of $G_D$ (the rank 2 $E_6$). Hence $(V_1')_{\bar N_2, \psi_{C_\Sigma}} \cong \mathbb C$ precisely when $E\oplus C_{\Sigma}$ is a division algebra.  Combining 
with Proposition \ref{P:N_2_spectrum}, we see that the image of the map $\Theta(1)  \longrightarrow I(1/2)$ is exactly as predicted and the kernel consists of small representations, 
i.e.  those for which $(\bar N_2, \psi_{C_{\Sigma}})$ co-invariants vanish for all nondegenerate cubes $\Sigma$.  Since we know that $\Theta(1)$ is a quotient of $J(1/2)$, to finish the proof,  it suffices to show that any irreducible constituent $\pi$ of $J(1/2)$ satisfies $\pi_{\bar N_2, \psi_{C_{\Sigma}}}\neq 0$, for some nondegenerate $\Sigma$. 
\vskip 5pt

 To that end, we claim that it suffices to check one of the following two conditions: 
\begin{itemize} 
\item[(a)] The Jacquet functor of $\pi$ for any parabolic subgroup with Levi subgroup of type $A_2$ is Whittaker generic; 
\item[(b)]  The Jacquet functor of $\pi$ with respect to $N_2$ is a Whittaker generic representation of the Levi subgroup $M_2$. 
\end{itemize} 
Indeed, if (a)  holds, then $\pi_{\bar N_2, \psi_{C_{\Sigma}}}\neq 0$ by \cite[Thm. A]{GGS}, interpreted in our setting for the nilpotent orbit $A_2$. By the same result of 
\cite{GGS}, the condition (b) implies that $\pi_{[\bar N_1, \bar N_1], \psi}\neq 0$  for a generic character of 
$\psi$ of $[\bar N_1, \bar N_1]$, which in turn  implies the existence of a nondegenerate $\Sigma$ such that $\pi_{\bar N_2, \psi_{C_\Sigma}} \neq 0$, by the main result in \cite{JLS} and the fact that the  nilpotent orbit $3A_1$ is not special. 

\vskip 5pt 

If $E$ is a field, we have only one additional constituent $V_1''$ in $J(1/2)$ (see Theorem \ref{T:degen_E'}). 
Its Jacquet functor with respect to 
$\bar N_2$ is a twist of the Steinberg representation of $M_2$, hence the condition (b)  holds and we are done in this case.
\vskip 5pt
  
If $E$ is not a field,  then we have not analyzed $J(1/2)$. In these remaining cases, we shall treat all representations
whose exponents lie in the Weyl group orbit of the leading exponent of the spherical quotient of $I(1/2)$, namely $(1,1,0,0)$  if $E = F^3$ or $(1,1,0)$  if $E=F \times K$ for $K$ a field.  
In both cases, we have two tempered representations,
\begin{equation} \label{E:DSt}
D(\mathrm{St}) = D(\mathrm{St})_{\mathrm{gen}} \oplus D(\mathrm{St})_{\mathrm{deg}}, 
\end{equation}
which are the generic and non-generic summands of the unitary representation $D(\mathrm{St})$ obtained by parabolic induction from the Steinberg representation of the Levi group of type $A_2$. There are three 
such parabolic groups if $E=F^3$, but the resulting representation does not depend on this choice, just as in the case of $D(1)$, which is the Aubert involute of $D(\mathrm{St})$. Observe that these tempered representations satisfy the condition (a).   

\vskip 5pt 
In order to tabulate all possible standard modules, let us recall their properties, working with a general root system $\Delta= \{\alpha_1, \ldots ,\alpha_n\}$. Let 
$\{\lambda_1, \ldots, \lambda_n\}$ be the corresponding fundamental weights.  
 A parabolic subgroup in standard position corresponds to a subset $S \subset \Delta $. A standard module 
associated to the parabolic subgroup has leading exponents  
\[ 
\lambda =( \sum_{i\in S} x_i \alpha_ i ) + ( \sum_{i\notin S} y_i \lambda_ i )
\] 
where $x_i\leq 0$,  $y_i>0$ and the first summand is an exponent of the tempered representation defining the standard module. 
 Now it is easy to determine all leading exponents in the cases at hand, and thus determine all irreducible Langlands quotients in both cases: 
\vskip 5pt 

\noindent{\bf \underline{Case $E=F^3$}}: 
\vskip 5pt

We have three Langlands quotients of $G_E$ for the three maximal parabolic subgroups whose Levi subgroups are of type  $A_3$. The tempered 
representation  on the Levi subgroup is obtained by inducing the Steinberg representation of the Levi subgroup of the type $A_1\times A_1$, that is, whose derived group is 
$\SL_2(F) \times \SL_2(F)$.   These Langlands  quotients clearly satisfy the condition (a).  
\vskip 5pt

There are three remaining representation: the spherical quotient of $I(1/2)$, 
 the Langlands quotient $J_2(\mathrm{St}_E, 1/2)$ and the Langlands quotient $J_1(\mathrm{St}, 1/2)$. For these representations we have complete  
control of their $(\bar N_2, \psi_{C_\Sigma})$-coinvariants, since the spherical representation and 
$J_1( \mathrm{St}, 1/2)$ are the theta lifts  $\Theta_{M_3(F)}(1)$ and $\Theta_{M_3(F)}(\epsilon)$ respectively, and 
 $J_2(\mathrm{St}_E, 1/2)$ is a submodule and the only other constituent of $I(1/2)$.   This settles the case $E = F^3$.

\vskip 5pt 

\noindent{\bf \underline{Case $E=F \times K$}}: 
\vskip 5pt

Here we have an interesting twist, when compared to the split case: there are two Langlands quotients of $G_E$ forming an $L$-packet which prove especially challenging.  
\vskip 5pt

More precisely,  
instead of the three $A_3$ maximal parabolic subgroups considered in the split case, we have a maximal parabolic subgroup in the standard position 
 with Levi subgroup of the type $B_2$, so that its derived group is a quasi-split $\mathrm{SU}_4(K)$.  Inducing the Steinberg representation of the Levi subgroup of $\mathrm{SU}_4(K)$ whose derived group is $\SL_2(K)$, gives a representation of $\mathrm{SU}_4(K)$ with two irreducible summands. They in turn give two Langlands quotients of $G_E$ 
 with the leading exponent  $(1,0,-1)$.  One of these two  representation is the summand of $D(1)$, denoted by $V_1'$, with $(1,0,-1)$ as its only exponent.  The other representation $V$  is the potentially troublesome one. 
\vskip 5pt

 Finally, we have three additional representations:
 the spherical quotient of $I(1/2)$ (which is the other  summand of $D(1)$ besides $V_1'$, by Proposition \ref{P:degen_KA'}(4)),  
 the Langlands quotient $J_2(\mathrm{St}_E, 1/2)$ and the Langlands quotient $J_1(\mathrm{St}, 1/2)$. 
 Clearly, $J_2(\mathrm{St}_E, 1/2)$ satisfies the condition (b) above. 
 Now $J_1(\mathrm{St}, 1/2)$ is a submodule of $I(1/2)$, while the spherical representation and $V_1'$ are the theta lifts  
 $\Theta_{M_3(F)}(1)$ and $\Theta_{M_3(F)}(\epsilon)$. For these representations, we 
 have a similar situation as in the split case, with complete control of their $(\bar N_2, \psi_{C_\Sigma})$-coinvariants, and in particular non-vanishing for some $\Sigma$.
 \vskip 5pt
 
  It remains to deal with the other representation $V$ with  leading exponent $(1,0,-1)$.  Recall that, counting two tempered representations in (\ref{E:DSt}), we have seven representations in all. Let us examine the effect of the Aubert involution on this set of representations:
  \vskip 5pt
  \begin{itemize}
  \item The  Aubert  involution takes the two summands of $D(\mathrm{St})$ to the two summands of $D(1)$. 
  
  \item It takes the degenerate series $I(1/2)$ to the generalized principal series $I(\mathrm{St}_E, -1/2)$.  It follows that the Aubert involution takes $J_1(\mathrm{St}, 1/2)\subset I(1/2)$ to $J_2(\mathrm{St}_E, 1/2)\subset I(\mathrm{St}_E, -1/2)$. 
  \end{itemize}
  From this, one deduces that the involution fixes the 
 remaining representation $V$, and hence $(-1,0,1)$ is also an exponent of  $V$. But with respect to the $A_2$ Levi subgroup, this is the exponent of the Steinberg representation and hence condition (a)  holds for $V$.   This completes the proof in the case $E = F \times K$.

\end{proof}

This theorem is used in our paper \cite{GS3}.

  \vskip 15pt

\section{\bf Appendix B: Degenerate principal series}
In this section, we analyze unramified degenerate principal series representations for $G_E$ (the quasi-split simply connected  reductive group of  absolute 
type $D_4$ determined by $E$). The results here are new if  $E$ is a field and a mixture of new and known results if $E=F \times K$.  
We have used the results and language introduced in this appendix for the description of theta lifting in the main body of the paper.
\vskip 5pt

\subsection{Affine Weyl groups, when $E$  a field} 

Let $A=\{(x,y,z)\in \mathbb R^3 ~|~ x+y+z=0\}$ be the 2-dimensional euclidean space equipped with the 
usual dot product. Let $\Phi\subseteq A^*$ (we identify $A$ with $A^*$ using the dot product) 
 be the root space of type $G_2$ such that $\alpha_1=(1,-1,0)$ and 
$\alpha_2=(-\frac13,\frac23,-\frac13)$ are the simple roots. Let $W$ be the corresponding Weyl group. 
It is generated by the simple reflections $s_1$ and $s_2$ corresponding to the simple roots.

{\bf \underline{Assume first that $E$ is unramified}}. 
\vskip 5pt

Affine roots are the affine functions $\alpha+k$ on $A$ where 
 $\alpha\in\Phi$ and $k\in\mathbb Z$.
The affine Weyl group $W_a$ is generated by reflections about  the lines where the affine 
roots vanish. Let $\alpha_l\in \Phi$ be the highest root. The fundamental cell in $A$ for $W_a$ 
is given by the inequalities 
\[  0 <\alpha_1 ,\quad  0< \alpha_2 \quad \text{and} \quad \alpha_l <1. \]
 In particular, $W_a$ is 
generated by $s_1$, $s_2$ and $s_0$, the reflections about the three lines bounding the fundamental cell. 
Let $X\subseteq A$ be the lattice spanned by 
\[ 
\omega_1=(1,0,-1) \text{ and } \omega_2=(1,1,-2). 
\] 
Then $W_a$ is a semi direct product of $W$ and the group of translations $t_{\omega}$ where 
$\omega\in X$.  We note the following relations in $W_a$: 
\[ 
t_{\omega_1}=s_0 s_1 s_2 s_1 s_2 s_1  \text{ and } t_{\omega_2}=(s_0 s_1 s_2 s_1 s_0)(s_2 s_1 s_2 s_1 s_2).  
\] 

\vskip 5pt

{\bf  \underline{Assume now that $E$ is ramified}}. 
\vskip 5pt

Affine roots are the affine functions $\alpha+k$ on $A$ where 
 $\alpha\in\Phi$ and $k\in\mathbb Z$, if $\alpha$ is long, and $k\in\frac{1}{3}\mathbb Z$, if 
 $\alpha$ is short. 
The affine Weyl group $W_a$ is generated by reflections about  the lines where the affine 
roots vanish. Let $\alpha_s\in \Phi$ be the highest short root. The fundamental cell in $A$ for $W_a$ 
is given by the inequalities 
\[  0 <\alpha_1,\quad  0< \alpha_2 \quad \text{and} \quad  \alpha_s <1/3. \]
In particular, $W_a$ is 
generated by $s_1$, $s_2$ and $s_0$, the reflections about the three lines bounding the fundamental cell. 
Let $X\subseteq A$ be the lattice spanned by 
\[ 
\omega_1=(1,0,-1) \text{ and } \omega_2=(\frac13,\frac13,-\frac23). 
\] 
Then $W_a$ is a semi direct product of $W$ and the group of translations $t_{\omega}$ where 
$\omega\in X$.  We note the following relations in $W_a$: 
\[ 
t_{\omega_2}=s_0 s_2 s_1 s_2 s_1 s_2  \text{ and } t_{\omega_1}=(s_0 s_2 s_1 s_2 s_0)(s_1 s_2 s_1 s_2 s_1).  
\] 

\smallskip 
Let $G_E$ be the simply connected quasi-split group of type $D_4$ corresponding to the cubic field $E$. 
 Let $I$ be the Iwahori subgroup corresponding to the fundamental cell. 
 Let  $\ell: W_a \rightarrow \mathbb Z$ be the length function such that, for every $w\in W_a$,  
\[ 
[IwI : I] = q^{\ell(w)}. 
\]

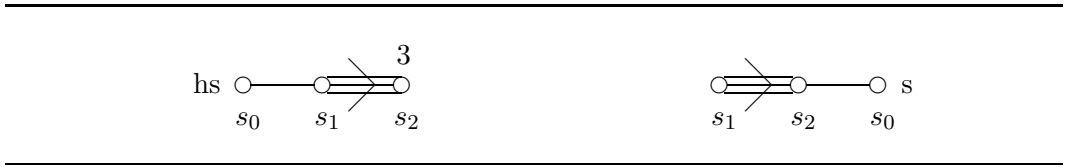
\begin{picture}(400,140)(-35,60)

\put(-30,120){\line(1,0){400}}
\put(-30,180){\line(1,0){400}}

\put(63,150){\line(1,0){24}}
\put(93,150){\line(1,0){24}}
\put(92,153){\line(1,0){28}}
\put(92,147){\line(1,0){28}}

\put(60,150){\circle{6}}
\put(90,150){\circle{6}}
\put(120,150){\circle{6}}

\put(110,150){\line(-1,-1){10}}
\put(110,150){\line(-1,1){10}}

\put(41,147){hs}
\put(118, 158){3} 

\put(57,135){$s_0$}
\put(87,135){$s_1$}
\put(117,135){$s_2$}

\put(243,150){\line(1,0){24}}
\put(273,150){\line(1,0){24}}  
\put(242,153){\line(1,0){26}}
\put(242,147){\line(1,0){26}}

\put(240,150){\circle{6}}
\put(270,150){\circle{6}}
\put(300,150){\circle{6}}

\put(260,150){\line(-1,-1){10}} 
\put(260,150){\line(-1,1){10}}

\put(237,135){$s_1$}
\put(267,135){$s_2$}
\put(297,135){$s_0$} 

\put(309,147){s}

\end{picture} 

 Let $H$ be the Iwahori Hecke algebra. It is spanned by $T_w$, the 
characteristic functions of $IwI$ for all $w\in W_a$. As an abstract algebra, $H$ is generated by 
$T_0$, $T_1$ and $T_2$ corresponding to simple reflections, modulo braid and quadratic relations 
given by the diagrams in the above picture, where the left diagram corresponds to the case when $E$ is unramified.  
Let $\hat{T}_w= q^{-\frac{\ell(w)}{2}} T_w$.  The elements 
$\hat{T}_{\omega}$ for dominant $\omega=n_1\omega_1 + n_2\omega_2$ (i.e. $n_1,n_2 \geq 0$) 
form a commutative semi-group 
\[ 
\hat {T}_{\omega} \cdot \hat{T}_{\omega'}= \hat{T}_{\omega+ \omega'}.
\] 

Let $V$ be a finite-dimensional $H$-module. 
Since $\hat{T}_{\omega}$, for dominant $\omega$, commute and are invertible, we can decompose 
\[ 
V =\oplus_{\mu} V_{\mu} 
\] 
where, for every $\mu\in A\otimes_{\mathbb R} \mathbb C$, 
\[ 
V_{\mu}=\{ v\in V ~|~ (\hat{T}_{\omega} - q^{\mu(\omega)})^{\infty} v=0 \text{ for all dominant } \omega\}. 
\] 
Note that $V_{\mu}=V_{\mu+\frac{2\pi i}{\ln q} \nu}$ for any $\nu\in X^*$, the lattice dual to $X$.  
Thus, we say that  $\mu, \mu'$  are {\em congruent} if $\mu- \mu' \in {\frac{2\pi i}{\ln q} X^*}$. 
The congruence class of  $\mu$ such that $V_{\mu}\neq 0$ is called an exponent of $V$.  A representation 
$V$ is a discrete series if 
\[ 
\Re(\mu(\omega_i)) < 0
\] 
for $i=1,2$ for all exponents $\mu$ of $V$. 
Exponents  represented by $\mu\in A$ are called {\em real}.  The exponent of the trivial representation 
(i.e. $T_w \mapsto q^{\ell(w)}$ for all $w\in W_a$) is 
\[ 
(2,1,-3) . 
\] 
The Iwahori-Matsumoto (IM) involution changes the exponents by the sign. In particular, the exponent of the 
Steinberg representation is $(-2,-1,3)$. It is a discrete series representation. 

\vskip 10pt

\subsection{Some representations, when $E$ is a field} 

We shall now construct small dimensional representations of  the Hecke algebra $H$ that will appear in the description of degenerate 
principal series.   

\vskip 5pt 

{\bf \underline{Assume first that $E$ is unramified.}}

\subsubsection{One dimensional representations} \label{SS:E-1d}
Let $V$ be a one dimensional complex vector space 
spanned by $e$.  Let $V'_1$ be the representation of $H$ on $V$
defined  by 
\[ T_0 e=-e,  \quad T_1 e=-e \quad \text{and} \quad  T_2 e= q^3 e. \]
 The exponent of $V'_1$ is  
\[ 
(0,1,-1). 
\] 
Let $V''_1$ be be the representation of $H$ on $V$
defined  by 
\[  T_0 e=qe, \quad  T_1 e=qe \quad \text{and} \quad  T_2 e= - e. \]
 Then $V''_1$ is the IM-involute of $V'_1$ and is a discrete series representation. 

\vskip 5pt

\subsubsection{Two dimensional representations} The subalgebra generated by $T_0$ and $T_1$ is 
isomorphic to the group algebra of $S_3$. It is not too difficult to see that any
irreducible two dimensional representation of $H$, when restricted to this sub algebra, must be isomorphic 
to the reflection representation of $S_3$.  Thus let $V$ be a two dimensional complex vector space 
spanned by $e_0$ and $e_1$ on which $T_0$ and $T_1$ act as matrices 
\[ 
\left(\begin{array}{cc} 
-1 & q^{\frac{1}{2}} \\
0 & q 
\end{array}\right) 
\text{ and } 
\left(\begin{array}{cc} 
 q & 0 \\
q^{\frac{1}{2}} & -1
\end{array}\right), 
\] 
respectively. We can extend this representation to $H$ in three different ways. Two of these extensions 
are easy to construct. Let $V_2'$ be the representation of $H$ on $V$  such that $T_2$ acts the scalar $q^3$. 
The exponents of $V_2'$ are 
\[ 
(1-\frac{2\pi i}{3\ln q}, 1+\frac{2\pi i}{3\ln q}, -2) \text{ and } (1+\frac{2\pi i}{3\ln q}, 1-\frac{2\pi i}{3\ln q}, -2). 
\] 
This is the minimal representation. Let $V''_2$ be the representation of $H$ on $V$ such that $T_2$ acts the scalar $-1$. 
Then $V''_2$ is the IM-involute of $V'_2$ and is a discrete series representation. 

\vskip 5pt

 These two representations do not have  real exponents, however. We shall be interested in the third extension such that $T_2$ acts as the matrix 
\[ 
\left(\begin{array}{cc} 
-1 & q^{\frac{1}{2}}\Phi_6(q) \\
0 & q^3
\end{array}\right),
\] 
where $\Phi_6$ is the characteristic polynomial (over $\Q$ of the primitive 6-th roots of unity.
This representation, henceforth denoted by $V_2$, is invariant under the involution. Its exponents are real and  given by: 
\[ 
(1,-1,0) \text{ and } (-1,1,0). 
\] 

\subsubsection{Three dimensional representations} Let $V$ be a three dimensional complex vector space 
spanned by $e_0$, $e_1$  and $e_2$. Let $V'_3$ be a representation of $H$ on 
$V$  such that $T_0$, $T_1$ and $T_2$ act as matrices 
\[ 
\left(\begin{array}{ccc} 
-1 & q^{\frac{1}{2}} & 0  \\
0 & q & 0 \\
0  & 0 & q
\end{array}\right), 
\left(\begin{array}{ccc} 
 q & 0 & 0  \\
q^{\frac{1}{2}} & -1 & q^{\frac{1}{2}}\\
0 &  0 & q 
\end{array}\right) 
\text{ and } 
\left(\begin{array}{ccc} 
 q^3 & 0 & 0  \\
0 &  q^3 & 0\\
0 &  q^{\frac{1}{2}}\Phi_3(q) & -1
\end{array}\right)
\] 
respectively. This is the reflection representation. The exponents of $V'_3$, counted with multiplicities,  are 
\[ 
(0,1,-1), (1,0,-1) \text{ and } (1,0,-1). 
\] 
Let $V''_3$ be the IM-involute of $V'_3$. It is a discrete series representation. 

\vskip 5pt 

{\bf \underline{Assume now that $E$ is ramified}.}

\subsubsection{One dimensional representations}  \label{SS:E-1dram}
Let $V$ be a one dimensional complex vector space 
spanned by $e$. Let $V_1'$ be a representation of $H$ on $V$ defined by 
\[  T_0 e=qe,\quad  T_1 e=-e \quad \text{and} \quad  T_2 e= q e. \]  
The exponent of $V'_1$ is  
\[ 
(0,1,-1). 
\] 
Let $V''_1$ be be the representation of $H$ on $V$
defined  by 
\[  T_0 e=-e,\quad  T_1 e=qe \quad  \text{and} \quad  T_2 e= -e.\]
  Then $V''_1$ is the IM-involute of $V'_1$ and is a discrete series representation.
\vskip 5pt

\subsubsection{Two dimensional representations} The subalgebra generated by $T_0$ and $T_2$ is 
isomorphic to the group algebra of $S_3$. It is not too difficult to see that any
irreducible two dimensional representation of $H$, when restricted to the subalgebra, must be isomorphic 
to the reflection representation of $S_3$.  Thus let $V$ be a two dimensional complex vector space 
spanned by $e_0$ and $e_2$. Then $T_0$ and $T_2$ act on $V$  as matrices 
\[ 
\left(\begin{array}{cc} 
-1 & q^{\frac{1}{2}} \\
0 & q 
\end{array}\right) 
\text{ and } 
\left(\begin{array}{cc} 
 q & 0 \\
q^{\frac{1}{2}} & -1
\end{array}\right), 
\] 
respectively. We can extend this representation to $H$
  in three different ways. Two of these extensions
are easy to construct. Let $V_2'$ be the representation of $H$ on $V$  such that $T_1$ acts as the scalar $q$.
The exponents of $V_2'$ are
\[
(1-\frac{2\pi i}{3},\frac{4\pi i}{3\ln q},-1-\frac{2\pi i}{3\ln q})\text{ and }(1+\frac{2\pi i}{3},- \frac{4\pi i}{3\ln q},-1+\frac{2\pi i}{3\ln q}).
\]
This is not the minimal representation. Let $V''_2$ be the representation of $H$ on $V$ such that $T_1$ acts as the scalar $-1$.
Then $V''_2$ is the IM-involute of $V'_2$ and is a discrete series representation. Again, these representations do not have  
real exponents. 
\vskip 5pt

We shall be interested in the third extension such that $T_1$ acts as $T_0$. 
This representation, henceforth denoted by $V_2$, is invariant under the involution. Its exponents are real and given by
\[ 
(1,-1,0) \text{ and } (-1,1,0). 
\] 
\vskip 5pt

\subsubsection{Three dimensional representations} Let $V$ be a three dimensional complex vector space 
spanned by $e_1$, $e_2$  and $e_0$. 
 Let $V'_3$ be a representation of $H$ on
$V$  such that $T_0$, $T_1$ and $T_2$ act as matrices
\[ 
\left(\begin{array}{ccc} 
-1 & q^{\frac{1}{2}} & 0  \\
0 & q & 0 \\
0  & 0 & q
\end{array}\right), 
\left(\begin{array}{ccc} 
 q & 0 & 0  \\
3q^{\frac{1}{2}} & -1 & q^{\frac{1}{2}}\\
0 &  0 & q 
\end{array}\right) 
\text{ and } 
\left(\begin{array}{ccc} 
 q & 0 & 0  \\
0 &  q & 0\\
0 &  q^{\frac{1}{2}} & -1
\end{array}\right)
\] 
respectively. This is the reflection representation. The exponents of $V'_3$, counted with multiplicities,  are 
\[ 
(0,1,-1), (1,0,-1) \text{ and } (1,0,-1). 
\] 
Let $V''_3$ be the involute of $V'_3$. It is a discrete series representation.

\vskip 10pt

\subsection{Degenerate principal series, when $E$ is a field} 

We now study the unramified degenerate principal series representation of $G_E$ associated to the Heisenberg parabolic subgroup $P_E$.
Let $e$ and $f$ be the ramification and inertia indices of $E$ over $F$, so that $e\cdot f=3$. The simple coroots are 
\[ 
\alpha_1^{\vee}= (1,-1,0)\text{ and } \alpha_2^{\vee}=(-\frac1e,\frac2e,-\frac1e).
\] 
  Let $V$ be an irreducible representation of $H$. Let $\mu\in A\otimes\mathbb C$ such that 
  $V_{\mu}\neq 0$ i.e. the class of $\mu$ is an exponent of $V$. 
   Then, from the 
representation theory of $SL_2(F)$ and $SL_2(E)$,  
\begin{itemize} 
\item If $\mu(\alpha_1^{\vee}) \neq \pm 1 + \frac{2\pi i}{\ln q}\mathbb Z$ then  $s_1(\mu)$ is an exponent of $V$. 
\item If $\mu(\alpha_2^{\vee}) \neq \pm f + \frac{2\pi i}{\ln q}\mathbb Z$ then  $s_2(\mu)$ is 
an exponent of $V$.  
\item If  $s_i(\mu)$ is congruent to $\mu$ and $\mu(\alpha_i^{\vee}) =0$ then $V_{\mu}$ is at least two dimensional.

\end{itemize} 
Two exponents are {\em equivalent} if one is obtained from another by a repeated use of the first two bullets. 

\vskip 5pt

In the following, we shall consider the decomposition of various unramified degenerate principal series representations of $G_E$. 
The representations $V$ of the affine Hecke algebra that we constructed above will occur in the subspace of Iwahori-fixed vectors in these principal series representations
So as not to introduce more notation, we will use $V$ to denote the corresponding representation of $G_E(F)$ (whose space of Iwahori-fixed vectors is $V$) as well.
\vskip 5pt

\subsubsection{\bf Degenerate series $I(s)$}  \label{SS:degpsI}
Let 
\[ 
\mu_s= (s-\frac{1}{2},1,-s-\frac{1}{2}) 
\] 
where $s\in \mathbb C$. 
Note that $\mu_s$ and $\mu_{s'}$ are congruent if $s-s'\in \frac{2\pi i}{\ln q}\mathbb Z$.  
Since $\mu_s(\alpha_2^{\vee})=f$, the equivalence class of 
$\mu_s$, for a generic $s$,  contains the following six elements
\[ 
 (s-\frac{1}{2},1,-s-\frac{1}{2}),  
\] 
\[ 
(1, s-\frac{1}{2},-s-\frac{1}{2}),
\] 
\[
(s+\frac{1}{2},-s+\frac{1}{2},-1),
\] 
\[
(-s+\frac{1}{2},s+\frac{1}{2},-1),
\] 
\[ 
(1, -s-\frac{1}{2},s-\frac{1}{2}),
\] 
\[
(-s-\frac{1}{2},1,s-\frac{1}{2}).
\] 
These are the exponents of  a degenerate principal series $I(s)$, attached to the Heisenberg  maximal parabolic subgroup $P_E$. 
Since the representations $I(s)$ form an algebraic family, these are the exponents for any $s$. 
The first exponent ($\mu_s$) is  a {\em leading} exponent of $I(s)$. 
The last exponents is a {\em trailing} exponent of $I(s)$.  (It is a leading exponent of $I(-s)$.)  If $V$ is a quotient 
of $I(s)$ then the leading exponent is an exponent of $V$. If $V$ is a submodule of $I(s)$,  then the trailing 
exponent of $I(s)$ is also an exponent of $V$. We would like to determine the points of reducibility of $I(s)$. 

\vskip 5pt 

We say that an exponent $\mu$ is {\em regular}, if the stabilizer of $\mu$ in the Weyl group is trivial. 
A representation $V$ of $H$ is regular if the exponents of $V$ are regular. It is well known that irreducible 
regular representations correspond to equivalence classes of regular exponents.  One checks that 
$I(s)$ is regular if
 \[
\pm  s\neq \frac{3}{2},   \frac{1}{2},  0,  \frac{\pi i}{\ln q} \text{ and }  \frac{1}{2} \pm\frac{2\pi i}{3\ln q} 
 \] 
 where the last possibility occurs only if $E$ is unramified. 
If $I(s)$ is regular, one checks that all exponents are equivalent, and hence $I(s)$ is irreducible, if 
 \[
\pm  s\neq   \frac{5}{2}, \frac{1}{2} + \frac{\pi i}{\ln q} \text{ and }    \frac{3}{2}\pm\frac{2\pi i}{3\ln q} 
 \] 
 and reducibility in the last case occurs only when $E$ is unramified. In particular, $I(s)$ is irreducible unless $s$ is on one of the two lists.  
  
\begin{thm} \label{T:degen_E}
The representation $I(s)\cong I(-s)^*$ is reducible only if 
\[
\pm s= \frac{5}{2}, \frac{1}{2}, \frac{1}{2}+\frac{\pi i}{\ln q} \text{ and } \frac32 \pm \frac{2\pi i}{3\ln q}
\] 
and the last case occurs only if $E$ is unramified. At the points of reducibility, we have: 
\begin{enumerate} 
\item $I(\frac{5}{2})$ has length 2. The trivial representation is the unique irreducible quotient. 
\item $I(\frac{1}{2})$ has length 3. The representation $V_2$ is the unique irreducible submodule. 
The representations $V'_1$ and $V'_3$ are irreducible quotients. 
\item $I(\frac{1}{2}+ \frac{\pi i}{\ln q})$ has length 2. There is a unique irreducible submodule and a unique 
irreducible quotient. 
\item $I(\frac{3}{2}\pm \frac{2\pi i}{3\ln q})$ has length 2. The  minimal representation $V'_2$ is the unique irreducible quotient.  
\end{enumerate} 
\end{thm} 
\begin{proof}   It remains to analyze the finite set of cases. We do so by considering the space of Iwahori-fixed vectors in $I(s)$, which is a $H$-module.

\smallskip 
\noindent 
\underline{Case $s=\frac{5}{2}$.}  The exponents are 
\[ 
(2,1,-3), 
\] 
\[ 
(1, 2,-3),
\] 
\[
(3,-2,-1),
\] 
\[
(-2,3,-1),
\] 
\[ 
(1, -3,2),
\] 
\[
(-3,1,2).
\] 
The leading exponent belongs to the trivial representation, the unique irreducible quotient of $I(\frac{5}{2})$.
The other five exponents are equivalent to the trailing exponent.  Thus $I(\frac{5}{2})$ has length 2. 

\smallskip 

\noindent 
\underline{Case $s=\frac{3}{2}$.}  The exponents are 
\[ 
(1,1,-2), 
\] 
\[ 
(1, 1,-2),
\] 
\[
(2,-1,-1),
\] 
\[
(-1,2,-1),
\] 
\[ 
(1, -2, 1),
\] 
\[
(-2,1,1).
\] 
The last four exponents are equivalent. Let $V$ be an irreducible subquotient such that $V_{(1,1,-2)}\neq 0$. 
The third bullet implies that this space is 2 dimensional. Thus, either $I(\frac{3}{2})$ is irreducible or it has a 2 
dimensional irreducible quotient.  But the exponents of $I(\frac{3}{2})$ are different from the exponents of irreducible 2 dimensional representations of $H$. Thus $I(\frac{3}{2})$ is irreducible. 

\smallskip 
\noindent 
\underline{Case $s=\frac{3}{2}+ \frac{2\pi i}{3\ln(q)}$.}  We assume that $E$ is unramified. The exponents are 
\[ 
(1+ \frac{2\pi i}{3\ln(q)},1,-2-\frac{2\pi i}{3\ln(q)}), 
\] 
\[ 
(1, 1+  \frac{2\pi i}{3\ln(q)},-2-\frac{2\pi i}{3\ln(q)}),
\] 
\[
(2+\frac{2\pi i}{3\ln(q)},-1-\frac{2\pi i}{3\ln(q)},-1),
\] 
\[
(-1-\frac{2\pi i}{3\ln(q)},2+\frac{2\pi i}{3\ln(q)},-1),
\] 
\[ 
(1, -2-\frac{2\pi i}{3\ln(q)}, 1+ \frac{2\pi i}{3\ln(q)}),
\] 
\[
(-2-\frac{2\pi i}{3\ln(q)},1, 1+ \frac{2\pi i}{3\ln(q)}).
\] 
All exponents are different. The first two are equivalent and so are  the last four. Since 
\[ 
(1+ \frac{2\pi i}{3\ln(q)},1,-2-\frac{2\pi i}{3\ln(q)})- (1-\frac{2\pi i}{3\ln q}, 1+\frac{2\pi i}{3\ln q}, -2) =
\frac{2\pi i}{\ln q}\cdot (\frac{2}{3}, -\frac13,-\frac13) 
\]
the first two are the exponents of the minimal representation $V'_2$.  The induced representation 
has length 2, with unique irreducible quotient  $V'_2$. 

\smallskip 
\noindent 
\underline{Case $s=\frac{1}{2}$.}  The exponents are 
\[ 
(0,1,-1), 
\] 
\[ 
(1, 0,-1),
\] 
\[
(1,0,-1),
\] 
\[
(0,1,-1),
\] 
\[ 
(1, -1, 0),
\] 
\[
(-1,1,0).
\] 
In this case, $V_2$ is a unique irreducible submodule. The quotient is isomorphic to a direct sum of $V'_1$ and $V'_3$. 

\smallskip 

\noindent 
\underline{Case $s=\frac{1}{2}+ \frac{\pi i}{\ln(q)}$.}  The exponents are 
\[ 
(\frac{\pi i}{\ln(q)},1,-1-\frac{\pi i}{\ln(q)}), 
\] 
\[ 
(1, \frac{\pi i}{\ln(q)},-1-\frac{\pi i}{\ln(q)}),
\] 
\[
(1+\frac{\pi i}{\ln(q)},-\frac{\pi i}{\ln(q)},-1),
\] 
\[
(-\frac{\pi i}{\ln(q)},1+\frac{\pi i}{\ln(q)},-1),
\] 
\[ 
(1, -1-\frac{\pi i}{\ln(q)}, \frac{\pi i}{\ln(q)}),
\] 
\[
(-1-\frac{\pi i}{\ln(q)},1,\frac{\pi i}{\ln(q)}).
\] 
All exponents are different. The first three are equivalent and so are the last three exponents. In particular, 
$I(\frac{1}{2}+ \frac{\pi i}{\ln(q)})$ has length 2. 

\smallskip 

\noindent 
\underline{Case $s=\frac{1}{2}+ \frac{2\pi i}{3\ln(q)}$.}  We assume that $E$ is unramified. 
This representation is irreducible. The argument is similar to the argument  for $s=\frac{3}{2}$. We omit details.

\smallskip 

\noindent 
\underline{Case $s=0$.}  The exponents are 
\[ 
 (-\frac{1}{2},1,-\frac{1}{2}). 
\] 
\[ 
(1, -\frac{1}{2},-\frac{1}{2}),
\] 
\[
(\frac{1}{2}, \frac{1}{2},-1),
\] 
\[
(\frac{1}{2}, \frac{1}{2},-1),
\] 
\[ 
(1, -\frac{1}{2},-\frac{1}{2}),
\] 
\[
(-\frac{1}{2},1,-\frac{1}{2}).
\] 
We have three equivalent exponents each with multiplicity 2. Thus, either $I(0)$ is irreducible or it is a sum of 
two three dimensional  representations with the same exponents. However, if  $V_{(\frac12,\frac12,-1)}\neq 0$, 
then the third bullet implies that this space is 2 dimensional. Thus $I(0)$ is irreducible. 

\smallskip 

\noindent 
\underline{Case $s=\frac{\pi i}{\ln q}$.} This representation is irreducible. The argument is the same as 
for $s=0$. We omit details. 

\end{proof} 

\subsubsection{\bf Degenerate series $J(s)$} 

We now study the unramified degenerate principal series associated to the 3-step parabolic subgroup $Q_E$ of $G_E$.
 Let 
\[ 
\mu_s= (s+\frac{1}{2}, s-\frac{1}{2}, -2s) 
\] 
where $s\in \mathbb C$. 
Note that $\mu_s$ and $\mu_{s'}$ are congruent if $s-s'\in \frac{2\pi i}{f \ln q}\mathbb Z$.  
Since $\mu_s(\alpha_1^{\vee})=1$, the equivalence class of 
$\mu_s$, for a generic $s$,  contains the following six elements
\[ 
 (s+\frac{1}{2},s-\frac{1}{2},-2s),  
\] 
\[ 
(2s , -s+\frac{1}{2},-s-\frac{1}{2}),
\] 
\[
(-s+\frac{1}{2}, 2s, -s-\frac{1}{2}),
\] 
\[
(s+\frac{1}{2}, -2s ,s-\frac{1}{2}),
\] 
\[ 
(-2s, s+\frac{1}{2},s-\frac{1}{2}),
\] 
\[
(- s+\frac{1}{2}, -s-\frac{1}{2}, 2s).
\] 
These are the exponents of  a degenerate principal series $J(s)$, attached to the 3-step maximal parabolic subgroup $Q_E$.   
Since the representations $J(s)$ form an algebraic family, these are the exponents for any $s$. 
The first exponent ($\mu_s$) is  a {\em leading} exponent of $J(s)$. 
The last exponents is a {\em trailing} exponent of $J(s)$.  (It is a leading exponent of $J(-s)$.)  If $V$ is a quotient 
of $J(s)$ then the leading exponent is an exponent of $V$. If $V$ is a submodule of $J(s)$ then the trailing 
exponent of $J(s)$ is also an exponent of $V$. 

\vskip 5pt

We would like to determine points of reducibility of $J(s)$. 
  One checks that 
$J(s)$ is regular if
 \[
\pm  s\neq   \frac{1}{2},  \frac{1}{6}, 0,  \frac{\pi i}{\ln q} \text{ and }  \frac{1}{6} \pm\frac{2\pi i}{3\ln q} 
 \] 
 where the last possibility occurs only if $E$ is ramified. 
If $J(s)$ is regular, one checks that all exponents are equivalent, and hence $J(s)$ is irreducible, if 
 \[
\pm  s\neq   \frac{3}{2}, \frac{1}{2} + \frac{\pi i}{\ln q} \text{ and }    \frac{1}{2}\pm\frac{2\pi i}{3\ln q} 
 \] 
  and reducibility in the last case occurs only when $E$ is ramified.  Hence, again, $J(s)$ is irreducible unless $s$ is on the two finite lists. 
  
  \begin{thm} \label{T:degen_E'} 
The representation $J(s)\cong J(-s)^*$ is reducible only if 
\[
\pm s\neq \frac{3}{2}, \frac{1}{2}, \frac{1}{2}+\frac{\pi i}{\ln q} \text{ and } \frac12 \pm \frac{2\pi i}{3\ln q}
\] 
and the last case occurs occurs only if $E$ is ramified. At the points of reducibility, we have: 
\begin{enumerate} 
\item $J(\frac{3}{2})$ has length 2. The trivial representation is the unique irreducible quotient. 
\item $J(\frac{1}{2})$ has length 3. The representation $V''_1$ is the unique irreducible submodule. 
The representation  $V'_3$ is  the unique irreducible quotient. The remaining subquotient is $V_2$.  
\item $J(\frac{1}{2}+ \frac{\pi i}{\ln q})$ has length 2. There is a unique irreducible submodule and a unique 
irreducible quotient. 
\item $J(\frac{1}{2}\pm \frac{2\pi i}{3\ln q})$ has length 2. The representation $V'_2$ is the unique irreducible quotient.  
\end{enumerate} 
\end{thm}

\begin{proof} We shall provide details for $s=1/2$, which is the only case used in the paper.
  
 \smallskip 
\noindent 
\underline{Case $s=\frac{1}{2}$.}  The exponents are 
\[ 
(1,0,-1), 
\] 
\[ 
(1, 0,-1),
\] 
\[
(0,1,-1),
\] 
\[
(1,-1,0),
\] 
\[ 
(-1, 1, 0),
\] 
\[
(0,-1,1).
\] 
We see that $V''_1$ is the unique irreducible submodule, $V'_3$ is the unique irreducible quotient, and $V_2$ is the remaining  subquotient. 

\end{proof} 
\vskip 5pt

\subsection{Affine Weyl group, when $K$ is a field} 
We now discuss the quasi split $G_E$ where $E=F\times K$ with $K$  a quadratic field. Let $e$ and $f$ be the ramification and inertia indices, so that $e \cdot f=2$. 
\vskip 5pt

Let $A=\R^3$ equipped with the 
usual dot product. Let $\Phi\subseteq A^*$ (we identify $A$ with $A^*$ using the dot product) 
 be the root space of type $B_2$ such that 
 \[ 
 \alpha_1=(1,-1,0), \, \alpha_2=(0,1,-1) \text{ and } \alpha_3=(0,0,1) 
\] 
 are the simple roots. The co-roots are 
 \[ 
 \alpha^{\vee}_1=(1,-1,0), \, \alpha^{\vee}_2=(0,1,-1) \text{ and } \alpha^{\vee}_3=(0,0,\frac{2}{e}).  
\] 
Let $W$ be the corresponding Weyl group. 
It is generated by the simple reflections $s_1$, $s_2$ and $s_3$  corresponding to the simple roots. 
\vskip 5pt

\noindent{\bf \underline{Assume first that $K$ is unramified}.} 
\vskip 5pt

Affine roots are the affine functions $\alpha+k$ on $A$ where 
 $\alpha\in\Phi$ and $k\in\mathbb Z$.
The affine Weyl group $W_a$ is generated by reflections about  the lines where the affine 
roots vanish. Let $\alpha_l=(1,1,0)\in \Phi$ be the highest root. The fundamental cell in $A$ for $W_a$ 
is given by the inequalities $0 <\alpha_1$ , $0< \alpha_2$,  $0< \alpha_3$ and $\alpha_l <1$. In particular, $W_a$ is 
generated by $s_1$, $s_2$, $s_3$ and $s_0$, the reflections about the three planes bounding the fundamental cell. 
\vskip 5pt

Let $X\subseteq A$ be the lattice consisting of $(x,y,z)\in \mathbb Z^3$ such that $x+y+z$ is even. 
Then $W_a$ is a semi direct product of $W$ and the group of translations $t_{\omega}$ where 
$\omega\in X$.  It will be convenient to work with the extended affine Weyl group $\tilde{W}_a= W_a \cup \tau W_a$ where 
$\tau$ is the involution defined by  $\tau(x,y,z)=(1-x,y,z)$. Note that 
 $\tau s_0=s_1 \tau$ and $\tau$ commutes with $s_2$ and $s_3$. The extended affine Weyl group is a semi direct product of $W$ and 
$\tilde X=\mathbb Z^3$. Let 
\[ 
\omega_1=(1,0,0), \,\omega_2=(1,1,0)  \text{ and } \omega_3=(1,1,1). 
\] 
 We note the following relations in $\tilde W_a$: 
\[ 
t_{\omega_1}=\tau  s_1 s_2  s_3 s_2 s_1, \,  t_{\omega_2}=  s_0 s_2  s_3 s_2  s_1 s_2  s_3 s_2 
\text{ and } t_{\omega_3}= s_0 s_2 s_3 s_1 s_2 s_3 \tau s_1 s_2 s_3. 
\] 

\noindent{\bf \underline{Assume now that $K$ is ramified}.} 
\vskip 5pt

Affine roots are the affine functions $\alpha+k$ on $A$ where 
 $\alpha\in\Phi$ and $k\in\frac{1}{2}\mathbb Z$,  but integral if $\alpha$ is long. 
The affine Weyl group $W_a$ is generated by reflections about  the lines where the affine 
roots vanish. Let $\alpha_s=(1,0,0)\in \Phi$ be the highest short root. The fundamental cell in $A$ for $W_a$ 
is given by the inequalities $0 <\alpha_1$ , $0< \alpha_2$,  $0< \alpha_3$ and $\alpha_s <1/2$. In particular, $W_a$ is 
generated by $s_1$, $s_2$, $s_3$ and $s_0$, the reflections about the three planes bounding the fundamental cell. 
\vskip 5pt

Let $X=\mathbb Z^3 \subseteq A$.  
Then $W_a$ is a semi direct product of $W$ and the group of translations $t_{\omega}$ where 
$\omega\in X$.  The extended affine Weyl group is  $\tilde{W}_a= W_a \cup \tau W_a$ where $\tau$ is the involution defined by 
 $\tau(x,y,z)=(1/2 -x,1/2-y,1/2-x)$. Note that 
 $\tau s_0=s_1 \tau$ and $\tau s_2=s_3\tau$. The extended affine Weyl group is a semi direct product of $W$ and 
$\tilde X$ generated by $X$ and $(1/2,1/2,1/2)$. Let 
\[ 
\omega_1=(1,0,0), \,\omega_2=(1,1,0)  \text{ and } \omega_3=(1/2,1/2,1/2). 
\] 
 We note the following relations in $\tilde W_a$: 
\[ 
t_{\omega_1}=s_0  s_1 s_2  s_3 s_2 s_1, \,  t_{\omega_2}=  s_0 s_1  s_2 s_3 s_2    s_0 s_1  s_2 s_3 s_2 
\text{ and } t_{\omega_3}= \tau  s_3 s_2 s_1 s_3  s_2 s_3. 
\] 
 \vskip 5pt
 
For any $E = F \times K$, the Iwahori Hecke algebra $H$ of $G_E$ is generated by the elements $T_0$, $T_1$, $T_2$ and $T_3$ corresponding to the simple reflections, modulo braid and quadratic relations
 given by the following diagrams, with the one on the left for the case of unramified $K$ and the one on the right for the case of ramified $K$.
 \vskip 5pt
 
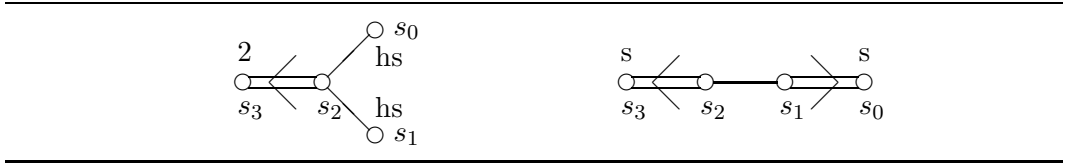
 \begin{picture}(400,100)(20,-20)  


\put(20,60){\line(1,0){400}}
\put(20,0){\line(1,0){400}}

\put(112,32){\line(1,0){26}}
\put(112,28){\line(1,0){26}}
\put(142,32){\line(1,1){16}}
\put(142,28){\line(1,-1){16}}

\put(110,30){\circle{6}}
\put(140,30){\circle{6}}

\put(160,50){\circle{6}}
\put(160,10){\circle{6}} 

\put(108,38){2}
\put(108,18){$s_3$}
\put(138,18){$s_2$}

\put(160,36){hs}
\put(160,17){hs}

\put(167,48){$s_0$}
\put(167,08){$s_1$}

\put(120,30){\line(1,-1){10}}
\put(120,30){\line(1,1){10}}

\put(257,32){\line(1,0){26}}
\put(257,28){\line(1,0){26}}
\put(288,30){\line(1,0){24}}
\put(317,32){\line(1,0){26}}
\put(317,28){\line(1,0){26}}

\put(255,30){\circle{6}}
\put(285,30){\circle{6}}
\put(315,30){\circle{6}}
\put(345,30){\circle{6}}

\put(253,38){s}
\put(253,18){$s_3$}
\put(283,18){$s_2$}
\put(313,18){$s_1$}
\put(343,38){s}
\put(343,18){$s_0$}

\put(265,30){\line(1,-1){10}}
\put(265,30){\line(1,1){10}}

\put(335,30){\line(-1,-1){10}}
\put(335,30){\line(-1,1){10}}

\end{picture}

\vskip 10pt

\subsection{Some representations, when $K$ is a field}  
We shall now construct some small dimensional representations of  the Hecke algebra $H$ that will appear in the description of the degenerate 
principal series representations.

\vskip 5pt

\noindent{\bf \underline{Assume that  $K$ is unramified}.} 
\vskip 5pt

\subsubsection{One dimensional representations}  \label{SS:EFK-1d}

Let $V$ be a one dimensional complex vector space 
spanned by $e$.  There are four representations of $H$ on $V$. We shall firstly describe two representations where 
\[ T_0 e=qe,\quad  T_1 e=qe \quad  \text{and} \quad  T_2e=qe. \]
 The remaining two are obtained by applying the IM-involution. If $T_3 e=q^2 e$, this is the 
trivial representation. Its exponent is 
\[ 
(3,2,1). 
\] 
 Let $V'_1$ be the representation of $H$ on $V$ such that $T_3 e=-e$.  The exponent of $V'_1$ is  
\[ 
(1,0,-1). 
\] 
Let  $V''_1$ be the IM-involute of $V'_1$. It is a tempered representation.  
 \vskip 5pt
 
\subsubsection{Two dimensional representations}  Let $V$ be a two dimensional complex vector space 
spanned by $e_0$ and $e_1$ on which $T_0$,  $T_1$  and $T_2$ act by 
\[ 
T_0=T_1=\left(\begin{array}{cc} 
-1 & q^{\frac{1}{2}} \\
0 & q 
\end{array}\right) 
\text{ and } 
T_2=\left(\begin{array}{cc} 
 q & 0 \\
q^{\frac{1}{2}} & -1
\end{array}\right). 
\] 
 We can extend this representation to $H$ in two ways.  Let $V_2'$ be the representation of $H$ on $V$  such that $T_3$ acts the scalar $q^2$. 
The exponents of $V_2'$ are 
\[ 
(2,0,1) \text{ and } (0,2,1). 
\] 
 Let $V''_2$ be the representation of $H$ on $V$ such that $T_3$ acts the scalar $-1$. 
Then $V''_2$ is the IM-involute of $V'_2$ and is a discrete series representation. 

\vskip 5pt 

\noindent{\bf \underline{Assume that  $K$ is ramified}.} 
\vskip 5pt

\subsubsection{One dimensional representations} \label{SS:EFK-1dram}
Let $V$ be a one dimensional complex vector space 
spanned by $e$. There are eight  representations of $H$ on $V$. We shall firstly describe four representations where
$T_1 e=qe$, $T_2e=qe$. The remaining four representations are obtain by the IM-involution. 
The trivial representation is the one where $T_0e=qe$ and $T_3e=qe$. Its exponent is 
\[ (3,2,1). 
\] 
Next, we have two representations  where $T_0$ and $T_3$ act by different eigenvalues.  These two representations occur in 
a restriction of a 2-dimensional representation of the extended affine Hecke algebra $\tilde H$. Their exponents are the same, 
\[ 
(2+ \frac{\pi i}{ \ln q}, 1+ \frac{\pi i}{ \ln q}, \frac{\pi i}{ \ln q}). 
\] 
 Let $V'_1$ be the representation of $H$ on $V$ such that $T_0 e=-e$ and $T_3 e=-e$.  The exponent of $V'_1$ is  
\[ 
(1,0,-1). 
\] 
Let  $V''_1$ be the IM-involute of $V'_1$. 
 It is a tempered representation.   
 
\vskip 5pt

\subsubsection{Two dimensional representations}  Let $V$ be a two dimensional complex vector space 
spanned by $e_0$ and $e_1$ on which  $T_1$  and $T_2$ act as matrices 
\[ 
T_1=\left(\begin{array}{cc} 
-1 & q^{\frac{1}{2}} \\
0 & q 
\end{array}\right) 
\text{ and } 
T_2=\left(\begin{array}{cc} 
 q & 0 \\
q^{\frac{1}{2}} & -1
\end{array}\right), 
\] 
respectively. We can extend this representation to $H$ in four ways.  Let $V_2'$ be the representation of $H$ on $V$  such that $T_0$ and $T_3$ act as the scalar $q$. 
The exponents of $V_2'$ are 
\[ 
(2,0,1) \text{ and } (0,2,1). 
\] 
Let $V''_2$ be the representation of $H$ on $V$ such that $T_0$ and $T_3$ act as the scalar $-1$. 
Then $V''_2$ is the IM-involute of $V'_2$. It is a tempered representation.  Finally, we have two additional representations, one where $T_0$ and 
$T_3$ act by different scalars.  These two representations occur in 
a restriction of a 4-dimensional representation of the extended affine Hecke algebra $\tilde H$. Their exponents are the same and given by:
\[ 
(1 +\frac{\pi i}{ \ln q}, -1+\frac{\pi i}{ \ln q}, \frac{\pi i}{ \ln q}) \text{ and } (-1 +\frac{\pi i}{ \ln q}, 1+\frac{\pi i}{ \ln q}, \frac{\pi i}{ \ln q}). 
\]  
The sum of these two representation is an irreducible representation of $\tilde H$, the extended affine Hecke algebra. 

\vskip 10pt

\subsection{\bf Degenerate principal series, when $K$ is a field} 

\subsubsection{\bf $B_2$ parabolic} 

Let $\lambda_s=(s,2,1)$. We have a degenerate principal series $B(s)$ (associated to the $B_2$-parabolic) whose exponents are 
\[ 
 (s,2,1), \, (2,s,1), \, (2,1,s), \, (2,1,-s), \, (2,-s, 1) \, (-s,2,1). 
\] 
Here $\lambda_s$ is a leading exponent and $\lambda_{-s}$ is the trailing exponent. In particular, the trivial representation is the
unique irreducible quotient of $B(3)$.

  \begin{prop}  \label{T:degen_KB}
The representation $B(s) = B(-s)^*$ is reducible only if 
\[
\pm s= 3, 1+\frac{\pi i}{\ln q}, 0,  \text{ and } \frac{\pi i}{\ln q}
\] 
where $\pm s=1+\frac{\pi i}{\ln q}$ occurs if $K$ is unramified and $\pm s= \frac{\pi i}{\ln q}$ if $K$ is ramified. At the points of reducibility, we have 
\begin{enumerate} 
\item $B(3)$ has length 2. The trivial representation is the unique irreducible quotient. 
\item $B(1+\frac{\pi i}{\ln q} )$ has length 2. The minimal representation is the unique irreducible quotient. 
\item $B(0)$ is a direct sum of two non-isomorphic representations where one is $V_2'$. 
\item $B(\frac{\pi i}{\ln q})$  is a direct sum of two non-isomorphic representations.  
\end{enumerate} 
\end{prop} 
\begin{proof} This can be proved as in \cite{WeRT}. Roughly speaking, off the unitary axis, i.e. $\Re(s) \neq 0$, 
 reducibility happens only if the trivial or the 
minimal representations appear as subquotients. The case $s=1+\frac{\pi i}{\ln q}$ merits a special discussion, as it illustrates a difference between ramified and 
unramified cases. In both cases, $B(1+\frac{\pi i}{\ln q} )$ is regular;  however, the number of equivalence classes is one, if $K$ is unramified, and 2 otherwise. 
 This is due to the fact 
that $\mu=(2,1, 1+\frac{\pi i}{\ln q} )$ is equivalent to $s_3(\mu)= (2,1, -1-\frac{\pi i}{\ln q} )$ if  and only if $K$ is ramified. 

On the unitary axis, all exponents are equivalent and $B(s)$ is irreducible, unless 
$s=0$ or $s=\frac{\pi i}{\ln q}$ and $K$ ramified. By the Frobenius reciprocity,  $V_2$ is a summand of $B(0)$, so (3) follows. 
Finally, $B(\frac{\pi i}{\ln q})$ must reduce, otherwise $B(s)$ with $\Re(s)= \frac{\pi i}{\ln q}$ would be all unitary, a contradiction. 
\end{proof} 

\subsubsection{\bf $A_2$ parabolic} 

Let $\lambda_s=(s+1,s,s-1)$. We have a degenerate principal series $A(s)$ (associated to the $A_2$-parabolic) whose exponents are 
\[ 
 (s+1,s,s-1), \, (s+1,s,-s+1), \, (s+1,-s+1,s), \, (-s+1,s+1,s),
 \]
 \[ 
 (s+1,-s+1,-s),  \, (-s+1,s+1,-s), \, (-s+1,-s, s+1),\, (-s+1,-s, -s-1). 
\] 
Here $\lambda_s$ is a leading exponent and $\lambda_{-s}$ is the trailing exponent. In particular, the trivial representation is the
unique quotient of $A(2)$. Note that $\lambda_s$ is congruent to $\lambda_{s+\frac{\pi i}{\ln q}}$ if $K$ is unramified. 

\begin{prop}  \label{P:degen_KA}
The degenerate principal series representation $A(s)$ (with $Re(s) \geq 0$) is irreducible except in the following cases: 

\begin{enumerate} 
\item $A(2)$ has length 2. The unique irreducible quotient is the trivial representation.  
\item $A(1)$ has length 2. The unique irreducible quotient is the orthogonal complement of $V_2'$ in $B(0)$. 
\item when $K$ is ramified, $A(1+\frac{\pi i}{\ln q} )$ has length 3.  It has two irreducible quotients, corresponding to two one-dimensional representations of $H$ with the 
exponent 
\[ 
(2+\frac{\pi i}{\ln q}, 1+\frac{\pi i}{\ln q} , \frac{\pi i}{\ln q}).
\]  
\item $A(0)$ is a direct sum of two non-isomorphic representations where one of them  is $V_1'$.  
\end{enumerate} 
\end{prop} 

\begin{proof}  (1) is trivial. For (2), observe that the spherical summand of $B(0)$ is a unique irreducible quotient of $A(1)$. The remaining sub quotients of $A(1)$ have four exponents. As these exponents are equivalent, the length of $A(1)$ is 2, as claimed. The statement (3)  is proved similarly. For (4), observe that $A(0)$ is semi-simple, and has at most two summands, since any summand contributes the exponent $(1,0,-1)$. Since $V_1'$ is a summand of $A(0)$ by the Frobenius reciprocity, we have two summands as claimed.
\end{proof} 
\vskip 5pt

Note that the complement of $V_1'$ in $A(0)$ is spherical, and has seven exponents. We shall use this fact shortly. 
\vskip 5pt

\subsubsection{\bf $A_1 \times A_1$ parabolic} 

Let $\lambda_s=(s+\frac12,s-\frac12, 1)$. We have a degenerate principal series $I(s)$ (associated to the $A_1 \times A_1$-parabolic, which is the Heisenberg parabolic), whose exponents are 
\[ 
 (s+\frac12,s-\frac12, 1), \, (s+\frac12,1,s-\frac12), \, (1,s+\frac12,s-\frac12),
 \]
 \[ 
(s+\frac12,1,-s+\frac12) ,  \, (s+\frac12,-s+\frac12,1), \, (-s+\frac12,s+\frac12,1),  
\] 
\[ 
(1,s+\frac12,-s+\frac12), \, (1,-s+\frac12,s+\frac12), \, (-s+\frac12,1,s+\frac12), 
\]
\[
(1,-s+\frac12,-s-\frac12), \, (-s+\frac12,1,-s-\frac12), \, (-s+\frac12,-s-\frac12,1). 
\] 

Here $\lambda_s$ is a leading exponent and $\lambda_{-s}$ is the trailing exponent. In particular, the trivial representation is the
unique quotient of $I(5/2)$. Points of reducibility of $I(s)$ and its co-socle if $Re(s) \geq 0$ was determined by Segal, Theorem 4.1 in \cite{Se2}. Here we determine the complete 
composition series.

\begin{prop} \label{P:degen_KA'}
The points of reducibility of $I(s)$ (with $Re(s) \geq 0$) are given as follows: 
\begin{enumerate} 
\item $I(5/2)$ has length 2. The unique irreducible quotient is the trivial representation. 
\item $I(3/2)$ has length 2. The unique irreducible quotient is $B(1)$. 
\item $I(3/2+\frac{\pi i}{\ln q} )$ has length 2 when $K$ is unramfied, with the minimal representation as its  unique irreducible quotient. 
\item $I(1/2)$ has length 2. The unique irreducible quotient is the orthogonal complement of $V_1'$ in $A(0)$. 
\item $I(1/2+\frac{\pi i}{\ln q} )$ has length 2 when $K$ is unramfied,  and 3 with two irreducible quotients if $K$ is ramified.  
\end{enumerate} 
\end{prop} 
\begin{proof} (1) is trivial. For (2),  we observe that 
$B(1)$ is the unique irreducible quotient of $I(3/2)$. Since the remaining six exponents are equivalent, $I(3/2)$ has length 2.  
The case (3) is regular, so the irreducible subquotients are easily determined by working out the equivalence classes of exponents. For (4), 
the spherical summand of $A(0)$ is the unique quotient of $I(1/2)$. The remaining subquotients of $I(1/2)$ have five exponents in total. Hence, if there are more than two irreducible 
subquotients in $I(1/2)$, there would be one with one or two exponents. But, 
by inspection, these five exponents are not among the exponents of one and two-dimensional $H$-modules. Hence, $I(1/2)$ has length 2, as asserted in (4). For the last case, by the result of A. Segal, the 
representation has one, respectively two irreducible quotients. By working out equivalence classes of exponents, it is seen that there are no more irreducible subquotients than as stated. 
\end{proof} 

\vskip 10pt

\subsection{Split $D_4$} 
Assume now that $E = F^3$ is split, so that $G_E$ is the split $\Spin_8$.
Let $A=\mathbb R^4$ and we identify $A^*$ with $A$ using the usual dot product. Let $\Phi\subset A^*$ be the root system of type $D_4$, so that the simple roots are 
\[ 
\alpha_1=(1,-1,0,0), \quad \alpha_2=(0,1,-1,0), \quad \alpha_3=(0,0,1,-1), \quad \alpha_4=(0,0,1,1). 
\] 
Let $W$ be the corresponding Weyl group. For every $k\in \mathbb Z$ and $\alpha\in \Phi$, we have an affine root $\alpha+k$. Let  $W_a$ be the 
corresponding affine Weyl group. It is a semi-direct direct product of $W$ and  $X=\{ (x,y,z,w) \in \mathbb Z^4 ~|~ x+y+z+w \equiv 0\pmod{2}\}$. 
\vskip 5pt

In this case, degenerate principal series representations have been well studied, and there are references in the literature, such as \cite{BJ} and \cite{WeRT}. So we shall be brief and put an emphasis on explaining, rather than giving the details. 
\vskip 5pt

Let $T_i$, $i=0, \ldots , 4$ be the standard generators of the affine Hecke algebra $H$, such that $T_2$ corresponds to the branching point of the extended Dynkin diagram. The algebra $H$ has a 2-dimensional irreducible representation $V_2$ such that 
\[ 
T_0=T_1=T_3=T_4=\left(\begin{array}{cc} 
-1 & q^{\frac{1}{2}} \\
0 & q 
\end{array}\right) 
\text{ and } 
T_2=\left(\begin{array}{cc} 
 q & 0 \\
q^{\frac{1}{2}} & -1
\end{array}\right). 
\] 
The exponents of this representations are 
\[ 
(0,1,-1,0) \text{ and } (0,-1,1,0). 
\]  
The minimal representation corresponds to the reflection representation of $H$ and its exponents are (the superscript 2 means that the exponent appears with 
multiplicity 2) 
\[ 
(2,1,1,0)^2, (1,2,1,0), (2,1,0,1), (2,1,0,-1). 
\] 

There are 3 maximal parabolic subgroups in standard position, of the type $A_3$, permuted by the group of outer automorphisms. 
 Let $A(s)$, $B(s)$ and $C(s)$ be the degenerate principal series, corresponding to these parabolic subgroups, normalized so that 
the trivial representation occurs as the unique irreducible quotient for $s=3$. For example,  assuming that $A(s)$ corresponds to the maximal parabolic whose Levi does not have 
$\alpha_1$ as a root,  the leading exponent of $A(s)$ is $(s,2,1,0)$.  There are eight exponents:
\[ 
(s,2,1,0), (2,s,1,0), (2,1,s,0), (2,1,0,s), 
\] 
\[ 
(2,1,0,-s), (2,1,-s,0),(2,-s,1,0),(-s,2,1,0). 
\] 
By a result of Weissman \cite{WeRT},  
$A(1)$, $B(1)$ and $C(1)$ have length 2, and the minimal representation is the unique irreducible quotient. 
Let  $V_3^A\subset A(1)$, $V_3^B\subset B(1)$ and $V_3^C\subset C(1)$ be the unique irreducible submodules.  
 These representations are non-isomorphic, as they have different exponents. 
\vskip 5pt

Let $I(s)$ be the principal series corresponding to the Heisenberg maximal parabolic (i.e. the Levi factor is $A_1^3$), normalized so that the trivial representation is the 
unique irreducible quotient for $s=5/2$. The leading exponent is $(s+\frac{1}{2}, s-\frac{1}{2}, 1,0)$. There are 24 exponents in all. They are in 4 groups of 6 exponents 
\[ 
(1,0,x,y), (1,x,0,y), (1,x,y,0),  (x,1,0,y), (x,1,y,0), (x,y,1,0) 
\] 
where 
\[ 
(x,y)=( s+\frac{1}{2}, s-\frac{1}{2}),\, ( s+\frac{1}{2}, -s+\frac{1}{2}),\, ( -s+\frac{1}{2}, s+\frac{1}{2}), \, ( -s+\frac{1}{2}, -s-\frac{1}{2}). 
\] 

The only other reducibility points are $s=\pm 1/2$ and $s=\pm 3/2$, which we examine in turn:
\vskip 5pt

\begin{itemize}
\item $s = 3/2$:  the minimal representation is the unique irreducible quotient of $I(3/2)$. Moreover, we have an intertwining map 
$I(3/2) \rightarrow A(1)$, obtained by composing standard intertwining operators, which are non-trivial on the spherical vector. Hence $A(1)$ (and analogously $B(1)$ and $C(1)$) is a quotient of $I(3/2)$. By removing these quotients, we are left with  
an irreducible submodule since its 10 exponents are equivalent.
\vskip 5pt

\item$s=1/2$:  By the Frobenius reciprocity, $V_2$ is the unique irreducible submodule of $I(1/2)$. The quotient is an irreducible spherical representation that appears as a summand of the representation induced from the trivial representation of (any) parabolic subgroup of the 
type $A_2$. 
\end{itemize}
Summarizing, we have:
\vskip 5pt

\begin{prop} (Theorems 5.3 and 5.5 in \cite{BJ}) \label{P:degen}
\begin{itemize} 
\item $I(3/2)$ has a filtration of length 3, consisting of a unique irreducible submodule and  a unique irreducible quotient (the minimal representation). The intermediate subquotient is isomorphic to $V_3^A\oplus V_3^B \oplus V_3^C$. 
\vskip 5pt

\item $I(1/2)$ has length 2, and $V_2$ is the unique irreducible submodule. 
\end{itemize}  
\end{prop}


\end{document}